\documentclass[10pt,reqno]{amsart}
\usepackage[margin=1in]{geometry}
\usepackage{amssymb,mathrsfs,color, mathtools}
\usepackage{textcomp}
\usepackage{enumerate}
\usepackage{graphicx}
\usepackage{xcolor} 
\usepackage{geometry}
\usepackage{amsfonts,amssymb,amsmath}
\usepackage{slashed}
\usepackage{ mathrsfs }
\usepackage[titletoc,title]{appendix}
\usepackage{wrapfig}

\usepackage{cancel}

\usepackage{cite}

\usepackage{nth}

\usepackage{marginnote}

\usepackage{pgfplots}

\usepackage{verbatim}

\usepackage{dsfont}

\definecolor{green}{rgb}{0,0.8,0} 

\definecolor{deepgreen}{cmyk}{1,0,1,0.5}

\newcommand{\Del}[1]{}

\numberwithin{equation}{section}

\newtheorem{theorem}{Theorem}[section]
\newtheorem{lemma}[theorem]{Lemma}
\newtheorem{proposition}[theorem]{Proposition}

\newcommand{\japanese}[1]{\langle #1\rangle}

\renewcommand{\Im}{\mathrm{Im}}

\renewcommand{\hbar}{{\underline h}}

\newcommand{\bbC}{\mathbb C}

\newcommand{\bbE}{\mathbb E}

\newcommand{\bbN}{\mathbb N}

\newcommand{\bbP}{\mathbb P}

\newcommand{\bbR}{\mathbb R}
\newcommand{\bbS}{\mathbb S}

\newcommand{\bbZ}{\mathbb Z}

\newcommand{\calA}{\mathcal A}

\newcommand{\calC}{\mathcal C}

\newcommand{\calE}{\mathcal E}
\newcommand{\calF}{\mathcal F}

\newcommand{\calH}{\mathcal H}
\newcommand{\calI}{\mathcal I}

\newcommand{\calM}{\mathcal M}
\newcommand{\calN}{\mathcal N}
\newcommand{\calO}{\mathcal O}
\newcommand{\calP}{\mathcal P}
\newcommand{\calQ}{\mathcal Q}
\newcommand{\calR}{\mathcal R}

\newcommand{\ud}{\mathrm{d}}

\newcommand{\pt}{\partial_t}

\newcommand{\hf}{\frac{1}{2}}

\newcommand{\bA}{\mathbf{A}}
\newcommand{\bM}{\mathbf{M}}

\newcommand{\ulk}{\underline{k}}

\usepackage{cancel}

\begin{document}

\title[Probabilistic small data GWP of the energy-critical Maxwell-Klein-Gordon equation]{Probabilistic small data global well-posedness \\ of the energy-critical Maxwell-Klein-Gordon equation} 

\author[J. Krieger]{Joachim Krieger}
\address{B\^atiment des Math\'ematiques \\ EPFL \\ Station 8 \\ 1015 Lausanne \\ Switzerland}
\email{joachim.krieger@epfl.ch}

\author[J. L\"uhrmann]{Jonas L\"uhrmann}
\address{Department of Mathematics \\ Texas A\&M University \\ College Station, TX 77843, USA}
\email{luhrmann@math.tamu.edu}

\author[G. Staffilani]{Gigliola Staffilani}
\address{Department of Mathematics \\ Massachusetts Institue of Technology \\ Cambridge, MA 02139, USA}
\email{gigliola@math.mit.edu}

\thanks{J. Krieger is supported in part by Consolidator Grant BSCGI0-157694 of the Swiss National Science Foundation. J. L\"uhrmann is supported in part by NSF grant DMS-1954707. G. Staffilani is supported in part by NSF grant DMS-1764403 and by the Simons Foundation}

\begin{abstract}
 We establish probabilistic small data global well-posedness of the energy-critical Maxwell-Klein-Gordon equation relative to the Coulomb gauge for scaling super-critical random initial data. The proof relies on an induction on frequency procedure and a modified linear-nonlinear decomposition furnished by a delicate ``probabilistic'' parametrix construction. This is the first global existence result for a geometric wave equation for random initial data at scaling super-critical regularity. 
\end{abstract}

\maketitle

\tableofcontents

\section{Introduction}

The study of the local and global well-posedness of nonlinear dispersive and hyperbolic PDEs for scaling super-criticial random initial data has seen an enormous amount of progress in recent years. 
The goal of our work is to consider the energy-critical Maxwell-Klein-Gordon equation, a prime example of a geometric semilinear wave equation, and to establish a probabilistic small data global well-posedness result for scaling super-critical random initial data. 

\medskip 

The energy-critical Maxwell-Klein-Gordon equation on $(1+4)$-dimensional Minkowski space $\bbR^{1+4}$ models the interaction of an electromagnetic field with a charged particle field. The electromagnetic field is described by a real-valued connection $1$-form $A_\alpha \colon \bbR^{1+4} \to \bbR$, $\alpha = 0, 1, \ldots, 4$, and the particle field in terms of a complex-valued scalar function $\phi \colon \bbR^{1+4} \to \bbC$. Upon introducing the curvature $2$-form 
\begin{equation*}
 F_{\alpha \beta} = \partial_\alpha A_\beta - \partial_\beta A_\alpha, \quad 0 \leq \alpha, \beta \leq 4,
\end{equation*} 
and the covariant derivatives
\begin{equation*}
 D_\alpha = \partial_\alpha + i A_\alpha, \quad 0 \leq \alpha \leq 4,
\end{equation*}
the Maxwell-Klein-Gordon system of equations reads 
\begin{equation} \label{equ:MKG}
 \left\{ \begin{aligned}
   \partial^\beta F_{\alpha \beta} &= \Im \bigl( \phi \overline{D_\alpha \phi} \bigr), \\
   D^\alpha D_\alpha \phi &= 0.
 \end{aligned} \right. \tag{MKG}
\end{equation}
Here we use the standard conventions of raising and lowering indices with respect to the Minkowski metric~$\text{diag}[1,-1,\ldots, -1]$, and of summing over repeated upper and lower indices. 

The system of equations \eqref{equ:MKG} is invariant under the scaling transformation
\begin{equation*}
 A_{\alpha}(t,x) \rightarrow \lambda A_{\alpha}(\lambda t, \lambda x), \quad \phi(t,x) \rightarrow \lambda \phi(\lambda t, \lambda x) \quad \text{for } \lambda > 0.
\end{equation*}
It also admits the conserved energy functional
\begin{equation*} 
 E(A, \phi) := \int_{\bbR^4} \Big( \frac{1}{4}\sum_{\alpha, \beta} F_{\alpha\beta}^2 + \frac{1}{2} \sum_{\alpha} \big|D_{\alpha}\phi\big|^2 \Big) \, dx,
\end{equation*}
which is invariant under the above scaling. For this reason the Maxwell-Klein-Gordon system of equations is referred to as energy-critical in four space dimensions.

Another important feature of the (MKG) system of equations is its gauge invariance. If $(A_\alpha, \phi)$ is a solution to (MKG), then so is $(A_\alpha - \partial_\alpha \gamma, e^{i \gamma} \phi)$ for any suitably regular function $\gamma \colon \bbR^{1+4} \to \bbR$. This yields an equivalence relation on the set of pairs $(A_\alpha, \phi)$ satisfying (MKG). In order to uniquely determine the solutions to (MKG), we therefore have to add an additional set of constraints to fix the ambiguity arising from this gauge invariance.

Imposing the global Coulomb gauge condition
\begin{equation*}
 \partial^j A_j = 0,
\end{equation*}
(MKG) becomes a system of wave equations in the dynamical variables $\phi$ and $A_j$, $j = 1, \ldots, 4$, coupled to an elliptic equation in the temporal component $A_0$, given by
\begin{equation} \label{equ:MKG-CG}
 \left\{ \begin{aligned}
          \Box A_j &= - \calP_j \Im \bigl( \phi \overline{D_x \phi} \bigr), \\
          D^\alpha D_\alpha \phi &= 0, \\
          \Delta A_0 &= - \Im \bigl( \phi \overline{D_0 \phi} \bigr),
         \end{aligned} \right. \tag{MKG-CG}
\end{equation}
where $\calP_j v =  v_j - \partial_j \Delta^{-1} \partial^\ell v_\ell$ is the Leray projection to divergence-free vector fields. In the formulation (MKG-CG), at any fixed time the temporal component $A_0$ is uniquely determined in terms of the dynamical variables $(A_x, \phi)$ by the elliptic equation. It therefore suffices to prescribe 
\[
 A_x[0] := (A_x, \partial_t A_x)(0) = (a, b), \quad \phi[0] := (\phi, \partial_t \phi)(0) = (\phi_0, \phi_1)
\]
as initial data for (MKG-CG) with $a$ and $b$ obeying the Coulomb gauge condition 
\begin{equation*}
 \partial^j a_j = \partial^j b_j = 0.
\end{equation*}

\medskip

Relative to the Coulomb gauge, the nonlinearities in the wave equations for the dynamical variables in (MKG-CG) have a favorable algebraic structure, the so-called null structure, which damps the worst interactions. Schematically, the system of equations (MKG-CG) is of the form 
\begin{equation*} 
 \left\{ \begin{aligned}
          \Box A_j &= - \calP_j \Im \bigl( \phi \overline{\nabla_x \phi} \bigr) + \text{``cubic terms''}, \\
          \Box \phi &= - 2i A^j \partial_j \phi - 2i A_0 \partial_t \phi + \text{``cubic terms''}, \\
          \Delta A_0 &= - \Im \bigl( \phi \overline{\partial_t \phi} \bigr) + \text{``cubic terms''},
         \end{aligned} \right. 
\end{equation*}
where the quadratic terms in the wave equations for $A_x$ and $\phi$ exhibit the null structures
\begin{align*}
 \calP_j \bigl( \phi \overline{\nabla_x \phi} \bigr) &= \partial^k \Delta^{-1} \calN_{kj}\bigl( \phi, \overline{\phi} \bigr), \\
 A^j \partial_j \phi &= \calN_{kj} \bigl( \partial^k \Delta^{-1} A^j, \phi \bigr),
\end{align*}
with $\calN_{ij}(\phi, \psi) = (\partial_i \phi) (\partial_j \psi) - (\partial_j \phi) (\partial_i \psi)$. The discovery of the presence of null structure in the nonlinearities of (MKG-CG) is due to Klainerman-Machedon~\cite{KlMa_MKG} and marked the beginning of the study of low-regularity well-posedness of solutions to the Maxwell-Klein-Gordon system of equations, which we briefly review now. Finite energy global well-posedness of (MKG-CG) in energy sub-critical $d=3$ space dimensions was established by Klainerman-Machedon~\cite{KlMa_MKG}. 
Through a deep structural analysis of the (MKG-CG) equations, Machedon-Sterbenz~\cite{MaSte} obtained an almost optimal local existence result for (MKG-CG) for $d=3$. An analogous almost optimal local existence result was obtained by Selberg~\cite{Selberg} in $d=4$ space dimensions.  For small critical Sobolev data, Rodnianski-Tao~\cite{RT} proved global existence for (MKG-CG) in $d \geq 6$ space dimensions. Their approach was further advanced in joint work of the first author with Sterbenz and Tataru~\cite{KST} to show global existence for small energy data for the energy-critical Maxwell-Klein-Gordon equation in $d = 4$ space dimensions. More recently, global existence and scattering for arbitrary finite energy data was established for the energy-critical Maxwell-Klein-Gordon equation independently by Oh-Tataru~\cite{OT1, OT2, OT3} and by the first two authors~\cite{KL15}.

\subsection{Randomization procedure}

In this work we consider the Cauchy problem for (MKG-CG) in four space dimensions for random initial data at scaling super-critical regularity, i.e. below the energy regularity.
Before stating our main theorem and putting it into perspective with prior random data results in the next subsection, we first describe our randomization procedure for the initial data. It relies on a unit-scale decomposition of frequency space and was introduced in~\cite{ZF12, LM14}. This procedure was subsequently coined ``Wiener randomization'' in \cite{BOP1, BOP2} to emphasize its natural association with the Wiener decomposition~\cite{Wiener32} and the modulation spaces introduced by H. Feichtinger~\cite{Feichtinger}. 

We pick an even, non-negative bump function $\varphi \in C_c^{\infty}(\bbR^4)$ with $\text{supp} (\varphi) \subseteq B(0,1)$ and such that 
\begin{equation*}
 \sum_{m \in \bbZ^4} \varphi(\xi - m) = 1 \quad \text{for all } \xi \in \bbR^4.
\end{equation*}
Then we let $\{g_m\}_{m \in \bbZ^4}$, $\{ \tilde{g}_m \}_{m \in \bbZ^4}$, $\{h_m\}_{m \in \bbZ^4}$, and $\{ \tilde{h}_m \}_{m \in \bbZ^4}$ be sequences of complex-valued standard (zero-mean) Gaussian random variables  on a probability space $(\Omega, \calF, \bbP)$. 
We assume the symmetry conditions $g_{-m} = \overline {g_m}$ and $\tilde{g}_{-m} = \overline {\tilde{g}_m}$ for all $m \in \bbZ^4$. Moreover, we suppose that $\{g_0, \textup{Re}(g_m), \textup{Im}(g_m)\}_{m \in \calI}$ are independent, zero-mean, real-valued random variables, where $\calI \subset \bbZ^4$ is such that we have a \textit{disjoint} union $\bbZ^4 = \calI \cup (-\calI) \cup \{0\}$, and similarly for the $\tilde{g}_m$. The Gaussians $\{h_m\}_{m \in \bbZ^4}$, and $\{ \tilde{h}_m \}_{m \in \bbZ^4}$ are just assumed to be independent random variables without any additional constraints.
We remark that we could more generally work with sequences of independent uniformly sub-Gaussian random variables with zero mean.

Let $0 < \delta_\ast \ll 1$ be some small absolute constant whose size will be specified later on. For any regularity exponent $1-\delta_\ast < s < 1$, we consider a pair of real-valued $1$-forms obeying the Coulomb gauge condition
\begin{equation*}
 A_x[0] = (a, b) \in H^s_x(\bbR^4) \times H^{s-1}_x(\bbR^4), \quad \partial^j a_j = \partial^j b_j = 0,
\end{equation*}
and a pair of complex-valued functions
\begin{equation*}
 \phi[0] = (\phi_0, \phi_1) \in H^s_x \times H^{s-1}_x(\bbR^4).
\end{equation*} 
Then we define the randomization of $(a,b)$ and $(\phi_0,\phi_1)$ by
\begin{equation} \label{equ:def_random_data}
 \begin{aligned}
  A_x^\omega[0] = (a^\omega, b^\omega) &:= \biggl( \sum_{m \in \bbZ^4} g_m(\omega) \varphi(D-m) a, \sum_{m \in \bbZ^4} \tilde{g}_m(\omega) \varphi(D-m) b \biggr), \\
  \phi^\omega[0] = (\phi_0^\omega, \phi_1^\omega) &:= \biggl( \sum_{m \in \bbZ^4} h_m(\omega) \varphi(D-m) \phi_0, \sum_{m \in \bbZ^4} \tilde{h}_m(\omega) \varphi(D-m) \phi_1 \biggr).
 \end{aligned}
\end{equation}
These quantities are to be understood as Cauchy limits in $L^2_\omega\bigl(\Omega; H^s_x(\bbR^4) \times H^{s-1}_x(\bbR^4)\bigr)$. The randomization almost surely does not regularize at the level of Sobolev spaces, see for instance \cite[Lemma B.1]{BT08_1}. 
It is crucial that the symmetry assumptions on the random variables $\{g_m\}_{m \in \bbZ}$, $\{\tilde{g}_m\}_{m \in \bbZ}$ together with the assumption that the bump function $\varphi$ is even, ensure that the randomization of the pair of real-valued $1$-forms $(a,b)$ is again real-valued and in Coulomb gauge.

We will frequently use the following truncation operators $T_n$ defined for all integers $n \geq 1$ by
\begin{align*}
 T_n \phi_0^\omega := \sum_{ \substack{ m \in \bbZ^4 \\ 2^{n-1} \leq |m| < 2^n} } g_m(\omega) \varphi(D-m) \phi_0,
\end{align*}
with analogous definitions for $T_n \phi_1^\omega$, $T_n a^\omega$, and $T_n b^\omega$, where we denote by $|m| = (m_1^2 + \ldots + m_4^2)^{\frac{1}{2}}$ the Euclidean norm of a vector $m = (m_1, \ldots, m_4) \in \bbZ^4$. In the same manner we introduce the truncation operators $T_{<n}$ for all integers $n \geq 1$ by
\begin{equation*}
 T_{<n} \phi_0^\omega:= \sum_{ \substack{ m \in \bbZ^4 \\ |m| < 2^n} } g_m(\omega) \varphi(D-m) \phi_0,
\end{equation*}
with corresponding definitions for $T_{<n} \phi_1^\omega$, $T_{<n} a^\omega$, and $T_{<n} b^\omega$. 
Moreover, we set 
\begin{align*}
 T_0 \phi_0^\omega := g_0(\omega) \varphi(D) \phi_0
\end{align*}
with analogous definitions for $T_0 \phi_1^\omega$, $T_0 a^\omega$, and $T_0 b^\omega$.
Finally, for every integer $n \geq 0$ we denote by 
\begin{equation*}
 \calF_n := \sigma \Bigl( g_m, \tilde{g}_m, h_m, \tilde{h}_m \, : \, |m| < 2^n \Bigr)
\end{equation*}
the $\sigma$-algebra generated by the Gaussians $g_m$, $\tilde{g}_m$, $h_m$, $\tilde{h}_m$ with $|m| < 2^n$.

\subsection{Main result}

In recent years there has been enormous progress in the development of a combination of probabilistic and deterministic techniques to prove the existence of strong local-in-time or even global-in-time solutions to nonlinear wave and Schr\"odinger equations almost surely (or with high probability) for random initial data of super-critical regularity. This approach was initiated in the pioneering work of Bourgain~\cite{B94_1, B96} for the periodic nonlinear Schrödinger equation in dimensions one and two, building upon the constructions of invariant measures in~\cite{Glimm_Jaffe} and~\cite{LRS}. Subsequently, the influential papers of Burq-Tzvetkov~\cite{BT08_1, BT08_2}, see also Oh~\cite{Oh09}, led to a burst of activity in this line of research by introducing a more general randomization method in the context of establishing almost sure local and global well-posedness results at super-critical regularities for nonlinear wave equations posed on compact Riemannian manifolds. 
We refer to a sample of recent random data results, primarily for nonlinear wave equations~\cite{BT14, BuKr19, LM14, LM16, DLM17, OhPoc16, Poc17, Bringmann18_1, Bringmann18_2, Bringmann18_3, CCMNS17, DNY19, DNY20, GKO18, KenigMend19} that are most closely related to this work. This list is by no means exhaustive and we also refer to the recent surveys \cite{BOP19, NS19} and references therein. We point out that the large majority of random data results so far is for equations with pure power-type nonlinearities.

Oversimplifying a bit here, in order to deal with the Cauchy problem for a nonlinear wave equation with super-critical random initial data, one typically decomposes the solution into the free wave evolution of the random data and into an inhomogeneous component satisfying a nonlinear wave equation with forcing terms. Using the randomization one then shows that almost surely (or with high probability) the free wave evolution of the rough random data enjoys improved (``redeeming'') space-time integrability properties that beat the scaling, and tend to allow one to solve the equation for the inhomogeneous component at a critical or sub-critical regularity. 
This type of linear-nonlinear decomposition can be attributed to the work of Bourgain~\cite{B96} in the field of dispersive PDEs, and is referred to as the Da
Prato-Debussche trick~\cite{PratoDeb02} in the field of stochastic parabolic PDEs. 
 
In the context of the energy-critical Maxwell-Klein-Gordon system of equations, a semilinear geometric wave equation with derivative nonlinearities, this standard linear-nonlinear decomposition is bound to fail due to certain low-high interactions in the equation for the scalar field $\phi$ that do not exhibit a smoothing effect when a rough input is at high frequency. 
Such a difficulty has already been observed by Bringmann~\cite{Bringmann18_2} in the context of a quadratic derivative nonlinear wave equation in three space dimensions and was overcome by building the corresponding problematic low-high interactions into the definition of the rough linear evolution of the random data. This step crucially relies on the fact that the high-frequency and the low-frequency parts are independent. 
Similar ideas for dealing with problematic low-high frequency interactions with the rough linear evolution at high frequency play a major role in the development of the theory of random averaging operators and of the theory of random tensors in the recent works of Deng-Nahmod-Yue~\cite{DNY19, DNY20}, too.
We also note that the treatment of related delicate low-high interactions are a key feature of the theory of paracontrolled calculus developed by Gubinelli-Imkeller-Perkowski~\cite{GIP15} to prove local well-posedness for singular parabolic stochastic PDEs, see also the theory of regularity structures put forth by Hairer~\cite{Hairer14}, the work of Kupiainen~\cite{Kupi16} using renormalization group techniques, and the approach of Otto-Weber~\cite{OttoWeber16}.

Already in the deterministic study of the Maxwell-Klein-Gordon equation at scaling-critical regularity, certain low-high interactions in the magnetic interaction term in the equation for $\phi$ are non-perturbative at scaling-critical regularity. A key idea of Rodnianski-Tao~\cite{RT} to overcome this issue was to incorporate these low-high interactions into the linear magnetic wave operator of the $\phi$ equation and to construct a corresponding parametrix to solve that linear magnetic wave equation. 

The probabilistic small data global well-posedness problem for the energy-critical Maxwell-Klein-Gordon equation relative to the Coulomb gauge for scaling super-critical random data features all of the obstacles described above. Our proof builds on the deterministic small data global existence results for the Maxwell-Klein-Gordon equation~\cite{RT, KST} at scaling-critical regularity, 
on the first two authors' induction on frequency procedure for the finite energy global regularity result for (MKG-CG)~\cite{KL15} (see also \cite{KS}),
and on the recent progress on almost sure well-posedness~\cite{Bringmann18_2, DNY19, DNY20}.

\medskip 

We are now in a position to present our main result.
The spaces $S^1$ and $Y^1$ in the following statement are at energy regularity and their precise definitions are provided in Section~\ref{sec:function_spaces}.

\begin{theorem} \label{thm:main}
 There exist small absolute constants $0 < \delta_\ast \ll 1$ and $0 < \varepsilon \ll 1$ with the following properties: For any $1-\delta_\ast < s < 1$, let $(a,b) \in H^s_x \times H^{s-1}_x$ be a pair of real-valued $1$-forms in Coulomb gauge, and let $(\phi_0, \phi_1) \in H^s_x \times H^{s-1}_x$. Denote by $(a^\omega, b^\omega)$ and by $(\phi_0^\omega, \phi_1^\omega)$ the associated random initial data as defined in~\eqref{equ:def_random_data}. Then there exists an event $\Sigma \subset \Omega$ with 
 \begin{equation*}
  \bbP(\Sigma^c) \lesssim \exp \Bigl( - c \frac{\varepsilon^2}{D^2} \Bigr), \quad D := \|(a,b)\|_{H^s_x \times H^{s-1}_x} + \| (\phi_0, \phi_1) \|_{H^s_x \times H^{s-1}_x}
 \end{equation*}
 such that for any $\omega \in \Sigma$, there exists a unique global solution 
 \begin{equation*}
  (A_x, A_0, \phi) \in \bigl( C^0_t H^s_x + S^1 \bigr) \times Y^1 \times \bigl( C^0_t H^s_x + S^1 \bigr)
 \end{equation*}
 to (MKG-CG) with initial data given by $A_x[0] = (a^\omega, b^\omega)$ and $\phi[0] = (\phi_0^\omega, \phi_1^\omega)$. 
 For every $\omega \in \Sigma$, $(A_x, A_0, \phi)$ is defined as the unique limit in $( C^0_t H^s_x + S^1 ) \times Y^1 \times ( C^0_t H^s_x + S^1 )$ of the sequence of canonical smooth approximations $\{ (A^{<n}_x, A^{<n}_0, \phi^{<n}) \}_{n \geq 0}$ to (MKG-CG) for frequency truncated random data $\{ (T_{<n} A_x^\omega[0], T_{<n} \phi^\omega[0] ) \}_{n \geq 0}$, and $(A_x, A_0, \phi)$ solves (MKG-CG) in the distributional sense.
\end{theorem}

\subsection{Overview of proof ideas}

We give an outline of the main aspects of the proof of Theorem~\ref{thm:main}.

\medskip 

\noindent {\it Small energy global regularity for the energy-critical Maxwell-Klein-Gordon equation~\cite{KST}.} 

\noindent The proof of Theorem~\ref{thm:main} relies on the functional framework, the multilinear estimates, and a parametrix construction from the small energy global regularity result for (MKG-CG) established in joint work of the first author with Sterbenz and Tataru~\cite{KST}. The key difficulty in the treatment of the Maxwell-Klein-Gordon equation relative to the Coulomb gauge at scaling-critical regularity are low-high interactions in the magnetic interaction term in the equation for the scalar field $\phi$ of the following schematic form, where the free wave evolution of the spatial part of the connection form is at low frequency,
\begin{equation*}
 \Box \phi = A^{free, j}_{low} \partial_j \phi_{high} + \ldots 
\end{equation*}
Even for small initial data, these low-high frequency interactions in the magnetic interaction term turn out to be non-perturbative at scaling-critical regularity due to a logarithmic divergence (the inhomogeneous part of~$A_x$ turns out to satisfy an improved $\ell^1$ bound and its contribution can therefore be treated perturbatively). Rodnianski and Tao~\cite{RT} resolved this impasse in the context of proving critical small data global regularity for (MKG-CG) in dimensions $d \geq 6$ by incorporating the problematic low-high interactions into the linear wave operator and by deriving Strichartz estimates via a parametrix construction for the resulting paradifferential magnetic wave operator
\begin{equation*}
 \Box_A^p \equiv \Box + 2i \sum_{k \in \bbZ} P_{\leq k-C} A^{free,j} \partial_j P_k.
\end{equation*}
This approach was significantly further advanced in~\cite{KST} through the realization that the parametrix construction from~\cite{RT} is also compatible with the more delicate $X^{s,b}$ type and null frame spaces. In this work we have to slightly adapt the parametrix construction from~\cite{KST} to allow for a rough free wave evolution $A^{free}$, which is at scaling super-critical regularity, but enjoys redeeming space-time integrability properties thanks to the randomization, see Section~\ref{sec:renormalization} for the details. We also take the parametrix construction from~\cite{KST} into a novel modified ``probabilistic'' direction as outlined in the next paragraphs.

\medskip 

\noindent {\it Failure of the Bourgain-da Prato-Debussche linear-nonlinear decomposition and induction on frequency.}

\noindent If one tries to treat the Cauchy problem for the energy-critical Maxwell-Klein-Gordon equation with (small) random initial data at scaling super-critical regularity, the usual approach of decomposing the dynamical variables $A_x$ and $\phi$ into free wave evolutions of the rough random data and into inhomogeneous nonlinear components is bound to partially fail. Owing to the favorable null structure in the wave equation for $A_x$, it suffices to just decompose the spatial part $A_x$ of the connection form according to the standard Bourgain-da Prato-Debussche trick. However, in the low-high frequency interactions in the magnetic interaction term $A^j_{low} \partial_j \phi^{free}_{high}$ of the $\phi$ equation, when the rough free wave evolution of the random data for $\phi$ is at high frequency, despite the null structure one cannot gain regularity and treat the term at the scaling-critical energy regularity. The way out is to build this low-high interaction term into the definition of the rough linear evolution of the random data for the scalar field $\phi$. This in turn requires to construct the solutions to (MKG-CG) for (small) random initial data $( A_x^\omega[0], \phi^\omega[0])$ via an induction on frequency procedure. More specifically, (on a suitable event) we construct the solutions as the limit of the sequence of solutions $(A^{<n}_x, A^{<n}_0, \phi^{<n})$ to (MKG-CG) for frequency truncated random initial data $( T_{<n} A_x^\omega[0], T_{<n} \phi^\omega[0] )$. To this end we derive uniform bounds on the dyadic solution increments $(\calA^n_x, \calA^n_0, \Phi^n)$, $n \geq 0$, defined by
\begin{align*}
 A^{<n}_x &= A^{<n-1}_x + \calA^n_x, \quad n \geq 0, \\
 A^{<n}_0 &= A^{<n-1}_0 + \calA^n_0, \quad n \geq 0, \\
 \phi^{<n} &= \phi^{<n-1} + \Phi^n, \quad n \geq 0,
\end{align*}
where we set $A^{<-1}_x = A^{<-1}_0 = \phi^{<-1} = 0$.
For every $n \geq 1$,  we decompose the increments $\calA^n_x$ and $\Phi^n$ of the dynamical variables into
\begin{align*}
 \calA^n_x &= \calA^n_{x,r} + \calA^n_{x,s}, \\
 \Phi^n &= \Phi^n_r + \Phi^n_s,
\end{align*}
where $\calA^n_{x,r}$, $\Phi^n_r$ are the rough linear components and $\calA^n_{x,s}$, $\Phi^n_s$ are the (``smooth'') inhomogeneous components satisfying a forced Maxwell-Klein-Gordon system of equations (fMKG-CG\textsubscript{n}) stated precisely in Subsection~\ref{subsec:decomp_lin_nonlin}.
As alluded to above, it suffices to define the rough part $\calA^n_{x,r}$ as the free wave evolution of the (rough) random initial data $T_n A_x^\omega[0]$, that is 
\begin{equation*}
 \calA^n_{x,r}(t) := S(t)\bigl[ T_n a^\omega, T_n b^\omega \bigr] = \cos(t|\nabla|) T_n a^\omega + \frac{\sin(t|\nabla|)}{|\nabla|} T_n b^\omega, \quad n \geq 1.
\end{equation*}
Instead, the rough part $\Phi^n_r$ of the scalar field is defined as an \emph{approximate solution} to the linear magnetic wave equation
\begin{equation*}
 \Box_{A^{<n-1}}^{p, mod} \Phi^n_r \equiv \Bigl( \Box + 2i P_{\leq (1-\gamma)n} A^{<n-1, \alpha} \partial_\alpha P_n \Bigr) \Phi^n_r \approx 0, \quad \Phi^n_r[0] \approx T_n \phi^\omega[0], \quad n \geq 1.
\end{equation*}
Here, it is crucial that the entire connection form $A^{<n-1}$ from the prior induction stages is built into the linear magnetic wave operator on the left-hand side. Moreover, it is important for the whole argument that only the ``strongly low-high'' interactions are incorporated into the modified magnetic wave operator, which is specified by the small absolute constant $0 < \gamma \ll 1$.
The precise linear-nonlinear decomposition and the induction on frequency procedure are set up and explained in more detail in Subsections~\ref{subsec:decomp_lin_nonlin}--\ref{subsec:global_forcedMKG}.

Some care has to be taken to ensure that at every stage of the induction procedure various smallness requirements on the rough linear evolutions and on the nonlinear components (from prior stages of the induction) are satisfied. This is achieved by working with suitable probabilistic cutoffs in the proof of Theorem~\ref{thm:main} in Section~\ref{sec:proof_of_thm}. Their use is perhaps somewhat reminiscent of the truncation method of de Bouard and Debussche\cite{BouDeb99}.

\medskip 

\noindent {\it ``Probabilistic'' parametrix, redeeming functional framework, and generalized multilinear estimates.}

\noindent A key difficulty in the proof of Theorem~\ref{thm:main} is the construction of the adapted linear evolution $\Phi^n_r$ of the rough random data $T_n \phi^\omega[0]$ as a suitable approximate solution to the modified paradifferential magnetic wave equation $\Box_{A^{<n-1}}^{p, mod} \Phi^n_r \approx 0$. The adapted linear evolution $\Phi^n_r$ has to have two main properties. On the one hand $\Phi^n_r$ has to satisfy suitable redeeming space-time integrability properties (on a suitable event) in order to close all nonlinear estimates. On the other hand, the accrued renormalization error estimate $\Box_{A^{<n-1}}^{p, mod} \Phi^n_r$ has to gain regularity so that it can be treated as a ``smooth'' source term in the equation for $\Phi^n_s$.

The subtle iterative definition of $\Phi^n_r$ in terms of a modified ``probabilistic'' parametrix is carefully laid out in Subsection~\ref{subsec:prob_definition_Phi}.
Then we exploit the randomness and derive redeeming space-time integrability properties (on a suitable event) of the rough linear evolution~$\Phi^n_r$ in Subsection~\ref{subsec:prob_strichartz_phi}.
Here a delicate point is that the connection form $A^{<n-1}$ from the prior induction stages enters the definition of the parametrix for $\Phi^n_r$ and is also a random function.
The key point that makes this construction work is that the random data $T_n \phi^\omega[0]$ for the rough linear evolution $\Phi^n_r$ is independent of the random data $\bigl( T_{<n-1} A_x^\omega[0], T_{<n-1} \phi^\omega[0] \bigr)$ on which the connection form $A^{<n-1}$ depends.
Since $\Phi^n_r$ is only an approximate solution, we need to show that the resulting data error $T_n \phi^\omega[0] - \Phi^n_r[0]$ in fact gains regularity (on a suitable event) and that the resulting renormalization error $\Box_{A^{<n-1}}^{p, mod} \Phi^n_r$ can be treated as a ``smooth'' source term (on a suitable event). This is accomplished in Subsection~\ref{subsec:prob_data_error} and Subsection~\ref{subsec:prob_renormalization_error}. 

The precise definitions of the ``redeeming'' space-time integrability properties that the rough linear evolutions enjoy (on a suitable event) are provided in Subsection~\ref{subsec:redeeming_properties}. They are designed so that the relevant multilinear estimates from~\cite{KST} can be generalized to allow for rough inputs. These generalized multilinear estimates are derived in Section~\ref{sec:multilinear_estimates}.

\medskip 

\noindent {\it Acknowledgements: The authors are grateful to Patricia Alonso Ruiz for helpful discussions.}

\section{Preliminaries}

\subsection{Global small constants}

We work with a string of globally defined small constants satisfying
\begin{equation*}
 0 < \delta_\ast \ll \sigma \ll \gamma \ll \delta_2 \ll \delta_1 \ll \delta \ll 1,
\end{equation*}
where 
\begin{itemize}
 \item $\delta$ specifies the off-diagonal gain in multilinear estimates; 
 \item $\delta_1$ is used for the sum $\sum_{\ell < 0} 2^{\delta_1 \ell}$ in the definition of the redeeming $R_k L^2_t L^\infty_x$, $R_k L^2_t L^6_x$, and $R_k L^\infty_t L^\infty_x$ norms;
 \item $\delta_2$ is used for capturing the frequency localization to $\sim 2^n$ (up to tails) of the smooth nonlinear components $\calA^n_{x,s}$ and $\Phi^n_s$ at dyadic frequency level $n$;
 \item $\gamma$ specifies the frequency restriction $k \leq (1-\gamma)n$ to distinguish ``moderately low-high'' and ``strongly low-high'' interactions, it therefore plays a key role in the definition of the ``probabilistic'' phase function $\psi^{n, mod}_\pm$ in Section~\ref{sec:probabilistic_bounds};
 \item $\sigma$ is used for specifying the cutoff of small angle interactions in the definitions of the ``deterministic'' and ``probabilistic'' phase functions;
 \item $\delta_\ast$ specifies the Sobolev regularity $1-\delta_\ast < s < 1$ of the random data.
\end{itemize}

\subsection{Probability theory}

The derivation of the redeeming space-time integrability properties of the linear evolutions of the rough random data crucially relies on the classical Khintchine inequality.
\begin{lemma}[Khintchine's inequality] \label{lem:khintchine}
 For any choice of a positive integer $N$, a sequence $\{X_j\}_{j=1}^N$ of independent standard zero-mean Gaussian random variables, and a sequence $\{c_j\}_{j=1}^N \subset \bbC$, we have for $1 \leq p < \infty$ that
 \begin{equation*}
  \biggl( \bbE \, \biggl| \sum_{j=1}^N c_j X_j \biggr|^p \biggr)^{\frac1p} \lesssim \sqrt{p} \biggl( \sum_{j=1}^N |c_j|^2 \biggr)^{\frac12}.
 \end{equation*}
\end{lemma}

We use the following lemma to estimate the probability of certain events. Its proof is a simple consequence of Chebyshev's inequality.
\begin{lemma}[Tail estimate] \label{lem:probability_estimate}
 Let $X$ be a real-valued random variable on a probablility space $(\Omega, \calF, \bbP)$. Suppose that there exists $D > 0$ such that for every $1 \leq p < \infty$ we have 
 \begin{equation*}
  \bigl( \bbE \, |X|^p \bigr)^{\frac1p} \lesssim \sqrt{p} D.
 \end{equation*}
 Then there exist absolute constants $C, c > 0$ such that for every $\lambda \geq 0$ it holds that
 \begin{equation*}
  \bbP \bigl( |X| > \lambda \bigr) \leq C \exp \Bigl( - c \frac{\lambda^2}{D^2} \Bigr).
 \end{equation*}
\end{lemma}

\subsection{Frequency and sector projections}

In order to define several Littlewood-Paley projection operators, we pick a non-negative even bump function $\chi_0 \in C^\infty(\bbR)$ satisfying $\chi_0(y) = 1$ for $|y| \leq 1$ and $\chi_0(y) = 0$ for $|y| > 2$ and set $\chi(y) = \chi_0(y) - \chi_0(2 y)$. Then we introduce the standard Littlewood-Paley projection operators for $k \in \bbZ$ by
\[
 \widehat{P_k f}(\xi) = \chi \big( 2^{-k} |\xi| \big) \hat{f}(\xi).
\]
To measure proximity of the space-time Fourier support to the light cone we use the concept of modulation. For $j \in \bbZ$ we define the projection operators
\begin{align*}
 \calF \big(Q_j f\big)(\tau, \xi) &= \chi \big( 2^{-j} | |\tau| - |\xi| | \big) \, \calF(f)(\tau, \xi), \\
 \calF \big(Q_j^\pm f\big)(\tau, \xi) &= \chi \big( 2^{-j} | |\tau| - |\xi| | \big) \, \chi_{\{ \pm \tau > 0 \}} \, \calF(f)(\tau, \xi), 
\end{align*}
where $\calF$ denotes the space-time Fourier transform. On occasion, we also need multipliers $S_l$ to restrict the space-time frequency and correspondingly set for $l \in \bbZ$,
\[
 \calF \big(S_l f\big)(\tau, \xi) = \chi \big( 2^{-l} |(\tau, \xi)| \big) \, \calF(f)(\tau, \xi).
\]
Moreover, we use projection operators $P_l^\kappa$ to localize the homogeneous variable $\frac{\xi}{|\xi|}$ to caps $\kappa \subset \mathbb{S}^3$ of diameter~$\sim 2^l$ for integers $l < 0$ via smooth cutoff functions. We assume that for each such $l < 0$ these cutoffs form a smooth partition of unity subordinate to a uniformly finitely overlapping covering of $\mathbb{S}^3$ by caps $\kappa$ of diameter~$\sim 2^l$.

Finally, for any $\eta \in \bbS^3$ and any angle $0 < \theta \lesssim 1$, we define the sector projection $\Pi_{>\theta}^\eta$ in frequency space by the formula
\[
 \widehat{\Pi_{> \theta}^\eta f}(\xi) := \Bigl( 1 - \chi_0\Bigl(\frac{\angle (\xi, \eta)}{\theta}\Bigr) \Bigr) \Bigl( 1 - \chi_0\Bigl(\frac{\angle (-\xi, \eta)}{\theta}\Bigr) \Bigr) \hat{f}(\xi),
\]
where $\angle(\xi, \eta)$ is the angle between $\xi$ and $\eta$. Thus, $\Pi_{>\theta}^\eta$ restricts $f$ smoothly (except at the frequency origin) to the sector of frequencies $\xi$ whose angle with both $\eta$ and $-\eta$ is $\gtrsim \theta$. Similarly, we define the Fourier multipliers $\Pi_\theta^\eta$, $\Pi_{\leq \theta}^\eta$, and $\Pi^\eta_{\theta_1 > \cdot > \theta_2}$.

\section{Function spaces} \label{sec:function_spaces}

In this section we first recall the functional framework from~\cite{KST} that we will use throughout.  We also set up some notation for function spaces that will be convenient for the induction on frequency procedure in this work. 
Finally, we introduce the redeeming space-time integrability properties that the rough linear evolutions of the random data will enjoy and that beat the scaling.  

\subsection{Review of the functional framework from~\cite{KST}}

We use the same definitions and notations as in~\cite{KST} for the spaces $S_k$, $S^1$, $N$, $Y^1$, and $Z$. The (smooth) solutions of the nonlinear wave equations for the spatial part of the connection form and for the scalar field will be placed in the scaling-critical space $S^1$, while the inhomogeneous terms of the wave equations will be placed in the space $N$. The (smooth) elliptic variable $A_0$ will be measured in the $Y$ space.

We begin by introducing the convention that for any norm $\|\cdot\|_S$ and any $p \in [1,\infty)$, 
\[
 \|F\|_{\ell^p S} = \biggl( \sum_{k \in \bbZ} \|P_k F\|_S^p \biggr)^{\frac{1}{p}}.
\]
Then we define the $X^{s,b}$ type norms applied to functions at spatial frequency $\sim 2^k$,
\[
 \|F\|_{X^{s,b}_p} = 2^{s k} \biggl( \sum_{j\in\bbZ} \Bigl( 2^{b j} \|Q_j P_k F\|_{L^2_t L^2_x} \Bigr)^p \biggr)^{\frac{1}{p}}
\]
for $s, b \in \bbR$ and $p \in [1,\infty)$ with the obvious analogue for $p = \infty$.

We will mainly use three function spaces $N, N^\ast, \text{ and } S$. Their dyadic subspaces $N_k, N_k^\ast$ and $S_k$ satisfy 
\[
 N_k = L^1_t L^2_x + X_1^{0, -\frac{1}{2}}, \quad N_k^\ast = L^\infty_t L^2_x \cap X_\infty^{0, \frac{1}{2}}, \quad X_1^{0, \frac{1}{2}} \subseteq S_k \subseteq N_k^\ast.
\]
Then it holds that
\[
 \|F\|_N^2 = \sum_{k \in \bbZ} \|P_k F\|_{N_k}^2, \quad \|F\|_{N^\ast}^2 = \sum_{k \in \bbZ} \|P_k F\|_{N_k^\ast}^2.
\]
The space $S_k$ is defined by
\[
 \|\phi\|_{S_k}^2 = \|\phi\|_{S_k^{Str}}^2 + \|\phi\|_{S_k^{ang}}^2 + \|\phi\|_{X^{0,\frac{1}{2}}_\infty}^2,
\]
where
\begin{equation*}
 \begin{split}
  S_k^{Str} &= \bigcap_{\frac{1}{q} + \frac{3/2}{r} \leq \frac{3}{4}} 2^{(\frac{1}{q} + \frac{4}{r} - 2)k} L^q_t L^r_x, \\
  \|\phi\|_{S_k^{ang}}^2 &= \sup_{l < 0} \sum_\eta \|P_l^\eta Q_{< k + 2l} \phi \|_{S_k^\eta(l)}^2,
 \end{split}
\end{equation*}
and the angular sector norms $S_k^\eta(l)$ are defined below. The sum over $\eta$ in the definition of $S_k^{ang}$ is over a covering of $\bbS^3$ by caps of diameter $\sim 2^{l}$ with uniformly finite overlaps, and the symbols of $P_l^\eta$ form a smooth partition of unity subordinate to this covering.

\medskip

To introduce the angular sector norms $S_k^\eta(l)$ we first define the plane wave space
\[
 \|\phi\|_{PW^\pm_\eta(l)} = \inf_{\phi = \int \phi^{\eta'}} \int_{|\eta - \eta'| \leq 2^l} \|\phi^{\eta'}\|_{L^2_{\pm \eta'} L^\infty_{(\pm \eta')^\perp}} \, d\eta'
\]
and the null energy space
\[
 \|\phi\|_{NE} = \sup_\eta \|\slashed{\nabla}_\eta \phi\|_{L^\infty_\eta L^2_{\eta^\perp}},
\]
where the norms are with respect to $\ell_\eta^\pm = t \pm \eta \cdot x$ and the transverse variable, while $\slashed{\nabla}_\eta$ denotes spatial differentiation in the $(\ell_\eta^+)^\perp$ plane. Then we set 
\begin{equation*}
 \begin{split}
  \|\phi\|_{S_k^\eta(l)}^2 &= \|\phi\|_{S_k^{Str}}^2 + 2^{-2k} \|\phi\|_{NE}^2 + 2^{-3k} \sum_{\pm} \|Q^\pm \phi\|_{PW_\eta^{\mp}(l)}^2 \\
  &\quad \quad + \sup_{\substack{ k' \leq k, l' \leq 0, \\ k+2l \leq k' + l' \leq k+l} } \sum_{{\mathcal C}_{k'}(l')} \bigg( \|P_{{\mathcal C}_{k'}(l')} \phi\|_{S_k^{Str}}^2 + 2^{-2k} \|P_{{\mathcal C}_{k'}(l')} \phi\|_{NE}^2 \\
  &\quad \quad \quad \quad + 2^{-2k'-k} \|P_{{\mathcal C}_{k'}(l')} \phi\|_{L^2_t L^\infty_x}^2 + 2^{-3(k'+l')} \sum_{\pm} \|Q^{\pm} P_{{\mathcal C}_{k'}(l')} \phi \|_{PW_\eta^\mp(l)}^2 \bigg),
 \end{split}
\end{equation*}
where $P_{{\mathcal C}_{k'}(l')}$ is a projection operator to a radially directed block ${\mathcal C}_{k'}(l')$ of dimensions $2^{k'} \times (2^{k'+l'})^3$.

\medskip

Now we define
\[
 \|\phi\|_{S^1}^2 = \sum_{k \in \bbZ} \|\nabla_{t,x} P_k \phi\|_{S_k}^2 + \| \Box \phi \|^2_{\ell^1 L^2_t \dot{H}^{-\frac{1}{2}}_x}
\]
and the higher derivative norms
\[
 \|\phi\|_{S^N} := \|\nabla_{t,x}^{N-1} \phi\|_{S^1}, \quad N \geq 2.
\]
Moreover, we introduce
\[
 \|u\|_{S_k^\sharp} = \|\nabla_{t,x} u\|_{L^\infty_t L^2_x} + \|\Box u\|_{N_k}.
\]
Occasionally we need to separate the two characteristic cones $\{ \tau = \pm |\xi| \}$, for which we define
\begin{eqnarray*}
 N_{k,\pm},&  &N_k = N_{k,+} \cap N_{k, -} \\
 S_{k,\pm}^\sharp,&  &S_k^\sharp = S_{k,+}^\sharp + S_{k,-}^\sharp \\
 N_{k, \pm}^\ast,&  &N_k^\ast = N_{k,+}^\ast + N_{k,-}^\ast.
\end{eqnarray*}
We will also use an auxiliary space of $L^1_t L^\infty_x$ type,
\begin{equation*}
 \|\phi\|_{Z} = \sum_{k \in \bbZ} \|P_k \phi\|_{Z_k}, \quad \|\phi\|_{Z_k}^2 = \sup_{l < C} \sum_\eta 2^l \|P_l^\eta Q_{k+2l} \phi \|_{L^1_t L^\infty_x}^2.
\end{equation*}
Finally, to control the component $A_0$, we define
\[
 \|A_0\|_{Y^1}^2 = \|\nabla_{t,x} A_0\|_{L^\infty_t L^2_x}^2 + \|A_0\|_{L^2_t \dot{H}^{3/2}_x}^2 + \|\partial_t A_0\|_{L^2_t \dot{H}^{1/2}_x}^2
\]
and the higher derivative norms
\[
 \|A_0\|_{Y^N} = \|\nabla_{t,x}^{N-1} A_0\|_{Y^1}, \quad N \geq 2.
\]

The link between the $S$ and $N$ spaces is provided by the following energy estimate from \cite{KST},
\[
 \|\nabla_{t,x} \phi \|_{S} \lesssim \|\nabla_{t,x} \phi(0)\|_{L^2_x} + \| \Box \phi \|_{N}.
\]

We will also use the notation 
\begin{equation*}
 \|P_k \phi\|_{S^1_k} := \|\nabla_{t,x} P_k \phi\|_{S_k}, 
\end{equation*}
and we set 
\begin{equation*}
 \|\phi\|_{S^{1-\delta_\ast}_k} := 2^{- \delta_\ast k} \|\nabla_{t,x} P_k \phi\|_{S_k} = 2^{- \delta_\ast k} \|\phi\|_{S^1_k}.
\end{equation*}

\medskip 

Finally, in order to capture the frequency localization (up to tails) of the solution increments in our induction on frequency procedure, we introduce for any $n \geq 1$ the norms 
\begin{align*}
 \|\phi\|_{S^1[n]} &:= \sup_{k \in \bbZ} \, 2^{+\delta_2 |n-k|} \Bigl( \|P_k \phi\|_{S^1_k} + \|P_k \Box \phi\|_{L^2_t \dot{H}^{-\hf}_x} \Bigr), \\
 \|A_0\|_{Y^1[n]} &:= \sup_{k \in \bbZ} \, 2^{+\delta_2 |n-k|} \|P_k A_0\|_{Y^1}.
\end{align*}

\subsection{The redeeming ``probabilistic'' functional framework} \label{subsec:redeeming_properties}

Here we introduce the ``redeeming'' function spaces capturing the improved space-time integrability properties that the rough linear evolutions of the random data will enjoy (on a suitable event).
For any $k \geq 1$ we define the redeeming $R_k$ norm of a rough linear evolution localized to frequencies $\sim 2^k$ by 
\begin{equation}
 \begin{aligned}
  \|v\|_{R_k} &:= \|(v, 2^{-k} \nabla_{t,x} v)\|_{R_kL^2_t L^\infty_x} + \|(v, 2^{-k} \nabla_{t,x} v)\|_{R_k L^2_t L^6_x}  \\
  &\qquad + \|(v, 2^{-k} \nabla_{t,x} v)\|_{R_kL^\infty_t L^\infty_x} + \|(v, 2^{-k} \nabla_{t,x} v)\|_{R_kStr} + \|v\|_{S_k^{1-\delta_\ast}},
 \end{aligned} 
\end{equation}
where the components $R_k L^2_t L^\infty_x$, $R_k L^2_t L^6_x$, $R_k L^\infty_t L^\infty_x$, and $R_k Str$ are given by
\begin{align*}
 \|v\|_{R_kL^2_t L^\infty_x} &:= 2^{(\frac12-20\sigma)k} \sum_{l<0}2^{\delta_1 l} \biggl( \sum_{\kappa} \sum_{\substack{k'\leq k,\,l'\leq 0\\k+2l\leq k'+l'\leq k+l }}\sum_{\mathcal{C}_{k'}(l')}\gamma^{-2}(k',l')\big\|P_{\mathcal{C}_{k'}(l')}P_l^{\kappa}Q_{<k+2l} v \big\|_{L_t^2 L_x^\infty}^2 \biggr)^{\frac12}, \\
 &\qquad \text{with }  \gamma(k',l') := \bigl( \min\{2^{k'}, 1 \} \bigr)^{\frac12-} \cdot \bigl( \min\{2^{k'+l'}, 1 \} \bigr)^{\frac12-}, \\
 \|v\|_{R_kL^2_t L^6_x} &:= 2^{(\frac12-20\sigma)k} \sum_{l<0}2^{\delta_1 l} \biggl( \sum_{\kappa} \, \bigl\| P_l^{\kappa} Q_{<k+2l} v \bigr\|_{L_t^2 L_x^6}^2 \biggr)^{\frac12},  \\
 \|v\|_{R_kL^\infty_t L^\infty_x} &:= 2^{(1-20\sigma)k} \sum_{l < 0} 2^{\delta_1 l} \biggl( \sum_\kappa \bigl( \min \{ 2^{(\frac32 -) (k+l)}, 1 \} \bigr)^{-2} \bigl\| P_l^\kappa P_k v \bigr\|_{L^\infty_t L^\infty_x}^2 \biggr)^{\frac{1}{2}}, \\
 \|v\|_{R_k Str} &:= 2^{(1-20\sigma)k} \sum_{\frac{1}{q} + \frac{3}{2r} \leq \frac{3}{4}} 2^{-\frac{1}{q}k} \|v\|_{L^q_t L^r_x}.
\end{align*}
Let us briefly comment on the definition and the use of the different components of the redeeming $R_k$ norm. 
In the definitions of the $R_k L^2_t L^\infty_x$ and $R_k L^\infty_t L^\infty_x$ components, for each $l < 0$ the sum over $\kappa$ refers to a sum over caps $\kappa \subset \bbS^3$ of diameter $\sim 2^l$ with uniformly finite overlaps and the symbols $P_l^\kappa$ form a corresponding subordinate smooth partition of unity.

The $R_k L^2_t L^\infty_x$ and $R_k L^2_t L^6_x$ components are designed to be used in conjunction with the $S^{1-\delta_\ast}_k$ component to control rough linear evolution inputs in null form estimates. The factor $\gamma(k',l') $ helps gain additional smallness for very thin and/or short rectangular boxes $\mathcal{C}_{k'}(l')$, which come up many times in the null form estimates in~\cite{KST}. 

The $R_k L^\infty_t L^\infty_x$ component incorporates a gain from frequency localization to caps $\kappa \subset \bbS^3$ of diameter $\sim 2^l$. It plays an important role in $L^\infty$ estimates of the ``rough'' parts of the ``deterministic'' and the ``probabilistic'' phase functions, see Lemma~\ref{lem:det_phase_function_Linfty_bounds} and Lemma~\ref{lem:prob_phase_function_Linfty_bounds}. 

Finally, the $R_k Str$ component consists of finitely many wave-admissible exponent pairs and encompasses a Klainerman-Tataru gain from the unit-scale frequency localization of the ``atoms'' of the Wiener randomization. The $R_k Str$ bounds for the rough linear evolutions are used in many places, in particular they suffice to estimate all cubic nonlinearities in (MKG-CG) with rough linear evolutions as inputs.

\section{Induction on frequency procedure}

In this section we begin with the construction of solutions $(A_x, A_0, \phi)$ to the (MKG-CG) system of equations for scaling super-critical random initial data $A_x[0] = A_x^\omega[0]$, $\phi[0] = \phi^\omega[0]$ on an event with high probability. 
We will define the solutions $(A_x, A_0, \phi)$ as the limit in $( C^0_t H^s_x + S^1 ) \times Y^1 \times ( C^0_t H^s_x + S^1 )$ of a sequence $\bigl\{ (A^{<n}_x, A^{<n}_0, \phi^{<n}) \bigr\}_{n \geq 0}$ of solutions to (MKG-CG) with frequency-truncated data given by $A^{<n}_x[0] = T_{<n} A_x^\omega[0]$, $\phi^{<n}[0] = T_{<n} \phi^\omega[0]$. 
Since the frequency-truncated random data is smooth, we would in fact have global existence of the solutions $(A^{<n}_x, A^{<n}_0, \phi^{<n})$ for every $n \geq 0$, even for large data, by the (deterministic) global regularity results~\cite{KST, OT1, OT2, OT3, KL15} for the energy-critical (MKG-CG) equation. 
However, in order to show the convergence of this sequence on a suitable event, we need to establish refined uniform bounds on the sequence of solutions. To this end we construct the sequence inductively, adding in one dyadic frequency block of the random data at a time and 
decomposing the spatial parts of the connection form as well as the scalar field into suitable rough linear components and smooth nonlinear components.

The main result of this section is a (deterministic) global existence result for a \emph{forced Maxwell-Klein-Gordon system of equations} for the nonlinear components of the solution increments at each induction step, assuming that certain smallness assumptions on the forcing terms hold. The main work in the proof of Theorem~\ref{thm:main} then goes into establishing the existence of an event with high probability on which these smallness assumptions are satisfied at all induction stages so that the corresponding sequence of solutions converges.

\subsection{Decomposition of the nonlinearity}

We begin by examining the nonlinearities in the (MKG-CG) system of equations more carefully and we introduce some notation that will be useful in the following.
Recall that the (MKG-CG) system is given by
\begin{equation*}
\left\{ \begin{aligned}
 \Box A_j &= - \calP_j \Im ( \phi \overline{D_x \phi} ), \\
 D^\alpha D_\alpha \phi &= 0, \\
 \Delta A_0 &= - \Im ( \phi \overline{D_0 \phi} ),
\end{aligned} \right.
\end{equation*}
and that it suffices to prescribe initial data for $A_x[0]$ and $\phi[0]$, because the temporal component of the connection form $A_0$ is at any time determined in terms of $A_x$ and $\phi$ by an elliptic equation.

\medskip 

\noindent {\bf The $A_x$ equation.}
We decompose $A_j$ into its free wave evolution part and its nonlinear part
\begin{equation*}
 A_j = A_j^{free} + A_j^{nl}.
\end{equation*}
Then we write
\begin{equation*}
 A_j^{nl} = \bA_j(\phi, \phi, A),
\end{equation*}
where $\bA_j$ is extended to a symmetric quadratic form in the first two variables 
\begin{equation*}
 \bA_j(\phi^{(1)}, \phi^{(2)}, A) = \bA_j^2(\phi^{(1)}, \phi^{(2)}) + \bA_j^3(\phi^{(1)}, \phi^{(2)}, A)
\end{equation*}
with 
\begin{align*}
 \bA_j^2(\phi^{(1)}, \phi^{(2)}) &= -\hf \Box^{-1} \calP_j \Im \bigl( \phi^{(1)} \overline{\nabla_x \phi^{(2)}} + \overline{\nabla_x \phi^{(1)}} \phi^{(2)} \bigr), \\
 \bA_j^3(\phi^{(1)}, \phi^{(2)}, A) &= +\hf \Box^{-1} \calP_j \bigl( \phi^{(1)} \overline{\phi^{(2)}} A_x + \overline{\phi^{(1)}} \phi^{(2)} A_x \bigr).
\end{align*}
Recall that the quadratic part $\bA_j^2$ of the nonlinearity exhibits the favorable null structure 
\begin{equation*}
 \calP_j \bigl( \phi^{(1)} \overline{\nabla_x \phi^{(2)}} \bigr) = \partial^k \Delta^{-1} \calN_{kj}\bigl( \phi^{(1)}, \overline{\phi^{(2)}} \bigr).
\end{equation*}

\medskip 

\noindent {\bf The $A_0$ equation.} 
Here we introduce the notation
\begin{align*}
 A_0 &= \bA_0(\phi, \phi, A), \\
 \pt A_0 &= \pt \bA_0(\phi, \phi, A),
\end{align*}
where we set
\begin{align*}
 \bA_0(\phi, \phi, A) &= - \Delta^{-1} \Im ( \phi \overline{\pt \phi} ) + \Delta^{-1} ( \phi \overline{\phi} A_0 ) \equiv \bA_0^2(\phi, \phi) + \bA_0^3(\phi, \phi, A_0), \\
 \pt \bA_0(\phi, \phi, A) &= - \Delta^{-1} \partial^j \Im( \phi \overline{\partial_j \phi} ) + \Delta^{-1} \partial^j( \phi \overline{\phi} A_j ) \equiv \pt \bA_0^2(\phi, \phi) + \pt \bA_0^3(\phi, \phi, A_x).
\end{align*}
In the following we think of $\bA_0(\phi, \phi, A)$ and $\pt \bA_0(\phi, \phi, A)$ as being extended to symmetric quadratic forms in the first two variables.

\medskip 

\noindent {\bf The $\phi$ equation.} 
Expanding the covariant wave operator $D^\alpha D_\alpha$ leads to the following equation for the scalar field
\begin{equation*}
 \Box \phi = - 2i A^\alpha \partial_\alpha \phi + i (\pt A_0) \phi + A^\alpha A_\alpha \phi.
\end{equation*}
In the Coulomb gauge the magnetic interaction term $A^j \partial_j \phi$ exhibits the null structure
\begin{equation*}
 A^j \partial_j \phi = \calN_{kj} \bigl( \partial^k \Delta^{-1} A^j, \phi \bigr).
\end{equation*}
Even in the purely deterministic case, the low-high interactions in the magnetic interaction term involving the free wave part of $A_x$ turn out to be non-perturbative at energy regularity and have to be retained into the linear wave operator. 
In the current setting with scaling super-critical random data, the low-high interactions in the magnetic interaction term become even more problematic.
In preparation for a refined decomposition of the $\phi$ equation in the next subsection, we isolate the low-high interactions in the magnetic interaction term $A^\alpha \partial_\alpha \phi$ and correspondingly rewrite $D^\alpha D_\alpha \phi = 0$ as 
\begin{align*}
 \Bigl( \Box + 2i \sum_k (P_{\leq k-C} A^\alpha) \partial_\alpha P_k \Bigr) \phi &= - 2i \sum_k (P_{>k-C} A^\alpha) \partial_\alpha P_k \phi + i (\pt A_0) \phi + A^\alpha A_\alpha \phi \\
 &\equiv \bM^1(A, \phi) + \bM^2(A_0, \phi) + \bM^3(A, A, \phi).
\end{align*}
Additionally, we decompose the nonlinear term $\bM^1(A, \phi)$ into
\begin{align*}
 \bM^1(A, \phi) = - 2i \sum_k (P_{>k-C} A^\alpha) \partial_\alpha P_k \phi &= - 2i \sum_k (P_{>k-C} A^0) \partial_t P_k \phi - 2i \sum_k (P_{>k-C} A^j) \partial_j P_k \phi \\
 &\equiv \bM^1_0(A_0, \phi) + \bM^1_x(A_x, \phi). 
\end{align*}

\subsection{Decomposition into rough and smooth components} \label{subsec:decomp_lin_nonlin}

We now turn to setting up the precise construction of the sequence $\bigl\{ (A^{<n}_x, A^{<n}_0, \phi^{<n}) \bigr\}_{n \geq 0}$ of solutions to (MKG-CG) with frequency-truncated random initial data given by
\begin{equation*}
 A^{<n}_x[0] = T_{<n} A^\omega_x[0], \quad \phi^{<n}[0] = T_{<n} \phi^\omega[0].
\end{equation*}
The sequence will be constructed inductively. To this end we introduce the dyadic decompositions
\begin{align*}
 A^{<n}_x &= A^{<n-1}_x + \calA^n_x, \quad n \geq 0, \\
 A^{<n}_0 &= A^{<n-1}_0 + \calA^n_0, \quad n \geq 0, \\
 \phi^{<n} &= \phi^{<n-1} + \Phi^n, \quad n \geq 0,
\end{align*}
where we set $A^{<-1}_x = A^{<-1}_0 = \phi^{<-1} = 0$.

We let $(A^{<0}_x, A^{<0}_0, \phi^{<0})$ be the solution to (MKG-CG) with (smooth) random initial data $A_x^{<0}[0] = T_0 A_x^\omega[0]$, $\phi^{<0}[0] = T_0 \phi^\omega[0]$, which we obtain from the small energy global regularity result by~\cite{KST} if this frequency-truncated data has sufficiently small energy. 

Then having constructed $(A^{<n-1}_x, A^{<n-1}_0, \phi^{<n-1})$, we construct $(A^{<n}_x, A^{<n}_0, \phi^{<n})$ by solving the (MKG-CG) difference equations for $(\calA_x^n, \calA_0^n, \Phi^n)$ with random initial data sharply localized to frequencies~$\sim 2^n$. Specifically, the random initial data for the spatial parts of the connection forms $\calA^n_x$ is given by
\begin{align*}
 \calA_x^n[0] = T_n A_x^\omega[0] = \bigl( T_n a^\omega, T_n b^\omega \bigr), \quad n \geq 1,
\end{align*}
and the random initial data for the scalar fields $\Phi^n$ is given by
\begin{align*}
 \Phi^n[0] = \bigl( T_n \phi_0^\omega, T_n \phi_1^\omega \bigr), \quad n \geq 1.
\end{align*}
At each dyadic frequency level $n \geq 1$ we decompose the (spatial part of) the connection form $\calA^n_x$ as well as the scalar field $\Phi^n$ into a rough (linear) component and a smooth (nonlinear) component. In the following ``smooth'' refers to having scaling-critical energy regularity. Crucially, on a suitable event the rough evolutions will have redeeming space-time integrability properties that beat the scaling. It is worth pointing out that such a decomposition is not necessary for the temporal component $\calA^n_0$, because at any time it is determined by $\calA^n_x$ and $\Phi^n$ via an elliptic equation.  

For the spatial part of the connection form we use the standard Bourgain-Da Prato-Debussche decomposition and write
\begin{align*}
 \calA^n_x = \calA^n_{x,r} + \calA^n_{x,s},
\end{align*}
where the rough part $\calA^n_{x,r}$ is just defined as the linear wave evolution of the (rough) random initial data
\begin{equation*}
 \calA^n_{x,r}(t) := S(t)\bigl[ T_n a^\omega, T_n b^\omega \bigr] = \cos(t|\nabla|) T_n a^\omega + \frac{\sin(t|\nabla|)}{|\nabla|} T_n b^\omega, \quad n \geq 1.
\end{equation*}
We emphasize that $\calA^n_{x,r}$ is sharply localized to frequencies $\sim 2^n$. 

Instead for the scalar field $\Phi^n$, we introduce an adapted linear-nonlinear decomposition
\begin{equation*}
 \Phi^n = \Phi^n_r + \Phi^n_s,
\end{equation*}
where the rough part $\Phi^n_r$ is defined as an \emph{approximate solution} to the linear magnetic wave equation
\begin{equation*}
 \Bigl( \Box + 2i P_{\leq (1-\gamma)n} A^{<n-1, \alpha} \partial_\alpha P_n \Bigr) \Phi^n_r \approx 0, \quad \Phi^n_r[0] \approx T_n \phi^\omega[0], \quad n \geq 1.
\end{equation*}
Here, $0 < \gamma \ll 1$ is a small constant that enacts a ``strongly low-high'' frequency separation. 
This choice will emerge and will be explained further below as we will derive the system of equations for the smooth components $(\calA^n_{x,s}, \calA^n_0, \Phi^n_s)$. 
Observe that the entire connection form $A^{<n-1}$ from the prior induction stages is built into the linear magnetic wave operator on the left-hand side. 
While $A^{<n-1}$ is a random function depending on the random initial data $T_{<n} A_x^\omega[0]$, $T_{<n} \phi^\omega[0]$ from the induction on frequency stages $\leq n-1$, the key point that will make this construction work is that the latter are independent of the random data $T_n A_x^\omega[0]$, $T_n \phi^\omega[0]$ at the induction on frequency stage $n$.
The precise definition of $\Phi^n_r$ via a parametrix will be given in Section~\ref{sec:probabilistic_bounds}. At this point we stress that by construction $\Phi^n_r$ will also be sharply localized to frequencies ~$\sim 2^n$. Moreover, it will follow from Proposition~\ref{prop:prob_data_error} that the \emph{data error}, i.e. the initial data for the nonlinear component of the scalar field, gains smoothness and is at the better energy regularity (on a suitable event)
\begin{equation*}
 \Phi_s^n[0] = \bigl( T_n \phi_0^\omega, T_n \phi_1^\omega \bigr) - \Phi_r^n[0] \in \dot{H}^1 \times L^2.
\end{equation*}

In order to systematically use the subscripts $s$, respectively $r$, to indicate smooth, respectively rough components, it will be convenient to denote the smooth solution $(A^{<0}_x, A^{<0}_0, \phi^{<0})$ to (MKG-CG) for the lowest frequency random initial data block $T_0 A_x^\omega[0]$, $T_0 \phi^\omega[0]$ by
\begin{equation*}
 (\calA^0_{x, s}, \calA^0_0, \Phi^0_s) \equiv (A^{<0}_x, A^{<0}_0, \phi^{<0}).
\end{equation*}

After these preparations, we are now in a position to derive the system of ``forced Maxwell-Klein-Gordon equations in Coulomb gauge'' for the nonlinear components $(\calA_{x,s}^n, \calA_0^n, \Phi^n_s)$, $n \geq 1$.
Subtracting the equations for $A^{<n}_x$ and $A^{<n-1}_x$ from each other, we obtain that the nonlinear component $\calA^n_{x,s}$ satisfies the forced (wave) equation
\begin{align*}
 \calA^n_{x,s} &= \bA_x( \phi^{<n}, \phi^{<n}, A^{<n} ) - \bA_x(\phi^{<n-1}, \phi^{<n-1}, A^{<n-1}), \quad n \geq 1.
\end{align*}
In the same manner, we find that $\calA_0^n$ satisfies the forced (elliptic) equation
\begin{align*}
 \calA_0^n = \bA_0(\phi^{<n}, \phi^{<n}, A_0^{<n}) - \bA_0(\phi^{<n-1}, \phi^{<n-1}, A_0^{<n-1}), \quad n \geq 1.
\end{align*}
To determine the equation for $\Phi^n_s$, we first subtract the equations for $\phi^{<n}$ and $\phi^{<n-1}$ from each other to obtain that
\begin{align*}
 &\Bigl( \Box + 2i \sum_k P_{\leq k-C} A^{<n, \alpha} \partial_\alpha P_k \Bigr) \phi^{<n} - \Bigl( \Box + 2i \sum_k P_{\leq k-C} A^{<n-1, \alpha} \partial_\alpha P_k \Bigr) \phi^{<n-1} \\
 &= \bM^1(A^{<n}, \phi^{<n}) - \bM^1(A^{<n-1}, \phi^{<n-1}) \\
 &\quad + \bM^2(A^{<n}, \phi^{<n}) - \bM^2(A^{<n-1}, \phi^{<n-1}) \\
 &\quad + \bM^3(A^{<n}, A^{<n}, \phi^{<n}) - \bM^3(A^{<n-1}, A^{<n-1}, \phi^{<n-1}).
\end{align*}
Inserting the decompositions $\phi^{<n} = \phi^{<n-1} + \Phi^n$ and $A^{<n} = A^{<n-1} + \calA^n$, this gives 
\begin{align*}
 \Bigl( \Box + 2i \sum_k P_{\leq k-C} A^{<n-1, \alpha} \partial_\alpha P_k \Bigr) \Phi^n =& - 2i \sum_k P_{\leq k-C} \calA^{n,\alpha} \partial_\alpha P_k \phi^{<n} \\
 &+ \bM^1(A^{<n}, \phi^{<n}) - \bM^1(A^{<n-1}, \phi^{<n-1}) \\
 &+ \bM^2(A^{<n}, \phi^{<n}) - \bM^2(A^{<n-1}, \phi^{<n-1}) \\
 &+ \bM^3(A^{<n}, A^{<n}, \phi^{<n}) - \bM^3(A^{<n-1}, A^{<n-1}, \phi^{<n-1}).
\end{align*}
Next, we insert the presumptive decomposition of $\Phi^n$ into its rough and smooth components $\Phi^n = \Phi^n_r + \Phi^n_s$, where the precise definition of $\Phi^n_r$ will now emerge. Then we find that $\Phi^n_s$ is a solution to the following forced magnetic wave equation
\begin{align*}
 \Bigl( \Box + 2i \sum_k P_{\leq k-C} A^{<n-1, \alpha} \partial_\alpha P_k \Bigr) \Phi^n_s =& - \Bigl( \Box + 2i \sum_k P_{\leq k-C} A^{<n-1, \alpha} \partial_\alpha P_k \Bigr) \Phi^n_r \\
 &- 2i \sum_k P_{\leq k-C} \calA^{n,\alpha} \partial_\alpha P_k \phi^{<n} \\ 
 &+ \bM^1(A^{<n}, \phi^{<n}) - \bM^1(A^{<n-1}, \phi^{<n-1}) \\
 &+ \bM^2(A^{<n}, \phi^{<n}) - \bM^2(A^{<n-1}, \phi^{<n-1}) \\
 &+ \bM^3(A^{<n}, A^{<n}, \phi^{<n}) - \bM^3(A^{<n-1}, A^{<n-1}, \phi^{<n-1})
\end{align*}
with initial data
\begin{equation*}
 \Phi^n_s[0] = T_n \phi^\omega[0] - \Phi^n_r[0].
\end{equation*}
In order to derive a priori bounds for $\Phi^n_s$, it is more favorable to only retain the free wave evolution part of the spatial components of the connection form $A^{<n-1}$ in the linear magnetic wave operator on the left-hand side. The other parts can be treated as perturbative nonlinear source terms at energy regularity. Keeping in mind that $\Phi_r^n$ will be chosen such that it is sharply localized to frequencies $\sim 2^n$, this leads to the equation
\begin{align*}
 \Bigl( \Box + 2i \sum_k P_{\leq k-C} \bigl( A^{<n-1, j}_r + \calA^{0,free, j}_s \bigr) \partial_j P_k \Bigr) \Phi^n_s =& - \Bigl( \Box + 2i P_{\leq n-C} A^{<n-1, \alpha} \partial_\alpha P_n \Bigr) \Phi^n_r \\
 &- 2i \sum_k P_{\leq k-C} \calA^{n,\alpha} \partial_\alpha P_k \phi^{<n} \\ 
 &- 2i \sum_k P_{\leq k-C} \bigl( A^{<n-1, j}_s - \calA^{0,free, j}_s \bigr) \partial_j P_k \Phi^n_s \\
 &+ 2i \sum_k P_{\leq k-C} A^{<n-1}_0 \partial_t P_k \Phi^n_s \\
 &+ \bM^1(A^{<n}, \phi^{<n}) - \bM^1(A^{<n-1}, \phi^{<n-1}) \\
 &+ \bM^2(A^{<n}, \phi^{<n}) - \bM^2(A^{<n-1}, \phi^{<n-1}) \\
 &+ \bM^3(A^{<n}, A^{<n}, \phi^{<n}) - \bM^3(A^{<n-1}, A^{<n-1}, \phi^{<n-1})
\end{align*}
with initial data
\begin{equation*}
 \Phi^n_s[0] = T_n \phi^\omega[0] - \Phi^n_r[0].
\end{equation*}
For the \emph{paradifferential magnetic d'Alembertian} on the left-hand side of the above equation for $\Phi^n_s$, we introduce the convenient short-hand notation
\begin{equation*}
 \Box^p_{A^{<n-1}} := \Box + 2i \sum_k P_{\leq k-C} \bigl( A^{<n-1, j}_r + \calA^{0,free, j}_s \bigr) \partial_j P_k.
\end{equation*}
To derive a priori bounds for $\Phi^n_s$, in Section~\ref{sec:renormalization} we will establish linear estimates for the inhomogeneous magnetic wave equation $\Box_{A^{<n-1}}^p u = F$ that are compatible with the delicate functional framework of the $S^1$ and $N$ spaces. This part will be based on a ``deterministic'' parametrix construction.

It remains to examine the low-high magnetic interaction term $P_{\leq n-C} A^{<n-1, \alpha} \partial_\alpha P_n \Phi^n_r$ with the rough component $\Phi^n_r$ at high frequency. We further decompose it into a ``strongly low-high'' interaction term and a ``moderately low-high'' interaction term
\begin{equation*}
 P_{\leq n-C} A^{<n-1, \alpha} \partial_\alpha P_n \Phi^n_r = P_{\leq (1-\gamma)n} A^{<n-1, \alpha} \partial_\alpha P_n \Phi^n_r + P_{[(1-\gamma)n, n-C]} A^{<n-1, \alpha} \partial_\alpha P_n \Phi^n_r,
\end{equation*}
where $0 < \gamma \ll 1$ is a suitable small constant.
Using the redeeming space-time integrability properties of $\Phi^n_r$, the ``moderately low-high interactions'' $P_{[(1-\gamma)n, n-C]} A^{<n-1, \alpha} \partial_\alpha P_n \Phi^n_r$ will turn out to be still perturbative at energy regularity. 
However, it is not possible for the ``strongly low-high'' interaction term to gain regularity and become treatable at energy regularity. The way out is to build it into a \emph{modified paradifferential magnetic d'Alembertian}
\begin{equation*}
 \Box^{p, mod}_{A^{<n-1}} := \Box + 2i P_{\leq (1-\gamma)n} A^{<n-1, \alpha} \partial_\alpha P_n
\end{equation*}
that \emph{defines} the rough linear evolution $\Phi^n_r$. More precisely, we define $\Phi^n_r$ as an \emph{approximate solution} to the modified linear magnetic wave equation
\begin{equation} \label{equ:derivation_forced_system_mod_magnetic_wave_equ}
 \Box^{p, mod}_{A^{<n-1}} \phi = 0, \quad \phi[0] = T_n \phi^\omega[0], \quad n \geq 1.
\end{equation}
The subtle iterative construction of $\Phi^n_r$ in terms of a ``probabilistic'' parametrix is carried out in Section~\ref{sec:probabilistic_bounds}.
Importantly, while $\Phi^n_r$ is not an exact solution to~\eqref{equ:derivation_forced_system_mod_magnetic_wave_equ} and thus does not completely remove the ``strongly low-high'' interaction term, it will follow from Proposition~\ref{prop:prob_renormalization_error_estimate} that the accrued renormalization error $\Box_{A^{<n-1}}^{p, mod} \Phi^n_r$ gains regularity and can be treated as a smooth source term (on a suitable event). This argument relies on a suitable control of a certain redeeming error control quantity $\calE \calC^n$ defined in~\eqref{equ:definition_EC_n}.

\medskip 

To summarize, we have arrived at the following {\bf system of forced MKG-CG equations for $(\calA_{x,s}^n, \calA^n_0, \Phi^n_s)$ at frequency level $n \geq 1$},
\begin{equation} 
\left\{ \begin{aligned}
 \calA^n_{j,s} &= \bA_j( \phi^{<n}, \phi^{<n}, A^{<n} ) - \bA_j(\phi^{<n-1}, \phi^{<n-1}, A^{<n-1}) \\
 \calA_0^n &= \bA_0(\phi^{<n}, \phi^{<n}, A_0^{<n}) - \bA_0(\phi^{<n-1}, \phi^{<n-1}, A_0^{<n-1}) \\
 \Box^p_{A^{<n-1}} \Phi^n_s &= - \Box^{p, mod}_{A^{<n-1}} \Phi^n_r - 2i P_{[(1-\gamma)n, n-C]} A^{<n-1, \alpha} \partial_\alpha P_n \Phi^n_r \\
 &\quad - 2i \sum_k P_{\leq k-C} \calA^{n,\alpha} \partial_\alpha P_k \phi^{<n} \\ 
 &\quad - 2i \sum_k P_{\leq k-C} \bigl( A^{<n-1, j}_s - \calA^{0,free, j}_s \bigr) \partial_j P_k \Phi^n_s \\
 &\quad + 2i \sum_k P_{\leq k-C} A^{<n-1}_0 \partial_t P_k \Phi^n_s \\
 &\quad + \bM^1(A^{<n}, \phi^{<n}) - \bM^1(A^{<n-1}, \phi^{<n-1}) \\
 &\quad + \bM^2(A^{<n}, \phi^{<n}) - \bM^2(A^{<n-1}, \phi^{<n-1}) \\
 &\quad + \bM^3(A^{<n}, A^{<n}, \phi^{<n}) - \bM^3(A^{<n-1}, A^{<n-1}, \phi^{<n-1})
\end{aligned} \right. \tag{fMKG-CG\textsubscript{n}}
\end{equation}
with initial data for the scalar field $\Phi^n_s$ given by
\begin{equation*}
 \Phi^n_s[0] = T_n \phi^\omega[0] - \Phi^n_r[0].
\end{equation*}
It is important to keep in mind that the nonlinearities in (fMKG-CG\textsubscript{n}) contain $A^{<n-1}$, $\phi^{<n-1}$, $\calA^n_{x,r}$, and $\Phi^n_r$ as forcing terms. Additionally, the right-hand side of the magnetic wave equation for $\Phi^n_s$ features the ``probabilistic'' renormalization error term $\Box^{p, mod}_{A^{<n-1}} \Phi^n_r$ as another forcing term.

\subsection{Global existence for the forced MKG-CG system of equations} \label{subsec:global_forcedMKG}

We now present two global existence results on which the iterative construction of the sequence of (smooth) solutions $\{ (\calA_{x,s}^n, \calA^n_0, \Phi^n_s) \}_{n \geq 0}$ relies. These should be viewed and are formulated as purely deterministic global existence results at energy regularity under suitable smallness assumptions on the respective data and forcing terms. 

If the energy of the lowest frequency block $(T_0 A_x^\omega[0], T_0 \phi^\omega[0])$ of the random data $(A_x^\omega[0], \phi^\omega[0])$ is sufficiently small, we can invoke the small energy global regularity result for (MKG-CG) by~\cite{KST} and start the induction on frequency procedure by solving the standard Maxwell-Klein-Gordon system of equations (MKG-CG) with (smooth) initial data given by $(T_0 A_x^\omega[0], T_0 \phi^\omega[0])$. We denote this (smooth) solution by $(\calA_{x,s}^0, \calA_0^0, \Phi^0_s) \in S^1 \times Y^1 \times S^1$ and we set $\calA_r^0 = \Phi_r^0 = 0$. 
In order to capture the frequency localization of the solution $(\calA_{x,s}^0, \calA_0^0, \Phi^0_s)$ to frequencies $|\xi| \lesssim 1$ up to tails, we use the norms
\begin{align*}
 \|\calA_{x,s}^0\|_{S^1[0]} &:= \biggl( \sum_k \bigl( \max\{ 2^{\delta_2 k}, 1 \} \bigr)^2 \| P_k \calA^0_{x,s} \|_{S^1_k}^2 \biggr)^{\frac12} + \| \calA_{x,s}^{0,nl} \|_{\ell^1 S^1}, \\
 \|\calA_0^0\|_{Y^1[0]} &:=  \biggl( \sum_k \bigl( \max\{ 2^{\delta_2 k}, 1 \} \bigr)^2 \| P_k \calA_0^0 \|_{Y^1}^2 \biggr)^{\frac12}, \\
 \|\Phi^0_s\|_{S^1[0]} &:= \biggl( \sum_k \bigl( \max\{ 2^{\delta_2 k}, 1 \} \bigr)^2 \| P_k \calA^0_{x,s} \|_{S^1_k}^2 \biggr)^{\frac12}.
\end{align*}

\begin{proposition}[Induction base case] \label{prop:induction_base_case}
 There exist absolute constants $0 < \varepsilon \ll 1$ and $C_0 \geq 1$ with the following property: If 
 \begin{equation*}
  \| T_0 A_x^\omega[0] \|_{\dot{H}^1_x \times L^2_x} + \|T_0 \phi^\omega[0] \|_{\dot{H}^1_x \times L^2_x} \leq \varepsilon,
 \end{equation*}
 then there exists a unique global solution $(\calA_{x,s}^0, \calA_0^0, \Phi_s^0) \in S^1 \times Y^1 \times S^1$ to (MKG-CG) with initial data
 \begin{equation*}
  \calA_{x,s}^0[0] = T_0 A_x^\omega[0], \quad \Phi^0_s[0] = T_0 \phi^\omega[0].
 \end{equation*}
 Moreover, it holds that
 \begin{equation*}
  \|\calA_{x,s}^0\|_{S^1[0]} + \|\calA_0^0\|_{Y^1[0]} + \|\Phi^0_s\|_{S^1[0]} \leq C_0 \bigl( \| T_0 A_x^\omega[0] \|_{\dot{H}^1_x \times L^2_x} + \|T_0 \phi^\omega[0] \|_{\dot{H}^1_x \times L^2_x} \bigr).
 \end{equation*}
\end{proposition}
\begin{proof}
 We may freely assume that the constant $0 < \varepsilon \leq 1$ is sufficiently small so that we can invoke the small energy global regularity result by Sterbenz-Tataru and the first author~\cite[Theorem 1]{KST}. From~\cite{KST} we obtain the following refined information about the solution $(\calA_{x,s}^0, \calA_0^0, \Phi_s^0)$.
Let $\{ c_k \}_{k \in \bbZ}$ be an $\dot{H}^1_x \times L^2_x$ frequency envelope for the initial data $(T_0 A^\omega[0], T_0 \phi^\omega[0])$ defined by
\begin{equation*}
 c_k := \sum_{\ell \in \bbZ} 2^{-2\delta_2 |k-\ell|} \Bigl( \bigl\| P_\ell T_0 A^\omega[0] \bigr\|_{\dot{H}^1_x \times L^2_x} + \bigl\|P_\ell T_0 \phi^\omega[0] \bigr)\bigr\|_{\dot{H}^1_x \times L^2_x} \Bigr).
\end{equation*}
Then it holds that
\begin{align*}
 \bigl\| P_k \calA_{x,s}^{0, nl} \bigr\|_{S_k} \lesssim c_k^2, \quad \| P_k \Phi_s^0 \|_{S^1} \lesssim c_k, \quad \| P_k \calA_{0,s}^0 \|_{Y^1} \lesssim c_k.
\end{align*}
In particular, it follows that
\begin{equation*}
 \|\calA_{x,s}^0\|_{S^1[0]} + \|\calA_0^0\|_{Y^1[0]} + \|\Phi^0_s\|_{S^1[0]} \leq C_0 \bigl( \| T_0 A_x^\omega[0] \|_{\dot{H}^1_x \times L^2_x} + \|T_0 \phi^\omega[0] \|_{\dot{H}^1_x \times L^2_x} \bigr)
\end{equation*}
for some absolute constant $C_0 \geq 1$. 
\end{proof}

The main result of this section is the following (deterministic) induction step global existence result at energy regularity for the system of forced MKG-CG equations (fMKG-CG\textsubscript{n}) at stage $n \geq 1$. Note that the corresponding solution $(\calA_{x,s}^n, \calA_0^n, \Phi^n_s)$ to (fMKG-CG\textsubscript{n}) is localized to frequencies~$\sim 2^n$ up to tails, as quantified by the estimate~\eqref{equ:induction_step_solution_bound} below.
The statement assumes smallness of the redeeming error control quantity $\calE \calC^n$ defined in~\eqref{equ:definition_EC_n}, which is used to bound the accrued renormalization error $\Box_{A^{<n-1}}^{p, mod} \Phi^n_r$.

\begin{proposition}[Induction step] \label{prop:induction_step}
 There exist absolute constants $0 < \varepsilon \ll 1$ and $C_0 \geq 1$ with the following property: Let $n \geq 1$ be arbitrary. Suppose that the linear rough components from previous stages satisfy 
 \begin{equation} \label{equ:induction_step_cond1}
  \sum_{m=1}^{n-1} \| \Phi^m_r \|_{R_m} + \sum_{m=1}^{n-1} \| \calA_{x,r}^m \|_{R_m} \leq \varepsilon 
 \end{equation}
 and that the smooth components from previous stages satisfy 
 \begin{equation} \label{equ:induction_step_cond2}
  \sum_{m=0}^{n-1} \| \Phi^m_s \|_{S^1[m]} + \sum_{m=0}^{n-1} \| \calA_{x,s}^m \|_{S^1[m]} + \sum_{m=0}^{n-1} \|\calA_0^m\|_{Y^1[m]} \leq C_0 \varepsilon.
 \end{equation}
 Assume that 
 \begin{equation} \label{equ:induction_step_cond3}
  \bigl\| \Phi_s^n[0] \bigr\|_{\dot{H}^1_x \times L^2_x} + \|\Phi_r^n\|_{R_n} + \|\calA_{x,r}^n\|_{R_n} + \|T_n \phi^\omega[0]\|_{H^{1-\delta_\ast}_x \times H^{-\delta_\ast}_x} + \calE \calC^n \leq \varepsilon.
 \end{equation}
 Then there exists a unique global solution $(\calA_{x,s}^n, \calA_0^n, \Phi_s^n) \in S^1 \times Y^1 \times S^1$ to ($\text{fMKG-CG\textsubscript{n}}$) satisfying
 \begin{equation} \label{equ:induction_step_solution_bound}
  \begin{aligned}
   &\|\Phi^n_s\|_{S^1[n]} + \| \calA_{x,s}^n \|_{S^1[n]} + \| \calA_0^n \|_{Y^1[n]} \\
   &\qquad \qquad \qquad \leq C_0 \Bigl( \| \Phi_n^s[0] \|_{\dot{H}^1_x \times L^2_x} + \|\Phi_r^n\|_{R_n} + \|\calA_{x,r}^n\|_{R_n} + \|T_n \phi^\omega[0]\|_{H^{1-\delta_\ast}_x \times H^{-\delta_\ast}_x} + \calE \calC^n \Bigr).
  \end{aligned}
 \end{equation}
\end{proposition}

For the proof of Theorem~\ref{thm:main} it is crucial to observe that in view of the conditions~\eqref{equ:induction_step_cond1}--\eqref{equ:induction_step_cond3} and in view of the bound~\eqref{equ:induction_step_solution_bound} on the solutions $(\calA_{x,s}^n, \calA^n_0, \Phi^n_s)$, we can invoke the induction base case Proposition~\ref{prop:induction_base_case} and keep iterating the induction step Proposition~\ref{prop:induction_step} for all $n \geq 1$ on an event on which we have that
\begin{align*}
 &\| (T_0 A_x^\omega[0], T_0 \phi^\omega[0]) \|_{\dot{H}^1_x \times L^2_x} + \sum_{m=1}^\infty \| \Phi^m_r \|_{R_m} + \sum_{m=1}^{\infty} \| \calA_{x,r}^m \|_{R_m} \\
 &\qquad + \sum_{m=1}^\infty \| \Phi^m_s[0] \|_{\dot{H}^1 \times L^2_x} + \sum_{m=1}^\infty \|T_m \phi^\omega[0]\|_{H^{1-\delta_\ast}_x \times H^{-\delta_\ast}_x} + \sum_{m=1}^\infty \calE \calC^m \leq \varepsilon.
\end{align*}

We now outline the proof of Proposition~\ref{prop:induction_step} using the nonlinear estimates established in the next Subsection~\ref{subsec:nonlinear_estimates}

\begin{proof}[Proof of Proposition~\ref{prop:induction_step}]
 As in the proof of the small energy global regularity result for (MKG-CG) in~\cite{KST}, the scheme of the proof is a Picard iteration. Here it is important to keep in mind that in terms of estimates this is really a two-step iteration, because to obtain good bounds, the equations for $\calA^n_{x,s}$ and $\calA^n_0$ have to be reinserted. The nonlinear estimates in Subsection~\ref{subsec:nonlinear_estimates} have to be understood in this sense. In the following we use superscripts $(\ell)$ to denote the Picard iterates.
 
 We initialize the Picard iteration by setting
 \begin{equation*}
  \bigl( \calA_{x,s}^{n, (0)}, \calA_0^{n, (0)}, \Phi^{n, (0)}_s \bigr) = (0, 0, 0).
 \end{equation*}
 Then for any $\ell \geq 1$, we define $\bigl( \calA_{x,s}^{n, (\ell)}, \calA_0^{n, (\ell)}, \Phi^{n, (\ell)}_s \bigr)$ as the solution to the system (fMKG-CG\textsubscript{n}) where $\bigl( \calA_{x,s}^{n}, \calA_0^{n}, \Phi^{n}_s \bigr)$ on the left-hand sides is replaced by $\bigl( \calA_{x,s}^{n, (\ell)}, \calA_0^{n, (\ell)}, \Phi^{n, (\ell)}_s \bigr)$ and $\bigl( \calA_{x,s}^{n}, \calA_0^{n}, \Phi^{n}_s \bigr)$ on the right-hand sides is replaced by $\bigl( \calA_{x,s}^{n, (\ell-1)}, \calA_0^{n, (\ell-1)}, \Phi^{n, (\ell-1)}_s \bigr)$, and with initial data for the scalar field
 \begin{equation*}
  \Phi^{n, (\ell)}_s[0] = (T_n \phi_0^\omega, T_n \phi_1^\omega) - \Phi^n_r[0].
 \end{equation*}

 We first derive the bound 
 \begin{equation} \label{equ:induction_step_prop_first_iterate_bound}
  \begin{aligned}
   &\| \Phi^{n,(1)}_s \|_{S^1[n]} + \|\calA_{x,s}^{n, (1)}\|_{S^1[n]} + \|\calA_0^{n,(1)}\|_{Y^1[n]} \\ 
   &\qquad \leq C_1 \bigl( \| \Phi^n_s[0] \|_{\dot{H}^1_x \times L^2_x} + \|\Phi^n_r\|_{R_n} + \|\calA_{x,r}^n\|_{R_n} + \|T_n \phi^\omega[0]\|_{H^{1-\delta_\ast}_x \times H^{-\delta_\ast}_x} + \calE \calC^n \bigr) 
  \end{aligned}
 \end{equation}
 for some absolute constant $C_1 > 0$ such that $C_1 \varepsilon \ll 1$.
 This bound follows from the nonlinear estimates in the next Subsection~\ref{subsec:nonlinear_estimates} by observing that, since $\Phi^{n,(0)}_s = 0$ and $\calA_{x,s}^{n,(0)} = 0$ vanish on the right-hand sides of the equations for $\Phi^{n,(1)}_s$, $\calA_{x,s}^{n,(1)}$, $\calA_0^{n,(1)}$, all (non-vanishing) multilinear terms have at least one copy of $\Phi^n_r$ or $\calA_{x,r}^n$ in one slot (while the forcing terms $A^{<n-1}$ or $\phi^{<n-1}$ sitting in one or more of the other slots just give an additional $\varepsilon$ of smallness). The term $\| \Phi^n_s[0] \|_{\dot{H}^1_x \times L^2_x}$ on the right-hand side of~\eqref{equ:induction_step_prop_first_iterate_bound} just comes from the initial data for $\Phi^{n,(1)}_s$, while the terms $\|T_n \phi^\omega[0]\|_{H^{1-\delta_\ast}_x \times H^{-\delta_\ast}_x}$ and $\calE \calC^n$ are a consequence of the bound provided by Proposition~\ref{prop:prob_renormalization_error_estimate} on the renormalization error term $\Box_{A^{<n-1}}^{p,mod} \Phi^n_r$ on the right-hand side of the equation for $\Phi^{n,(0)}_s$.
 In particular, by the assumptions in the statement of Proposition~\ref{prop:induction_step} the bound~\eqref{equ:induction_step_prop_first_iterate_bound} implies 
 \begin{equation}
  \| \Phi^{n,(1)}_s \|_{S^1[n]} + \|\calA_{x,s}^{n, (1)}\|_{S^1[n]} + \|\calA_0^{n,(1)}\|_{Y^1[n]} \leq C_1 \varepsilon.
 \end{equation}
 
 Next, we establish that for all $\ell \geq 2$ we have 
 \begin{equation} \label{equ:induction_step_prop_difference_bound}
  \begin{aligned}
   &\| \Phi^{n, (\ell)}_s - \Phi^{n, (\ell-1)}_s \|_{S^1[n]} + \|\calA_{x, s}^{n,(\ell)} - \calA_{x, s}^{n, (\ell-1)}\|_{S^1[n]} + \| \calA^{n, (\ell)}_0 - \calA^{n, (\ell-1)}_0 \|_{Y^1[n]} \\
   &\quad \leq C_\ast \Bigl( 2 C_1 \bigl( \| \Phi^n_s[0] \|_{\dot{H}^1_x \times L^2_x} + \|\Phi^n_r\|_{R_n} + \|\calA_{x,r}^n\|_{R_n} + \|\Psi_r^{n,\pm}\|_{S_n^{1-\delta_\ast}} \bigr) \Bigr)^\ell \leq C_\ast  (2 C_1 \varepsilon)^\ell
  \end{aligned}
 \end{equation}
 for some absolute constant $C_\ast \geq 1$.
 The proof of~\eqref{equ:induction_step_prop_difference_bound} proceeds inductively. The induction base case $\ell=1$ is provided by~\eqref{equ:induction_step_prop_first_iterate_bound}. To carry out the induction step $\ell-1 \to \ell$ we note that by summing up~\eqref{equ:induction_step_prop_first_iterate_bound} and~\eqref{equ:induction_step_prop_difference_bound} we may add to our induction hypothesis (as long as $C_1 \varepsilon \ll 1$ is sufficiently small) that for $\tilde{\ell} = 1, 2, \ldots, \ell-1$, it holds that 
 \begin{equation} \label{equ:induction_step_prop_additional_induction_assumption}
  \begin{aligned}
   &\| \Phi^{n,(\tilde{\ell})}_s \|_{S^1[n]} + \|\calA_{x,s}^{n, (\tilde{\ell})}\|_{S^1[n]} + \|\calA_0^{n,(\tilde{\ell})}\|_{Y^1[n]} \\
   &\qquad \leq 2 C_1 \bigl( \| \Phi^n_s[0] \|_{\dot{H}^1_x \times L^2_x} + \|\Phi^n_r\|_{R[n]} + \|\calA_{x,r}^n\|_{R[n]} + \|T_n \phi^\omega[0]\|_{H^{1-\delta_\ast}_x \times H^{-\delta_\ast}_x} + \calE \calC^n \bigr). 
  \end{aligned}
 \end{equation}
 The bound~\eqref{equ:induction_step_prop_difference_bound} follows from the nonlinear estimates in Subsection~\ref{subsec:nonlinear_estimates} by observing that the equations for $\Phi^{n, (\ell)}_s - \Phi^{n, (\ell-1)}_s$, $\calA^{n, (\ell)}_{x,s} - \calA^{n, (\ell-1)}_{x,s}$, and $\calA^{n, (\ell)}_0 - \calA^{n, (\ell-1)}_0$ have zero initial data and do not involve forcing terms that come up to linear order on the right-hand sides. Moreover, in all multilinear terms on the right-hand sides there will be at least one copy of $\Phi^{n,(\ell-1)}_s - \Phi^{n,(\ell-2)}_s$, $\calA^{n,(\ell-1)}_{x,s} - \calA^{n, (\ell-2)}_{x,s}$, or $\calA^{n, (\ell-1)}_0 - \calA^{n, (\ell-2)}_0$ in one slot, while the other slots at least give an additional $\varepsilon$ of smallness.
 
 Then~\eqref{equ:induction_step_prop_difference_bound} implies that the sequence $\bigl\{ ( \calA_{x,s}^{n, (\ell)}, \calA_0^{n, (\ell)}, \Phi^{n, (\ell)}_s ) \bigr\}_{\ell \geq 0}$ converges in $S^1 \times Y^1 \times S^1$ to a solution $(\calA^n_{x,s}, \calA^n_0, \Phi^n_s)$ to (fMKG-CG\textsubscript{n}). Moreover, assuming that $2C_1 \leq C_0$, from~\eqref{equ:induction_step_prop_additional_induction_assumption} it follows that this solution satisfies
 \begin{align*}
  &\| \Phi^{n}_s \|_{S^1[n]} + \|\calA_{x,s}^{n}\|_{S^1[n]} + \|\calA_0^{n}\|_{Y^1[n]} \\
  &\qquad \leq C_0 \bigl( \| \Phi^n_s[0] \|_{\dot{H}^1_x \times L^2_x} + \|\Phi^n_r\|_{R_n} + \|\calA_{x,r}^n\|_{R_n} + \|T_n \phi^\omega[0]\|_{H^{1-\delta_\ast}_x \times H^{-\delta_\ast}_x} + \calE \calC^n \bigr).
 \end{align*}
\end{proof}

\subsection{The main nonlinear estimates} \label{subsec:nonlinear_estimates}

In this subsection we establish all estimates for the source terms that appear in the forced MKG-CG system of equations (fMKG-CG\textsubscript{n}), $n \geq 1$. The derivations rely on the frequency-localized multilinear estimates in~\cite[Section 12]{KST} and their generalized versions established in Section~\ref{sec:multilinear_estimates}, which allow for rough inputs satisfying redeeming bounds.

We begin with the source terms of the $A_x$ equation.
\begin{proposition}[The $A_x$ equation] \label{prop:nonlinear_est_Ax}
 For arbitrary $n \geq 1$ the following estimates hold
 \begin{align}
   &\sup_{k\in\bbZ} \, 2^{+\delta_2 |k-n|} \bigl\| P_k \Box \bigl( \bA_x^2(\phi^{<n}, \phi^{<n}) - \bA_x^2(\phi^{<n-1}, \phi^{<n-1}) \bigr) \bigr\|_{N_k \cap L^2_t \dot{H}^{-\frac12}_x}   \label{equ:nonlinear_est_Ax2} \\
   &\quad \lesssim \bigl( \|\Phi^n_s\|_{S^1[n]} + \|\Phi^n_r\|_{R_n} \bigr) \biggl( \sum_{m=0}^n \|\Phi^m_s\|_{S^1[m]} + \sum_{m=1}^n \|\Phi^m_r\|_{R_m} \biggr) \nonumber \\
   &\sup_{k\in\bbZ} \, 2^{+\delta_2 |k-n|} \bigl\| P_k \Box \bigl( \bA_x^3(\phi^{<n}, \phi^{<n}, A_x^{<n}) - \bA_x^3(\phi^{<n-1}, \phi^{<n-1}, A_x^{<n-1}) \bigr) \bigr\|_{N_k \cap L^2_t \dot{H}^{-\frac12}_x}   \label{equ:nonlinear_est_Ax3} \\
   &\quad \lesssim \bigl( \|\Phi^n_s\|_{S^1[n]} + \|\calA_{x,s}^n\|_{S^1[n]} + \|\Phi^n_r\|_{R_n} + \|\calA_{x,r}^n\|_{R_n} \bigr) \times \nonumber \\
   &\quad \quad \quad \times \biggl( \sum_{m=0}^n \|\Phi^m_s\|_{S^1[m]} + \sum_{m=0}^n \|\calA_{x,s}^m\|_{S^1[m]} + \sum_{m=1}^n \|\Phi^m_r\|_{R_m} + \sum_{m=1}^n \|\calA_{x,r}^m\|_{R_m} \biggr). \nonumber
 \end{align}
\end{proposition}
\begin{proof}
We provide the details for the proof of~\eqref{equ:nonlinear_est_Ax2}. The proof of~\eqref{equ:nonlinear_est_Ax3} for the simpler cubic terms is analogous, only that it just relies on the core generic product estimates from Lemma~\ref{lem:core_generic_product_est}. Correspondingly, we omit the details for~\eqref{equ:nonlinear_est_Ax3}. 
 
The proof of~\eqref{equ:nonlinear_est_Ax2} can be further reduced to the following two bilinear estimates
\begin{align}
 \sup_{k \in \bbZ} \, 2^{+\delta_2 |k-n|} \bigl\| P_k \Box \bA_x^2( \Phi^n, \Phi^n ) \bigr\|_{N_k \cap L^2_t \dot{H}^{-\hf}_x} &\lesssim \bigl( \|\Phi^n_s\|_{S^1[n]} + \|\Phi^n_r\|_{R_n} \bigr)^2, \label{equ:nonlinear_est_Ax2_one} \\
 \sup_{k \in \bbZ} \, 2^{+\delta_2 |k-n|} \bigl\| P_k \Box \bA_x^2( \Phi^n, \phi^{<n-1} ) \bigr\|_{N_k \cap L^2_t \dot{H}^{-\hf}_x} &\lesssim \bigl( \| \Phi^n_s \|_{S^1[n]} + \|\Phi^n_r\|_{R_n} \bigr) \biggl( \sum_{m=0}^{n-1} \| \Phi^m_s \|_{S^1[m]} + \sum_{m=1}^{n-1} \|\Phi^m_r\|_{R_m} \biggr). \label{equ:nonlinear_est_Ax2_two}
\end{align}
Their proofs are analogous and so we focus on~\eqref{equ:nonlinear_est_Ax2_one}, beginning with the $N_k$ bounds. Recall that the $\Box \bA_x^2$ term has the favorable null form structure
\begin{equation*}
 \calP_j \bigl( \Phi^n \overline{\nabla_x \Phi^n} \bigr) = \partial^k \Delta^{-1} \calN_{kj}\bigl( \Phi^n, \overline{\Phi^n} \bigr).
\end{equation*}
Here we have to distinguish smooth-smooth, smooth-rough, and rough-rough interactions. These cases can all be dealt with analogously using the generalized bilinear null form estimate~\eqref{equ:generalization_131_KST} and keeping in mind that the rough component $\Phi_r^n$ is sharply localized to frequencies $\sim 2^n$, while the smooth component $\Phi^n_s$ is localized to frequencies $|\xi| \sim 2^n$ up to exponential tails captured by the $S^1[n]$ norm. We also recall the smallness relations $0 < \delta_2 \ll \delta \ll 1$, which are important here to close the estimates.

\medskip 

\noindent {\bf smooth-smooth interactions:}
Let $k \in \bbZ$ be arbitrary. In the high-low case we obtain from the bilinear null form estimate~\eqref{equ:generalization_131_KST} that
\begin{align*}
 2^{+\delta_2 |k-n|} \sum_{k_2 \leq k-C} \bigl\| P_k \Box \bA_x^2( P_k \Phi^n_s, P_{k_2} \Phi^n_s ) \bigr\|_{N_k} &\lesssim \sum_{k_2 \leq k-C} 2^{-\delta(k-k_2)} \bigl( 2^{+\delta_2 |k-n|} \|P_k \Phi^n_s\|_{S^1_k} \bigr) \|P_{k_2} \Phi^n_s \|_{S^1_{k_2}} \\
 &\lesssim \|\Phi^n_s\|_{S^1[n]}^2.
\end{align*}
The low-high case is the same and in the high-high case we find that
\begin{align*}
 &2^{+\delta_2|k-n|} \sum_{k_1 \geq k+O(1)} \sum_{k_2 = k_1 + O(1)} \bigl\| P_k \Box \bA_x^2( P_{k_1} \Phi^n_s, P_{k_2} \Phi^n_s ) \bigr\|_{N_k} \\
 &\lesssim 2^{+\delta_2|k-n|} \sum_{k_1 \geq k+O(1)} \sum_{k_2 = k_1 + O(1)} 2^{-\delta(k_1-k)} \| P_{k_1} \Phi^n_s \|_{S^1_{k_1}} \|P_{k_2} \Phi^n_s \|_{S^1_{k_2}} \\
 &\lesssim 2^{+\delta_2|k-n|} \sum_{k_1 \geq k+O(1)} 2^{-\delta(k_1-k)} 2^{-\delta_2 |n-k_1|} \|\Phi^n_s\|_{S^1[n]}^2 \\
 &\lesssim \sum_{k_1 \geq k+O(1)} 2^{-(\delta-\delta_2)(k_1-k)} \|\Phi^n_s\|_{S^1[n]}^2 \\
 &\lesssim \|\Phi^n_s\|_{S^1[n]}^2.
\end{align*}

\medskip 

\noindent {\bf smooth-rough interactions:}
If $k \geq n+C$, only the following high-low interaction is possible
\begin{align*}
 2^{+\delta_2(k-n)} \bigl\| P_k \Box \bA_x^2( P_k \Phi^n_s, P_n \Phi^n_r ) \bigr\|_{N_k} &\lesssim 2^{+\delta_2(k-n)} 2^{-\delta(k-n)} \|P_k \Phi^n_s\|_{S^1_k} \|\Phi^n_r\|_{R_n} \\
 &\lesssim \|\Phi^n_s\|_{S^1[n]} \|\Phi^n_r\|_{R_n}.
\end{align*}
Instead if $k \leq n+C$, we may have the high-high interaction 
\begin{align*}
 2^{+\delta_2(n-k)} \bigl\| P_k \Box \bA_x^2( P_n \Phi^n_s, P_n \Phi^n_r ) \bigr\|_{N_k} &\lesssim 2^{+\delta_2(n-k)} 2^{-\delta(n-k)} \|P_n \Phi^n_s\|_{S^1_n} \|\Phi_r^n\|_{R_n} \\
 &\lesssim \|\Phi^n_s\|_{S^1[n]} \|\Phi^n_r\|_{R_n}, 
\end{align*}
and specifically for $k = n + \calO(1)$ the low-high interaction 
\begin{align*}
 \sum_{k_1 \leq n-C} \bigl\| P_n \Box \bA_x^2( P_{k_1} \Phi^n_s, P_n \Phi^n_r ) \bigr\|_{N_k} &\lesssim \sum_{k_1 \leq n-C} 2^{-\delta(n-k_1)} \|P_{k_1} \Phi^n_s\|_{S^1_{k_1}} \|\Phi^n_r\|_{R_n} \lesssim \|\Phi^n_s\|_{S^1[n]} \|\Phi^n_r\|_{R_n}.
\end{align*}

\medskip 

\noindent {\bf rough-rough interactions:}
Due to the sharp frequency localization of the rough components, here the output frequency must satisfy $k \leq n +C$ and only the following high-high interaction is possible
\begin{align*}
 2^{+\delta_2 (n-k)} \bigl\| P_k \Box \bA_x^2( P_n \Phi^n_r, P_n \Phi^n_r ) \bigr\|_{N_k} &\lesssim 2^{+\delta_2(n-k)} 2^{-\delta(n-k)} \|\Phi^n_r\|_{R_n} \|\Phi^n_r\|_{R_n} \lesssim \|\Phi^n_r\|_{R_n}^2.
\end{align*}

\medskip 

The high-modulation bounds for~\eqref{equ:nonlinear_est_Ax2_one} can be proved similarly using the core generic product estimates from Lemma~\ref{lem:core_generic_product_est}.
\end{proof}

The next proposition treats the source terms of the $A_0$ equation.
\begin{proposition}[The $A_0$ equation] \label{prop:nonlinear_est_A0}
For any $n \geq 1$ we have that 
\begin{equation} \label{equ:nonlinear_est_A0}
 \begin{aligned}
  &\sup_{k \in \bbZ} \, 2^{+\delta_2|k-n|} \, \bigl\| P_k \bigl( \bA_0\bigl( \phi^{<n}, \phi^{<n}, A^{<n} \bigr) - \bA_0\bigl( \phi^{<n-1}, \phi^{<n-1}, A^{<n-1} \bigr) \bigr) \bigr\|_{Y^1} \\
  &\quad \lesssim \bigl( \|\Phi^n_s\|_{S^1[n]} + \|\calA_{x,s}^n\|_{S^1[n]} + \|\calA^n_0\|_{Y^1[n]} + \|\Phi^n_r\|_{R_n} + \|\calA_{x,r}^n\|_{R_n} \bigr) \times \\
  &\quad \quad \quad \times \biggl( \sum_{m=0}^n \|\Phi^m_s\|_{S^1[m]} + \sum_{m=0}^n \|\calA_{x,s}^m\|_{S^1[m]} + \sum_{m=0}^n \|\calA_0^m\|_{Y^1[m]} + \sum_{m=1}^n \|\Phi^m_r\|_{R_m} + \sum_{m=1}^n \|\calA_{x,r}^m\|_{R_m} \biggr)^2.
 \end{aligned}
\end{equation}
\end{proposition}
\begin{proof}
 The proof of~\eqref{equ:nonlinear_est_A0} proceeds analgously to the one of Proposition~\ref{prop:nonlinear_est_A0}, using the generalized core product estimates from Lemma~\ref{lem:core_generic_product_est}, see also Subsection~4.2 in~\cite{KST} and~\cite[Proposition~4.2]{KST}. 
\end{proof}

\noindent {\bf The $\phi$ equation.}
We now turn to the heart of the matter, namely the magnetic wave equation for the scalar field. The derivation of a priori bounds for $\Phi^n_s$ hinges on the following linear estimate for the inhomogeneous magnetic wave equation $\Box_{A^{<n-1}}^p \phi = F$. We recall the definition of the paradifferential magnetic wave operator
\begin{equation*}
 \Box^p_{A^{<n-1}} := \Box + 2i \sum_k P_{\leq k-C} \bigl( A^{<n-1, j}_r + \calA^{0,free, j}_s \bigr) \partial_j P_k,
\end{equation*}
where $A^{<n-1}_{x,r} = \sum_{m=1}^{n-1} \calA_{x,r}^m$. In Proposition~\ref{prop:main_linear_estimate} in Section~\ref{sec:renormalization} we use a ``deterministic'' parametrix construction to establish the following main linear estimate that provides the link between the $S^1$ and $N$ spaces.

\medskip 

\noindent {\it Main linear estimate: Let $n \geq 1$ and assume that
\begin{equation*}
 \sum_{m=1}^{n-1} \|\calA_{x,r}^m\|_{R_m} + \|\calA_{x,s}^{0, free}\|_{S^1[0]} \leq 2 C_0 \varepsilon. 
\end{equation*}
For any $(f, g) \in \dot{H}^1_x \times L^2_x$ and any $F \in N \cap \ell^1 L^2_t \dot{H}^{-\frac12}_x$, there exists a unique global solution to the linear magnetic wave equation $\Box_{A^{<n-1}}^p \phi = F$ with initial data $\phi[0] = (f,g)$ and it holds that
\begin{equation} \label{equ:main_linear_estimate_section_phi_equation}
 \|\phi\|_{S^1} \lesssim \|f\|_{\dot{H}^1_x} + \|g\|_{L^2_x} + \| F \|_{N \cap \ell^1 L^2_t \dot{H}^{-\hf}_x}.
\end{equation}
}

We now turn to estimating the terms on the right-hand side of the equation for $\Phi^n_s$ in (fMKG-CG\textsubscript{n}). Here we first consider the most delicate terms, namely the error term $\Box_{A^{<n-1}}^{p, mod} \Phi^n_r$ and the low-high magnetic interaction term $-2i \sum_k P_{\leq k-C} \calA^{n,\alpha} \partial_\alpha P_k \phi^{<n}$, where the high-frequency input $\phi^{<n}$ can be rough. Subsequently, we will describe how to deal with the other terms on the right-hand side of the equation for $\Phi^n_s$.  

The error term $\Box_{A^{<n-1}}^{p, mod} \Phi^n_r$ acts as a forcing term (at linear order) on the right-hand side of the equation for~$\Phi^n_s$ and correspondingly has to be sufficiently small in $N_n \cap L^2_t \dot{H}^{-\frac12}_x$ to run the induction step global existence Proposition~\ref{prop:induction_step}. We defer the treatment of $\Box_{A^{<n-1}}^{p, mod} \Phi^n_r$ to Subsection~\ref{subsec:prob_renormalization_error} since it relies on the precise definition of the adapted rough linear evolution $\Phi^n_r$. 
There we show in Proposition~\ref{prop:prob_renormalization_error_estimate} that the desired smallness follows if $\|T_n \phi^\omega[0]\|_{H^{1-\delta_\ast}_x \times H^{-\delta_\ast}_x}$ and the redeeming error control quantity $\calE \calC^n$ are sufficiently small, which in turn gives rise to the corresponding smallness assumptions in the statement of the induction step global existence Proposition~\ref{prop:induction_step}.
 
The next proposition provides the bounds for the low-high interaction term $-2i \sum_k P_{\leq k-C} \calA^{n,\alpha} \partial_\alpha P_k \phi^{<n}$.

\begin{proposition} \label{prop:nonlinear_est_lowhigh_phi_less_n}
 Let $n \geq 1$ be arbitrary. Assume that $\calA_{x,s}^n$ and $\calA^n_0$ are given by
 \begin{align}
  \calA^n_{x,s} &= \bA_x( \phi^{<n}, \phi^{<n}, A^{<n} ) - \bA_x(\phi^{<n-1}, \phi^{<n-1}, A^{<n-1}), \label{equ:nonlinear_est_lowhigh_defAx} \\
  \calA_0^n &= \bA_0(\phi^{<n}, \phi^{<n}, A_0^{<n}) - \bA_0(\phi^{<n-1}, \phi^{<n-1}, A_0^{<n-1}). \label{equ:nonlinear_est_lowhigh_defA0}
 \end{align}
 Then it holds that
 \begin{align}
  &\sup_{k\in\bbZ} \, 2^{+\delta_2 |n-k|} \bigl\| P_{\leq k-C} \calA^{n,\alpha} \partial_\alpha P_k \phi^{<n} \bigr\|_{N_k \cap L^2_t \dot{H}^{-\frac12}_x} \nonumber \\
  &\quad \lesssim \Bigl( \|\calA_{x,s}^n\|_{S^1[n]} + \|\calA_{x,r}^n\|_{R_n} + \|\calA_0^n\|_{Y^1[n]} \Bigr) \biggl( \sum_{m=0}^{n} \|\Phi^m_s\|_{S^1[m]} + \sum_{m=1}^{n} \|\Phi^m_r\|_{R_m} \biggr) \label{equ:nonlinear_est_lowhighAnphinm1}  \\
  &\quad \quad \quad + \Bigl( \|\Phi^n_s\|_{S^1[n]} + \|\Phi^n_r\|_{R_n} + \|\calA^n_{x,s}\|_{S^1[n]} + \|\calA_0^n\|_{Y^1[n]} + \|\calA^n_{x,r}\|_{R_n} \Bigr) \times \nonumber \\ 
  &\quad \quad \quad \quad \quad \times \biggl( \sum_{m=0}^n \Bigl( \|\Phi^m_s\|_{S^1[m]} + \|\calA_{x,s}^m\|_{S^1[m]} + \|\calA_0^m\|_{Y^1[m]} \Bigr) + \sum_{m=1}^n \Bigl( \|\Phi^m_r\|_{R_m} + \|\calA_{x,r}\|_{R_m} \Bigr) \biggr)^2. \nonumber
 \end{align}
\end{proposition}
Observe that one of the bounds on the right-hand side of~\eqref{equ:nonlinear_est_lowhighAnphinm1} is in fact cubic. The reason is that when $\calA^n$ is smooth, certain parts of this low-high interaction can only be bounded by exploiting a subtle trilinear null form structure, which emerges upon reinserting the equations~\eqref{equ:nonlinear_est_lowhigh_defAx}--\eqref{equ:nonlinear_est_lowhigh_defA0} for $\calA^n_{x,s}$ and $\calA^n_0$.
\begin{proof}
The high-modulation bound is as usual a consequence of the generalized generic product estimates from Lemma~\ref{lem:core_generic_product_est}, so it suffices to discuss the delicate $N_k$ bound. 
We further expand this low-high term into
\begin{equation} \label{equ:magn_int_expand}
 \begin{aligned}
  P_{\leq k-C} \calA^{n,\alpha} \partial_\alpha P_k \phi^{<n} = P_{\leq k-C} \calA^{n,j} \partial_j P_k \phi^{<n} + P_{\leq k-C} \calA^n_0 \partial_t P_k \phi^{<n}.
 \end{aligned}
\end{equation}
Then we have to distinguish different types of interactions depending on whether $\calA^{n}_x$ and $\phi^{<n}$ are rough or smooth. Recall that the entire temporal component $\calA^n_0$ is smooth in the sense that it belongs to the space~$Y^1[n]$. 

\medskip 

\noindent {\bf rough-rough interactions:} 
Due to the sharp frequency localization of the rough evolution $\calA^n_{x,r}$ to frequencies $|\xi| \sim 2^n$ and of $\phi^{<n}_r$ to frequencies $|\xi| \leq 2^{n-1} + \calO(1)$, the first term on the RHS of~\eqref{equ:magn_int_expand} in fact vanishes $P_{\leq k-C} \calA^{n,j}_r \partial_j P_k \phi^{<n}_r = 0$, $k \in \bbZ$, in the case of rough-rough interactions.

\medskip 

\noindent {\bf rough-smooth interactions:} 
In view of the sharp frequency localization of the rough evolution $\calA^n_{x,r}$ to frequencies $|\xi| \sim 2^n$, we have $P_{\leq k-C} \calA^{n,j}_r \partial_j P_k \phi^{<n}_s = 0$ for $k-C < n$ and 
\begin{equation*}
 P_{\leq k-C} \calA^{n,j}_r \partial_j P_k \phi^{<n}_s = P_n \calA^{n,j}_r \partial_j P_k \phi^{<n}_s \quad \text{ for } k \geq n+C.
\end{equation*}
In the latter case we use the generalized null form estimate~\eqref{equ:generalization_132_KST} (and the fact that $\calA^n_{x,r}$ is just the free wave evolution of the random data $(T_n a^\omega, T_n b^\omega)$) to obtain for any $k > n+C$ that 
\begin{align*}
 \bigl\| P_n \calA^{n,j}_r \partial_j P_k \phi^{<n}_s \bigr\|_{N_k} &\simeq \bigl\| \calN \bigl( \Delta^{-1} \nabla \calA^n_{x,r}, P_k \phi^{<n}_s \bigr) \bigr\|_{N_k} \\
 &\lesssim 2^n 2^{-n} \|\calA_{x,r}^n\|_{R_n} \|P_k \phi^{<n}_s\|_{S^1_k} \\
 &\lesssim \|\calA_{x,r}^n\|_{R_n} \biggl( \sum_{m=0}^{n} 2^{-\delta_2|k-m|} \| \Phi^m_s \|_{S^1[m]} \biggr) \\
 &\lesssim 2^{-\delta_2|k-n|} \|\calA_{x,r}^n\|_{R_n} \biggl( \sum_{m=0}^{n} \| \Phi^m_s \|_{S^1[m]} \biggr).
\end{align*}

\medskip 

\noindent {\bf smooth-smooth interactions:} 
This case can be handled by proceeding as in Steps 3--5 in~\cite[Subsection 4.3]{KST} using the generalized multilinear estimates from Section~\ref{sec:multilinear_estimates}. 
The idea is to first ``peel off'' the good parts of $P_{\leq k-C} \calA^{n,j}_s \partial_j P_k \phi^{<n}_s$ and of $P_{\leq k-C} \calA^n_0 \partial_t P_k \phi^{<n}_s$ that can be handled by bilinear (null form) estimates). Then we reinsert the equations for $\calA^n_{x,s}$ and $\calA^n_0$ to unveil the core trilinear null forms, which can be bounded using the generalized trilinear null form estimates~\eqref{equ:generalization_136_KST}--\eqref{equ:generalization_138_KST}. Since we implement a version of this procedure in the treatment of the ``strongly low-high'' case of the smooth-rough interactions below, we omit the details for the smooth-smooth interactions.

\medskip 

\noindent {\bf smooth-rough interactions:} 
Due to the sharp frequency localization of $\phi^{<n}_r$ to frequencies $1 \lesssim |\xi| \lesssim 2^{n}$, we may assume that $0 \leq k \leq n$. Then we distinguish \emph{``moderately low-high interactions''} $P_{[(1-\gamma) k, k-C]} \calA^{n,\alpha} \partial_\alpha P_k\phi^{<n}$ and \emph{``strongly low-high interactions''} $P_{\leq (1-\gamma)k} \calA^{n,\alpha} \partial_\alpha P_k \phi^{<n}$ where the small constant $0 < \gamma \ll 1$ is chosen sufficiently small (depending on the size of $\delta_\ast$ and $\sigma$). 

\medskip 

\noindent \underline{\it ``Moderately low-high interactions'': the bound for $P_{[(1-\gamma) k, k-C]} \calA^{n,j}_s \partial_j P_k \phi^{<n}_r$}: 
We decompose this term schematically according to the various possibilities for the modulations, assuming for now that we are in the ``elliptic situation'' where all modulations are less than $(1-\gamma) k$
\begin{equation} \label{equ:nonlinear_est_moderately_lh_expand}
\begin{aligned}
P_{[(1-\gamma) k, k-C]} \calA^{n,j}_s P_k \partial_j \phi^{<n}_r &= \sum_{l\leq (1-\gamma) k} Q_{<l-10} \big( P_{[(1-\gamma)k, k-C]} Q_l \calA^{n,j}_s \partial_j P_k Q_{<l-10} \phi^{<n}_r \big) \\
&\quad + \sum_{l \leq (1-\gamma)k} Q_{l} \big( P_{[(1-\gamma)k, k-10]} Q_{<l-10} \calA^{n, j}_s \partial_j P_k Q_{<l-10} \phi^{<n}_r \big) \\
&\quad + \sum_{l \leq (1-\gamma)k}Q_{<l-10}\big(P_{[(1-\gamma)k, k-10]} Q_{<l-10} \calA^{n,j}_s \partial_j P_k Q_{l} \phi^{<n}_r \big).
\end{aligned}
\end{equation}
Observe that in all cases the angle between the inputs may be localized to size $\sim 2^{\frac{l-k_1}{2}}$, where $(1-\gamma) k \leq k_1 \leq k-C$ is the low frequency of the input $\calA^{n,j}_s$. Then using the null form, we can bound the first term on the right by 
\begin{align*}
&\sum_{l\leq (1-\gamma) k} \bigl\| Q_{<l-10} \big( P_{[(1-\gamma)k, k-C]} Q_l \calA^{n,j}_s \partial_j P_k Q_{<l-10} \phi^{<n}_r \big) \bigr\|_{L^1_t L^2_x} \\
&\quad \lesssim \sum_{l\leq (1-\gamma) k} \sum_{k_1 = (1-\gamma)k}^{k-C} \sum_{\kappa}2^{\frac{l-k_1}{2}}  2^{k}\big\|P_{k_1,\kappa} Q_l \calA^n_{x,s} \big\|_{L^2_t L^2_x} \big\|P_{k,\kappa}Q_{<l-10}\phi^{<n}_r\big\|_{L_t^2 L_x^\infty}.
\end{align*}
Thanks to the redeeming bound and a simple application of Bernstein's inequality 
\begin{equation*}
 \biggl( \sum_{\kappa} \big\|P_{k,\kappa}Q_{<l-10} \phi^{<n}_r \big\|_{L_t^2 L_x^\infty}^2 \biggr)^{\frac12} \lesssim 2^{-\frac{k}{2+}}\cdot 2^{(0+)(l-k)}\cdot \| P_k \phi^{<n}_r \|_{R_k}, 
\end{equation*}
we obtain by Cauchy-Schwarz in $\kappa$ that 
\begin{align*}
 &\sum_{l\leq (1-\gamma) k} \sum_{k_1 = (1-\gamma)k}^{k-C} 2^{\frac{l-k_1}{2-}}  2^{k} \big\|P_{k_1} Q_l \calA^n_{x,s} \big\|_{L^2_t L^2_x} 2^{-\frac{k}{2+}} \| P_k \phi^{<n}_r \|_{R_k} \\
 &\quad \lesssim \sum_{l\leq (1-\gamma) k} \sum_{k_1 = (1-\gamma)k}^{k-C} 2^{\frac{l-k_1}{2-}} 2^{k-k_1} 2^{-\frac12 l} \biggl( 2^{k_1} \sup_{l'} 2^{\frac12 l'}  \big\|P_{k_1} Q_{l'} \calA^n_{x,s} \big\|_{L^2_t L^2_x} \biggr) 2^{-\frac{k}{2+}} \| P_k \phi^{<n}_r \|_{R_k}.
\end{align*}
We can then complete the preceding estimate by Cauchy-Schwarz
\begin{align*}
 &\lesssim \sum_{k_1 = (1-\gamma)k}^{k-10}2^{-\delta_2|k_1-n|} 2^{k-k_1} 2^{-\frac{k}{2+}} \big\|\calA^n_{x,s}\big\|_{S^1[n]} \biggl( \sum_{m=0}^n \bigl\| \Phi^m_r \big\|_{R_m} \biggr) \\
 &\lesssim 2^{-\delta_2|k-n|}  \big\|\calA^n_{x,s} \big\|_{S^1[n]} \biggl(\sum_{m=0}^n \bigl\| \Phi^m_r\big\|_{R_m} \biggr),
\end{align*}
provided $\gamma+\delta_*<\frac12$. The bound if $l\geq (1-\gamma)k$ is similar except that one no longer needs to localize to angular sectors. 

The remaining terms above are bounded similarly. For example, consider the last of the three terms in~\eqref{equ:nonlinear_est_moderately_lh_expand}, again assuming $l \leq (1-\gamma)k$, i.e. 
\begin{align*}
\sum_{l < (1-\gamma)n} Q_{<l-10}\big(P_{[k(1-\gamma), k-10]}Q_{<l-10}\calA^{n,j}_s\partial^j P_kQ_{l}\phi^{<n}_r\big) 
& = \sum_{l < (1-\gamma)n} \sum_{k_1 = (1-\gamma)k}^{k-10} Q_{<l-10}\big(P_{k_1} Q_{<l-10} \calA^{n,j}_s \partial^j P_kQ_{l}\phi^{<n}_r\big).
\end{align*}
Here we place the large frequency term $P_nQ_{l}\phi^{<n}_r$ into $L_t^2 L_x^3$, and the low frequency term into $L_t^2 L_x^6$. Precisely, we may localize the factors to caps $\kappa_{1,2}$ of radius $\sim 2^{\frac{l-k_1}{2}}$ and either aligned or anti-aligned,
\begin{align*}
&\sum_{l<(1-\gamma)n}\sum_{k_1 = (1-\gamma)k}^{k-10} Q_{<l-10}\big(P_{k_1}Q_{<l-10}\calA^{n,j}_s \partial^j P_kQ_{l}\phi^{<n}_r\big) \\&=  \sum_{l<(1-\gamma)n}\sum_{k_1 = (1-\gamma)k}^{k-10}\sum_{\kappa_1\sim \pm \kappa_2}  Q_{<l-10}\big(P^{\kappa_1}P_{k_1}Q_{<l-10}\calA^{n,j}_s \partial^j P^{\kappa_2}P_kQ_{l}\phi^{<n}_r\big),\\
\end{align*}
and using interpolation we have the bound 
\begin{align*}
 \big(\sum_{\kappa_2}\big\|\partial^j P^{\kappa_2}P_kQ_{l}\phi^{<n}_r\big\|_{L_t^2 L_x^3}^2\big)^{\frac12} &\lesssim \big(2^{1-\frac{k}{2+}}\big)^{\frac13}\cdot \big(2^{\delta_*k - \frac{l}{2}}\big)^{\frac23}\cdot \big\|P_k\phi^{<n}_r\big\|_{R_k}\\
 \big(\sum_{\kappa_1}\big\|P^{\kappa_1}P_{k_1}Q_{<l-10}\calA^{n,j}_s\big\|_{L_t^2 L_x^6}^2\big)^{\frac12} &\lesssim 2^{-\delta_2|k_1-n|}\cdot 2^{-\frac{k_1}{6}}\cdot \big\|\calA^{n,j}_s\big\|_{S^1[n]}
\end{align*}
Observe that the null form gains $2^{\frac{l-k_1}{2}}$, and we infer that 
\begin{align*}
 &\biggl\|\sum_{l<(1-\gamma)n}\sum_{k_1 = (1-\gamma)k}^{k-10}\sum_{\kappa_1\sim \pm \kappa_2}  Q_{<l-10}\big(P^{\kappa_1}P_{k_1}Q_{<l-10}\calA^{n,j}_s \partial^j P^{\kappa_2}P_kQ_{l}\phi^{<n}_r\big) \biggr\|_{L_t^1 L_x^2}\\
 &\lesssim \sum_{l<(1-\gamma)n}\sum_{k_1 = (1-\gamma)k}^{k-10} 2^{\frac{l-k_1}{2}}\cdot \big(2^{1-\frac{k}{2+}}\big)^{\frac13}\cdot \big(2^{\delta_*k - \frac{l}{2}}\big)^{\frac23}\cdot 2^{-\delta_2|k_1-n|}\cdot 2^{-\frac{k_1}{6}} \big\|\calA^{n,j}_s\big\|_{S^1[n]} \big\|P_k\phi^{<n}_r\big\|_{R_k} \\
 &\lesssim 2^{-\delta_2|k-n|} \big\|\calA^n_{x,s} \big\|_{S^1[n]} \biggl( \sum_{m=0}^{n-1} \bigl\| \Phi^m_r\big\|_{R_m} \biggr).
\end{align*}
The case of large modulations $l>(1-\gamma)k$ is again handled in a simpler fashion, without having to take recourse to angular decompositions. 

\medskip 

\noindent \underline{\it ``Moderately low-high interactions'': the bound for $P_{[(1-\gamma) k, k-C]} \calA^{n}_0 \partial_t P_k \phi^{<n}_r$}: 
In view of the frequency localization of $\phi^{<n}_r$ to $1 \lesssim |\xi| \lesssim 2^{n}$, we may assume that $0 \leq k \leq n$. Then we easily obtain that 
\begin{align*}
 \bigl\| P_{[(1-\gamma) k, k-C]} \calA^{n}_0 \partial_t P_k \phi^{<n}_r \bigr\|_{L^1_t L^2_x} &\lesssim \sum_{(1-\gamma) k \leq k_1 \leq k-C} \| P_{k_1} \calA^n_0 \|_{L^2_t L^2_x} \| P_k \partial_t \phi^{<n}_r \|_{L^2_t L^\infty_x} \\
 &\lesssim \sum_{(1-\gamma) k \leq k_1 \leq k-C} 2^{-\frac{3}{2} k_1} \| P_{k_1} \calA^n_0 \|_{Y^1} 2^k 2^{-\frac{1}{2+}k} \bigl( 2^{-k} \|P_k \partial_t \phi^{<n}_r\|_{R_k} \bigr) \\
 &\lesssim \sum_{(1-\gamma) k \leq k_1 \leq k-C} 2^{-\frac{3}{2} k_1} 2^{-\delta_2|k_1-n|} \|\calA^n_0\|_{Y^1[n]} 2^k 2^{-\frac{1}{2+}k} \bigl( 2^{-k} \|P_k \partial_t \phi^{<n}_r\|_{R_k} \bigr) \\
 &\lesssim 2^{-\delta_2 |k-n|} \|\calA^n_0\|_{Y^1[n]}\bigl( 2^{-k} \|P_k \partial_t \phi^{<n}_r\|_{R_k} \bigr).
\end{align*}

\medskip 

\noindent \underline{\it ``Strongly low-high interactions'':} Here the idea is to proceed analogously to the treatment of the low-high interactions in~\cite[Subsection 4.3]{KST}. We first peel off the ``good parts'' of $P_{\leq (1-\gamma)k} \calA^{n,j}_s \partial_j P_k \phi^{<n}_r$ and of $P_{\leq (1-\gamma)k} \calA^{n}_0 \partial_t P_k \phi^{<n}_r$, using generalized bilinear null form estimates. Then we reinsert the equations for $\calA^n_{x,s}$ and $\calA^n_0$ to unveil the crucial trilinear null forms, which we bound using the generalized trilinear null form estimates~\eqref{equ:generalization_136_KST}--\eqref{equ:generalization_138_KST}. However, in this approach we can only place the rough high-frequency term $P_k \phi^{<n}_r$ in $S^1_k$. This costs $2^{\delta_\ast k}$ below, which we can then compensate (for suitable choice of $\gamma$ depending on $\delta_\ast, \delta_2$) using the off-diagonal decay in all multilinear estimates combined with the off-diagonal decay of $P_{k_1} \calA^n_{x,s}$ and $P_{k_1} \calA^n_0$ (and $P_{k_1} \Phi^n_s$) away from frequency $|\xi| \sim 2^n$ and the fact that we are in the strongly low-high case $k_1 \leq (1-\gamma)k$, $k \geq 1$. We now turn to the details.

In what follows we try to closely mimic the notation in~\cite{KST}. To decompose the nonlinearity and to ``peel off its good parts'', it will be useful to introduce the following notation. For any bilinear operator $\calM(D_{t,x}, D_{t,y})$ we set 
\begin{align*}
 \calH_k \calM(\phi, \psi) &:= \sum_{j < k+C} Q_j P_k \calM \bigl( Q_{<j-C} \phi, Q_{<j-C} \psi \bigr), \\
 \calH_k^\ast \calM(\phi, \psi) &:= \sum_{j < k+C} Q_{<j-C} \calM \bigl(Q_j P_k \phi, Q_{<j-C} \psi \bigr).
\end{align*}
Moreover, we introduce the following short-hand notation for the ``strongly low-high'' interaction terms (for the spatial components of the connection form)
\begin{align*}
 \calN_{str}^{lh, k}\bigl(\calA^n_{x,s}, \phi_r^{<n} \bigr) &:= P_{\leq (1-\gamma) k} \calA_s^{n, j} \partial_j P_k \phi^{<n}_r \\
 &= \sum_{i \neq j} \calN_{ij} \bigl( \nabla_i \Delta^{-1} P_{\leq (1-\gamma) k} \calA_{j,s}^n, P_k \phi^{<n}_r \bigr),
\end{align*}
where in the last line we recalled the null structure of this interaction term in the Coulomb gauge. In order to decompose the ``strongly low-high'' interaction term for the temporal component of the connection form, we introduce the short-hand notation
\begin{equation*}
 \calN_{0, str}^{lh, k} \bigl(\calA_0^n, \phi^{<n}_r) := P_{\leq (1-\gamma) k} Q_{<k-C} \calA^n_0 \partial_t P_k \phi^{<n}_r.
\end{equation*}

We can bound $\calN_{str}^{lh, k}$ via a bilinear null form estimate for the most part, except for $high \times low \to low$ modulation interactions. We group the latter into an expression denoted by
\begin{equation} \label{equ:calHast_def}
 \calH^\ast \calN_{str}^{lh, k}\bigl(\calA^n_{x,s}, \phi_r^{<n} \bigr) := \sum_{k' \leq k-C} \calH^\ast_{k'} \calN_{str}^{lh, k}\bigl(\calA^n_{x,s}, \phi_r^{<n} \bigr).
\end{equation}
Then by the bilinear null form estimate (132) from~\cite{KST} (and recalling that $1 \leq k \leq n$) we obtain for the difference that
\begin{align*}
 &2^{+\delta_2 |k-n|} \bigl\| \calN_{str}^{lh, k}\bigl(\calA^n_{x,s}, \phi_r^{<n} \bigr) - \calH^\ast \calN_{str}^{lh, k}\bigl(\calA^n_{x,s}, \phi_r^{<n} \bigr) \bigr\|_{N_k} \\
 &\quad \lesssim 2^{+\delta_2 |k-n|} \sum_{k_1 \leq (1-\gamma) k} \|P_{k_1} \calA_{x,s}^n\|_{S^1_{k_1}} \|P_k \phi^{<n}_r\|_{S^1_k} \\
 &\quad \lesssim 2^{+\delta_2 |k-n|} \sum_{k_1 \leq (1-\gamma) k} 2^{-\delta_2 |k_1-n|} \|\calA_{x,s}^n\|_{S^1[n]} 2^{\delta_\ast k} \|P_k \phi^{<n}_r\|_{S^{1-\delta_\ast}_k} \\
 &\quad \lesssim 2^{-\delta_2 \gamma k} 2^{\delta_\ast k} \|\calA_{x,s}^n\|_{S^1[n]} \|P_k \phi^{<n}_r\|_{R_k},
\end{align*}
which is of the desired form as long as we choose $\delta_2 \gamma > \delta_\ast$.

The remaining term $\calH^\ast \calN_{str}^{lh, k}\bigl(\calA^n_{x,s}, \phi_r^{<n} \bigr)$ can for the most part be estimated using the stronger $Z$ norm except for the following delicate part of $\calA^n_{x,s}$ given by
\begin{equation*}
 \calH \calA_{j,s}^n := \sum_{\substack{k, k_i \\ k < \min\{k_1, k_2\}-C}} \calH_k \bigl( \bA^2_j( P_{k_1} \phi^{<n}, P_{k_2} \phi^{<n}) - \bA^2_j( P_{k_1} \phi^{<n-1}, P_{k_2} \phi^{<n-1}) \bigr).
\end{equation*}
Then by the bilinear null form estimate (133) from~\cite{KST} we obtain for the difference that
\begin{equation} \label{equ:nonlinear_est_lowhigh_pickup_Z}
 \begin{aligned}
  &2^{+\delta_2|k-n|} \bigl\| \calH^\ast \calN_{str}^{lh, k}\bigl(\calA^n_{x,s}, \phi_r^{<n} \bigr) - \calH^\ast \calN_{str}^{lh, k}\bigl( \calH \calA^n_{x,s}, \phi_r^{<n} \bigr) \bigr\|_{N_k} \\
  &\quad \lesssim 2^{+\delta_2 |k-n|} \sum_{k_1 \leq (1-\gamma) k} \bigl\| P_{k_1} \bigl( \calA^n_{x,s} - \calH \calA_{x,s}^n \bigr) \bigr\|_{Z} 2^{\delta_\ast k} \|P_k \phi_r^{<n}\|_{S^{1-\delta_\ast}_k}.
 \end{aligned}
\end{equation}
In order to bound $\bigl\| P_{k_1} \bigl( \calA^n_{x,s} - \calH \calA_{x,s}^n \bigr) \bigr\|_{Z}$ we use the generalized bilinear null form estimate~\eqref{equ:generalization_134_KST} for the quadratic contributions to $\calA^n_{x,s}$ and the generalized generic product estimates from Lemma~\ref{lem:core_generic_product_est} for the cubic contributions to $\calA_{x,s}^n$. Since at least one of the inputs for the difference $\calA^n_{x,s} - \calH \calA_{x,s}^n$ must be $\Phi^n$ or $\calA^n$, their localization around frequency $\sim 2^n$ combined with the off-diagonal decay of all multilinear estimates involved as well as the ``strongly low-high'' separation, allows to compensate the factor $2^{\delta_\ast k}$ (as long as $\delta_2 \gamma > \delta_\ast$). In this manner we can bound the right-hand side of~\eqref{equ:nonlinear_est_lowhigh_pickup_Z} by
\begin{align*}
 &\bigl( \|\Phi^n_{s}\|_{S^1[n]} + \|\Phi^n_r\|_{R_n} + \|\calA^n_{x,s}\|_{S^1[n]} + \|\calA^n_{x,r}\|_{R_n} \bigr) \times \\
 &\qquad \qquad \qquad \times \biggl( \sum_{m=0}^n \Bigl( \|\Phi^m_s\|_{S^1[m]} + \|\calA_{x,s}^m\|_{S^1[m]} \Bigr) + \sum_{m=1}^n \Bigl( \|\Phi^m_r\|_{R_m} + \|\calA_{x,r}\|_{R_m} \Bigr) \biggr) \|P_k \phi^{<n}_r\|_{R_k}. 
\end{align*}

At this point we are left to bound the term $\calH^\ast \calN_{str}^{lh, k}\bigl( \calH \calA^n_{x,s}, \phi_r^{<n} \bigr)$. This only turns out to be possible after exploiting cancellations that occur by combining it with an analogous contribution from the low-high interactions $\calN_{0,str}^{lh,k}(\calA_0^n, \phi^{<n}_r)$ involving the temporal component of the connection form. Proceeding as in Step~4 of Subsection~4.3 in~\cite{KST} and using the generalized generic product estimates from Lemma~\ref{lem:core_generic_product_est} (as well as the more microlocal generalized product estimate~\eqref{equ:generalization_141_KST}), we may peel off the good parts of $\calN_{0, str}^{lh,k}$ in a similar manner as above, until we are left with the following part
\begin{equation*}
 \calH^\ast \calN_{0,str}^{lh,k} \bigl( \calH \calA_0^n, \phi^{<n}_r \bigr),
\end{equation*}
where 
\begin{align*}
 \calH^\ast \calN_{0, str}^{lh,k} \bigl( \calH \calA_0^n, \phi^{<n}_r \bigr) &:= \sum_{k' < k-C} \calH^\ast_{k'} \calN_{0, str}^{lh,k} \bigl( \calH \calA_0^n, \phi^{<n}_r \bigr), \\
 \calH \calA_0^n &:= \sum_{\substack{k, k_i \\ k < \min\{k_1, k_2\}-C}} \calH_k \bigl( \bA^2_0(P_{k_1} \phi^{<n}, P_{k_2} \phi^{<n}) - \bA^2_0(P_{k_1} \phi^{<n-1}, P_{k_2} \phi^{<n-1}) \bigr).
\end{align*}

Finally, we collect the portion of $\calN_{str}^{lh,k}$ and the portion of $\calN_{0, str}^{lh, k}$ that have not been estimated yet, and combine them into the expression
\begin{equation*}
 - \calH^\ast \calN_{str}^{lh, k}\bigl( \calH \calA^n_{x,s}, \phi_r^{<n} \bigr) + \calH^\ast \calN_{0,str}^{lh,k} \bigl( \calH \calA_0^n, \phi^{<n}_r \bigr),
\end{equation*}
which exhibits a striking trilinear null structure. As detailed in the appendix of~\cite{KST}, we may write
\begin{equation*}
 -\bA^2_j \bigl( \phi^{(1)}, \phi^{(2)} \bigr) \partial_j \phi^{(3)} + \bA^2_0 \bigl( \phi^{(1)}, \phi^{(2)} \bigr) \partial_t \phi^{(3)} = \bigl( \calQ_1 + \calQ_2 + \calQ_3 \bigr)(\phi^{(1)}, \phi^{(2)}, \phi^{(3)})
\end{equation*}
with 
\begin{align*}
 \calQ_1\bigl( \phi^{(1)}, \phi^{(2)}, \phi^{(3)} \bigr) &= - \Box^{-1} \Im \bigl( \phi^{(1)} \partial_\alpha \phi^{(2)} \bigr) \cdot \partial^\alpha \phi^{(3)}, \\
 \calQ_2\bigl( \phi^{(1)}, \phi^{(2)}, \phi^{(3)} \bigr) &= \Delta^{-1} \Box^{-1} \partial_t \partial_\alpha \Im \bigl( \phi^{(1)} \partial^\alpha \phi^{(2)} \bigr) \cdot \partial_t \phi^{(3)}, \\
 \calQ_3\bigl( \phi^{(1)}, \phi^{(2)}, \phi^{(3)} \bigr) &= \Delta^{-1} \Box^{-1} \partial_\alpha \partial^j \Im \bigl( \phi^{(1)} \partial_j \phi^{(3)} \bigr) \cdot \partial^\alpha \phi^{(3)}.
\end{align*}
Using the generalized trilinear null form estimates~\eqref{equ:generalization_136_KST}--\eqref{equ:generalization_138_KST}, we then obtain for $1 \leq k \leq n$ the desired bound 
\begin{equation*}
 \begin{aligned}
  &2^{+\delta_2 |k-n|} \bigl\| P_k \bigl( -\calH^\ast \calN_{str}^{lh, k}\bigl( \calH \calA^n_{x,s}, \phi_r^{<n} \bigr) + \calH^\ast \calN_{0,str}^{lh,k} \bigl( \calH \calA_0^n, \phi^{<n}_r \bigr) \bigr) \bigr\|_{N_k} \\
  &\quad \lesssim \bigl( \|\Phi^n_s\|_{S^1[n]} + \|\Phi^n_r\|_{R_n} \bigr) \biggl( \sum_{m=0}^n \|\Phi^m_s\|_{S^1[m]} + \sum_{m=1}^n \|\Phi^m_r\|_{R_m} \biggr) \|P_k \phi^{<n}_r\|_{S^{1-\delta_\ast}_k}. 
 \end{aligned}
\end{equation*}
Here it is again crucial that while we place $P_k \phi^{<n}_r$ into $S^1_k$ at a loss of a factor $2^{\delta_\ast k}$, we can compensate this loss using the off-diagonal decay in the trilinear estimates along with the fact that at least one of the first two inputs of the trilinear expressions must be $\Phi^n$, which is localized around frequency $\sim 2^n$. 
\end{proof}

It remains to describe how to deal with the other terms on the right-hand side of the equation for~$\Phi^n_s$ in (fMKG-CG\textsubscript{n}). The ``moderately low-high'' interaction term $- 2i P_{[(1-\gamma)n, n-C]} A^{<n-1, \alpha} \partial_\alpha P_n \Phi^n_r$ with the rough component $\Phi^n_r$ at high frequency and the entire connection form $A^{<n-1}$ from prior induction stages at low frequency can be estimated analogously to the ``moderately low-high'' interactions in the proof of Proposition~\ref{prop:nonlinear_est_lowhigh_phi_less_n}. 
\begin{proposition}
For arbitrary $n \geq 1$ it holds that
\begin{equation*} 
  \bigl\| P_{[(1-\gamma)n, n-C]} A^{<n-1, \alpha} \partial_\alpha P_n \Phi^n_r \bigr\|_{N_n \cap L^2_t \dot{H}^{-\frac12}_x} \lesssim \biggl( \sum_{m=0}^{n-1} \bigl( \|\calA_{x,s}^m\|_{S^1[m]} + \|\calA_0^m\|_{Y^1[m]} \bigr) + \sum_{m=1}^{n-1} \|\calA_{x,r}^m\|_{R_m} \biggr) \|\Phi^n_r\|_{R_n}. 
\end{equation*} 
\end{proposition}

One more low-high interaction term appears on the right-hand side of the equation for $\Phi^n_s$ that only involves smooth-smooth interactions with the smooth component $\Phi^n_s$ at high frequency and all smooth components of the connection form $A^{<n-1}$ from prior induction stages at low frequency. The treatment of this low-high interaction term essentially exactly follows the approach in Subsection 4.3 of~\cite{KST}. First the good parts are peeled off using bilinear estimates, and then the equations for $A^{<n-1}_{x,s}$ and $A^{<n-1}_0$ are inserted to unveil the crucial trilinear null forms. (This scheme was basically detailed in the treatment of the ``strongly low-high'' smooth-rough interactions in Proposition~\ref{prop:nonlinear_est_lowhigh_phi_less_n} above, although there additional work is needed to compensate the derivative loss when placing the rough component $\phi^{<n}_r$ at high frequency into the critical space $S^1$.)
\begin{proposition}
 Let $n \geq 1$. Assume that $A^{<n-1}_{x, s} - \calA^{0,free}_{x,s}$ and $A^{<n-1}_0$ are given by
  \begin{align*}
  A^{<n-1}_{x, s} - \calA^{0,free}_{x,s} &= \bA_x( \phi^{<n-1}, \phi^{<n-1}, A^{<n-1}), \\
  A^{<n-1}_0 &= \bA_0( \phi^{<n-1}, \phi^{<n-1}, A_0^{<n-1} ).
 \end{align*} 
 Then it holds that
 \begin{equation*}
  \begin{aligned}
   &\sup_{k \in \bbZ} \, 2^{+\delta_2|k-n|} \bigl\| P_{\leq k-C} \bigl( A^{<n-1, j}_s - \calA^{0,free, j}_s \bigr) \partial_j P_k \Phi^n_s + P_{\leq k-C} A^{<n-1}_0 \partial_t P_k \Phi^n_s \bigr\|_{N_k \cap L^2_t \dot{H}^{-\frac12}_x} \\
   &\quad \lesssim \biggl( \sum_{m=0}^{n-1} \Bigl( \|\calA_{x,s}^m\|_{S^1[m]} + \|\calA_0^m\|_{Y^1[m]} \Bigr) + \sum_{m=1}^{n-1} \|\calA_{x,r}^m\|_{R_m} \biggr) \|\Phi^n_s\|_{S[n]} \\
   &\quad \quad \quad + \biggl( \sum_{m=0}^{n-1} \Bigl( \|\Phi^m_s\|_{S^1[m]} + \|\calA_{x,s}^m\|_{S^1[m]} + \|\calA_0^m\|_{Y^1[m]} \Bigr) + \sum_{m=1}^{n-1} \Bigl( \|\Phi^m_r\|_{R_m} + \|\calA_{x,r}\|_{R_m} \Bigr) \biggr)^2 \|\Phi^n_s\|_{S[n]}.
  \end{aligned}
 \end{equation*}
\end{proposition}

Finally, we dispose of the easier multilinear terms $\bM^1$, $\bM^2$, and $\bM^3$. 
\begin{proposition}
For any $n \geq 1$ we have 
 \begin{align}
  &\sup_{k \in \bbZ} \, 2^{+\delta_2 |k-n|} \bigl\| P_k \bigl( \bM^1(A^{<n}, \phi^{<n}) - \bM^1(A^{<n-1}, \phi^{<n-1}) \bigr) \bigr\|_{N_k \cap L^2_t \dot{H}^{-\frac12}_x} \nonumber \\
  &\quad \lesssim \bigl( \|\Phi^n_s\|_{S^1[n]} + \|\calA_{x,s}^n\|_{S^1[n]} + \|\calA^n_0\|_{Y^1[n]} + \|\Phi^n_r\|_{R_n} + \|\calA_{x,r}^n\|_{R_n} \bigr) \times \label{equ:nonlinear_est_M1} \\
  &\quad \quad \quad \times \biggl( \sum_{m=0}^n \Bigl( \|\Phi^m_s\|_{S^1[m]} + \|\calA_{x,s}^m\|_{S^1[m]} + \|\calA^m_0\|_{Y^1[m]} \Bigr) + \sum_{m=1}^n \Bigl( \|\Phi^m_r\|_{R_m} + \|\calA_{x,r}^m\|_{R_m} \Bigr) \biggr) \nonumber \\
  &\sup_{k \in \bbZ} \, 2^{+\delta_2 |k-n|} \bigl\| P_k \bigl( \bM^2(A^{<n}, \phi^{<n}) - \bM^2(A^{<n-1}, \phi^{<n-1}) \bigr) \bigr\|_{N_k \cap L^2_t \dot{H}^{-\frac12}_x} \nonumber \\
  &\quad \lesssim \bigl( \|\Phi^n_s\|_{S^1[n]} + \|\calA^n_0\|_{Y^1[n]} + \|\Phi^n_r\|_{R_n} \bigr) \biggl( \sum_{m=0}^n \Bigl( \|\Phi^m_s\|_{S^1[m]} + \|\calA^m_0\|_{Y^1[m]} \Bigr) + \sum_{m=1}^n \|\Phi^m_r\|_{R_m} \biggr) \label{equ:nonlinear_est_M2} \\
  &\sup_{k \in \bbZ} \, 2^{+\delta_2 |k-n|} \bigl\| P_k \bigl( \bM^3(A^{<n}, A^{<n}, \phi^{<n}) - \bM^3(A^{<n-1}, A^{<n-1}, \phi^{<n-1}) \bigr) \bigr\|_{N_k \cap L^2_t \dot{H}^{-\frac12}_x} \nonumber \\
  &\quad \lesssim \bigl( \|\Phi^n_s\|_{S^1[n]} + \|\calA_{x,s}^n\|_{S^1[n]} + \|\calA^n_0\|_{Y^1[n]} + \|\Phi^n_r\|_{R_n} + \|\calA_{x,r}^n\|_{R_n} \bigr) \times \label{equ:nonlinear_est_M3} \\
  &\quad \quad \quad \times \biggl( \sum_{m=0}^n \Bigl( \|\Phi^m_s\|_{S^1[m]} + \|\calA_{x,s}^m\|_{S^1[m]} + \|\calA^m_0\|_{Y^1[m]} \Bigr) + \sum_{m=1}^n \Bigl( \|\Phi^m_r\|_{R_m} + \|\calA_{x,r}^m\|_{R_m} \Bigr) \biggr)^2. \nonumber
 \end{align}
\end{proposition}
\begin{proof}
 The high-modulation bounds for~\eqref{equ:nonlinear_est_M1}--\eqref{equ:nonlinear_est_M3} all follow readily using the generalized generic product estimates from Lemma~\ref{lem:core_generic_product_est}, and so it remains to discuss the $N_k$ bounds.
 The corresponding proof of~\eqref{equ:nonlinear_est_M1} for the $\bM^1_x$ component follows from the generalized core bilinear null form estimate~\eqref{equ:generalization_131_KST}, while the proof for the $\bM^1_0$ component follows analogously to the proof of the estimate (56) in~\cite{KST}, using the generalized generic product estimates from Lemma~\ref{lem:core_generic_product_est}. Finally, the corresponding bounds for~\eqref{equ:nonlinear_est_M2} and \eqref{equ:nonlinear_est_M3} just rely on Strichartz-type estimates and Sobolev embeddings, and therefore follow using the generalized generic product estimates from Lemma~\ref{lem:core_generic_product_est}. 
\end{proof}

\section{The ``deterministic'' parametrix} \label{sec:renormalization}

The goal of this section is to establish a key linear estimate for the linear magnetic wave equation $\Box_{A^{<n-1}}^p \phi = F$, which establishes a link between the $S^1$ and $N$ spaces. 
We recall that the paradifferential magnetic wave operator $\Box_{A^{<n-1}}^p$ is given by
\begin{equation*}
 \Box^p_{A^{<n-1}} := \Box + 2i \sum_k P_{\leq k-C} \bigl( A^{<n-1, j}_r + \calA^{0,free, j}_s \bigr) \partial_j P_k.
\end{equation*}
All results in this section are deterministic in the sense that they hold as long as $A^{<n-1}_{x,r}$ and $\calA^{0,free}_x$ satisfy suitable smallness assumptions.
\begin{proposition}[Main linear estimate for $\phi$ equation] \label{prop:main_linear_estimate}
Let $n \geq 1$ and assume that
\begin{equation*}
 \sum_{m=1}^{n-1} \|\calA_{x,r}^m\|_{R_m} + \|\calA_{x,s}^{0, free}\|_{S^1[0]} \lesssim \varepsilon. 
\end{equation*}
Then for any $(f, g) \in \dot{H}^1_x \times L^2_x$ and any $F \in N \cap \ell^1 L^2_t \dot{H}^{-\frac12}_x$, there exists a unique global solution to the linear magnetic wave equation $\Box_{A^{<n-1}}^p \phi = F$ with initial data $\phi[0] = (f,g)$ and it holds that
\begin{equation} 
 \|\phi\|_{S^1} \lesssim \|f\|_{\dot{H}^1_x} + \|g\|_{L^2_x} + \| F \|_{N \cap \ell^1 L^2_t \dot{H}^{-\hf}_x}.
\end{equation}
\end{proposition}

The proof of Proposition~\ref{prop:main_linear_estimate} proceeds as in~\cite{KST}. We first define an approximate solution via a parametrix construction. Then we obtain an exact solution satisfying the desired linear estimate by iterating away the error. To this end we build approximate solutions $\phi_{app,k}$ at each spatial frequency $k \in \bbZ$ to the frequency localized problems 
\begin{equation} \label{equ:approx_sol_magn_wave_freq_k}
 \bigl( \Box + 2i P_{\leq k-C} ( A^{<n-1, j}_r + \calA_{s}^{0, free, j} )\partial_j P_k \bigr) \phi = P_k F, \quad P_k \phi[0] = (P_k f, P_k g)
\end{equation}
and assemble these to a full approximate solution $\phi_{app} := \sum_{k \in \bbZ} \phi_{app, k}$. The approximate solution $\phi_{app, k}$ at frequency $k \in \bbZ$ is essentially defined as
\begin{equation} \label{equ:definition_phi_app_k}
 \begin{aligned}
  \phi_{app, k}(t,x) &:= \frac{1}{2} \sum_\pm e^{-i \psi_{\pm}^{n,<k}}_{<k-C}(t,x,D) \frac{e^{\pm i t |D|}}{i|D|} e^{+i\psi_{\pm}^{n, <k}}_{<k-C}(D,y,0) \bigl( i |D| P_k f \pm P_k g \bigr) \\
  &\quad \, \pm \frac{1}{2} \sum_\pm e^{-i\psi_{\pm}^{n, <k}}_{<k-C}(t,x,D) \frac{K^\pm}{i|D|} e^{+i\psi_{\pm}^{n, <k}}_{<k-C}(D,y,s) P_k F,
 \end{aligned}
\end{equation}
where the phase $\psi_{\pm}^{n, <k}(t,x,\xi)$ is defined in Subsection~\ref{subsec:det_phase_function} below and where $K^\pm G$ are the Duhamel terms
\begin{equation*}
 K^\pm G(t) = \int_0^t e^{\pm i (t-s)|D|} G(s) \, \ud s.
\end{equation*}
The renormalization operators $e^{-i \psi_{\pm}^{n,<k}}_{<k-C}(t,x,D)$ and $e^{+i\psi_{\pm}^{n, <k}}_{<k-C}(D,y,s)$ denote the left and right quantization of the symbol $e^{+i\psi_{\pm}^{n, <k}}_{<k-C}(t,x, \xi)$, where the subscript $<k-C$ denotes space-time $(t,x)$-frequency localization to frequencies $\leq k-C$, pointwise in $\xi$. 

The definition of the phase function $\psi_\pm^{n,<k}(t,x,\xi)$ in Subsection~\ref{subsec:det_phase_function} below is the exact analogue of the corresponding definition of the phase function introduced in~\cite[Section 6]{KST}. However, here we build the rough free wave evolution $A_{x,r}^{<n-1}$ into the phase that does not belong to the critical $S^1$ space and only enjoys the redeeming spacetime bounds of the $R_k$ spaces.
Despite the different (redeeming) bounds on these rough components of the phase function, the construction from~\cite{KST} turns out to (largely) go through. 
Following Section 6 in~\cite{KST} the proof of Proposition~\ref{prop:main_linear_estimate} reduces to establishing the following mapping properties of the frequency-localized renormalization operators $e^{\pm i \psi^{n,<k}_{\pm}}_{<k-C}(t,x,D)$, which are the same as the ones in~\cite{KST}.
\begin{proposition} \label{prop:det_renormalization_mapping_properties} 
Let $n \geq 1$ and assume that
\begin{equation*}
 \sum_{m=1}^{n-1} \|\calA_{x,r}^m\|_{R_m} + \|\calA_{x,s}^{0, free}\|_{S^1[0]} \lesssim \varepsilon. 
\end{equation*}
For every $k \in \bbZ$ the frequency-localized renormalization operators $e^{\pm i \psi_{\pm}^{n, <k}}_{<k-C}(t,x,D)$ have the following mapping properties with $Z \in \{ N_k, L^2, N_k^\ast \}$:
\begin{align}
&e^{\pm i \psi_{\pm}^{n, <k}}_{<k-C}(t,x,D) \colon Z \longrightarrow Z, \\
&\partial_t e^{\pm i \psi_{\pm}^{n, <k}}_{<k-C}(t,x,D) \colon Z \longrightarrow \varepsilon Z, \\
&e^{-i \psi_{\pm}^{n, <k}}_{<k-C}(t,x,D) e^{+i \psi_{\pm}^{n, <k}}_{<k-C}(D, y, s) - I \colon Z \longrightarrow \varepsilon Z, \\
&e^{-i \psi_{\pm}^{n, <k}}_{<k-C}(t,x,D) \Box - \Box^p_{A^{<n-1}_{<k}} e^{-i \psi_{\pm}^{n, <k}}_{<k-C}(t,x,D) \colon N_{k,\pm}^\ast \longrightarrow \varepsilon N_{k,\pm}, \label{equ:det_renormalization_error_est} \\
&e^{-i \psi_{\pm}^{n, <k}}_{<k-C}(t,x,D) \colon S_k^{\sharp} \longrightarrow S_k. 
\end{align}
\end{proposition}
The proof of Proposition~\ref{prop:det_renormalization_mapping_properties} proceeds exactly as in Sections~6--11 in~\cite{KST} once we have established certain pointwise and decomposable estimates for the ``deterministic'' phase functions $\psi_{\pm}^{n, <k}(t,x,\xi)$. This is accomplished in Subsection~\ref{subsec:det_bounds_phase_function} below. 
We remark that the proof of the conjugation estimate~\eqref{equ:det_renormalization_error_est} is essentially identical to the corresponding proof of the conjugation estimate (82) in~\cite{KST}, only that we use the redeeming $R L^2_t L^\infty_x$ norm for the rough evolution $A_{x,r}^{<n-1}$.

\subsection{The ``deterministic'' phase function} \label{subsec:det_phase_function}
We begin with a heuristic motivation for the choice of the phase function, see~\cite[Section 6]{KST} or \cite[Section 7]{RT} for a more detailed account. It is reasonable to expect that a linear magnetic wave equation of the form $(\Box + 2i A^j \partial_j) \phi = 0$ can be approximately conjugated to $\Box$ via some phase correction $e^{i\psi}$ with $\nabla \psi \approx A$. To define an approximate solution to $(\Box + 2i A^j \partial_j) \phi = 0$, let us therefore consider distorted waves of the form
\begin{equation*}
 \phi(t,x) = e^{-i\psi_\pm(t,x)} e^{\pm i t |\xi| + i x \cdot \xi}
\end{equation*}
and compute 
\begin{equation*}
 \begin{aligned}
  \bigl( \Box + 2i A^j \partial_j \bigr) \phi &= 2 \bigl( \pm |\xi| (\partial_t \psi_{\pm}) - \xi \cdot (\nabla_x \psi_\pm) + A_x \cdot \xi \bigr) \phi \\
  &\quad + \bigl( -2 A_x \cdot (\nabla_x \psi_\pm) + |\nabla_x \psi_\pm|^2 - i (\Box \psi_\pm) \bigr) \phi.
 \end{aligned}
\end{equation*}
While the terms in the second parenthesis can be expected to be error terms, we would ideally like to choose the phase correction $\psi_\pm$ so that the expression in the first parenthesis vanishes.
Introducing the differential operators
\begin{align*}
 L_\pm^{\eta} := \pm \partial_t + \eta \cdot \nabla_x, \quad \Delta_{\eta^\perp} := \Delta - (\eta \cdot \nabla_x)^2, \quad \eta := \frac{\xi}{|\xi|} \in \bbS^3,
\end{align*}
we may formulate this requirement more succinctly as
\begin{equation*}
 L_\mp^\eta \psi_\pm = A_x \cdot \frac{\xi}{|\xi|} = A_x \cdot \eta.
\end{equation*}
Applying $L_\pm^\eta$, noting that $L_\pm^\eta L_\mp^\eta = - \Box - \Delta_{\eta^\perp}$, and neglecting $\Box$ (since we may assume that $\Box A_x = 0$), we obtain for fixed $\xi$ that formally we would like to choose
\begin{equation*}
 \psi_\pm = - \Delta_{\eta^\perp}^{-1} L_\pm^\eta \bigl( A_x \cdot \eta \bigr), \quad \eta := \frac{\xi}{|\xi|}.
\end{equation*}
Unfortunately, this symbol is too singular due to the degeneracy of $\Delta_{\eta^\perp}^{-1}$ when $\phi$ and $A$ have parallel frequencies. 
Nevertheless, a viable choice is to smoothly cut off small angle interactions in the above expression for $\psi_\pm$ and to observe that the arising additional error terms turn out to be manageable, because one can gain from the small interaction angle. For general initial data $\int e^{ix \cdot \xi} \hat{f}(\xi) \, \ud \xi$, we obtain by linearity the approximate solution 
\begin{equation*}
 \phi(t,x) = \int e^{-i \psi_\pm(t,x,\xi)} e^{\pm i t |\xi| + i x \cdot \xi} \hat{f}(\xi) \, \ud \xi,
\end{equation*}
in other words we apply the pseudodifferential renormalization operator $e^{-i\psi_\pm}(t,x,D)$. 

\medskip 

Let us now turn to the exact choice of the phase correction for our magnetic wave operator $\Box_{A^{<n-1}}^p$, where $n \geq 1$ is arbitrary. In view of the above considerations, for every frequency $k \in \bbZ$ we are led to define the ``deterministic'' phase function by
\begin{align*}
 \psi_{\pm}^{n, <k}(t,x,\xi) := \sum_{0 \leq j \leq k-C} \psi_{\pm, j}^{n, <k, r}(t,x,\xi) + \sum_{j \leq k-C} \psi_{\pm, j}^{n, <k, s}(t,x,\xi), 
\end{align*}
where its rough part is defined as
\begin{align*}
 \psi_{\pm, j}^{n, <k, r}(t,x,\xi) :=& - L_\pm^\eta \Delta_{\eta^\perp}^{-1} \biggl( \Pi_{ > 2^{\sigma (j-k)}}^{\eta}  P_j A_{x, r}^{<n-1} \cdot \eta \biggr), \qquad \eta := \frac{\xi}{|\xi|} \in \bbS^{3},
\end{align*}
and its smooth part is defined as
\begin{align*}
 \psi_{\pm, j}^{n, <k, s}(t,x,\xi) :=& - L_\pm^\eta \Delta_{\eta^\perp}^{-1} \biggl( \Pi_{ > 2^{\sigma (j-k)}}^{\eta}  P_j \calA_{x, s}^{0, free} \cdot \eta \biggr), \qquad \eta := \frac{\xi}{|\xi|} \in \bbS^{3}.
\end{align*}

\subsection{Pointwise and decomposable estimates for the ``deterministic'' phase function} \label{subsec:det_bounds_phase_function}
First, we establish some $L^\infty$ bounds on the ``deterministic'' phase function that are used throughout this section.
To this end it is helpful to introduce some notation for the sector projection of $\psi_{\pm, j}^{n, <k}$ in frequency space for an angle $0 < \theta \lesssim 1$,
\begin{equation*}
 \psi_{\pm, j, (\theta)}^{n, <k}(t,x,\xi) = \bigl( \Pi_\theta^\eta \psi_{\pm, j}^{n, <k} \bigr)(t,x,\xi), \qquad \eta := \frac{\xi}{|\xi|}.
\end{equation*}

\begin{lemma} \label{lem:det_phase_function_Linfty_bounds}
 Let $n \geq 1$. For the rough part of the ``deterministic'' phase function we have for any $k \in \bbZ$, any $0 \leq j \leq k-C$, and for any $1 \gtrsim \theta > 2^{\sigma (j-k)}$ that
 \begin{align}
  \bigl| \psi_{\pm, j, (\theta)}^{n, <k, r}(t,x,\xi) \bigr| &\lesssim 2^{-(2-20\sigma)j} \theta^{-1-\delta_1} \min\{ (\theta 2^j)^{\frac{3}{2}-}, 1 \} \| P_j A_{x,r}^{<n-1} \|_{R_j}, \\
  \bigl| \psi_{\pm, j}^{n, <k, r}(t,x,\xi) \bigr| &\lesssim 2^{-(1-20\sigma-\delta_1)j} \| P_j A_{x,r}^{<n-1} \|_{R_j}, \\
  \bigl| \nabla_{t,x} \psi_{\pm, j, (\theta)}^{n, <k, r}(t,x,\xi) \bigr| &\lesssim 2^{-(1-20\sigma)j} \theta^{-1-\delta_1} \min\{ (\theta 2^j)^{\frac{3}{2}-}, 1 \} \| P_j A_{x,r}^{<n-1} \|_{R_j},  \\
  \bigl| \nabla_{t,x} \psi_{\pm, j}^{n, <k, r}(t,x,\xi) \bigr| &\lesssim 2^{+(20\sigma+\delta_1)j} \| P_j A_{x,r}^{<n-1} \|_{R_j}.
 \end{align}
 For the smooth part of the ``deterministic'' phase function we have for any $k \in \bbZ$, any $j \leq k-C$, and any $1 \gtrsim \theta > 2^{\sigma (j-k)}$ that
 \begin{align}
  \bigl| \psi_{\pm, j, (\theta)}^{n, <k, s}(t,x,\xi) \bigr| &\lesssim \theta^{\frac{1}{2}} 2^j \| P_j \calA_{x,s}^{0, free} \|_{L^\infty_t L^2_x}, \\
  \bigl| \psi_{\pm, j}^{n,<k, s}(t,x,\xi) \bigr| &\lesssim 2^j \| P_j \calA_{x,s}^{0, free} \|_{L^\infty_t L^2_x},  \\
  \bigl| \nabla_{t,x} \psi_{\pm, j, (\theta)}^{n,<k, s}(t,x,\xi) \bigr| &\lesssim \theta^{\frac{1}{2}} 2^j \| \nabla_{t,x} P_j \calA_{x,s}^{0, free}\|_{L^\infty_t L^2_x}, \\
  \bigl| \nabla_{t,x} \psi_{\pm, j}^{n, <k, s}(t,x,\xi) \bigr| &\lesssim 2^j \| \nabla_{t,x} P_j \calA_{x,s}^{0, free} \|_{L^\infty_t L^2_x}.
 \end{align}
 Finally, for derivatives of the rough part of the phase function with respect to the frequency variable we have for any multi-index $\alpha$ with $|\alpha| \geq 1$, any $l \geq 0$, any $k \in \bbZ$, any $0 \leq j \leq k-C$, and any $1 \gtrsim \theta > 2^{\sigma (j-k)}$ that
 \begin{align}
  \bigl| \partial_{|\xi|}^l \partial_\eta^\alpha \psi_{\pm, j, (\theta)}^{n,<k, r}(t,x,\xi) \bigr| &\lesssim \theta^{-1-|\alpha|} 2^{-(2-20\sigma)j} \| P_j A_{x,r}^{<n-1} \|_{R_j}, \\ 
  \bigl| \partial_{|\xi|}^l \partial_\eta^\alpha \psi_{\pm, j}^{n, <k, r}(t,x,\xi) \bigr| &\lesssim 2^{\sigma(1+|\alpha|)(k-j)} 2^{-(2-20\sigma)j} \| P_j A_{x,r}^{<n-1} \|_{R_j}, \\   
  \bigl| \partial_{|\xi|}^l \partial_\eta^\alpha \nabla_{t,x} \psi_{\pm, j, (\theta)}^{n, <k, r}(t,x,\xi) \bigr| &\lesssim \theta^{-1-|\alpha|} 2^{-(1-20\sigma)j} \| P_j A_{x,r}^{<n-1} \|_{R_j}, \\
  \bigl| \partial_{|\xi|}^l \partial_\eta^\alpha \nabla_{t,x} \psi_{\pm, j}^{n, <k, r}(t,x,\xi) \bigr| &\lesssim 2^{\sigma(1+|\alpha|)(k-j)} 2^{-(1-20\sigma)j} \| P_j A_{x,r}^{<n-1} \|_{R_j}.
 \end{align}
 Similarly, for derivatives of the smooth part of the phase function with respect to the frequency variable we have for any multi-index $\alpha$ with $|\alpha| \geq 1$, any $l \geq 0$, any $k \in \bbZ$, any $j \leq k-C$, and any $1 \gtrsim \theta > 2^{\sigma (j-k)}$ that
 \begin{align}
  \bigl| \partial_{|\xi|}^l \partial_\eta^\alpha \psi_{\pm, j, (\theta)}^{n, <k, s}(t,x,\xi) \bigr| &\lesssim \theta^{\frac{1}{2}-|\alpha|} 2^j \| P_j \calA_{x,s}^{0, free} \|_{L^\infty_t L^2_x}, \\ 
  \bigl| \partial_{|\xi|}^l \partial_\eta^\alpha \psi_{\pm, j}^{n, <k, s}(t,x,\xi) \bigr| &\lesssim 2^{\sigma (|\alpha|-\frac{1}{2}) (k-j)} 2^j \| P_j \calA_{x,s}^{0, free} \|_{L^\infty_t L^2_x}, \\ 
  \bigl| \partial_{|\xi|}^l \partial_\eta^\alpha \nabla_{t,x} \psi_{\pm, j, (\theta)}^{n,<k, s}(t,x,\xi) \bigr| &\lesssim \theta^{\frac{1}{2}-|\alpha|} 2^j \| \nabla_{t,x} P_j \calA_{x,s}^{0, free} \|_{L^\infty_t L^2_x}, \\
  \bigl| \partial_{|\xi|}^l \partial_\eta^\alpha \nabla_{t,x} \psi_{\pm, j}^{n,<k, s}(t,x,\xi) \bigr| &\lesssim 2^{\sigma (|\alpha|-\frac{1}{2}) (k-j)} 2^j \| \nabla_{t,x} P_j \calA_{x,s}^{0, free} \|_{L^\infty_t L^2_x}.
 \end{align}
\end{lemma}
\begin{proof}
For the rough part $\psi_{\pm, j, (\theta)}^{n,<k,r}$ of the ``deterministic'' phase function we use the Coulomb gauge condition to gain an additional factor of $\theta$ and the redeeming $L^\infty_t L^\infty_x$ norm with angular gains from our $R_j$ space to obtain that
\begin{align*}
 \bigl| \psi_{\pm, j, (\theta)}^{n, <k, r}(t,x,\xi) \bigr| &\lesssim \sup_{\eta} \, \bigl\| L_\pm^\eta \Delta_{\eta^\perp}^{-1} \Pi_\theta^\eta P_j A_{x,r}^{<n-1} \cdot \eta \bigr\|_{L^\infty_t L^\infty_x} \\
 &\lesssim \theta^{-2} 2^{-j} \sup_{\eta} \,  \bigl\| \Pi_\theta^\eta P_j A_{x,r}^{<n-1} \cdot \eta \bigr\|_{L^\infty_t L^\infty_x} \\
 &\lesssim \theta^{-1} 2^{-j} \bigl\| P_j A_{x,r}^{<n-1} \bigr\|_{L^\infty_t L^\infty_x} \\
 &\lesssim 2^{-(2-20\sigma)j} \theta^{-1-\delta_1} \min\{ (\theta 2^j)^{(\frac{3}{2}-)}, 1 \} \bigl\| P_j A_{x,r}^{<n-1} \bigr\|_{R_j}.
\end{align*}
Then the other bounds on the rough part of the phase function follow upon summing over the dyadic angles $1 \gtrsim \theta \gtrsim 2^{\sigma (j-k)}$ and upon taking an additional $\nabla_{t,x}$ derivative. 

Next, we turn to estimating the smooth part $\psi_{\pm, j, (\theta)}^{n,<k,s}$ of the ``deterministic'' phase function. We again exploit the Coulomb gauge condition to gain another factor of $\theta$ and then use the Bernstein estimate $\Pi_{\theta}^\eta P_j L^2_x \to (\theta^3 2^{4j})^{\frac{1}{2}} L^\infty_x$ to obtain that
\begin{align*}
 \bigl| \psi_{\pm, j, (\theta)}^{n, <k, s}(t,x,\xi) \bigr| &\lesssim \sup_{\eta} \, \Bigl\| L_\pm^\eta \Delta_{\eta^\perp}^{-1} \Pi_\theta^{\eta} \bigl( P_j \calA_{x,s}^{0, free} \cdot \eta \bigr) \Bigr\|_{L^\infty_t L^\infty_x} \\
 &\lesssim \theta^{-2} 2^{-j} \sup_{\eta} \, \bigl\| \Pi_\theta^{\eta} P_j \calA_{x,s}^{0, free} \cdot \eta \bigr\|_{L^\infty_t L^\infty_x} \\
 &\lesssim \theta^{-1} 2^{-j} \bigl\| \Pi_\theta^{\eta} P_j \calA_{x,s}^{0, free} \bigr\|_{L^\infty_t L^\infty_x} \\
 &\lesssim \theta^{\frac{1}{2}} 2^j \|P_j \calA_{x,s}^{0, free} \|_{L^\infty_t L^2_x}.
\end{align*}
The other bounds on the smooth part of the phase function then again follow upon summing over the dyadic angles $1 \gtrsim \theta \gtrsim 2^{\sigma (j-k)}$ and upon taking an additional $\nabla_{t,x}$ derivative. 
 
Finally, the estimates for $\partial_{|\xi|}^l \partial_\eta^\alpha$ derivatives of the phase function follow analogously, upon noting that differentiating with respect to $\eta := \frac{\xi}{|\xi|}$ yields additional $\theta^{-1}$ factors, while differentiating with respect to the radial frequency variable $|\xi|$ is harmless since the definition of the phase function only involves $\eta$. In the corresponding estimates for the rough part of the phase function it suffices to just use the standard redeeming $L^\infty_t L^\infty_x$ Strichartz norm, because an additional gain in the angle cannot ultimately compensate the additional $\theta^{-1}$ factors produced by differentiating with respect to the angular frequency variable.
\end{proof}

Next we establish $L^\infty$ bounds for differences of two ``deterministic'' phase functions. The following lemma is the analogue of Lemma~7.4 in~\cite{KST}.
\begin{lemma}[Additional symbol bounds for differences of ``deterministic'' phase functions]
 Let $n \geq 1$ and assume that
 \begin{equation*}
  \sum_{m=1}^{n-1} \|\calA_{x,r}^m\|_{R_m} + \|\calA_{x,s}^{0, free}\|_{S^1[0]} \lesssim \varepsilon. 
 \end{equation*}
 Then we have for any $k \in \bbZ$, any multi-index $\alpha$ with $1 \leq |\alpha| + 1 < \sigma^{-1}$, and any $l \geq 0$ that
 \begin{align}
  \bigl| \psi_{\pm}^{n, <k}(t,x,\xi) - \psi_{\pm}^{n, <k}(s,y,\xi) \bigr| &\lesssim \varepsilon \log \bigl( 1 + 2^k (|t-s| + |x-y|) \bigr), \label{equ:det_phase_function_difference_psis} \\
  \bigl| \partial_{|\xi|}^l \partial_\eta^\alpha \bigl( \psi_{\pm}^{n, <k}(t,x,\xi) - \psi_{\pm}^{n, <k}(s,y,\xi) \bigr) \bigr| &\lesssim \varepsilon \bigl( 1 + 2^k (|t-s| + |x-y|) \bigr)^{\sigma (|\alpha|+1) }. \label{equ:det_phase_function_difference_xi_derivative_psis}
 \end{align}
\end{lemma}
\begin{proof}
 In the following we use the shorthand notation $T := |x-y| + |t-s|$.
 We establish the first estimate~\eqref{equ:det_phase_function_difference_psis} separately for the rough and the smooth part of the phase function. 
 Recall that the rough component $A_{x,r}^{<n-1}$ is sharply localized to frequencies $1 \lesssim |\xi| \lesssim 2^{n-1}$. It therefore suffices to consider $k \geq 0$ for the rough part and we just bound by
 \begin{align*}
  \bigl| \psi_{\pm}^{n, <k, r}(t,x,\xi) - \psi_{\pm}^{n, <k, r}(s,y,\xi) \bigr| &\lesssim \sum_{0 \leq j \leq k-C} \sup_\eta \, \bigl\| \psi_{\pm, j}^{n, <k, r}(t,x,\xi) \bigr\|_{L^\infty_t L^\infty_x} \\
  &\lesssim \sum_{0 \leq j \leq k-C} 2^{-(1-20\sigma-\delta_1)j} \| P_j A_{x,r}^{<n-1} \|_{R_j} \\
  &\lesssim \varepsilon. 
 \end{align*}
 Instead, to bound the smooth part, for arbitrary $k \in \bbZ$ we pick some $j_0 \leq k-C$ and decompose into
 \begin{align*}
  &\bigl| \psi_{\pm}^{n, <k, s}(t,x,\xi) - \psi_{\pm}^{n, <k, s}(s,y,\xi) \bigr| \\
  &\quad \lesssim \sum_{j \leq j_0} \sup_{\eta} \, \bigl\| \nabla_{t,x} \psi_{\pm, j}^{n, <k, s} \bigr\|_{L^\infty_t L^\infty_x} \bigl( |t-s| + |x-y| \bigr) + \sum_{j_0 \leq j \leq k-C} \sup_\eta \, \bigl\| \psi_{\pm, j}^{n, <k, s}(t,x,\xi) \bigr\|_{L^\infty_t L^\infty_x} \\
  &\quad \lesssim \sum_{j \leq j_0} 2^j \| P_j \nabla_{t,x} \calA_{x,s}^{0, free} \|_{L^\infty_t L^2_x} T + \sum_{j_0 \leq j \leq k-C} 2^j \| P_j \calA_{x,s}^{0, free} \|_{L^\infty_t L^2_x} \\
  &\quad \lesssim 2^{j_0} T \varepsilon + \bigl| k - j_0 \bigr| \varepsilon.
 \end{align*}
 Then choosing $k-j_0 \sim \log_2(2^k T)$ yields the desired estimate.
 
 We also establish the second estimate~\eqref{equ:det_phase_function_difference_xi_derivative_psis} separately for the rough and the smooth part of the phase function. 
 For the rough part, we distinguish several cases. If $T \lesssim 2^{-k}$, we bound by
 \begin{align*}
  \bigl| \partial_{|\xi|}^l \partial_\eta^\alpha \bigl( \psi_{\pm}^{n,<k, r}(t,x,\xi) - \psi_{\pm}^{n,<k, r}(s,y,\xi) \bigr) \bigr| &\lesssim \sum_{0 \leq j \leq k-C} \sup_\eta \, \bigl\| \nabla_{t,x} \partial_{|\xi|}^l \partial_\eta^\alpha \psi_{\pm, j}^{n,<k,r} \bigr\|_{L^\infty_t L^\infty_x} \bigl( |t-s| + |x-y| \bigr) \\
  &\lesssim \sum_{0 \leq j \leq k-C} 2^{\sigma(1+|\alpha|)(k-j)} 2^{-(1-20\sigma)j} \| P_j A_{x,r}^{<n-1} \|_{R_j} T \\
  &\lesssim 2^{\sigma (|\alpha| + 1)k} T \varepsilon \\
  &\lesssim \varepsilon.
 \end{align*}
 Instead, if $2^{-k} \ll T \lesssim 1$, we pick some $0 \leq j_0 \leq k-C$ and decompose into
 \begin{align*}
  &\bigl| \partial_{|\xi|}^l \partial_\eta^\alpha \bigl( \psi_{\pm}^{n, <k, r}(t,x,\xi) - \psi_{\pm}^{n,<k, r}(s,y,\xi) \bigr) \bigr| \\
  &\lesssim \sum_{0 \leq j \leq j_0} \sup_{\eta} \, \bigl\| \nabla_{t,x} \partial_{|\xi|}^l \partial_\eta^\alpha \psi_{\pm, j}^{n,<k, r} \bigr\|_{L^\infty_t L^\infty_x} \bigl( |t-s| + |x-y| \bigr) + \sum_{j_0 \leq j \leq k-C} \sup_{\eta} \, \bigl\| \partial_{|\xi|}^l \partial_\eta^\alpha \psi_{\pm, j}^{n, <k, r} \bigr\|_{L^\infty_t L^\infty_x} \\
  &\lesssim \sum_{0 \leq j \leq j_0} 2^{\sigma(1+|\alpha|)(k-j)} 2^{-(1-20\sigma)j} \| P_j A_{x,r}^{<n-1} \|_{R_j} T + \sum_{j_0 \leq j \leq k-C} 2^{\sigma(1+|\alpha|)(k-j)} 2^{-(2-20\sigma)j} \| P_j A_{x,r}^{<n-1} \|_{R_j} \\
  &\lesssim 2^{\sigma(|\alpha|+1)k} T \varepsilon + 2^{\sigma(|\alpha|+1)(k-j_0)} \varepsilon \\
  &\simeq 2^{\sigma(|\alpha|+1)(k-j_0)} \bigl( 2^{j_0} T + 1 \bigr) \varepsilon.
 \end{align*} 
 Choosing $2^{-j_0} \sim T$, we obtain the desired estimate.
 Finally, if $T \gtrsim 1$, we just bound by
 \begin{align*}
  \bigl| \partial_{|\xi|}^l \partial_\eta^\alpha \bigl( \psi_{\pm}^{n,<k, r}(t,x,\xi) - \psi_{\pm}^{n,<k, r}(s,y,\xi) \bigr) \bigr| &\lesssim \sum_{0 \leq j \leq k-C} \sup_\eta \, \bigl\| \partial_{|\xi|}^l \partial_\eta^\alpha \psi_{\pm, j}^{n,<k,r} \bigr\|_{L^\infty_t L^\infty_x} \\
  &\lesssim \sum_{0 \leq j \leq k-C} 2^{\sigma(1+|\alpha|)(k-j)} 2^{-(2-20\sigma)j} \| P_j A_{x,r}^{<n-1} \|_{R_j} \\ 
  &\lesssim 2^{\sigma (|\alpha| + 1)k} \varepsilon \lesssim \bigl( 2^k T \bigr)^{\sigma (|\alpha| + 1)} \varepsilon.
 \end{align*}
 For the smooth part we pick some $j_0 \leq k-C$ and decompose into
 \begin{align*}
  &\bigl| \partial_{|\xi|}^l \partial_\eta^\alpha \bigl( \psi_{\pm}^{n,<k, s}(t,x,\xi) - \psi_{\pm}^{n,<k, s}(s,y,\xi) \bigr) \bigr| \\
  &\quad \lesssim \sum_{j \leq j_0} \sup_\eta \, \bigl\| \nabla_{t,x} \partial_{|\xi|}^l \partial_\eta^\alpha \psi_{\pm, j}^{n,<k, s} \bigr\|_{L^\infty_t L^\infty_x} \bigl( |t-s| + |x-y| \bigr) + \sum_{j_0 \leq j \leq k-C} \sup_{\eta} \, \bigl\| \partial_{|\xi|}^l \partial_\eta^\alpha \psi_{\pm, j}^{n,<k, s} \bigr\|_{L^\infty_t L^\infty_x} \\
  &\quad \lesssim \sum_{j \leq j_0} 2^{\sigma (|\alpha|-\frac{1}{2}) (k-j)} 2^j \| \nabla_{t,x} P_j \calA_{x,s}^{0, free} \|_{L^\infty_t L^2_x} T + \sum_{j_0 \leq j \leq k-C} 2^{\sigma (|\alpha|-\frac{1}{2}) (k-j)} 2^j \| P_j \calA_{x,s}^{0, free} \|_{L^\infty_t L^2_x}  \\
  &\quad \lesssim \sum_{j \leq j_0} 2^{\sigma(|\alpha|+\frac{1}{2})(k-j)} 2^j T \varepsilon + \sum_{j_0 \leq j \leq k-C} 2^{\sigma(|\alpha|+\frac{1}{2})(k-j)} \varepsilon \\
  &\quad \lesssim 2^{\sigma(|\alpha|+\frac{1}{2})(k-j_0)} \bigl( 2^{j_0} T + 1 \bigr) \varepsilon.
 \end{align*}
 Choosing $2^{-j_0} \sim T$, we arrive at the desired estimate
 \begin{align*}
  \bigl| \partial_{|\xi|}^l \partial_\eta^\alpha \bigl( \psi_{\pm}^{n,<k, s}(t,x,\xi) - \psi_{\pm}^{n,<k, s}(s,y,\xi) \bigr) \bigr| &\lesssim \bigl( 2^k T \bigr)^{\sigma(|\alpha|+\frac{1}{2})} \varepsilon.
 \end{align*}
\end{proof}

Finally, we obtain certain decomposable estimates for the ``deterministic'' phase function. This is the analogue of Lemma~7.3 in~\cite{KST}, but here we have to restrict the allowed ranges of Strichartz exponents slightly in order not to lose derivatives.
We first briefly recall the definition of decomposable function spaces from \cite{RT, Krieger-Sterbenz, KST} and the basic decomposable calculus.

\medskip

Let $c(t,x,D)$ be a pseudodifferential operator whose symbol $c(t,x,\xi)$ is homogeneous of degree $0$ in $\xi$. Assume that $c$ has a representation 
\[
 c = \sum_{\theta \in 2^{-\bbN}} c^{(\theta)}.
\]
Let $1 \leq q,r \leq \infty$. For every $\theta \in 2^{-\bbN}$, we define 
\begin{equation*}
 \|c^{(\theta)}\|_{D_\theta (L^q_t L^r_x)} := \Biggl\| \biggl( \sum_{l=0}^{40} \sum_{\Gamma_\theta^\nu} \sup_{\eta \in \Gamma_\theta^\nu} \, \bigl\| b_\theta^\nu(\eta) (\theta \nabla_\xi)^l c^{(\theta)} \bigr\|_{L^r_x(\bbR^4)}^2 \biggr)^{1/2} \Biggr\|_{L^q_t(\bbR)},
\end{equation*}
where $\{\Gamma_\theta^\nu\}_{\nu \in \mathbb{S}^3}$ is a uniformly finitely overlapping covering of $\mathbb{S}^3$ by caps of diameter $\sim \theta$ and $\{ b_\theta^\nu \}_{\nu \in \mathbb{S}^3}$ is a smooth partition of unity subordinate to the covering $\{\Gamma_\theta^\nu\}_{\nu \in \mathbb{S}^3}$. Then we define the decomposable norm
\begin{equation*}
 \|c\|_{D L^q_t L^r_x} := \inf_{c = \sum_{\theta } c^{(\theta)}} \sum_{\theta \in 2^{-\mathbb{N}}} \|c^{(\theta)}\|_{D_\theta(L^q_t L^r_x)}.
\end{equation*}
We will frequently use the following decomposability lemma from~\cite{KST}. 
\begin{lemma}[Decomposability lemma, {\cite[Lemma 7.1]{KST}}] \label{lem:decomposability_lemma_KST}
 Let $P(t,x,D)$ be a pseudodifferential operator with symbol $p(t,x,\xi)$. Suppose that $P$ satisfies the fixed-time estimate
 \[
  \sup_{t \in \bbR} \|P(t,x,D)\|_{L^2_x \to L^2_x} \lesssim 1.
 \]
 Let $1 \leq q, q_1, q_2, r, r_1 \leq \infty$ such that $\frac{1}{q} = \frac{1}{q_1} + \frac{1}{q_2}$ and $\frac{1}{r} = \frac{1}{r_1} + \frac{1}{2}$. Then for any symbol $c(t,x,\xi) \in D L^{q_1}_t L^{r_1}_x$ that is zero homogeneous in $\xi$, we have
 \begin{equation*}
  \|(c p)(t,x,D) \phi\|_{L^{q}_t L^{r}_x} \lesssim \|c\|_{D L^{q_1}_t L^{r_1}_x} \|\phi\|_{L^{q_2}_t L^2_x}.
 \end{equation*}
\end{lemma}

We now turn to proving decomposable estimates for the ``deterministic'' phase function.
\begin{lemma}[Decomposable estimates for the ``deterministic'' phase function] \label{lem:det_decomposable_est}
 Let $n \geq 1$ and assume that
 \begin{equation*}
  \sum_{m=1}^{n-1} \|\calA_{x,r}^m\|_{R_m} + \|\calA_{x,s}^{0, free}\|_{S^1[0]} \lesssim \varepsilon. 
 \end{equation*}
 Let $k \in \bbZ$ and $j \leq k-C$. For $2 \leq q < \infty$ and $\frac{2}{q} + \frac{3}{r} \leq \frac{3}{2+}$ we have that 
 \begin{equation} \label{equ:decomposable_estimate_det_phase_angle}
  \bigl\| \bigl( \psi_{\pm, j, (\theta)}^{n,<k}, 2^{-j} \nabla_{t,x} \psi_{\pm, j, (\theta)}^{n,<k} \bigr) \bigr\|_{D L^q_t L^r_x} \lesssim 2^{-(\frac{1}{q} + \frac{4}{r})j} \theta^{\frac{1}{2+}-\frac{2}{q}-\frac{3}{r}} \varepsilon.
 \end{equation}
 For $4+ < q \leq \infty$ it holds that
 \begin{equation} \label{equ:decomposable_estimate_det_phase_qinfty}
  \bigl\| \bigl( \psi_{\pm, j}^{n,<k}, 2^{-j} \nabla_{t,x} \psi_{\pm, j}^{n,<k} \bigr) \bigr\|_{D L^q_t L^\infty_x} \lesssim 2^{- \frac{1}{q} j} \varepsilon.
 \end{equation}
\end{lemma}
\begin{proof}
 We establish the decomposable estimate~\eqref{equ:decomposable_estimate_det_phase_angle} separately for the rough and the smooth part of the phase function. We begin with the rough part. 
 As in~\cite[Lemma 7.3]{KST}, we interchange integration and the $\eta$ summation to obtain that
 \begin{align*}
  \bigl\| \bigl( \psi_{\pm, j, (\theta)}^{n,<k, r}, 2^{-j} \nabla_{t,x} \psi_{\pm, j, (\theta)}^{n,<k, r} \bigr) \bigr\|_{D L^q_t L^r_x} &\lesssim \theta^{-2} 2^{-j} \biggl( \sum_\eta \, \bigl\| \Pi_\theta^\eta P_j A_{x,r}^{<n-1} \cdot \eta \bigr\|_{L^q_t L^r_x}^2 \biggr)^{\frac{1}{2}} \\
  &\lesssim \theta^{-1} 2^{-j} \biggl( \sum_\eta \, \bigl\| \Pi_\theta^\eta P_j A_{x,r}^{<n-1} \bigr\|_{L^q_t L^r_x}^2 \biggr)^{\frac{1}{2}},
 \end{align*}
 where in the last step we used the Coulomb gauge to gain another factor of $\theta$. Next we pick an admissible Strichartz pair $(\tilde{q}, \tilde{r})$ with $\tilde{q} > q$ and $\tilde{r} < r$ such that an interpolate between $(\tilde{q}, \tilde{r})$ and $(2,\infty)$ gives $(q,r)$. 
 Then we obtain for some $\mu > 0$ that 
 \begin{align*}
  &\bigl\| \bigl( \psi_{\pm, j, (\theta)}^{n,<k, r}, 2^{-j} \nabla_{t,x} \psi_{\pm, j, (\theta)}^{n,<k, r} \bigr) \bigr\|_{D L^q_t L^r_x} \\
  &\quad \lesssim \theta^{-1} 2^{-j} \biggl( \sum_\eta \, \bigl\| \Pi_\theta^\eta P_j A_{x,r}^{<n-1} \bigr\|_{L^{\tilde{q}}_t L^{\tilde{r}}_x}^2 \biggr)^{\frac{\mu}{2}} \biggl( \sum_\eta \, \bigl\| \Pi_\theta^\eta P_j A_{x,r}^{<n-1} \bigr\|_{L^2_t L^\infty_x}^2 \biggr)^{\frac{1-\mu}{2}}. 
 \end{align*}
 Now let $r_0 > 6$ such that $(\tilde{q}, r_0)$ is a sharp admissible Strichartz pair in four space dimensions, i.e. $\frac{2}{q} + \frac{3}{r_0} = \frac{3}{2}$. Bernstein's estimate on an angular sector of size $\theta$ gives the inequality $\Pi_\theta^\eta P_j L^{r_0}_x \subset \theta^{3(\frac{1}{r_0}-\frac{1}{\tilde{r}})} 2^{4(\frac{1}{r_0}-\frac{1}{\tilde{r}})j} L^{\tilde{r}}_x$. Using also the redeeming $R_j L^2_t L^\infty_x$ norm, we find that 
 \begin{align*}
  &\bigl\| \bigl( \psi_{\pm, j, (\theta)}^{n,<k, r}, 2^{-j} \nabla_{t,x} \psi_{\pm, j, (\theta)}^{n, <k, r} \bigr) \bigr\|_{D L^q_t L^r_x} \\
  &\quad \lesssim \theta^{-1} 2^{-j} \biggl( \sum_\eta \, \Bigl( \theta^{\frac{3}{2} -\frac{2}{\tilde{q}} - \frac{3}{\tilde{r}}} 2^{4(\frac{1}{r_0}-\frac{1}{\tilde{r}})j} \bigl\| \Pi_\theta^\eta P_j A_{x,r}^{<n-1} \bigr\|_{L^{\tilde{q}}_t L^{r_0}_x} \Bigr)^2 \biggr)^{\frac{\mu}{2}} \bigl( 2^{-(\frac{1}{2}-20\sigma) j} \theta^{-\delta_1} \| P_j A_{x,r}^{<n-1} \|_{R_j} \bigr)^{1-\mu} \\
  &\quad \lesssim \theta^{-1} 2^{-j} \biggl( 2^{\delta_\ast j} \theta^{\frac{3}{2} -\frac{2}{\tilde{q}} - \frac{3}{\tilde{r}}} 2^{(1-\frac{1}{\tilde{q}}-\frac{4}{\tilde{r}})j} \bigl\| P_j A_{x,r}^{<n-1} \bigr\|_{S^{1-\delta_\ast}_j} \biggr)^\mu \bigl( 2^{-(\frac{1}{2}-20\sigma) j} \theta^{-\delta_1} \| P_j A_{x,r}^{<n-1} \|_{R_j} \bigr)^{1-\mu} \\
  &\quad \lesssim \theta^{\frac{\mu}{2}-(1-\mu)\delta_1-\frac{2}{q}-\frac{3}{r}} 2^{(\mu \delta_\ast + (1-\mu)20\sigma - (1-\mu) )j} 2^{-(\frac{1}{q} + \frac{4}{r})j} \varepsilon.
 \end{align*}
 Choosing $0 < \mu < 1$ close to $1$, gives the decomposable estimate~\eqref{equ:decomposable_estimate_det_phase_angle} for the rough part of the phase function. 
 
 For the smooth part of the phase function we remark that since $\calA_{x,s}^{0, free}$ belongs to the critical $S^1$ space, the corresponding bound follows exactly as in Lemma 7.3 in~\cite{KST}. 
 
 Finally, the second estimate~\eqref{equ:decomposable_estimate_det_phase_qinfty} in the statement of the lemma follows immediately for $4+ < q < \infty$ from the first estimate~\eqref{equ:decomposable_estimate_det_phase_angle} by summing over the dyadic angles $2^{\sigma (j-k)} \lesssim \theta \lesssim 1$. In the important case $q=\infty$ the second estimate~\eqref{equ:decomposable_estimate_det_phase_qinfty} can be proved directly using the redeeming $R_j L^\infty_t L^\infty_x$ norm with angular gains to bound the contributions of the rough part of the phase function.
\end{proof}

\section{The ``probabilistic'' parametrix} \label{sec:probabilistic_bounds}

In this section we turn to the precise definition of the adapted linear evolutions $\Phi^n_r$ of the rough random data $T_n \phi^\omega[0]$, $n \geq 1$, as approximate solutions to the modified linear magnetic wave equation
\begin{equation} \label{equ:prob_section_magnetic_wave_equation}
 \Box_{A^{<n-1}}^{p, mod} \Phi^n_r \equiv \bigl( \Box + 2i P_{\leq (1-\gamma) n} A^{<n-1, \alpha} \partial_\alpha P_n \bigr) \Phi^n_r \approx 0, \quad \Phi^n_r[0] \approx T_n \phi^\omega[0].
\end{equation}
We emphasize that $A^{<n-1}$ is the entire connection form of the solution $(A^{<n-1}, A^{<n-1}_0, \phi^{<n-1})$ to (MKG-CG) with random initial data $A_x^{<n-1}[0] = T_{<n-1} A_x^\omega[0]$, $\phi^{<n-1}[0] = T_{<n-1} \phi^\omega[0]$ that was constructed in the prior induction stages $\leq n-1$. 

We first carefully develop the iterative definition of $\Phi^n_r$ in terms of a modified ``probabilistic'' parametrix and prove mapping properties of the associated modified renormalization operators. Then we turn to the derivation of the redeeming space-time integrability properties of the rough linear evolutions $\Phi^n_r$ of the random data $T_n \phi^\omega[0]$. Finally, we discuss the delicate renormalization error estimate  and show that (on a suitable event) the error $\Box_{A^{<n-1}}^{p, mod} \Phi^n_r$ gains regularity and can be treated as a ``smooth'' source term.

\subsection{Definition of the adapted rough linear evolution $\Phi^n_r$} \label{subsec:prob_definition_Phi}

Similarly to the ``deterministic'' parametrix construction, our construction of the adapted linear evolution $\Phi^n_r$ will be based on modified renormalization operators $e^{\pm i \psi_{\pm}^{n, mod}}_{<n-C}(t,x,D)$. 
We begin by motivating heuristically the choice of the modified phase function~$\psi_{\pm}^{n, mod}$ in the context of the modified linear magnetic wave equation~\eqref{equ:prob_section_magnetic_wave_equation}. 
To this end let us again consider distorted waves of the form 
\begin{equation*}
 \phi(t,x) = e^{-i\psi_\pm(t,x)} e^{\pm i t |\xi| + i x \cdot \xi}
\end{equation*}
and compute how a magnetic wave operator of the form $(\Box + 2i A^\alpha \partial_\alpha)$ acts on them. 
In what follows it is important to keep in mind that the spatial part $A_x$ of the connection form is no longer a free wave, but that it also has inhomogeneous parts, and that the temporal component $A_0$ of the connection form is also built into the magnetic wave operator.
We find that
\begin{equation*}
 \begin{aligned}
  \bigl( \Box + 2i A^\alpha \partial_\alpha \bigr) \phi &= 2 \bigl( - |\xi| (L_{\mp}^\eta \psi_{\pm}) + A_x \cdot \xi \mp A_0 |\xi| \bigr) \phi \\
  &\quad + \bigl( 2 A^\alpha (\partial_\alpha \psi_\pm) - (\partial_t \psi_\pm)^2 + |\nabla_x \psi_\pm|^2 - i (\Box \psi_\pm) \bigr) \phi.
 \end{aligned}
\end{equation*}
The terms in the second parentheses are again expected to be error terms, while we would like to achieve as much cancellation as possible in the first parentheses. 
In order to largely cancel out the term $\mp A_0 |\xi|$, we would like to formally build a component $\mp (L_{\mp}^\eta)^{-1} A_0$ into the definition of the phase function. However, to deal with the degeneracy of the symbol $L_{\mp}^{\eta}$, we have to refine this choice depending on the size of the symbol of $L_\mp^\eta$. 
This leads to an analogous choice to largely cancel out the inhomogeneous part of the term $A_x \cdot \xi$. For the homogeneous (rough) part we use the same definition as for the ``deterministic'' phase function.

To arrive at the precise definitions we need to introduce some additional notation. We denote by $(\tau, \zeta)$ space-time Fourier variables. For $\eta \in \bbS^3$ and $|\zeta| \sim 2^k$ we introduce the space-time frequency regions
\begin{align*}
 S^{\pm, \eta}_{\gg 2^k |\measuredangle|^2} &:= \bigl\{ | \pm \tau + \eta \cdot \zeta | \gg 2^k |\measuredangle(\eta, \zeta)|^2 \bigr\}, \\
 S^{\pm, \eta}_{\lesssim 2^k |\measuredangle|^2} &:= \bigl\{ | \pm \tau + \eta \cdot \zeta | \lesssim 2^k |\measuredangle(\eta, \zeta)|^2 \bigr\}.
\end{align*}
Then we denote by $\Pi^{\pm, \eta}_{\gg 2^k |\measuredangle|^2}$ and by $\Pi^{\pm, \eta}_{\lesssim 2^k |\measuredangle|^2}$ smooth projections to these space-time frequency regions such that $\Pi^{\pm, \eta}_{\gg 2^k |\measuredangle|^2} + \Pi^{\pm, \eta}_{\lesssim 2^k |\measuredangle|^2} = 1$. Additionally, for any dyadic $\lambda \in 2^{\bbZ}$ we denote by $\Pi^{\pm, \eta}_\lambda$ a smooth projection to the frequency region $\{ |\pm \tau + \eta \cdot \zeta| \sim \lambda \}$.

\bigskip 

Now we are in the position to define for every integer $n \geq 1$ the modified ``probabilistic'' phase function
\begin{equation} \label{equ:definition_prob_phase_function}
 \psi_{\pm}^{n, mod}(t,x,\xi) := \sum_{0 \leq k \leq (1-\gamma) n} \psi_{\pm, k}^{n, mod, r}(t,x,\xi) + \sum_{k \leq (1-\gamma)n} \psi_{\pm, k}^{n, mod, s}(t,x,\xi), 
\end{equation}
where its rough part is given by
\begin{align*}
 \psi_{\pm, k}^{n, mod, r}(t,x,\xi) :=& - L_\pm^\eta \Delta_{\eta^\perp}^{-1} \biggl( \Pi_{> 2^{\sigma \min\{k, -n\}}}^{\eta}  P_k A_{x, r}^{<n-1} \cdot \eta \biggr), \qquad \eta := \frac{\xi}{|\xi|} \in \bbS^{3},
\end{align*}
while its smooth part is defined as 
\begin{align*}
 \psi_{\pm, k}^{n, mod, s}(t,x,\xi) :=& + (L_{\mp}^\eta)^{-1} \Pi_{\gg 2^k |\measuredangle|^2}^{\mp, \eta} \Pi_{> 2^{\sigma \min\{k, -n\}}}^{\eta} \bigl( P_k A_{x, s}^{<n-1} \cdot \eta \bigr), \qquad \eta := \frac{\xi}{|\xi|} \in \bbS^{3} \\
 &- L_\pm^\eta \Delta_{\eta^\perp}^{-1} \Pi_{\lesssim 2^k |\measuredangle|^2}^{\mp, \eta} \Pi_{> 2^{\sigma \min\{k, -n\}}}^{\eta} \bigl( P_k A_{x, s}^{<n-1} \cdot \eta \bigr) \\
 &\mp (L_{\mp}^\eta)^{-1} \Pi_{\gg 2^k |\measuredangle|^2}^{\mp, \eta} \Pi_{> 2^{\sigma \min\{k, -n\}}}^{\eta} \bigl( P_k A_0^{<n-1} \bigr) \\
 &\mp L_\pm^\eta \Delta_{\eta^\perp}^{-1} \Pi_{\lesssim 2^k |\measuredangle|^2}^{\mp, \eta} \Pi_{> 2^{\sigma \min\{k, -n\}}}^{\eta} \bigl( P_k A_0^{<n-1} \bigr). 
\end{align*}

\medskip 

For given initial conditions $(f,g)$ and a given source term $F$, we define an approximate solution to the linear magnetic wave equation $\Box_{A^{<n-1}}^{p,mod} u = F$ with initial data $u[0] = (f,g)$ by the parametrix
\begin{equation} \label{equ:definition_parametrix_modified}
 \begin{aligned}
  \phi_{app, n}^{mod} \bigl[ f, g; F \bigr] &:= \frac{1}{2} \sum_\pm e^{-i \psi_{\pm}^{n, mod}}_{<n-C}(t,x,D) \frac{e^{\pm i t |D|}}{i|D|} e^{+i \psi_{\pm}^{n, mod}}_{<n-C}(D,y,0) \bigl( i |D| f \pm g \bigr) \\
  &\quad \quad + \frac{1}{2} \sum_\pm \pm e^{-i\psi_{\pm}^{n, mod}}_{<n-C}(t,x,D) \frac{K^\pm}{i|D|} e^{+i\psi_{\pm}^{n, mod}}_{<n-C}(D,y,s) F,
 \end{aligned}
\end{equation}
where $K^\pm F$ are the Duhamel terms
\begin{equation*}
 K^\pm F(t) = \int_0^t e^{\pm i (t-s)|D|} F(s) \, \ud s.
\end{equation*}

\medskip 

We define the adapted rough linear evolution $\Phi^n_r$ as an infinite sum whose components are defined iteratively
\begin{equation} \label{equ:definition_Phi_n_rough}
 \Phi^n_r = \sum_{\ell = 0}^\infty \Phi^{n, [\ell]}_r.
\end{equation}
The zeroth term $\Phi^{n, [0]}_r$ is defined in terms of the homogeneoux parametrix
\begin{equation*}
 \Phi^{n,[0]}_r := \phi_{app, n}^{mod} \bigl[ T_n \phi_0^\omega, T_n \phi_1^\omega; 0 \bigr].
\end{equation*}
It would be desirable if it sufficed to take $\Phi^{n,[0]}_r$ as our choice for an approximate solution in the sense that $\Phi^{n,[0]}_r$ would be amenable to the probabilistic redeeming bounds and that the entire error $\Box_{A^{<n-1}}^{p, mod} \Phi^{n, [0]}_r$ would gain regularity.  However, $\Box_{A^{<n-1}}^{p, mod} \Phi^{n, [0]}_r$ produces several types of error terms that we colloquially group into ``mild'', ``delicate'', and ``rough'' error terms
\begin{equation*}
 \Box_{A^{<n-1}}^{p, mod} \Phi^{n, [0]}_r = \calE^{n, [0]}_{mild} + \calE^{n, [0]}_{del} + \calE^{n, [0]}_{rough}.
\end{equation*}
These are defined precisely further below.
Unfortunately, the ``rough'' error terms do not gain regularity and therefore cannot be treated as smooth source terms. But at least, they gain smallness. The way out is therefore to try to iterate these ``rough'' error terms away so that we end up with an infinite sum of ``mild'' and ``delicate'' error terms that can all be treated as smooth source terms.
Correspondingly, we inductively define the $\ell$-th iterate $\Phi^{n, [\ell]}_r$, by
\begin{equation*}
 \Phi^{n, [\ell]}_r := \phi_{app,n}^{mod} \bigl[ 0, 0; - \calE^{n, [\ell-1]}_{rough} \bigr], \quad \ell \geq 1,
\end{equation*}
which produces an error of the form
\begin{equation*}
 \calE^{n, [\ell-1]}_{rough} + \Box_{A^{<n-1}}^{p, mod} \Phi^{n, [\ell]}_r = \calE^{n, [\ell]}_{mild} + \calE^{n, [\ell]}_{del} + \calE^{n, [\ell]}_{rough}, \quad \ell \geq 1.
\end{equation*}
Then we have 
\begin{equation} \label{equ:overall_accrued_error}
 \Box_{A^{<n-1}}^{p,mod} \Phi^n_r = \sum_{\ell =0}^\infty \calE^{n, [\ell]}_{mild} + \sum_{\ell=0}^\infty \calE^{n, [\ell]}_{del}.
\end{equation}

\medskip 

Our next goal is to arrive at the precise definitions of the higher iterates $\Phi^{n,[\ell]}_r$.
To this end we first compute the errors accrued by the zeroth iterate 
\begin{equation*}
 \Phi^{n,[0]}_r := \phi_{app, n}^{mod} \bigl[ T_n \phi_0^\omega, T_n \phi_1^\omega; 0 \bigr] = \frac{1}{2} \sum_\pm e^{-i \psi_{\pm}^{n, mod}}_{<n-C}(t,x,D) \frac{e^{\pm i t |D|}}{i|D|} e^{+i \psi_{\pm}^{n, mod}}_{<n-C}(D,y,0) \bigl( i |D| T_n \phi_0^\omega \pm T_n \phi_1^\omega \bigr).
\end{equation*}
We obtain schematically that
\begin{align*}
 \Box_{A^{<n-1}}^{p, mod} \Phi^{n,[0]}_r &= \sum_{k \leq (1-\gamma) n} \sum_\pm 2 \Bigl[ \bigl( -|\xi| (L_\mp^\eta \psi_{\pm, k}^{n, mod}) + P_k A_x^{<n-1} \cdot \xi \mp P_k A_0^{<n-1} |\xi| \bigr) e^{-i \psi_\pm^{n, mod}} \Bigr]_{<n-C} \calR_r^{n, \pm, [0]} \\
  &\quad \quad + \sum_\pm 2\partial_te^{- i\psi_{\pm}^{n,mod}}_{<n-C}(i\partial_t\pm|\xi|) \calR_r^{n, \pm, [0]} \\
  &\quad \quad + \sum_\pm \Bigl[ \bigl(- (\partial_t\psi_{\pm}^{n,mod})^2 + (\partial_x\psi_{\pm}^{n,mod})^2 \bigr) e^{- i\psi_{\pm}^{n,mod}} \Bigr]_{<n-C} \calR_r^{n, \pm, [0]} \\
  &\quad \quad + \sum_\pm \Bigl[-i (\Box \psi_{\pm}^{n,mod}) e^{- i\psi_{\pm}^{n,mod}} \Bigr]_{<n-C} \calR_r^{n, \pm, [0]} \\
  &\quad \quad - \sum_\pm 2i A^{<n-1, \alpha}_{\leq (1-\gamma)n} \Bigl[ (\partial_\alpha \psi_{\pm}^{n,mod}) e^{- i\psi_{\pm}^{n,mod}} \Bigr]_{<n-C} \calR_r^{n, \pm, [0]} \\
  &\quad \quad + \sum_\pm 2i \Bigl[ A^{<n-1, \alpha}_{\leq (1-\gamma)n}, S_{<n-C} \Bigr] \xi_{\alpha} \calR_r^{n, \pm, [0]} \\
  &\equiv \text{Diff}^{[0]}_1 + \ldots + \text{Diff}^{[0]}_6
\end{align*}
with 
\begin{align*}
 \calR_r^{n, \pm, [0]} := \frac{1}{2} \frac{e^{\pm i t |D|}}{i|D|} e^{+i \psi_{\pm}^{n, mod}}_{<n-C}(D,y,0) \bigl( i |D| T_n \phi_0^\omega \pm T_n \phi_1^\omega \bigr).
\end{align*}
At this point we anticipate that the error terms $\text{Diff}^{[0]}_2$, ..., $\text{Diff}^{[0]}_6$ will be manageable since effectively a derivative falls on a low frequency term, which then allows to gain regularity thanks to the frequency separation.
The main error term $\text{Diff}^{[0]}_1$ in the first line on the right-hand side of $\Box_{A^{<n-1}}^{p, mod} \Phi^{n, [0]}_r$ can be further decomposed as
\begin{align*}
 \text{Diff}^{[0]}_1 &= \sum_{k \leq (1-\gamma) n} \sum_\pm 2 \Bigl[ \bigl( -|\xi| (L_\mp^\eta \psi_{\pm, k}^{n, mod}) + P_k A_x^{<n-1} \cdot \xi \mp P_k A_0^{<n-1} |\xi| \bigr) e^{-i \psi_\pm^{n, mod}} \Bigr]_{<n-C} \calR_r^{n, \pm, [0]} \\
 &= \sum_{0 \leq k \leq (1-\gamma) n} \sum_\pm 2 \Bigl[ \Pi_{\leq 2^{-\sigma n}} \bigl( P_k A_{x,r}^{<n-1} \cdot \xi \bigr) e^{-i \psi_\pm^{n, mod}} \Bigr]_{<n-C} \calR_r^{n, \pm, [0]} \\
 &\quad + \sum_{k \leq (1-\gamma) n} \sum_\pm 2 \Bigl[ \Pi_{\leq 2^{\sigma \min\{k, -n\}}} \bigl( P_k A_{x,s}^{<n-1} \cdot \xi \mp P_k A_0^{<n-1} |\xi| \bigr) e^{-i \psi_\pm^{n, mod}} \Bigr]_{<n-C} \calR_r^{n, \pm, [0]} \\
 &\quad + \sum_{k \leq (1-\gamma) n} \sum_\pm 2 \Bigl[ - \Box \Delta_{\eta^\perp}^{-1} \Pi_{\lesssim 2^k |\measuredangle|^2}^{\mp, \eta} \Pi_{> 2^{\sigma \min\{k, -n\}}}^{\eta} \bigl( P_k A_{x, s}^{<n-1} \cdot \xi \mp P_k A_0^{<n-1} |\xi| \bigr) e^{-i \psi_\pm^{n, mod}} \Bigr]_{<n-C} \calR_r^{n, \pm, [0]} \\
 &\equiv \text{Diff}^{[0]}_{1, (a)} + \text{Diff}^{[0]}_{1, (b)} + \text{Diff}^{[0]}_{1, rough}.
\end{align*}
We emphasize that in view of the definition of the modified ``probabilistic'' phase function $\psi^{n, mod}_\pm$, a crucial cancellation of the large angle part occurs so that the first two error terms come with tight angle cutoffs. We anticipate that owing to the tighter angle cutoffs $\Pi_{\leq 2^{-\sigma n}}$ and $\Pi_{\leq 2^{\sigma \min\{k, -n\}}}$, the errors $\text{Diff}^{[0]}_{1,(a)}$ and $\text{Diff}^{[0]}_{1,(b)}$ will turn out to gain regularity.
Correspondingly, we group together all the error terms $\text{Diff}^{[0]}_{1,(a)}$, $\text{Diff}^{[0]}_{1,(b)}$, and $\text{Diff}^{[0]}_{j}$, $2 \leq j \leq 6$, into the collection of ``mild'' error terms
\begin{equation} \label{equ:definition_mild_errors_stage0}
 \calE^{n, [0]}_{mild} = \text{Diff}^{[0]}_{1,(a)} + \text{Diff}^{[0]}_{1,(b)} + \sum_{j=2}^6 \text{Diff}^{[0]}_{j}.
\end{equation}

Unfortunately, only certain parts of the error term $\text{Diff}^{[0]}_{1, rough}$ can gain regularity (which we will call ``delicate'' errors and which we will denote by $\calE^{n, [0]}_{del}$), while the remaining parts will have to be iterated away. In order to arrive at precise definitions, it is helpful to first also determine the structure of the errors accrued by the higher iterates 
\begin{equation*}
 \Phi^{n, [\ell]}_r(t,x) := \frac{1}{2} \sum_\pm \pm \, e^{-i\psi_{\pm}^{n, mod}}_{<n-C}(t,x,D) \frac{K^\pm}{i|D|} e^{+i\psi_{\pm}^{n, mod}}_{<n-C}(D,y,s) \Bigl( - \calE_{rough}^{n, [\ell-1]} \Bigr), \quad \ell \geq 1.
\end{equation*}
Here we compute that schematically 
\begin{align*}
 &\calE^{n, [\ell-1]}_{rough} + \Box_{A^{<n-1}}^{p, mod} \Phi^{n,[\ell]}_r \\
 &= \sum_{k \leq (1-\gamma) n} \sum_\pm 2 \Bigl[ \bigl( -|\xi| (L_\mp^\eta \psi_{\pm, k}^{n, mod}) + P_k A_x^{<n-1} \cdot \xi \mp P_k A_0^{<n-1} |\xi| \bigr) e^{-i \psi_\pm^{n, mod}} \Bigr]_{<n-C} \calR_r^{n, \pm, [\ell]} \\
 &\quad \quad + \sum_\pm 2\partial_te^{- i\psi_{\pm}^{n,mod}}_{<n-C}(i\partial_t\pm|\xi|) \calR_r^{n, \pm, [\ell]}  \\
 &\quad \quad + \sum_\pm \Bigl[ \bigl(- (\partial_t\psi_{\pm}^{n,mod})^2 + (\partial_x\psi_{\pm}^{n,mod})^2 \bigr) e^{- i\psi_{\pm}^{n,mod}} \Bigr]_{<n-C} \calR_r^{n, \pm, [\ell]}  \\
 &\quad \quad + \sum_\pm \Bigl[-i (\Box \psi_{\pm}^{n,mod}) e^{- i\psi_{\pm}^{n,mod}} \Bigr]_{<n-C} \calR_r^{n, \pm, [\ell]}  \\
 &\quad \quad - \sum_\pm 2i A^{<n-1, \alpha}_{\leq (1-\gamma)n} \Bigl[ (\partial_\alpha \psi_{\pm}^{n,mod}) e^{- i\psi_{\pm}^{n,mod}} \Bigr]_{<n-C} \calR_r^{n, \pm, [\ell]}  \\
 &\quad \quad + \sum_\pm 2i \Bigl[ A^{<n-1, \alpha}_{\leq (1-\gamma)n}, S_{<n-C} \Bigr] \xi_{\alpha} \calR_r^{n, \pm, [\ell]} \\
 &\quad \quad + \calE_{rough}^{n, [\ell-1]} \pm \frac12 \sum_\pm e^{-i\psi_{\pm}^{n, mod}}_{<n-C}(t,x,D) \frac{\partial_t \pm i|D| + 2i A_0^{<n-1}}{i|D|} e^{+i\psi_{\pm}^{n, mod}}_{<n-C}(D,y,t) \Bigl( - \calE_{rough}^{n, [\ell-1]} \Bigr) \\
 &\equiv \text{Diff}^{[\ell]}_1 + \ldots + \text{Diff}^{[\ell]}_6 + \text{Diff}^{[\ell]}_7, 
\end{align*}
where
\begin{align*}
 \calR_r^{n, \pm, [\ell]} &= \pm \frac12 \frac{K^\pm}{i|D|} e^{+i\psi_{\pm}^{n, mod}}_{<n-C}(D,y,s) \Bigl( - \calE_{rough}^{n, [\ell-1]} \Bigr).
\end{align*}
We anticipate that the error terms $\text{Diff}^{[\ell]}_2$, ..., $\text{Diff}^{[\ell]}_6$ will again be manageable and gain regularity since effectively a derivative falls on a low frequency term. The additional error term $\text{Diff}^{[\ell]}_7$ arising for the inhomogeneous parametrix will also be manageable since we can gain regularity from the difference to the previous error $\calE_{rough}^{n, [\ell-1]}$, see Prop.~\ref{prop:prob_data_error}.
Then we may again further decompose the main error term $\text{Diff}^{[\ell]}_1$ in the first line on the right-hand side of the accrued error $\calE^{n, [\ell-1]}_{rough} + \Box_{A^{<n-1}}^{p, mod} \Phi^{n,[\ell]}_r$  as
\begin{align*}
 \text{Diff}^{[\ell]}_1 &= \sum_{k \leq (1-\gamma) n} \sum_\pm 2 \Bigl[ \bigl( -|\xi| (L_\mp^\eta \psi_{\pm, k}^{n, mod}) + P_k A_x^{<n-1} \cdot \xi \mp P_k A_0^{<n-1} |\xi| \bigr) e^{-i \psi_\pm^{n, mod}} \Bigr]_{<n-C} \calR_r^{n, \pm, [\ell]} \\
 &= \sum_{0 \leq k \leq (1-\gamma) n} \sum_\pm 2 \Bigl[ \Pi_{\leq 2^{-\sigma n}} \bigl( P_k A_{x,r}^{<n-1} \cdot \xi \bigr) e^{-i \psi_\pm^{n, mod}} \Bigr]_{<n-C} \calR_r^{n, \pm, [\ell]} \\
 &\quad + \sum_{k \leq (1-\gamma) n} \sum_\pm 2 \Bigl[ \Pi_{\leq 2^{\sigma \min\{k, -n\}}} \bigl( P_k A_{x,s}^{<n-1} \cdot \xi \mp P_k A_0^{<n-1} |\xi| \bigr) e^{-i \psi_\pm^{n, mod}} \Bigr]_{<n-C} \calR_r^{n, \pm, [\ell]} \\
 &\quad + \sum_{k \leq (1-\gamma) n} \sum_\pm 2 \Bigl[ - \Box \Delta_{\eta^\perp}^{-1} \Pi_{\lesssim 2^k |\measuredangle|^2}^{\mp, \eta} \Pi_{> 2^{\sigma \min\{k, -n\}}}^{\eta} \bigl( P_k A_{x, s}^{<n-1} \cdot \xi \mp P_k A_0^{<n-1} |\xi| \bigr) e^{-i \psi_\pm^{n, mod}} \Bigr]_{<n-C} \calR_r^{n, \pm, [\ell]} \\
 &\equiv \text{Diff}^{[\ell]}_{1, (a)} + \text{Diff}^{[\ell]}_{1, (b)} + \text{Diff}^{[\ell]}_{1, rough}.
\end{align*}
As in the treatment of the error $\Box_{A^{<n-1}}^{p, mod} \Phi^{n,[0]}_r$ accrued by the zeroth iterate, we anticipate that the errors $\text{Diff}^{[\ell]}_{1, (a)}$ and $\text{Diff}^{[\ell]}_{1, (b)}$ gain regularity thanks to the tighter angle cutoff. 
Then we again group together the terms $\text{Diff}^{[\ell]}_{1,(a)}$, $\text{Diff}^{[\ell]}_{1,(b)}$, and $\text{Diff}^{[\ell]}_{j}$, $1 \leq j \leq 7$, into the collection of ``mild'' error terms
\begin{equation} \label{equ:definition_mild_errors_stagel}
 \calE^{n, [\ell]}_{mild} = \text{Diff}^{[\ell]}_{1,(a)} + \text{Diff}^{[\ell]}_{1,(b)} + \sum_{j=2}^7 \text{Diff}^{[\ell]}_{j}.
\end{equation}
It remains to systematically define the ``delicate'' errors $\calE_{del}^{n, [\ell]}$ at every stage $\ell \geq 0$, which are those parts of $\text{Diff}^{[\ell]}_{1, rough}$ that gain regularity. 

To this end we need to introduce some more notation and terminology. We shall call a ``string of frequencies of length $\ell$'', $\ell \geq 0$, an expression
\begin{equation*}
 (\underline{k}) = (k_1 k_2 k_3 \ldots k_{3\ell +1} k_{3\ell+2} k_{3\ell+3}), 
\end{equation*}
where each $k_j$, $1 \leq j \leq 3\ell+3$, is either a positive integer $\leq (1-\gamma)n$ or the symbol $\leq 0$.
Given a string of frequencies $(\underline{k})$ of length $\ell \geq 1$, we denote by 
\begin{equation*}
 (\underline{k}') = (k_1 k_2 k_3 \ldots k_{3\ell -2} k_{3\ell-1} k_{3\ell})
\end{equation*}
the associated ``truncated string''.
Moreover, if a string of frequencies $(\underline{k})$ of length $\ell \geq 0$ has all frequencies $k_j \leq 3\sigma n$, $1 \leq j \leq 3\ell+3$, then we call $(\underline{k})$ a ``small string of length $\ell$''.

We also introduce the concept of an ascending sequence of ``dominating frequencies'' associated with a small string of frequencies. Specifically, given a small string of frequencies $(\underline{k}) = (k_1k_2k_3 \ldots k_{3\ell+1} k_{3\ell+2} k_{3\ell+3})$, we select an ascending sequence of ``dominating frequencies'' $r_1 < r_2 < \ldots < r_q$ as follows: 
We set $r_1 = k_1$ if $k_1>0$ and $r_1 = 0$ otherwise. Then we let $r_2$ be the first frequency among $k_2, k_3,\ldots, k_{3\ell+3}$ that is larger\footnote{By definition any positive frequency dominates the symbol $\leq 0$.} than $r_1$. Then if $r_2 = k_p$, we let $r_3$ be the first frequency among $k_{p+1},k_{p+2},\ldots, k_{3\ell+3}$ that is larger than $r_2$ and so on. Finally, we denote by $b_j$, $1 \leq j \leq q$, the length of the string starting\footnote{If $r_1=0$ we mean by this the string starting at $k_1$.} at $r_j$ and ending right before $r_{j+1}$\footnote{For the last dominating frequency $r_q$, we let $b_q$ be the length of the string starting at $r_q$ and ending with $k_{3\ell+3}$. Hence, if $r_q = k_{3\ell+3}$, we have $r_q = 1$.}. Then we say that the small string $(\underline{k})$ with dominating frequencies $r_1 < \ldots < r_q$ consists of segments of length $b_1, b_2,\ldots, b_q$, and we have $3 \ell + 3 = \sum b_j$. 
In addition, we introduce the notation
\begin{equation*}
 Z^{\eta, \mp}_k := - \Box \Delta_{\eta^\perp}^{-1} \Pi_{\lesssim 2^k |\measuredangle|^2}^{\mp, \eta} \Pi_{> 2^{\sigma \min\{k, -n\}}}^{\eta}.
\end{equation*}
Then we set
\begin{align*}
 \Psi^{n, (k_1 k_2 k_3)}_r &:= \sum_\pm Z_{k_3}^{\eta, \mp} \bigl( P_{k_3} A_{x, s}^{<n-1} \cdot \xi \mp P_{k_3} A_0^{<n-1} |\xi| \bigr) \times \\
 &\quad \quad \quad \times P_{k_2} \bigl( e^{-i \psi_\pm^{n, mod}} \bigr)(t,x,D) \frac{e^{\pm i t |D|}}{i|D|} P_{k_1} \bigl( e^{+i \psi_{\pm}^{n, mod}} \bigr)(D,y,0) \bigl( i |D| T_n \phi_0^\omega \pm T_n \phi_1^\omega \bigr)
\end{align*}
and for a given string $(\ulk)$ of length $\ell \geq 1$, we define inductively
\begin{align*}
 \Psi^{n, (\underline{k})}_r &:= \sum_\pm \pm Z_{k_{3\ell+3}}^{\eta, \mp} \bigl( P_{k_{3\ell+3}} A_{x, s}^{<n-1} \cdot \xi \mp P_{k_{3\ell+3}} A_0^{<n-1} |\xi| \bigr) \times \\
 &\quad \quad \quad \quad \times P_{k_{3\ell+2}} \bigl( e^{-i \psi_\pm^{n, mod}} \bigr)(t,x,D) \frac{K^{\pm}}{i|D|} P_{k_{3\ell+1}} \bigl( e^{+i \psi_{\pm}^{n, mod}} \bigr)(D,y,s) \Psi^{n, (\underline{k}')}_r.
\end{align*}
Moreover, for a given string $(\underline{k})$ of length $\ell$ and a given $\Psi^{n, (\underline{k})}_r$, we also introduce the notation
\begin{equation} \label{equ:definition_Phi_string}
 \Phi^{n, \pm, (\underline{k})}_r := P_{k_{3\ell+2}} \bigl( e^{-i \psi_\pm^{n, mod}} \bigr)(t,x,D) \frac{K^{\pm}}{i|D|} P_{k_{3\ell+1}} \bigl( e^{+i \psi_{\pm}^{n, mod}} \bigr)(D,y,s) \Psi^{n, (\underline{k}')}_r.
\end{equation}

\medskip 

We may from now on suppress the space-time localizations $[\ldots]_{<n-C}$, since they are given by convolution with $L^1$-bounded kernels and all spaces used are translation invariant. Using the above notation, we may then schematically write the worst part $\text{Diff}_{1, rough}^{[0]}$ of the error $\Box_{A^{<n-1}}^{p, mod} \Phi^{n, [0]}_r$ accrued at the zeroth stage as
\begin{align*}
 \text{Diff}_{1, rough}^{[0]} = \sum_{k_1, k_2, k_3 \leq (1-\gamma) n} \Psi^{n, (k_1 k_2 k_3)}_r.
\end{align*}
A key observation will be that whenever (at least) one of the frequencies $k_1$, $k_2$, or $k_3$ is $\geq 3 \sigma n$, the corresponding frequency localized error $\Psi^{n, (\underline{k})}_r$ gains smoothness and we can treat it as a smooth source term, see the proof of Proposition~\ref{prop:prob_renormalization_error_estimate}. We refer to such a situation as a ``terminating situation''. This means that $\Psi^{n, (\underline{k})}_r$ is only a rough error term for ``small strings'' $(k_1, k_2, k_3)$, and these rough error terms have to be iterated away by applying the inhomogeneous parametrix to them again. 

Thus, since only ``small strings'' have to be iterated away, at this point we anticipate that at stage $\ell \geq 1$ the worst part $\text{Diff}_{1, rough}^{[\ell]}$ of the error $\Box_{A^{<n-1}}^{p, mod} \Phi^{n, [\ell]}_r$ accrued at the $\ell$-th stage is approximately of the schematic form
\begin{align*}
 \text{Diff}_{1,rough}^{[\ell]} &\approx \sum_{ \substack{k_i \leq (1-\gamma)n \\ 3\ell+1\leq i \leq 3\ell+3 } } \sum_{\substack{\text{small strings } (\underline{k}') \\ \text{of length } \ell-1}} \sum_\pm \pm Z_{k_{3\ell+3}}^{\eta, \mp} \bigl( P_{k_{3\ell+3}} A_{x, s}^{<n-1} \cdot \xi \mp P_{k_{3\ell+3}} A_0^{<n-1} |\xi| \bigr) \times \\
 &\qquad \qquad \qquad \qquad \qquad \qquad \qquad \times P_{k_{3\ell+2}} \bigl( e^{-i \psi_\pm^{n, mod}} \bigr)(t,x,D) \frac{K^{\pm}}{i|D|} P_{k_{3\ell+1}} \bigl( e^{+i \psi_{\pm}^{n, mod}} \bigr)(D,y,s) \Psi^{n, (\underline{k}')}_r.
\end{align*}
Then at stage $\ell$ a ``terminating event'' occurs whenever (at least) one of the frequencies $k_{3\ell+1}$, $k_{3\ell+2}$, or $k_{3\ell+3}$ is at a higher frequency $\geq 3 \sigma n$. Correspondingly, the rough error terms accrued at the $\ell$th stage are approximately $\Psi^{n, (\underline{k})}_r$ for all small strings $(\underline{k})$ of length $\ell$.

In fact, we have to refine this definition by taking into account the angular cutoffs in the operators $Z^{\eta, \mp}_{k}$ and in the phases $\psi^{n, mod}_\pm$. What makes this somewhat delicate is that we do not carry out this refinement one function at a time, but only for the collection of all of them. 
To this end we introduce the notation for $0 \leq k \leq 3 \sigma n$ and integer-valued $- \sigma n \leq \alpha < 0$,
\begin{equation*}
 Z^{\eta, \mp}_{k, \alpha} := - \Box \Delta_{\eta^\perp}^{-1} \Pi_{\lesssim 2^k |\measuredangle|^2}^{\mp, \eta} \Pi_{2^\alpha}^{\eta},
\end{equation*}
and then expand $Z^{\eta, \mp}_k$ as
\begin{equation*}
 Z^{\eta, \mp}_k = \sum_{-\sigma n \leq \alpha < 0} Z^{\eta, \mp}_{k, \alpha}.
\end{equation*}
Keeping in mind that the inductive definition of $\Psi^{n, (\underline{k})}_r$ for a string $(\underline{k})$ of length $\ell$ implies the presence of $\ell+1$ such operators $Z^{\eta, \mp}_{k_j}$, and expanding each of these out, we encounter a string of operators $Z^{\eta, \mp}_{k_3, \alpha_0}, \ldots, Z^{\eta, \mp}_{k_{3\ell+3}, \alpha_\ell}$. Freezing the angles $\alpha_j$, we denote the corresponding contribution by $\Psi^{n, (\underline{k}), (\underline{\alpha})}_r$. 
We will see that if the sum of the angles is sufficiently small $\sum_{j} \alpha_j \leq -\frac{\sigma n}{100}$, the corresponding term gains regularity and does not have to be passed on to the next iteration stage, see the proof of Proposition~\ref{prop:prob_renormalization_error_estimate}.

Similarly, we also reduce the phases in the exponentials $e^{\pm i\psi^{n,mod}_{\pm}}$ to sufficiently large angular separation from the Fourier support of the high frequency factor. To this end we use the decomposition 
\begin{align*}
 e^{\pm i\psi^{n,mod}_{\pm}} &= e^{\pm i\Pi_{>-10}\psi^{n,mod}_{\pm}} +  \bigl( e^{\pm i\psi^{n,mod}_{\pm}} - e^{\pm i\Pi_{>-10}\psi^{n,mod}_{\pm}} \bigr) \\
 &= e^{\pm i\Pi_{>-10}\psi^{n,mod}_{\pm}}  \mp \int_{-\sigma n}^{-10} i\frac{\partial}{\partial h}\big(\Pi_{>h}\psi^{n,mod}_{\pm}\big)e^{\pm i\Pi_{>h}\psi^{n,mod}_{\pm}}\,dh,
\end{align*}
where we recall that $\Pi_{>a}$ denotes a smooth cutoff localizing the angular separation of the Fourier support to direction $\eta := \frac{\xi}{|\xi|}$ to an angle $\gtrsim 2^a$.
We substitute this formula for each instance of $e^{\pm i\psi^{n,mod}_{\pm}}$ in $\Psi^{n, (\ulk), (\underline{\alpha})}_r$.
If we use the integral part 
\[
\mp \int_{-\sigma n}^{-10} i\frac{\partial}{\partial h}\big(\Pi_{>h}\psi^{n,mod}_{\pm}\big)e^{\pm i\Pi_{>h}\psi^{n,mod}_{\pm}} \, \ud h
\]
in $m$ instances of these exponentials in $\Psi^{n, (\ulk), (\underline{\alpha})}_r$, we can write the corresponding expression as an iterated integral of the form 
\[
 \int_{-\sigma n}^{-10} \int_{-\sigma n}^{-10} \cdots  \int_{-\sigma n}^{-10} \ldots \, \ud h_1 \ud h_2 \ldots \ud h_m, 
\]
and we denote it by $\Psi^{n, (\ulk), (\underline{\alpha}), (\underline{h})}$. We anticipate that the contribution of the integral 
\[
 \int_{-\sigma n}^{-10} \int_{-\sigma n}^{-10} \cdots \int_{-\sigma n}^{-10} \chi_{ \{ \sum h_j \leq -\frac{\sigma}{10}n \} } \ldots \, \ud h_1 \ud h_2 \ldots \ud h_m
\]
is a smooth source term, see the proof of Proposition~\ref{prop:prob_renormalization_error_estimate}.
 
Hence, we arrive at the following precise definition of the ``rough'' error accrued at every stage $\ell \geq 0$
\begin{equation} \label{equ:definition_rough_error_stage_l}
 \calE^{n, [\ell]}_{rough} := \sum_{\sum_j h_j \geq -\frac{\sigma n}{10}} \sum_{\sum_j \alpha_j \geq -\frac{\sigma n}{100} } \sum_{\substack{\text{small strings } (\underline{k}) \\ \text{of length } \ell}}    \Psi^{n, (\underline{k}), (\underline{\alpha}), (\underline{h})}_r,
\end{equation}
which then gets iterated away by applying the inhomogeneous parametrix again.
Correspondingly, the ``delicate'' error terms accrued at stage $\ell \geq 0$ are precisely defined by
\begin{equation}
 \calE_{del}^{n, [\ell]} := \text{Diff}_{1,rough}^{n, [\ell]} - \calE_{rough}^{n, [\ell]}.
\end{equation}

\medskip

In what follows we will frequently make use of the short-hand notation
\begin{equation} \label{equ:definition_kappa}
  \kappa_{n-1} := \sum_{m=1}^{n-1} \bigl( \|\calA_{x,r}^m\|_{R_m} + \|\Phi^m_r\|_{R_m} \bigr) + \sum_{m=0}^{n-1} \bigl( \|\calA_{x,s}^m\|_{S^1[m]} + \|\calA_0^m\|_{Y^1[m]} + \|\Phi^m_s\|_{S^1[m]} \bigr).
\end{equation}

\subsection{The ``probabilistic'' phase function} \label{subsec:prob_phase_function}

We now turn to establishing mapping properties of the associated ``probabilistic'' renormalization operators $e^{\pm i \psi_{\pm}^{n, mod}}_{<n-C}(t,x,D)$. It is helpful to recall that the spatial and temporal parts of the connection form $A^{<n-1}$ are composed of
\begin{align*}
 A^{<n-1}_x = \sum_{m=0}^{n-1} \calA^m_{x,s} + \sum_{m=1}^{n-1} \calA^m_{x,r}, \quad A^{<n-1}_0 = \sum_{m=0}^{n-1} \calA^m_0, 
\end{align*}
and that the rough evolution $A_{x,r}^{<n-1} = \sum_{m=1}^{n-1} \calA^m_{x,r}$ is sharply localized to frequencies $1 \lesssim |\xi| \lesssim 2^{n-1}$. 
The next proposition on the mapping properties of $e^{\pm i \psi_{\pm}^{n, mod}}_{<n-C}(t,x,D)$ is purely deterministic in the sense that it only relies on certain smallness assumptions about the components of $A^{<n-1}$ and that their randomness does not play a role here.

\begin{proposition} \label{prop:prob_renormalization_mapping_properties} 
Let $n \geq 1$. Assume that
\begin{equation*}
 \sum_{m=1}^{n-1} \|\calA_{x,r}^m\|_{R_m} + \sum_{m=0}^{n-1} \|\calA_{x,s}^m\|_{S^1[m]} + \sum_{m=0}^{n-1} \|\calA_0^m\|_{Y^1[m]} \lesssim \varepsilon. 
\end{equation*} 
Then the frequency-localized ``probabilistic'' renormalization operator $e^{\pm i \psi_{\pm}^{n, mod}}_{<n-C}(t,x,D)$ has the following mapping properties with $Z \in \{ N_n, L^2, N_n^\ast \}$:
\begin{align} 
&e^{\pm i \psi_{\pm}^{n, mod}}_{<n-C}(t,x,D) \colon Z \longrightarrow Z, \\
&\partial_t e^{\pm i \psi_{\pm}^{n, mod}}_{<n-C}(t,x,D) \colon Z \longrightarrow \varepsilon Z, \\
&e^{-i \psi_{\pm}^{n, mod}}_{<n-C}(t,x,D) e^{+i \psi_{\pm}^{n, mod}}_{<n-C}(D, y, s) - I \colon Z \longrightarrow \varepsilon Z, \\
&e^{-i \psi_{\pm}^{n, mod}}_{<n-C}(t,x,D) \colon S_n^{\sharp} \longrightarrow S_n. 
\end{align}
\end{proposition}

The proof of the mapping properties in Proposition~\ref{prop:prob_renormalization_mapping_properties} again proceeds as in Sections~6--11 in~\cite{KST} once we have established certain pointwise and decomposable estimates for the ``probabilistic'' phase functions $\psi_{\pm}^{n, mod}(t,x,\xi)$ in Lemma~\ref{lem:prob_phase_function_Linfty_bounds}, Lemma~\ref{lem:prob_difference_phase_fct}, and Lemma~\ref{lem:prob_decomposable_est} below. The proofs of the latter are also just deterministic in the sense that they only rely on certain smallness assumptions about the components of $A^{<n-1}$. We note that the delicate proof of the error estimate for $\Box_{A^{<n-1}}^{p,mod} \Phi^n_r$ is deferred to Subsection~\ref{subsec:prob_renormalization_error} below. We start off with $L^\infty$ bounds on the ``probabilistic'' phase function. 
\begin{lemma} \label{lem:prob_phase_function_Linfty_bounds}
 Let $n \geq 1$. For the rough part of the ``probabilistic'' phase function we have for any $0 \leq k \leq n$ and for any $1 \gtrsim \theta > 2^{\sigma \min\{k, -n\}}$ that
 \begin{align*}
  \bigl| \psi_{\pm, k, (\theta)}^{n, mod, r}(t,x,\xi) \bigr| &\lesssim 2^{-(2-)k} 2^{+2\delta_\ast k} \theta^{-1-\delta_1} \min\{ (\theta 2^k)^{\frac{3}{2}-}, 1 \} \| P_k A_{x,r}^{<n-1} \|_{R_k}, \\
  \bigl| \psi_{\pm, k}^{n, mod, r}(t,x,\xi) \bigr| &\lesssim 2^{-((1-) - 2\delta_\ast - \delta_1)k} \| P_k A_{x,r}^{<n-1} \|_{R_k}, \\
  \bigl| \nabla_{t,x} \psi_{\pm, k, (\theta)}^{n, mod, r}(t,x,\xi) \bigr| &\lesssim 2^{-(1-)k} 2^{+2\delta_\ast k} \theta^{-1-\delta_1} \min\{ (\theta 2^k)^{\frac{3}{2}-}, 1 \} \| P_k A_{x,r}^{<n-1} \|_{R_k}, \\
  \bigl| \nabla_{t,x} \psi_{\pm, k}^{n, mod, r}(t,x,\xi) \bigr| &\lesssim 2^{-(0+)k} 2^{+(2\delta_\ast + \delta_1)k} \| P_k A_{x,r}^{<n-1} \|_{R_k}.
 \end{align*}
 For the smooth part of the ``probabilistic'' phase function we have for any $k \in \bbZ$ and for any $1 \gtrsim \theta > 2^{\sigma \min\{k, -n\}}$ that
 \begin{align*}
  \bigl| \psi_{\pm, k, (\theta)}^{n, mod, s}(t,x,\xi) \bigr| &\lesssim \theta^{\frac{1}{2}} \bigl( 2^k \| P_k A_{x,s}^{<n-1} \|_{L^\infty_t L^2_x} + 2^{\frac{3}{2} k} \|P_k A_0^{<n-1}\|_{L^2_t L^2_x} \bigr), \\
  \bigl| \psi_{\pm, k}^{n, mod, s}(t,x,\xi) \bigr| &\lesssim 2^k \| P_k A_{x,s}^{<n-1} \|_{L^\infty_t L^2_x} + 2^{\frac{3}{2} k} \|P_k A_0^{<n-1}\|_{L^2_t L^2_x},  \\
  \bigl| \nabla_{t,x} \psi_{\pm, k, (\theta)}^{n, mod, s}(t,x,\xi) \bigr| &\lesssim \theta^{\frac{1}{2}} 2^k \bigl( \| \nabla_{t,x} P_k A_{x,s}^{<n-1} \|_{L^\infty_t L^2_x} + 2^{\frac{1}{2} k} \|\nabla_{t,x} P_k A_0^{<n-1}\|_{L^2_t L^2_x} \bigr), \\
  \bigl| \nabla_{t,x} \psi_{\pm, k}^{n, mod, s}(t,x,\xi) \bigr| &\lesssim 2^k \bigl( \| \nabla_{t,x} P_k A_{x,s}^{<n-1} \|_{L^\infty_t L^2_x} + 2^{\frac{1}{2} k} \|\nabla_{t,x} P_k A_0^{<n-1}\|_{L^2_t L^2_x} \bigr).
 \end{align*}
 Finally, for derivatives of the rough part of the phase function with respect to the frequency variable we have for any multi-index $\alpha$ with $|\alpha| \geq 1$, any $l \geq 0$, any $0 \leq k \leq n$, and any $1 \gtrsim \theta > 2^{\sigma \min\{k, -n\}}$ that
 \begin{align*}
  \bigl| \partial_{|\xi|}^l \partial_\eta^\alpha \psi_{\pm, k, (\theta)}^{n, mod, r}(t,x,\xi) \bigr| &\lesssim \theta^{-1-|\alpha|} 2^{-2k} 2^{+2\delta_\ast k} 2^{(0+)k} \| P_k A_{x,r}^{<n-1} \|_{R_k}, \\ 
  \bigl| \partial_{|\xi|}^l \partial_\eta^\alpha \psi_{\pm, k}^{n, mod, r}(t,x,\xi) \bigr| &\lesssim 2^{\sigma(1+|\alpha|) n} 2^{-2k} 2^{+2\delta_\ast k} 2^{(0+)k} \| P_k A_{x,r}^{<n-1} \|_{R_k}, \\   
  \bigl| \partial_{|\xi|}^l \partial_\eta^\alpha \nabla_{t,x} \psi_{\pm, k}^{n, mod, r, (\theta)}(t,x,\xi) \bigr| &\lesssim \theta^{-1-|\alpha|} 2^{-k} 2^{+2\delta_\ast k} 2^{(0+)k} \| P_k A_{x,r}^{<n-1} \|_{R_k}, \\
  \bigl| \partial_{|\xi|}^l \partial_\eta^\alpha \nabla_{t,x} \psi_{\pm, k}^{n, mod, r}(t,x,\xi) \bigr| &\lesssim 2^{\sigma(1+|\alpha|) n} 2^{-k} 2^{+2\delta_\ast k} 2^{(0+)k} \| P_k A_{x,r}^{<n-1} \|_{R_k}.
 \end{align*}
 Similarly, for derivatives of the smooth part of the phase function with respect to the frequency variable we have for any multi-index $\alpha$ with $|\alpha| \geq 1$, any $l \geq 0$, any $k \in \bbZ$, and any $1 \gtrsim \theta > 2^{\sigma \min\{k, -n\}}$ that
 \begin{align*}
  \bigl| \partial_{|\xi|}^l \partial_\eta^\alpha \psi_{\pm, k, (\theta)}^{n, mod, s}(t,x,\xi) \bigr| &\lesssim \theta^{\frac{1}{2}-|\alpha|} \bigl( 2^k \| P_k A_{x,s}^{<n-1} \|_{L^\infty_t L^2_x} + 2^{\frac{3}{2} k} \|P_k A_0^{<n-1}\|_{L^2_t L^2_x} \bigr), \\ 
  \bigl| \partial_{|\xi|}^l \partial_\eta^\alpha \psi_{\pm, k}^{n, mod, s}(t,x,\xi) \bigr| &\lesssim 2^{\sigma (|\alpha|-\frac{1}{2}) \max\{-k, n\}} \bigl( 2^k \| P_k A_{x,s}^{<n-1} \|_{L^\infty_t L^2_x} + 2^{\frac{3}{2} k} \|P_k A_0^{<n-1}\|_{L^2_t L^2_x} \bigr), \\ 
  \bigl| \partial_{|\xi|}^l \partial_\eta^\alpha \nabla_{t,x} \psi_{\pm, k, (\theta)}^{n, mod, s}(t,x,\xi) \bigr| &\lesssim \theta^{\frac{1}{2}-|\alpha|} 2^k \bigl( \| \nabla_{t,x} P_k A_{x,s}^{<n-1} \|_{L^\infty_t L^2_x} +  2^{\frac{1}{2} k} \|\nabla_{t,x} P_k A_0^{<n-1}\|_{L^2_t L^2_x} \bigr), \\
  \bigl| \partial_{|\xi|}^l \partial_\eta^\alpha \nabla_{t,x} \psi_{\pm, k}^{n, mod, s}(t,x,\xi) \bigr| &\lesssim 2^{\sigma (|\alpha|-\frac{1}{2}) \max\{-k, n\}} 2^k \bigl( \| \nabla_{t,x} P_k A_{x,s}^{<n-1} \|_{L^\infty_t L^2_x} + 2^{\frac{1}{2}k} \| \nabla_{t,x} P_k A_0^{<n-1}\|_{L^\infty_t L^2_x} \bigr).
 \end{align*}
\end{lemma}
\begin{proof}
 The rough part of the ``probabilistic'' phase function coincides with the rough part of the ``deterministic'' phase function up to the tighter angle cut-off in the ``probabilistic'' phase. For this reason the proofs of the bounds for $\psi_{\pm, k, (\theta)}^{n, mod, r}(t,x,\xi)$ with localization to an angle $\theta$ is identical to the proof of the corresponding bounds for the ``deterministic'' phase function. We then obtain slightly different bounds for the rough component $\psi_{\pm, k}^{n, mod, r}(t,x,\xi)$ upon summing over the angles $1 \gtrsim \theta \gtrsim 2^{\sigma \min\{k, -n\}}$ due to the tighter angle cut-off in the definition of the ``probabilistic'' phase function. 
 
 While the smooth part of the ``deterministic'' phase function only contains the free wave evolution $\calA_{x,s}^{0,free}$ of the lowest frequency block, the smooth part of the ``probabilistic'' phase function incorporates the homogeneous and inhomogeneous spatial components of the connection form $A_{x,s}^{<n-1}$ (from frequency stages up to $n-1$) as well as the temporal components of the connection form $A_0^{<n-1}$ (from frequency stages up to $n-1$). Correspondingly, the bounds for the smooth part of the ``probabilistic'' phase function require more explanations. 
 
 In order to estimate the contribution of the spatial component of the connection form $A_{x,s}^{<n-1}$ we again exploit the Coulomb gauge condition to gain another factor of $\theta$ and use the Bernstein estimate $\Pi_{\theta}^\eta P_k L^2_x \to (\theta^3 2^{4k})^{\frac{1}{2}} L^\infty_x$. Specifically, we find that
 \begin{align*}
  \sup_{\eta} \, \Bigl\| (L_{\mp}^\eta)^{-1} \Pi_{\gg 2^k |\measuredangle|^2}^{\mp, \eta} \Pi_\theta^{\eta} \bigl( P_k A_{x,s}^{<n-1} \cdot \eta \bigr) \Bigr\|_{L^\infty_t L^\infty_x} &\lesssim \theta^{-2} 2^{-k} \sup_{\eta} \, \bigl\| \Pi_\theta^{\eta} P_k A_{x,s}^{<n-1} \cdot \eta \bigr\|_{L^\infty_t L^\infty_x} \\
  &\lesssim \theta^{-1} 2^{-k} \sup_\eta \, \bigl\| \Pi_\theta^{\eta} P_k A_{x,s}^{<n-1} \bigr\|_{L^\infty_t L^\infty_x} \\
  &\lesssim \theta^{\frac{1}{2}} 2^k \|P_k A_{x,s}^{<n-1} \|_{L^\infty_t L^2_x}
 \end{align*}
 as well as
 \begin{align*}
  \sup_{\eta} \, \Bigl\| L_\pm^\eta \Delta_{\eta^\perp}^{-1} \Pi_{\lesssim 2^k |\measuredangle|^2}^{\mp, \eta} \Pi_\theta^{\eta} \bigl( P_k A_{x,s}^{<n-1} \cdot \eta \bigr) \Bigr\|_{L^\infty_t L^\infty_x} &\lesssim \theta^{-2} 2^{-k} \sup_{\eta} \, \bigl\| \Pi_\theta^{\eta} P_k A_{x,s}^{<n-1} \cdot \eta \bigr\|_{L^\infty_t L^\infty_x} \\
  &\lesssim \theta^{-1} 2^{-k} \sup_\eta \, \bigl\| \Pi_\theta^{\eta} P_k A_{x,s}^{<n-1} \bigr\|_{L^\infty_t L^\infty_x} \\
  &\lesssim \theta^{\frac{1}{2}} 2^k \|P_k A_{x,s}^{<n-1} \|_{L^\infty_t L^2_x}.
 \end{align*}
In order to estimate the contribution of the temporal component $A_0^{<n-1}$ of the connection form, we dyadically decompose the size of the symbol of $L_\mp^\eta$, i.e. $|\tau \mp \eta \cdot \zeta| \sim \lambda$, $\lambda \in 2^{\bbZ}$. 
Using the Bernstein estimate $\Pi^{\pm, \eta}_\lambda \Pi_{\theta}^\eta P_k L^2_t L^2_x \to (\lambda \, \theta^3 2^{4k})^{\frac{1}{2}} L^\infty_t L^\infty_x$, we then find that
\begin{align*}
 \sup_{\eta} \, \bigl\| (L_{\mp}^\eta)^{-1} \Pi_{\gg 2^k |\measuredangle|^2}^{\mp, \eta} \Pi_\theta^{\eta} \bigl( P_k A_0^{<n-1} \bigr) \bigr\|_{L^\infty_t L^\infty_x} &\lesssim \sum_{\lambda \gg 2^k \theta^2} \sup_{\eta} \,   \bigl\| (L_{\mp}^\eta)^{-1} \Pi_{\lambda}^{\mp, \eta} \Pi_\theta^{\eta} \bigl( P_k A_0^{<n-1} \bigr) \bigr\|_{L^\infty_t L^\infty_x} \\
 &\lesssim \sum_{\lambda \gg 2^k \theta^2} \lambda^{-1} \bigl( \lambda \, \theta^3 2^{4k} \bigr)^{\frac{1}{2}} \|P_k A_0^{<n-1}\|_{L^2_t L^2_x} \\
 &\lesssim \sum_{\lambda \gg 2^k \theta^2} \lambda^{-\frac{1}{2}} \theta^{\frac{3}{2}} 2^{2k} \|P_k A_0^{<n-1}\|_{L^2_t L^2_x} \\
 &\lesssim \theta^{\frac{1}{2}} 2^{\frac{3}{2} k} \|P_k A_0^{<n-1}\|_{L^2_t L^2_x}
\end{align*}
as well as 
\begin{align*}
 \sup_{\eta} \, \bigl\| L_\pm^\eta \Delta_{\eta^\perp}^{-1} \Pi_{\lesssim 2^k |\measuredangle|^2}^{\mp, \eta} \Pi_\theta^{\eta} \bigl( P_k A_0^{<n-1} \bigr) \bigr\|_{L^\infty_t L^\infty_x} &\lesssim \sum_{\lambda \lesssim 2^k \theta^2} \sup_{\eta} \, \bigl\| L_\pm^\eta \Delta_{\eta^\perp}^{-1} \Pi_{\lambda}^{\mp, \eta} \Pi_\theta^{\eta} \bigl( P_k A_0^{<n-1} \bigr) \bigr\|_{L^\infty_t L^\infty_x} \\
 &\lesssim \sum_{\lambda \lesssim 2^k \theta^2} 2^k \theta^{-2} 2^{-2j} \bigl( \lambda \, \theta^3 2^{4k} \bigr)^{\frac{1}{2}} \|P_k A_0^{<n-1}\|_{L^2_t L^2_x} \\
 &\lesssim \sum_{\lambda \lesssim 2^k \theta^2} \lambda^{\frac{1}{2}} \theta^{-\frac{1}{2}} 2^k \|P_k A_0^{<n-1}\|_{L^2_t L^2_x} \\
 &\lesssim \theta^{\frac{1}{2}} 2^{\frac{3}{2} k} \|P_k A_0^{<n-1}\|_{L^2_t L^2_x}.
\end{align*}
Putting the above estimates together we arrive at the following bound on the smooth part of the ``deterministic'' phase function
\begin{equation*}
 \bigl| \psi_{\pm, k, (\theta)}^{n, mod, s}(t,x,\xi) \bigr| \lesssim \theta^{\frac{1}{2}} \bigl( 2^k \|P_k A_{x, s}^{<n-1}\|_{L^\infty_t L^2_x} + 2^{\frac{3}{2} k} \|P_k A_0^{<n-1}\|_{L^2_t L^2_x} \bigr).
\end{equation*}
Then the other bounds on the smooth part of the phase function again follow upon summing over the dyadic angles $1 \gtrsim \theta \gtrsim 2^{\sigma \min\{k,-n\}}$ and upon taking an additional $\nabla_{t,x}$ derivative. 

Finally, the estimates for $\partial_{|\xi|}^l \partial_\eta^\alpha$ derivatives of the ``probabilistic'' phase function are proved similarly, noting that differentiating with respect to $\eta := \frac{\xi}{|\xi|}$ yields additional $\theta^{-1}$ factors, while differentiating with respect to the radial frequency variable $|\xi|$ is harmless since the definition of the phase function only involves~$\eta$. 
\end{proof}

Next, we establish $L^\infty$ bounds for differences of two ``probabilistic'' phase functions. 
\begin{lemma}[Additional symbol bounds for differences of ``probabilistic'' phase functions] \label{lem:prob_difference_phase_fct}
 Let $n \geq 1$ and assume that 
 \begin{equation*}
  \sum_{m=1}^{n-1} \|\calA_{x,r}^m\|_{R_m} + \sum_{m=0}^{n-1} \|\calA_{x,s}^m\|_{S[m]} + \sum_{m=0}^{n-1} \|\calA_0^m\|_{Y^1[m]} \lesssim \varepsilon.
 \end{equation*}
 Then we have for any multi-index $\alpha$ with $1 \leq |\alpha| \leq \frac{\gamma}{2\sigma} + \frac{1}{2}$ and any $l \geq 0$ that
 \begin{align}
  \bigl| \psi_{\pm}^{n, mod}(t,x,\xi) - \psi_{\pm}^{n, mod}(s,y,\xi) \bigr| &\lesssim \varepsilon \log \bigl( 1 + 2^n (|t-s| + |x-y|) \bigr),  \label{equ:prob_phase_function_difference_psis} \\
  \bigl| \partial_{|\xi|}^l \partial_\eta^\alpha \bigl( \psi_{\pm}^{n, mod}(t,x,\xi) - \psi_{\pm}^{n, mod}(s,y,\xi) \bigr) \bigr| &\lesssim \varepsilon \bigl( 1 + 2^n (|t-s| + |x-y|) \bigr)^{\frac{2 \sigma}{\gamma} (|\alpha|-\frac{1}{2}) }. \label{equ:prob_phase_function_difference_xi_derivative_psis}
 \end{align}
\end{lemma}
\begin{proof}
In the following we again use the shorthand notation $T := |x-y| + |t-s|$.
The proof of the first estimate~\eqref{equ:prob_phase_function_difference_psis} is very similar to the corresponding proof of the estimate~\eqref{equ:det_phase_function_difference_psis} for the ``deterministic'' phase function.

In the proof of the second estimate~\eqref{equ:prob_phase_function_difference_xi_derivative_psis} we treat the rough and the smooth part of the phase function separately. 
The treatment of the rough part proceeds similarly to the treatment of the contribution of the rough part of the ``deterministic'' phase function in the proof of the estimate~\eqref{equ:det_phase_function_difference_xi_derivative_psis}. The contributions of the smooth part of the ``probabilistic'' phase function have to be discussed more carefully here.
We distinguish several cases depending on the size of $T$. Throughout we make use of the $L^\infty$ bounds on the ``probabilistic'' phase function from Lemma~\ref{lem:prob_phase_function_Linfty_bounds} without further mentioning.
If $2^{-(1-\gamma)n} \lesssim T \lesssim 2^{n}$, we pick some $-n \leq j_0 \leq (1-\gamma) n$ and decompose into
 \begin{align*}
  &\bigl| \partial_{|\xi|}^l \partial_\eta^\alpha \bigl( \psi_{\pm}^{n, mod, s}(t,x,\xi) - \psi_{\pm}^{n, mod, s}(s,y,\xi) \bigr) \bigr| \\
  &\lesssim \sum_{j \leq j_0} \sup_{\eta} \, \bigl\| \nabla_{t,x} \partial_{|\xi|}^l \partial_\eta^\alpha \psi_{\pm, j}^{n, mod, s} \bigr\|_{L^\infty_t L^\infty_x} \bigl( |t-s| + |x-y| \bigr) + \sum_{j_0 \leq j \leq (1-\gamma)n} \sup_\eta \, \bigl\| \partial_{|\xi|}^l \partial_\eta^\alpha \psi_{\pm, j}^{n, mod, s} \bigr\|_{L^\infty_t L^\infty_x} \\
  &\lesssim \sum_{j \leq j_0} 2^{\sigma(|\alpha|-\frac{1}{2}) \max\{-j, n\} } 2^j T \varepsilon + \sum_{j_0 \leq j \leq (1-\gamma)n} 2^{\sigma(|\alpha|-\frac{1}{2}) \max\{-j, n\} } \varepsilon \\
  &\lesssim 2^{\sigma(|\alpha|-\frac{1}{2})n} 2^{-n} T \varepsilon + \sum_{-n \leq j \leq j_0} 2^{\sigma(|\alpha|-\frac{1}{2}) n} 2^j T \varepsilon + \sum_{j_0 \leq j \leq (1-\gamma)n} 2^{\sigma(|\alpha|-\frac{1}{2}) n } \varepsilon.
 \end{align*}
 Using that $j \leq (1-\gamma) n$ we may further bound the last line by 
 \begin{align*}
  &2^{\sigma(|\alpha|-\frac{1}{2})n} 2^{-n} T \varepsilon + \sum_{-n \leq j \leq j_0} 2^{\frac{\sigma}{\gamma}(|\alpha|-\frac{1}{2}) (n-j)} 2^j T \varepsilon + \sum_{j_0 \leq j \leq (1-\gamma)n} 2^{\frac{\sigma}{\gamma}(|\alpha|-\frac{1}{2}) (n-j) } \varepsilon \\
  &\quad \lesssim \max \{ 1, (2^n T)^{\sigma(|\alpha|-\frac{1}{2})} \} \varepsilon + 2^{\frac{\sigma}{\gamma}(|\alpha|-\frac{1}{2}) (n-j_0)} (2^{j_0} T + 1) \varepsilon.
 \end{align*}
 Then choosing $2^{-j_0} \sim T$ yields the desired bound. If $T \gtrsim 2^n$ the argument proceeds similarly by decomposing with respect to a suitably chosen $j_0 \leq -n$. Finally, if $T \lesssim 2^{-(1-\gamma)n}$ we bound by
 \begin{align*}
  &\bigl| \partial_{|\xi|}^l \partial_\eta^\alpha \bigl( \psi_{\pm}^{n, mod, s}(t,x,\xi) - \psi_{\pm}^{n, mod, s}(s,y,\xi) \bigr) \bigr| \\
  &\quad \lesssim \sum_{j \leq (1-\gamma)n} \sup_{\eta} \, \bigl\| \nabla_{t,x} \partial_{|\xi|}^l \partial_\eta^\alpha \psi_{\pm, j}^{n, mod, s} \bigr\|_{L^\infty_t L^\infty_x} \bigl( |t-s| + |x-y| \bigr) \\
  &\quad \lesssim 2^{\sigma(|\alpha|-\frac{1}{2}) n} 2^{-\gamma n} 2^{n} T \varepsilon.
 \end{align*}
 Then if $2^n T \lesssim 2^{\frac{1}{2} \gamma n}$, we may just bound by $\varepsilon$ as long as $\sigma (|\alpha|-\frac{1}{2}) \leq \frac{1}{2} \gamma$, while if $2^{\frac{1}{2} \gamma n} \lesssim 2^n T \lesssim 2^{\gamma n}$, we can obtain a bound in terms of $(2^n T)^{\frac{2 \sigma}{\gamma} (|\alpha|-\frac{1}{2})} \varepsilon$.
\end{proof}

Finally, we record decomposable estimates for the ``probabilistic'' phase function. 
\begin{lemma}[Decomposable estimates for the ``probabilistic'' phase function] \label{lem:prob_decomposable_est}
 Let $n \geq 1$ and assume that 
 \begin{equation*}
  \sum_{m=1}^{n-1} \|\calA_{x,r}^m\|_{R_m} + \sum_{m=0}^{n-1} \|\calA_{x,s}^m\|_{S[m]} + \sum_{m=0}^{n-1} \|\calA_0^m\|_{Y^1[m]} \lesssim \varepsilon.
 \end{equation*}
 Let $k \leq (1-\gamma)n$. For $2 \leq q < \infty$ and $\frac{2}{q} + \frac{3}{r} \leq \frac{3}{2+}$ we have that 
 \begin{equation} \label{equ:decomposable_estimate_prob_phase_angle}
  \bigl\| \bigl( \psi_{\pm, k, (\theta)}^{n, mod}, 2^{-k} \nabla_{t,x} \psi_{\pm, k, (\theta)}^{n, mod} \bigr) \bigr\|_{D L^q_t L^r_x} \lesssim 2^{-(\frac{1}{q} + \frac{4}{r})k} \theta^{\frac{1}{2+}-\frac{2}{q}-\frac{3}{r}} \varepsilon. 
 \end{equation}
 Moreover, for $4+ < q \leq \infty$ it holds that
 \begin{equation} \label{equ:decomposable_estimate_prob_phase_qinfty}
  \bigl\| \bigl( \psi_{\pm, k}^{n, mod}, 2^{-k} \nabla_{t,x} \psi_{\pm, k}^{n, mod} \bigr) \bigr\|_{D L^q_t L^\infty_x} \lesssim 2^{- \frac{1}{q} k} \varepsilon. 
 \end{equation}
\end{lemma}
\begin{proof}
 The proofs of \eqref{equ:decomposable_estimate_prob_phase_angle}--\eqref{equ:decomposable_estimate_prob_phase_qinfty} for the rough part of the ``probabilistic'' phase function and for the contributions of the spatial components of the connection form $A_{x,s}^{<n-1}$ to the smooth part of the ``probabilistic'' phase function closely resemble the corresponding proofs of \eqref{equ:decomposable_estimate_det_phase_angle}--\eqref{equ:decomposable_estimate_det_phase_qinfty} for the ``deterministic'' phase function. It therefore only remains to discuss the contributions of the temporal component $A_{0}^{<n-1}$. Here it is straightforward to obtain the desired estimates.
 Interchanging integration and the $\eta$ summation as in~\cite[Lemma 7.3]{KST} we find that
 \begin{align*}
  &\biggl( \sum_\eta \, \bigl\| (L_{\mp}^\eta)^{-1} \Pi_{\gg 2^k |\measuredangle|^2}^{\mp, \eta} \Pi_\theta^{\eta} \bigl( P_k A_0^{<n-1} \bigr) \bigr\|_{L^q_t L^r_x}^2 \biggr)^{\frac{1}{2}} \\ 
  &\qquad \lesssim \biggl( \sum_\eta \, \Bigl( \sum_{\lambda \gg 2^k \theta^2} \bigl\| (L_{\mp}^\eta)^{-1} \Pi_{\lambda}^{\mp, \eta} \Pi_\theta^{\eta} \bigl( P_k A_0^{<n-1} \bigr) \bigr\|_{L^q_t L^r_x} \Bigr)^2 \biggr)^{\frac{1}{2}} \\
  &\qquad \lesssim \biggl( \sum_\eta \, \Bigl( \sum_{\lambda \gg 2^k \theta^2} \lambda^{-1} \lambda^{\frac{1}{2}-\frac{1}{q}} (\theta^3 2^{4k})^{\frac{1}{2}-\frac{1}{r}} \bigl\| \Pi_\theta^{\eta} \bigl( P_k A_0^{<n-1} \bigr) \bigr\|_{L^2_t L^2_x} \Bigr)^2 \biggr)^{\frac{1}{2}} \\
  &\qquad \lesssim \theta^{\frac{1}{2}-\frac{2}{q}-\frac{3}{r}} 2^{-(\frac{1}{q}+\frac{4}{r})k} 2^{\frac{3}{2}k} \biggl( \sum_\eta \, \bigl\| \Pi_\theta^{\eta} \bigl( P_k A_0^{<n-1} \bigr) \bigr\|_{L^2_t L^2_x}^2 \biggr)^{\frac{1}{2}} \\
  &\qquad \lesssim \theta^{\frac{1}{2}-\frac{2}{q}-\frac{3}{r}} 2^{-(\frac{1}{q}+\frac{4}{r})k} 2^{\frac{3}{2}k} \bigl\| P_k A_0^{<n-1} \bigr\|_{L^2_t L^2_x} \\
  &\qquad \lesssim \theta^{\frac{1}{2}-\frac{2}{q}-\frac{3}{r}} 2^{-(\frac{1}{q}+\frac{4}{r})k}
 \end{align*}
 and that
 \begin{align*}
  &\biggl( \sum_\eta \, \bigl\| L_\pm^\eta \Delta_{\eta^\perp}^{-1} \Pi_{\lesssim 2^k |\measuredangle|^2}^{\mp, \eta} \Pi_\theta^{\eta} \bigl( P_k A_0^{<n-1} \bigr) \bigr\|_{L^q_t L^r_x}^2 \biggr)^{\frac{1}{2}} \\
  &\qquad \lesssim \biggl( \sum_\eta \, \Bigl( \sum_{\lambda \lesssim 2^k \theta^2} \bigl\| L_\pm^\eta \Delta_{\eta^\perp}^{-1} \Pi_{\lambda}^{\mp, \eta} \Pi_\theta^{\eta} \bigl( P_k A_0^{<n-1} \bigr) \bigr\|_{L^q_t L^r_x} \Bigr)^2 \biggr)^{\frac{1}{2}} \\
  &\qquad \lesssim \biggl( \sum_\eta \, \Bigl( \sum_{\lambda \lesssim 2^k \theta^2} 2^k \theta^{-2} 2^{-2j} \lambda^{\frac{1}{2}-\frac{1}{q}} (\theta^3 2^{4k})^{\frac{1}{2}-\frac{1}{r}} \bigl\| \Pi_\theta^{\eta} \bigl( P_k A_0^{<n-1} \bigr) \bigr\|_{L^2_t L^2_x} \Bigr)^2 \biggr)^{\frac{1}{2}} \\ 
  &\qquad \lesssim \theta^{\frac{1}{2}-\frac{2}{q}-\frac{3}{r}} 2^{-(\frac{1}{q}+\frac{4}{r})k} 2^{\frac{3}{2}k} \biggl( \sum_\eta \, \bigl\| \Pi_\theta^{\eta} \bigl( P_k A_0^{<n-1} \bigr) \bigr\|_{L^2_t L^2_x}^2 \biggr)^{\frac{1}{2}} \\
  &\qquad \lesssim \theta^{\frac{1}{2}-\frac{2}{q}-\frac{3}{r}} 2^{-(\frac{1}{q}+\frac{4}{r})k} 2^{\frac{3}{2}k} \bigl\| P_k A_0^{<n-1} \bigr\|_{L^2_t L^2_x} \\
  &\qquad \lesssim \theta^{\frac{1}{2}-\frac{2}{q}-\frac{3}{r}} 2^{-(\frac{1}{q}+\frac{4}{r})k} \varepsilon.
 \end{align*}
\end{proof}

\subsection{Probabilistic Strichartz estimates for the adapted rough linear evolution $\Phi_r^n$} \label{subsec:prob_strichartz_phi}

We now turn to the derivation of the redeeming probabilistic space-time integrability properties (on a suitable event) of the adapted rough linear evolution $\Phi^n_r$ of the random data $T_n \phi^\omega[0]$, $n \geq 1$. These are a consequence of moment bounds for the redeeming $R_n$ norm of the evolution $\Phi_r^n$ established in the next proposition. 
At its core the proof is based on a combination of Bernstein's inequality, (refined) Strichartz estimates, Minkowski's integral inequality, and Khintchine's inequality, which allows one to decouple the ``atoms'' of the Wiener randomization and gain from their unit-sized frequency supports to beat the scaling. This idea was first used in~\cite{ZF12, LM14} for the Wiener randomization. 

However, in our setting $\Phi^n_r$ is not the free wave evolution of the random data $T_n \phi^\omega[0]$, but a modified linear evolution defined as an infinite sum in terms of iterative applications of the ``probabilistic'' parametrix~\eqref{equ:definition_parametrix_modified}. 
This comes with two main difficulties. 
First, the ``probabilistic'' parametrix~\eqref{equ:definition_parametrix_modified} is defined in terms of the modified phase functions $\psi_{\pm}^{n,mod}$. The definition of the latter involves the connection form $A^{<n-1}$ from the prior induction stages, which however depends in a highly nonlinear manner on the random initial data $T_{<n-1} A_x^\omega[0]$ and $T_{<n-1} \phi^\omega[0]$. Crucially, these are independent of the random data $T_n \phi^\omega[0]$ for the adapted linear evolution $\Phi^n_r$. One can therefore still decouple the ``atoms'' of the random data $T_n \phi^\omega[0]$ for the adapted linear evolution~$\Phi^n_r$ via Khintchine's inequality by conditioning on the $\sigma$-algebra $\calF_{n-1}$ generated by the Gaussians $\{ g_m, \tilde{g}_m, h_m, \tilde{h}_m \, : \, m \in \bbZ^4, |m| < 2^{n-1} \}$. This type of argument first appeared in~\cite{Bringmann18_2} for the Wiener randomization and we refer to~\cite[Proposition 4.4]{Bringmann18_2} for a nice illustration of this circle of ideas within a simpler functional framework.

A second difficulty is that the higher iterates $\Phi^{n, [\ell]}_r$ in the definition of $\Phi^n_r = \sum_{\ell=0}^\infty \Phi^{n, [\ell]}_r$ are defined in terms of iterative applications of the ``probabilistic'' parametrix~\eqref{equ:definition_parametrix_modified}. This could potentially more and more ``smear out'' the unit-sized frequency support of the ``atoms'' of the Wiener randomization and the desired gain from their unit-sized frequency support would eventually break down. However, this is prevented by the careful definition of the rough errors~\eqref{equ:definition_rough_error_stage_l} accrued at every stage in terms of ``small strings of frequencies''. It allows to essentially offset the loss due to the smearing out of the frequency supports by the gain in smallness of the higher iterates, at the expense of a very small regularity loss that is built into the definition of our redeeming norms.

\begin{proposition} \label{prop:prob_strichartz_phi}
 Let $n \geq 1$. Assume that the functions $\{ \calA_{x,r}^m \}_{m=1}^{n-1}$, $\{ \calA_{x,s}^m \}_{m=0}^{n-1}$, $\{ \calA_0^m \}_{m=0}^{n-1}$, $\{ \Phi_r^m \}_{m=1}^{n-1}$, and $\{ \Phi_s^m \}_{m=0}^{n-1}$ are measurable with respect to the $\sigma$-algebra $\calF_{n-1}$ and that we have almost surely 
 \begin{equation*}
  \sum_{m=1}^{n-1} \bigl( \|\calA_{x,r}^m\|_{R_m} + \|\Phi_r^m\|_{R_m} \bigr)  + \sum_{m=0}^{n-1} \bigl( \|\calA_{x,s}^m\|_{S^1[m]} + \|\calA_0^m\|_{Y^1[m]} + \|\Phi^m_s\|_{S^1[m]} \bigr) < \infty.
 \end{equation*}
 Let $\mathds{1}_{[0,2C_0\varepsilon]}$ be the characteristic function of the interval $[0, 2C_0\varepsilon]$ and set
 \begin{equation*}
  \mathds{1}_\varepsilon^{<n-1} := \mathds{1}_{[0,2C_0\varepsilon]} \biggl( \sum_{m=1}^{n-1} \bigl( \|\calA_{x,r}^m\|_{R_m} + \|\Phi_r^m\|_{R_m} \bigr)  + \sum_{m=0}^{n-1} \bigl( \|\calA_{x,s}^m\|_{S^1[m]} + \|\calA_0^m\|_{Y^1[m]} + \|\Phi^m_s\|_{S^1[m]} \bigr) \biggr).
 \end{equation*}
 Let $\Phi^n_r$ be defined as in~\eqref{equ:definition_Phi_n_rough}.
 Then we have for all $1 \leq p < \infty$ that
 \begin{equation} \label{equ:expectation_redeeming_phi}
  \bigl\| \mathds{1}_\varepsilon^{<n-1} \Phi^n_r \bigr\|_{L^p_\omega(\Omega; R_n)} \lesssim \sqrt{p} \, \bigl\| (P_n \phi_0, P_n \phi_1) \bigr\|_{H^{1-\delta_\ast}_x \times H^{-\delta_\ast}_x},
 \end{equation}
 with an analogous bound for $2^{-n} \nabla_{t,x} \Phi^n_r$.
\end{proposition}

Observe that the presence of the cutoff $\mathds{1}_\varepsilon^{<n-1}$ on the left-hand side of~\eqref{equ:expectation_redeeming_phi} is of utmost importance in the proof of Proposition~\ref{prop:prob_strichartz_phi}. It enforces the necessary smallness to invoke the mapping properties of the renormalization operators $e^{\pm i \psi_{\pm}^{n, mod}}_{<n-C}$ from Proposition~\ref{prop:det_renormalization_mapping_properties} and it ensures sufficient smallness to sum up all higher iterates $\Phi^{n,[\ell]}_r$ in the definition of the adapted linear evolution $\Phi^n_r$.

\begin{proof}[Proof of Proposition~\ref{prop:prob_strichartz_phi}]
 We first note that
 \begin{equation*} 
  \bigl\| \mathds{1}_\varepsilon^{<n-1} \Phi^n_r \bigr\|_{L^p_\omega(\Omega; R_n)} \lesssim \sum_{\ell = 0}^\infty \, \bigl\| \mathds{1}_\varepsilon^{<n-1} \Phi^{n, [\ell]}_r \bigr\|_{L^p_\omega(\Omega; R_n)}.
 \end{equation*}
 In what follows we will conclude for every stage $\ell \geq 0$ the moment bound
 \begin{equation} \label{equ:prob_strichartz_phi_moment_bound_stagel}
  \bigl\| \mathds{1}_\varepsilon^{<n-1} \Phi^{n, [\ell]}_r \bigr\|_{L^p_\omega(\Omega; R_n)} \lesssim 10^\ell (\kappa_{n-1})^{2\ell} \sqrt{p} \, \bigl\| (P_n \phi_0, P_n \phi_1) \bigr\|_{H^{1-\delta_\ast}_x \times H^{-\delta_\ast}_x},
 \end{equation} 
 where we recall the definition of $\kappa_{n-1}$ in~\eqref{equ:definition_kappa}.
 Thanks to the cutoff $\mathds{1}_\varepsilon^{<n-1}$ we may assume that $\kappa_{n-1} \ll 1$ so that the asserted moment bound~\eqref{equ:expectation_redeeming_phi} follows from summing the previous estimate~\eqref{equ:prob_strichartz_phi_moment_bound_stagel} over all $\ell \geq 0$.
 
We begin with a careful treatment of the moment bound~\eqref{equ:prob_strichartz_phi_moment_bound_stagel} for the zeroth iterate.
To this end it suffices to only consider the homogeneous parametrix
\begin{equation*}
 \widetilde{\Phi}^{n, [0]}_r(t,x) := \frac12 \sum_\pm e^{-i \psi_{\pm}^{n, mod}}_{<n-C}(t,x,D) e^{\pm i t |D|} e^{+i \psi_{\pm}^{n, mod}}_{<n-C}(D,y,0)  T_n \phi_0^\omega 
\end{equation*}
for the random initial condition 
\begin{equation}
 T_n \phi_0^\omega = \sum_{2^{n-1} \leq |m| < 2^n} h_m(\omega) \varphi(D-m) \phi_0
\end{equation} 
and to establish the bound corresponding to~\eqref{equ:prob_strichartz_phi_moment_bound_stagel} for $\ell =0$, i.e. to show for all $1 \leq p < \infty$ that 
\begin{equation} \label{equ:expectation_redeeming_phi_freqn}
 \bigl\| \mathds{1}_\varepsilon^{<n-1} \widetilde{\Phi}^{n, [0]}_r \bigr\|_{L^p_\omega(\Omega; R_n)} \lesssim \sqrt{p} \, \bigl\| P_n \phi_0  \bigr\|_{H^{1-\delta_\ast}_x}.
\end{equation}
We now establish \eqref{equ:expectation_redeeming_phi_freqn} separately for each component of our redeeming $R_n$ norm. Here it suffices to prove the corresponding bounds for $\widetilde{\Phi}^{n, [0]}_r$, noting that the bounds for $2^{-n} \nabla_{t,x} \widetilde{\Phi}^{n, [0]}_r$ follow analogously.

\medskip 

\noindent {\bf Moment bounds for $R_n L^2_t L^\infty_x$:} 
Recall that the Gaussians $\{ h_m \, : \, m \in \bbZ^4, 2^{n-1} \leq |m| < 2^n \}$ are independent of the $\sigma$-algebra $\calF_{n-1}$ generated by the Gaussians $\{ g_m, \tilde{g}_m, h_m, \tilde{h}_m \, : \, m \in \bbZ^4, |m| < 2^{n-1} \}$ (from the prior induction stages) and that the functions $\{ \calA_{x,r}^m \}_{m=1}^{n-1}$, $\{ \calA_0^m \}_{m=0}^{n-1}$, and $\{ \calA_{x,s}^m \}_{m=0}^{n-1}$ entering the definition of the phase functions $\psi_\pm^{n,mod}$ are assumed to be measurable with respect to~$\calF_{n-1}$. 
Conditioning on $\calF_{n-1}$ and using Bernstein's estimate to go down to $L^{\infty-}_x$, we have  
\begin{equation} \label{equ:moment_bound_RL2Linfty_est1}
\begin{aligned}
 &\bigl\| \mathds{1}_\varepsilon^{<n-1} P_n \widetilde{\Phi}^{n, [0]}_r \bigr\|_{L^p_\omega(\Omega; R_nL^2_t L^\infty_x)} \\
 &= \Biggl( \bbE \biggl[ \biggl( \bbE \, \Bigl[ \bigl\| \mathds{1}_\varepsilon^{<n-1} P_n \widetilde{\Phi}^{n, [0]}_r \bigr\|_{R_nL^2_t L^\infty_x}^p \, \Bigr| \, \calF_{n-1} \Bigr] \biggr)^{\frac{p}{p}} \biggr] \Biggr)^{\frac1p}  \\
 &\lesssim 2^{(0+)n} 2^{(\frac12 - 20 \sigma)n} \Biggl( \bbE \biggl[ \biggl( \bbE \biggl[ \biggl( \, \sum_{l<0}2^{\delta_1 l} \Bigl( \sum_{\substack{k'\leq n,\,l'\leq 0\\n+2l\leq k'+l'\leq n+l}} \gamma(k',l')^{-2} \times \\
 &\qquad \qquad \qquad \qquad \qquad \times \sum_{\kappa} \sum_{\mathcal{C}_{k'}(l')} \big\|P_{\mathcal{C}_{k'}(l')}P_l^{\kappa}Q_{<n+2l} \bigl( \mathds{1}_\varepsilon^{<n-1} P_n \widetilde{\Phi}^{n, [0]}_r \bigr) \big\|_{L_t^2 L_x^{\infty-}}^2 \Bigr)^{\frac12} \biggr)^p \, \biggr| \, \calF_{n-1} \biggr] \biggr)^{\frac{p}{p}} \biggr] \Biggr)^{\frac1p}, 
\end{aligned}
\end{equation}
where we recall from the definition of the redeeming space $R_n L^2_t L^\infty_x$ that
\begin{equation*}
 \gamma(k',l') := \bigl( \min\{ 2^{k'}, 1\} \bigr)^{\frac12 -} \bigl( \min\{ 2^{k'+l'}, 1\} \bigr)^{\frac12-}.
\end{equation*}
We distinguish summation in $k', l'$ over the frequency ranges $n+2l \leq k'+l' \leq \min\{n+l,0\}$ and over $\max\{n+2l, 0\} \leq k'+l' \leq n+l$. We start with the first case.
For any $p \geq \infty-$ we now use Minkowski's integral inequality and Khintchine's inequality (with respect to the conditional expectation), while for $1 \leq p \leq \infty-$ we first apply H\"older's inequality in~$\omega$. Then we can bound the last line~\eqref{equ:moment_bound_RL2Linfty_est1} by
\begin{equation} \label{equ:moment_bounds_RL2Linfty_after_khintchine}
 \begin{aligned}
 &2^{(0+)n} 2^{(\frac12 - 20 \sigma)n} \Biggl( \bbE \biggl[ \biggl( \, \sum_{l<0}2^{\delta_1 l} \biggl( \sum_{\substack{k'\leq n,\,l'\leq 0\\n+2l\leq k'+l'\leq n+l}} \gamma(k',l')^{-2} \times \\
 &\quad \times \sum_{\kappa} \sum_{\mathcal{C}_{k'}(l')} \bigg\| \biggl( \bbE \Bigl[ \, \Bigl| \sum_c h_c(\omega) P_{\mathcal{C}_{k'}(l')} P_l^{\kappa}Q_{<n+2l} \bigl( \mathds{1}_\varepsilon^{<n-1} P_n \widetilde{\Phi}^{n, [0], (c)}_r \bigr) \Bigr|^p \, \Bigr| \, \calF_{n-1} \Bigr] \biggr)^{\frac1p}  \biggr\|_{L_t^2 L_x^{\infty-}}^2 \biggr)^{\frac12}  \biggr)^p \biggr] \Biggr)^{\frac1p}, \\
 &\lesssim \Biggl\| \sqrt{p} \, 2^{(\frac12 - 19 \sigma)n} \sum_{l<0}2^{\delta_1 l} \biggl( \sum_c \sum_{\substack{k'\leq n,\,l'\leq 0\\n+2l\leq k'+l'\leq 0}} \gamma(k',l')^{-2} \times \\
 &\qquad \qquad \qquad \qquad \qquad \qquad \times \sum_{\kappa} \sum_{\mathcal{C}_{k'}(l')} \big\|P_{\mathcal{C}_{k'}(l')}P_l^{\kappa}Q_{<n+2l} \bigl( \mathds{1}_\varepsilon^{<n-1} P_n \widetilde{\Phi}^{n, [0], (c)}_r \bigr) \big\|_{L_t^2 L_x^{\infty-}}^2 \biggr)^{\frac12} \Biggr\|_{L^p_\omega}.
 \end{aligned}
\end{equation}
Here, $\sum_c$ denotes the sum over a covering of the annulus $|\xi| \sim 2^n$ by unit-sized balls with associated frequency projections $P_c$ (coming from the unit-scale frequency projections in the definition of the Wiener randomization) and where the parametrix applied to the (deterministic) initial datum $P_c \phi_0$ is denoted by
\begin{equation*}
 \widetilde{\Phi}^{n, [0], (c)}_r(t, x) := \frac12 \sum_\pm e^{-i \psi_{\pm}^{n, mod}}_{<n-C}(t,x,D) e^{\pm i t |D|} e^{+i \psi_{\pm}^{n, mod}}_{<n-C}(D,y,0)  P_c \phi_0.
\end{equation*}
Note that we tacitly changed the notation of the Gaussians and the unit-scale projections associated with the Wiener randomization to $h_c(\omega)$, respectively to $P_c$, to better distinguish the latter from the standard dyadic Littlewood-Paley projections $P_n$ within this proof.
Our goal is now to bound the integrand inside the $L^p_\omega$ norm on the right-hand side of~\eqref{equ:moment_bounds_RL2Linfty_after_khintchine} by $\sqrt{p} \|P_n \phi_0\|_{H^{1-\delta_\ast}_x}$. At that point the $L^p_\omega$ norm can be trivially dropped. To this end we have to distinguish several cases depending on the frequency localization of the symbols $e^{\pm i \psi_{\pm}^{n, mod}}(t,x,\xi)$. We introduce the corresponding short-hand notations
\begin{align*}
 \widetilde{\Phi}^{n, [0], (c)}_{r, LL}(t,x) &:= \frac12 \sum_\pm e^{-i \psi_{\pm}^{n, mod}}_{<-C}(t,x,D) e^{\pm i t |D|} e^{+i \psi_{\pm}^{n, mod}}_{<-C}(D,y,0) P_c \phi_0, \\
 \widetilde{\Phi}^{n, [0], (c)}_{r, LH}(t,x) &:= \sum_{-C \leq \ell_2 \leq n-C} \frac12 \sum_\pm e^{-i \psi_{\pm}^{n, mod}}_{<-C}(t,x,D) e^{\pm i t |D|} e^{+i \psi_{\pm}^{n, mod}}_{\ell_2}(D,y,0) P_c \phi_0, \\
 \widetilde{\Phi}^{n, [0], (c)}_{r, HL}(t,x) &:= \sum_{-C \leq \ell_1 \leq n-C} \frac12 \sum_\pm e^{-i \psi_{\pm}^{n, mod}}_{\ell_1}(t,x,D) e^{\pm i t |D|} e^{+i \psi_{\pm}^{n, mod}}_{<-C}(D,y,0) P_c \phi_0, \\
 \widetilde{\Phi}^{n, [0], (c)}_{r, HH}(t,x) &:= \sum_{-C \leq \ell_1, \ell_2 \leq n-C} \frac12 \sum_\pm e^{-i \psi_{\pm}^{n, mod}}_{\ell_1}(t,x,D) e^{\pm i t |D|} e^{+i \psi_{\pm}^{n, mod}}_{\ell_2}(D,y,0) P_c \phi_0.
\end{align*}

\medskip 

\noindent {\it Case 1: Bounding the contribution of $\widetilde{\Phi}^{n, [0], (c)}_{r, LL}$:}  
Here the frequency projection $P_c$ can essentially be moved through the parametrix to the outside (up to passing to a slight enlargement $\widetilde{P}_c$), and we reduce to bounding 
\begin{equation} \label{equ:moment_bounds_RL2Linfty_LL1}
 \sqrt{p} \, 2^{(\frac12 - 19 \sigma)n} \sum_{l<0}2^{\delta_1 l} \biggl( \sum_c  \sum_{\substack{k'\leq n,\,l'\leq 0\\n+2l\leq k'+l'\leq 0}} \gamma(k',l')^{-2} \sum_{\kappa} \sum_{\mathcal{C}_{k'}(l')} \big\|P_{\mathcal{C}_{k'}(l')} \widetilde{P}_c P_l^{\kappa}Q_{<n+2l} \bigl( \mathds{1}_\varepsilon^{<n-1} P_n \widetilde{\Phi}^{n, [0], (c)}_{r, LL} \bigr) \big\|_{L_t^2 L_x^{\infty-}}^2 \biggr)^{\frac12}.
\end{equation}
Then for fixed choice of $k', l'$ (with $k'+l' \leq 0$) we first use the Bernstein estimate $P_{\calC_{k'}(\ell')} \widetilde{P}_c L^6_x \to ( 2^{3(k'+l')} \min\{2^{k'}, 1\} )^{\frac16 -} L^{\infty-}_x$ to bound 
\begin{align*}
 &\sum_{\kappa} \sum_{\mathcal{C}_{k'}(l')} \big\| P_{\mathcal{C}_{k'}(l')} \widetilde{P}_c P_l^{\kappa}Q_{<n+2l} \bigl( \mathds{1}_\varepsilon^{<n-1} P_n \widetilde{\Phi}^{n, [0], (c)}_{r, LL} \bigr) \big\|_{L_t^2 L_x^{\infty-}}^2 \\
 &\quad \lesssim \Bigl( \bigl( 2^{3(k'+l')} \min\{2^{k'}, 1\} \bigr)^{\frac16 -} \Bigr)^2 \sum_{\kappa} \sum_{\mathcal{C}_{k'}(l')} \big\|P_{\mathcal{C}_{k'}(l')} \widetilde{P}_c P_l^{\kappa}Q_{<n+2l} \bigl( \mathds{1}_\varepsilon^{<n-1} P_n \widetilde{\Phi}^{n, [0], (c)}_{r, LL} \bigr) \big\|_{L_t^2 L_x^6}^2. 
\end{align*}
Then we invoke that by the mapping properties of the renormalization operator $e^{-i \psi_{\pm}^{n, mod}}_{<-C}(t,x,D)$ as in Proposition~\ref{prop:prob_renormalization_mapping_properties}, we have a square-summed $L^2_t L^6_x$ Strichartz estimate with a gain from frequency localization to a ball of diameter $\sim \min\{ 2^{k'}, 1\}$ (at distance $\sim 2^n$ from the origin of frequency space) for the wave operator $e^{-i \psi_{\pm}^{n, mod}}_{<-C}(t,x,D) e^{\pm i t |D|} e^{+i \psi_{\pm}^{n, mod}}_{<-C}(D,y,0) P_c$ (see Section 11 in~\cite{KST}). Thus, we can further estimate the previous line by
\begin{align*}
 &\Bigl( \bigl( 2^{3(k'+l')} \min\{2^{k'}, 1\} \bigr)^{\frac16 -} \Bigr)^2 \sum_{\kappa} \sum_{\mathcal{C}_{k'}(l')} \big\|P_{\mathcal{C}_{k'}(l')} \widetilde{P}_c P_l^{\kappa}Q_{<n+2l} \bigl( \mathds{1}_\varepsilon^{<n-1} P_n \widetilde{\Phi}^{n, [0], (c)}_{r, LL} \bigr) \big\|_{L_t^2 L_x^6}^2 \\
 &\quad \lesssim \Bigl( \bigl( 2^{3(k'+l')} \min\{2^{k'}, 1\} \bigr)^{\frac16 -} \bigl( \min\{ 2^{k'}, 1 \} 2^{-n} \bigr)^{\frac{1}{3}} 2^{\frac{5}{6}n} \Bigr)^2 \|P_c \phi_0\|_{L^2_x}^2 \\
 &\quad \simeq \Bigl( 2^{(\frac{1}{2}-)(k'+l')} \bigl(  \min\{2^{k'}, 1\} \bigr)^{\frac12 -} 2^{\frac{1}{2} n} \Bigr)^2 \|P_c \phi_0\|_{L^2_x}^2.
\end{align*}
Then we use a small portion of the factor $2^{(\frac{1}{2}-)(k'+l')}$ to sum over $k', l'$ in the indicated range. Square-summing over the unit-sized cubes $c$, we find that~\eqref{equ:moment_bounds_RL2Linfty_LL1} is safely bounded by $\sqrt{p} \, \|P_n \phi_0\|_{H^{1-\delta_\ast}_x}$.

\medskip 

\noindent {\it Case 2: Bounding the contribution of $\widetilde{\Phi}^{n, [0], (c)}_{r, LH}$:} 
In this case we first observe that for $-C \leq \ell_2 \leq n-C$ the operator $e^{+i \psi_{\pm}^{n, mod}}_{\ell_2}(D,y,0) P_c$ has the mapping property
\begin{equation} \label{equ:moment_bounds_RL2Linfty_LH_mapping_property}
 e^{+i \psi_{\pm}^{n, mod}}_{\ell_2}(D,y,0) P_c \colon L^2_y \to 2^{10\sigma n} 2^{-2 \ell_2} L^2_y.
\end{equation}
This follows from a crude decomposable estimate $\|\psi_{\pm, \ell_2}^{n, mod}\|_{D L^\infty L^2} \lesssim 2^{10 \sigma n} 2^{-2 \ell_2}$ and the unit-scale Bernstein estimate $P_c \colon L^2_y \to L^\infty_y$.
Moreover, we note that the evolution
\begin{equation*}
 e^{-i \psi_{\pm}^{n, mod}}_{<-C}(t,x,D) e^{\pm i t |D|} e^{+i \psi_{\pm}^{n, mod}}_{\ell_2}(D,y,0) P_c \phi_0, \quad -C \leq \ell_2 \leq n-C,
\end{equation*}
has spatial Fourier support in a ball of diameter $\sim 2^{\ell_2}$ (located at distance $\sim 2^n$ from the origin of frequency space). We may therefore freely insert outside a corresponding frequency projection $\widetilde{P}_{2^{\ell_2} c}$ adapted to a slight enlargement of that ball.
Then we proceed analogously to Case 1 and use the Bernstein estimate $P_{\calC_{k'}(\ell')} \widetilde{P}_{2^{\ell_2}c} L^6_x \to ( 2^{3(k'+l')} \min\{2^{k'}, 2^{\ell_2} \} )^{\frac16 -} L^{\infty-}_x$ and subsequently a square-summed $L^2_t L^6_x$ Strichartz estimate with gain from frequency localization to a ball of diameter $\sim \min\{ 2^{k'}, 2^{\ell_2} \}$ together with the mapping property~\eqref{equ:moment_bounds_RL2Linfty_LH_mapping_property} to find for fixed choice of $k', l'$ (with $k'+l' \leq 0$) that 
\begin{align*}
&\sum_{\kappa} \sum_{\mathcal{C}_{k'}(l')} \big\|P_{\mathcal{C}_{k'}(l')} P_l^{\kappa} Q_{<n+2l} \bigl( \mathds{1}_\varepsilon^{<n-1} P_n \widetilde{\Phi}^{n, [0], (c)}_{r, LH} \bigr) \big\|_{L_t^2 L_x^{\infty-}}^2 \\
&\quad \lesssim \biggl( \sum_{-C \leq \ell_2 \leq n-C} 2^{(\frac{1}{2}-)(k'+l')} \bigl(  \min\{2^{k'}, 2^{\ell_2} \} \bigr)^{\frac12 -} 2^{(\frac12 + 10 \sigma) n} 2^{-2 \ell_2} \|P_c \phi_0\|_{L^2_x} \biggr)^{2}.
\end{align*}
Thus, we have at least a gain of $2^{-\frac32 \ell_2}$, which can be summed over $-C \leq \ell_2 \leq n-C$. Then, again using a small fraction of the factor $2^{(\frac{1}{2}-)(k'+l')}$ to sum over $k', l'$ in the indicated range and square-summing over the unit-sized cubes $c$, we obtain that the contribution of $\widetilde{\Phi}^{n, [0], (c)}_{r, LH}$ to~\eqref{equ:moment_bounds_RL2Linfty_after_khintchine} is safely bounded by $\sqrt{p} \, \|P_n \phi_0\|_{H^{1-\delta_\ast}_x}$.

\medskip 

\noindent {\it Case 3: Bounding the contribution of $\widetilde{\Phi}^{n, [0], (c)}_{r, HL}$:} 
We begin by noting that the evolution
\begin{equation*}
 e^{-i \psi_{\pm}^{n, mod}}_{\ell_1}(t,x,D) e^{\pm i t |D|} e^{+i \psi_{\pm}^{n, mod}}_{<-C}(D,y,0) P_c \phi_0
\end{equation*}
has spatial Fourier support in a ball of diameter $\sim 2^{\ell_1}$ (located at distance $\sim 2^n$ from the origin of frequency space). We may therefore freely insert a corresponding frequency projection $\widetilde{P}_{2^{\ell_1} c}$ adapted to a slight enlargement of that ball. Hence, using the Bernstein estimate $P_{\calC_{k'}(l')} \widetilde{P}_{2^{\ell_1} c} \colon L^2_x \to \bigl( 2^{3(k'+l')} \min\{ 2^{k'}, 2^{\ell_1} \} \bigr)^{\frac{1}{2}-} L^{\infty-}_x$ and square-summing over the caps $\kappa$ and boxes $\calC_{k'}(l')$, we obtain for fixed choice of $k', l'$ (with $k' + l' \leq 0$) that 
\begin{align*}
&\sum_{\kappa} \sum_{\mathcal{C}_{k'}(l')} \big\|P_{\mathcal{C}_{k'}(l')} P_l^{\kappa} Q_{<n+2l} \bigl( \mathds{1}_\varepsilon^{<n-1} P_n \widetilde{\Phi}^{n, [0], (c)}_{r, HL} \bigr) \big\|_{L_t^2 L_x^{\infty-}}^2 \\
&\lesssim \Biggl( \sum_{-C \leq \ell_1 \leq n-C}  \biggl( \sum_{\kappa} \sum_{\mathcal{C}_{k'}(l')} \big\| \widetilde{P}_{2^{\ell_1} c} P_{\mathcal{C}_{k'}(l')} P_l^{\kappa} Q_{<n+2l} \bigl( \mathds{1}_\varepsilon^{<n-1} e^{-i \psi_{\pm}^{n, mod}}_{\ell_1}(t,x,D) e^{\pm i t |D|} e^{+i \psi_{\pm}^{n, mod}}_{<-C}(D,y,0) P_c \phi_0  \bigr) \big\|_{L_t^2 L_x^{\infty-}}^2 \biggr)^{\frac12} \Biggr)^2 \\
&\lesssim \Biggl( \sum_{-C \leq \ell_1 \leq n-C} \bigl( 2^{3(k'+l')} \min\{ 2^{k'}, 2^{\ell_1} \} \bigr)^{\frac{1}{2}-} \big\| \mathds{1}_\varepsilon^{<n-1} e^{-i \psi_{\pm}^{n, mod}}_{\ell_1}(t,x,D) \widetilde{P}_c e^{\pm i t |D|} e^{+i \psi_{\pm}^{n, mod}}_{<-C}(D,y,0) P_c \phi_0 \big\|_{L_t^2 L_x^2} \Biggr)^2.
\end{align*}
In the last line we already indicated that (a slight enlargement of) the frequency projection $P_c$ can be moved through to the outside of the renormalization operator $e^{+i \psi_{\pm}^{n, mod}}_{<-C}(D,y,0)$. 
Since the operator $e^{-i \psi_{\pm}^{n, mod}}_{\ell_1}(t,x,D) \widetilde{P}_c$ is essentially smooth at the scale of the unit-scale projection $\widetilde{P}_c$, we obtain from crude decomposability estimates the operator bound \begin{align*}
  e^{-i \psi_{\pm}^{n, mod}}_{\ell_1}(t,x,D) \widetilde{P}_c \colon L^\infty_x \to 2^{10\sigma n} 2^{-2\ell_1} L^2_x.
\end{align*}
By the mapping properties of the renormalization operator as in Proposition~\ref{prop:prob_renormalization_mapping_properties}, it follows that the wave operator $\widetilde{P}_c e^{\pm it|D|} e^{+i \psi_{\pm}^{n, mod}}_{<-C}(D,y,0)$ satisfies an improved Strichartz estimate $L^2_x \to 2^{\frac12 n} L^2_t L^\infty_x$ with gain from frequency localization to a ball of diameter $\sim 1$ (at distance $\sim 2^n$ from the origin of frequency space). Hence, the previous line can be further bounded by 
\begin{equation*}
 \biggl( 2^{10\sigma n} \sum_{-C \leq \ell_1 \leq n-C} 2^{(\frac32 -)(k'+l')} \bigl( \min\{ 2^{k'}, 2^{\ell_1} \} \bigr)^{\frac{1}{2}-} 2^{-2\ell_1} 2^{\frac12 n} \|P_c \phi_0 \|_{L_x^2} \biggr)^2.
\end{equation*}
Again, we have at least a gain of $2^{-\frac32 \ell_2}$, which can be summed over $-C \leq \ell_2 \leq n-C$, and subsequently we can use a small portion of the factor $2^{(\frac{3}{2}-)(k'+l')}$ to sum over $k', l'$ in the indicated range. After square-summing over the unit-sized cubes $c$, we obtain that the contribution of $\widetilde{\Phi}^{n, [0], (c)}_{r, HL}$ to~\eqref{equ:moment_bounds_RL2Linfty_after_khintchine} is also safely bounded by$\sqrt{p} \, \|P_n \phi_0\|_{H^{1-\delta_\ast}_x}$.

\medskip 

\noindent {\it Case 4: Bounding the contribution of $\widetilde{\Phi}^{n, [0], (c)}_{r, HH}$:} 
This case can be treated by combining the arguments from the previous two cases. 

\medskip 

To conclude the discussion of the derivation of the redeeming $R_n L^2_t L^\infty_x$ bounds, it remains to describe how to deal with the frequency range $\max\{n+2l, 0\} \leq k'+l' \leq n+l$. We proceed as above for the other frequency range, but in this regime we can only exploit the unit-scale frequency localization and cannot further gain from the radially directed frequency blocks $\calC_{k'}(l')$. The summation over $k',l'$ in this frequency range comes at the expense of a factor $2^{(0+)n}$ that can be safely compensated.

\medskip 

\noindent {\bf Moment bounds for $R_n L^2_t L^6_x$:}
This is effectively just a special case of the derivation of the moment bounds for the $R_n L^2_t L^\infty_x$ norm, we therefore omit the details.

\medskip 

\noindent {\bf Moment bounds for $R_n L^\infty_t L^\infty_x$:} 
We proceed analogously to the derivation of the moment bounds for the redeeming $R_n L^2_t L^\infty_x$ norm. First, we use Bernstein's inequality and fractional Sobolev embedding in time to go down to $L^{\infty-}_t L^{\infty-}_x$ at the expense of picking up a factor $2^{(0+)n}$. Then we condition on $\calF_{n-1}$, use Minkowski's integral inequality together with Khintchine's inequality as above, and distinguish the same four cases depending on the frequency localization of the renormalization operator symbols. In order to estimate the contribution of the main term $\widetilde{\Phi}^{n, [0], (c)}_{r, LL}(t)$, we use the Bernstein estimate $P_l^\kappa \widetilde{P}_c \colon L^{2+}_x \to \bigl( \min\{ 2^{3(n+l)}, 1 \} \bigr)^{\frac12 -} L^{\infty-}_x$ and subsequently a square-summed $L^{\infty-}_t L^{2+}_x$ Strichartz estimate. In the other three cases we proceed analogously to the treatment of the $R_n L^2_t L^\infty_x$ norm and play out the mismatched frequency localizations. 

\medskip 

\noindent {\bf Moment bounds for $R_n Str$:} 
This bound is also very analogous to the derivation of the moment bounds for the redeeming $R_n L^2_t L^\infty_x$ norm above. For a given admissible Strichartz pair $(q,r)$, we first condition on $\calF_{n-1}$ and use Minkowski's integral inequality together with Khintchine's inequality (after possibly going down to $L^{\infty-}_t$ or $L^{\infty-}_x$ at the expense of a factor $2^{(0+)n}$). Then we again distinguish the same four cases depending on the frequency localization of the renormalization operator symbols. For the contribution of the main term $\widetilde{\Phi}^{n, [0], (c)}_{r, LL}(t)$, we first use the unit-scale Bernstein estimate $\widetilde{P}_c \colon L^{r_0}_x \to L^r_x$ where $(q,r_0)$ is sharp-admissible and then apply a Strichartz estimate with gain from the frequency localization to a unit-sized ball (at distance $\sim 2^n$ from the origin of frequency space). In the other cases we take advantage of the mismatched frequency localizations. 

\medskip 

\noindent {\bf Moment bounds for $S^{1-\delta_\ast}_n$:} 
Here we do not look for a gain from probabilistic decoupling of the unit-scale frequency localized pieces. By the mapping properties from Proposition~\ref{prop:prob_renormalization_mapping_properties}, we right away have an $S^{1-\delta_\ast}_n$ norm bound for the homogeneous parametrix $\Phi^{n,[0]}_r$ (as in Section~11 in~\cite{KST}), and only subsequently apply Khintchine's inequality and square-sum over the unit-scale frequency localized pieces to find for any $p \geq 2$ that
\begin{equation*}
 \bigl\| \mathds{1}_\varepsilon^{<n-1} \Phi^n_r \bigr\|_{L^p_\omega(\Omega; S^{1-\delta_\ast}_n)} \lesssim \bigl\| (T_n \phi_0^\omega, T_n \phi_1^\omega) \bigr\|_{L^p_\omega(\Omega; H^{1-\delta_\ast}_x \times H^{-\delta_\ast}_x)} \lesssim \sqrt{p} \, \bigl\| (P_n \phi_0, P_n \phi_1) \bigr\|_{H^{1-\delta_\ast}_x \times H^{-\delta_\ast}_x}.
\end{equation*}
In the case $1 \leq p \leq 2$ we first use H\"older's inequality in $\omega$.

\medskip 

This finishes the proof of the moment bound~\eqref{equ:prob_strichartz_phi_moment_bound_stagel} for the zeroth iterate and it now remains to discuss the derivation of the moment bounds~\eqref{equ:prob_strichartz_phi_moment_bound_stagel} for all higher iterates $\ell \geq 1$.
Recall that the higher iterates $\Phi^{n,[\ell]}_r$, $\ell \geq 1$, are defined as the inhomogeneous parametrix applied to the error $-\calE^{[\ell-1]}_{rough}$. The key point that makes the derivation of the redeeming bounds work for all higher iterates $\Phi^{n,[\ell]}_r$, $\ell \geq 1$, is that the error $\calE^{[\ell-1]}_{rough}$ defined in~\eqref{equ:definition_rough_error_stage_l} consists only of ``small strings of frequencies of length $\ell-1$'' that are all at frequencies $\leq 2^{3 \sigma n}$. Thus, the Fourier support of the $\ell$-th iterate $\Phi^{n, [\ell], (c)}_r$ applied to a single unit-sized frequency localized piece smears out by at most $\lesssim 10^\ell 2^{3\sigma n}$. Moreover, the error $\calE^{[\ell-1]}_{rough}$ gains smallness and is of size $(\kappa_{n-1})^{2\ell}$. This smallness gain stems from repeatedly estimating the schematic magnetic potential terms 
\begin{equation*}
 Z_k^{\eta, \mp} \bigl( P_k A_{x,s}^{<n-1} \cdot \xi \mp P_{k} A_0^{<n-1} |\xi| \bigr)
\end{equation*}
using the equations for $A_{x,s}^{<n-1}$ and $A_0^{<n-1}$ (whose nonlinearities are at least quadratic). See the proof of Proposition~\ref{prop:prob_renormalization_error_estimate} for more details).

In the proof of the moment bound for the redeeming norms of $\Phi^{n,[\ell]}_r$, $\ell \geq 1$, we then only once have to argue analogously to the zeroth iterate above, namely for the inhomogeneoux parametrix applied to the error $-\calE^{[\ell-1]}_{rough}$, which costs some $2^{10\sigma n}$. Since we built enough room of size $2^{-20\sigma n}$ into the definition of our redeeming spaces, we can easily absorb the additional factor of $2^{3\sigma n}$ arising due to the smearing out of the frequency support. Overall, we therefore obtain the desired bound
\begin{equation*}
 \bigl\| \mathds{1}_\varepsilon^{<n-1} \Phi^{n, [\ell]}_r \bigr\|_{L^p_\omega(\Omega; R_n)} \lesssim 10^\ell (\kappa_{n-1})^{2\ell} \sqrt{p} \, \bigl\| (P_n \phi_0, P_n \phi_1) \bigr\|_{H^{1-\delta_\ast}_x \times H^{-\delta_\ast}_x}.
\end{equation*}
\end{proof}

\subsection{Probabilistic energy bounds for the data error $\Phi^n_s[0]$} \label{subsec:prob_data_error}

Next, we present a new type of moment bound, which ensures that the {\it data error} generated by the modified rough linear evolution $\Phi^n_r$ in fact gains smoothness (on a suitable event).

\begin{proposition} \label{prop:prob_data_error}
 Let $n \geq 1$. Assume that we have almost surely 
 \begin{equation*}
  \sum_{m=1}^{n-1} \bigl( \|\calA_{x,r}^m\|_{R_m} + \|\Phi_r^m\|_{R_m} \bigr)  + \sum_{m=0}^{n-1} \bigl( \|\calA_{x,s}^m\|_{S^1[m]} + \|\calA_0^m\|_{Y^1[m]} + \|\Phi^m_s\|_{S^1[m]} \bigr) < \infty.
 \end{equation*}
 Let $\mathds{1}_{[0,2C_0\varepsilon]}$ be the characteristic function of the interval $[0, 2C_0\varepsilon]$ and set
 \begin{equation*}
  \mathds{1}_\varepsilon^{<n-1} := \mathds{1}_{[0,2C_0\varepsilon]} \biggl( \sum_{m=1}^{n-1} \bigl( \|\calA_{x,r}^m\|_{R_m} + \|\Phi_r^m\|_{R_m} \bigr)  + \sum_{m=0}^{n-1} \bigl( \|\calA_{x,s}^m\|_{S^1[m]} + \|\calA_0^m\|_{Y^1[m]} + \|\Phi^m_s\|_{S^1[m]} \bigr) \biggr).
 \end{equation*}
 Let $\Phi^n_r$ be defined as in~\eqref{equ:definition_Phi_n_rough}. For every $1 \leq p < \infty$ it holds that
 \begin{equation} \label{equ:prob_data_error}
  \bigl\| \mathds{1}_\varepsilon^{<n-1} \bigl( \Phi^n_r[0] - T_n\phi^\omega[0] \bigr) \bigr\|_{L^p_\omega(\Omega; \dot{H}^1_x \times L^2_x)} \lesssim \sqrt{p} \, \bigl\| (P_n \phi_0, P_n \phi_1) \bigr\|_{H^{1-\delta_\ast}_x \times H^{-\delta_\ast}_x}.
 \end{equation}
\end{proposition}

We emphasize that the presence of the cut-off $\mathds{1}_\varepsilon^{<n-1}$ on the left-hand side of~\eqref{equ:prob_data_error} is again crucial to enforce the necessary smallness to invoke the mapping properties of the renormalization operators $e^{\pm i \psi_{\pm}^{n, mod}}_{<n-C}$ from Proposition~\ref{prop:det_renormalization_mapping_properties}, which enter the definition of~$\Phi_r^n$. Moreover, it guarantees the necessary smallness to sum up all higher iterates in the definition of $\Phi^n_r$.

\begin{proof}[Proof of Proposition~\ref{prop:prob_data_error}]
Here the main work actually goes into proving that almost surely we have
\begin{equation} \label{equ:prob_data_error_main_work}
  \bigl\| \mathds{1}_\varepsilon^{<n-1} \bigl( \Phi^n_r[0] - T_n\phi^\omega[0] \bigr) \bigr\|_{\dot{H}^1_x \times L^2_x} \lesssim \bigl\| \mathds{1}_\varepsilon^{<n-1} T_n \phi^\omega[0] \bigr\|_{H^{1-\delta_\ast}_x \times H^{-\delta_\ast}_x}.
 \end{equation}
 Then the asserted bound~\eqref{equ:prob_data_error} is a simple consequence of~\eqref{equ:prob_data_error_main_work} and a subsequent application of Khintchine's inequality (so that no conditioning on $\calF_{n-1}$ is necessary).
 We begin with the $\dot{H}^1_x$ bound. By definition of~$\Phi_r^n$, we have that
 \begin{equation} \label{equ:data_error_difference}
  \Phi^n_r(0) - T_n \phi^\omega_0 = \Phi^{n, [0]}_r(0) - T_n \phi^\omega_0 = \frac{1}{2} \sum_{\pm} \Bigl( e_{<n-C}^{-i \psi_{\pm}^{<n, mod}}(0,x,D) e_{<n-C}^{+i\psi_{\pm}^{<n, mod}}(D, y, 0) - 1 \Bigr) T_n \phi^\omega_0.
 \end{equation}
 Correspondingly, we need to show that
 \begin{equation} \label{equ:data_error_goal}
  \Bigl\| \mathds{1}_\varepsilon^{<n-1} \Bigl( e_{<n-C}^{-i \psi_{\pm}^{<n, mod}}(0,x,D) e_{<n-C}^{+i\psi_{\pm}^{<n, mod}}(D, y, 0) - 1 \Bigr) T_n \phi_0^\omega \Bigr\|_{\dot{H}^1_x} \lesssim 2^{-\delta_\ast n} \bigl\| T_n \phi_0^\omega \bigr\|_{\dot{H}^1_x}. 
 \end{equation}
 To this end we introduce the kernel 
 \begin{equation*}
  K_n(x,y) := \mathds{1}_\varepsilon^{<n-1} \int_{\bbR^4} e^{i(x-y)\cdot\xi} \Bigl( e^{-i \psi_{\pm}^{<n, mod}(0,x,\xi)} e^{+i\psi_{\pm}^{<n, mod}(\xi,y,0)} - 1 \Bigr) \chi \Bigl( \frac{\xi}{2^n} \Bigr) \, \ud \xi,
 \end{equation*}
 where $\chi(z)$ is a suitable bump function supported around $|z| \sim 1$. 
 Then by averaging arguments such as in Proposition~8.2 in~\cite{KST} and by the decomposable estimate $\| \nabla_{t,x} \psi_\pm^{<n,mod} \|_{D L^\infty_t L^\infty_x} \lesssim 2^{(1-\gamma)n} \varepsilon$ from~\eqref{equ:decomposable_estimate_prob_phase_qinfty}, the bound~\eqref{equ:data_error_goal} reduces  to proving that
 \begin{equation*}
  \biggl\| \int_{\bbR^4} K_n(x,y) (P_n f)(y) \, \ud y \biggr\|_{L^2_x} \lesssim 2^{-\delta_\ast n} \|P_n f\|_{L^2_x},
 \end{equation*}
 which in turn follows from Schur's test upon establishing that 
 \begin{equation} \label{equ:schur_test_needed}
  \sup_{y} \, \bigl\| K_n(x,y) \bigr\|_{L^1_x} + \sup_{x} \, \bigl\| K_n(x,y) \bigr\|_{L^1_y} \lesssim 2^{-\delta_\ast n}.
 \end{equation}
 We only estimate the first term on the left-hand side, the estimate for the second term being analogous. By non-stationary phase arguments using~\eqref{equ:prob_phase_function_difference_xi_derivative_psis}, we have 
 \begin{equation*}
  |K_n(x,y)| \lesssim_N 2^{4n} \japanese{2^n (x-y)}^{-N}, \quad \text{ for any } 1 \leq N \ll \frac{\gamma}{2\sigma}.
 \end{equation*}
 This decay estimate easily yields the desired bound for $|x-y| \gtrsim 2^{-(1-\delta_\ast)n}$ since we have for $N \gg 1$ that
 \begin{align*}
  \sup_y \, \int_{\{|x-y| \gtrsim 2^{-(1-\delta_\ast)n}\}} |K_n(x,y)| \, \ud y &\lesssim 2^{-(N-4)\delta_\ast n} \lesssim 2^{-\delta_\ast n}.
 \end{align*}
 If $|x-y| \lesssim 2^{-(1-\delta_\ast)n}$ we use that by Lemma~\ref{lem:prob_phase_function_Linfty_bounds} we have the bound 
 \begin{align*}
  |K_n(x,y)| 
  &\lesssim 2^{4n} \bigl\| \nabla_{t,x} \psi_{\pm}^{<n,mod} \bigr\|_{L^\infty_t L^\infty_x} |x-y| \lesssim 2^{4n} 2^{(1-\gamma)n} |x-y| \varepsilon,
 \end{align*}
 and hence, 
 \begin{align*}
  \sup_y \, \int_{\{|x-y| \lesssim 2^{-(1-\delta_\ast)n}\}} |K_n(x,y)| \, \ud y &\lesssim 2^{4n} 2^{(1-\gamma)n} \sup_y \, \int_{\{|x-y| \lesssim 2^{-(1-\delta_\ast)n}\}} |x-y| \, \ud y \\
  &\lesssim 2^{4n} 2^{(1-\gamma)n} 2^{-5(1-\delta_\ast)n} \\
  &\lesssim 2^{-(\gamma - 5\delta_\ast)n}.
 \end{align*}
 This yields the desired estimate~\eqref{equ:schur_test_needed} since $1 \gg \gamma \gg \delta_\ast > 0$.
 
 The proof of the $L^2_x$ bound for the time derivative in~\eqref{equ:prob_data_error_main_work} proceeds analogously. We note that here the time derivative produces additional terms when it falls onto the phase functions $\psi^{n,mod}_\pm$ (also of the higher iterates). However, these terms gain regularity easily by decomposable estimates such as $\| \nabla_{t,x} \psi_\pm^{<n,mod} \|_{D L^\infty_t L^\infty_x} \lesssim 2^{(1-\gamma)n} \varepsilon$ from~\eqref{equ:decomposable_estimate_prob_phase_qinfty}. 
\end{proof}

\subsection{Probabilistic Strichartz estimates for the rough linear evolution $\calA_{x,r}^n$} \label{subsec:prob_strichartz_A}

We also record moment bounds for the redeeming $R_n$ norm of the rough linear evolution $\calA_{x,r}^n$ of the random data $T_n A_x^\omega[0]$, $n \geq 1$. Since $\calA^n_{x,r} := S(t)[T_n a^\omega, T_n b^\omega]$ is just the free wave evolution of $T_n A_x^\omega[0] = (T_n a^\omega, T_n b^\omega)$, the proof is a (very simple) special case of the proof of Proposition~\ref{prop:prob_strichartz_phi} for the zeroth iterate $\Phi^{n,[0]}_r$ with the renormalization operators replaced by the identity (and no necessity for a probabilistic cutoff $\mathds{1}_{\varepsilon}^{<n-1}$).

\begin{proposition} \label{prop:prob_strichartz_A}
 Let $n \geq 1$. Then we have for all $1 \leq p < \infty$ that
 \begin{equation}
  \bigl\| \calA^n_{x,r} \bigr\|_{L^p_\omega(\Omega; R_n)} \lesssim \sqrt{p} \, \bigl\| (P_n a, P_n b) \bigr\|_{H^{1-\delta_\ast}_x \times H^{-\delta_\ast}_x}.
 \end{equation}
\end{proposition}

\subsection{The renormalization error estimate for $\Box_{A^{<n-1}}^{p,mod} \Phi^n_r$} \label{subsec:prob_renormalization_error}

In this section we turn to the subtle treatment of the error term $\Box_{A^{<n-1}}^{p, mod} \Phi^n_r$ produced by the rough linear evolution $\Phi^n_r$. We need to establish that it gains regularity and acts as a (small) ``smooth'' source term in the equation for $\Phi_s^n$ in the system of forced MKG-CG equations (fMKG-CG\textsubscript{n}) at dyadic level $n$.

Recall from~\eqref{equ:overall_accrued_error} that by construction of $\Phi^n_r$ the overall accrued error $\Box_{A^{<n-1}}^{p, mod} \Phi^n_r$ consists of ``mild'' and ``delicate'' error terms. To handle the ``mild'' error terms and show that they gain regularity we primarily rely on the tighter angle cutoff and the ``strongly low-high'' frequency separation. For the treatment of the ``delicate'' error terms (that do not enjoy tight angular localizations) we have to invoke as an additional key ingredient probabilistic redeeming bounds for the following error control quantity
\begin{equation} \label{equ:definition_EC_n}
 \calE \calC^n := \sum_{\ell \geq 0} \calE \calC^{n, [\ell]}, \quad n \geq 1,
\end{equation}
with 
\begin{equation*}
 \calE \calC^{n, [\ell]} := 2^{-\sigma n} 2^{-\delta_\ast n} 2^{-\nu n} \sum_{\sum_j h_j > - \frac{\sigma n}{10} } \sum_{\sum_j \alpha_j > - \frac{\sigma n}{100} } \sum_{\substack{\text{small strings } (\underline{k}) \\ \text{of length } \ell}} 2^{-\frac{r_a}{3}} \bigl\| \nabla_{t,x} \Phi^{n, \pm, (\underline{k}), (\underline{h}), (\underline{\alpha})}_r \bigr\|_{L^M_t L^6_x},
\end{equation*}
where $M \gg 1$ is sufficiently large, $0 < \nu \equiv \nu(M) \ll 1$ is sufficiently small with $\lim_{M\to\infty} \nu(M) = 0$, and for each small string $(\underline{k})$ of length $\ell$ we denote by $r_a$ its largest dominating frequency. 
We derive moment bounds for $\calE \calC^n$ in Proposition~\ref{prop:moment_bound_error_control_quantity} below. First, we turn to the treatment of the error estimate for $\Box_{A^{<n-1}}^{p, mod} \Phi^n_r$ in the following proposition (which should be regarded as an entirely deterministic estimate).

\begin{proposition}[Renormalization error estimate] \label{prop:prob_renormalization_error_estimate}
 Let $n \geq 1$.
 Let $(\calA_{x, s}^0, \calA_0^0, \Phi_s^0)$ be the solution to (MKG-CG) with initial data $(T_0 A_x^\omega, T_0 \phi^\omega)[0]$ and let $\{ (\calA_{x,s}^m, \calA_0^m, \Phi_s^m) \}_{m=0}^{n-1}$ be the solutions to (fMKG-CG\textsubscript{m}), $1 \leq m \leq n-1$, satisfying
 \begin{equation*}
  \sum_{m=0}^{n-1} \|\calA_{x,s}^m\|_{S^1[m]} + \sum_{m=0}^{n-1} \|\calA_0^m\|_{Y^1[m]} + \sum_{m=0}^{n-1} \|\Phi^m_s\|_{S^1[m]} \lesssim \varepsilon.
 \end{equation*}
 Moreover, assume that the corresponding rough linear evolutions satisfy 
 \begin{equation*}
  \sum_{m=1}^{n-1} \|\calA_{x,r}^m\|_{R_m} + \sum_{m=1}^{n-1} \|\Phi_r^m\|_{R_m} \lesssim \varepsilon,
 \end{equation*}
 and that the error control quantity satisfies $\calE \calC^n < \infty$.
 Let $\Phi^n_r$ be defined as in~\eqref{equ:definition_Phi_n_rough}.
 Then we have 
 \begin{equation}
  \bigl\| \Box_{A^{<n-1}}^{p, mod} \Phi^n_r \bigr\|_{N_n \cap \ell^1 L^2_t \dot{H}^{-\frac12}_x} \lesssim \bigl\| T_n \phi^\omega[0] \bigr\|_{H^{1-\delta_\ast}_x \times H^{-\delta_\ast}_x} + \calE \calC^n.
 \end{equation}
\end{proposition}
\begin{proof}
 By construction of the rough linear evolution $\Phi^n_r$, the error is given by
 \begin{align*}
  \Box_{A^{<n-1}}^{p,mod} \Phi^n_r = \sum_{\ell=0}^\infty \calE^{n, [\ell]}_{mild} + \sum_{\ell=0}^{\infty} \calE^{n, [\ell]}_{del}.
 \end{align*}
 We treat the ``mild'' and ``delicate'' error terms separately, starting with the former. We only focus on estimating the more difficult $N_n$ norm and omit the details for the high-modulation bound.
 
 \medskip 
 
 \noindent {\bf ``Mild'' error terms:} Our goal is to show that for all $\ell \geq 0$,
 \begin{equation*}
  \bigl\| \calE^{n, [\ell]}_{mild} \bigr\|_{N_n} \lesssim (\kappa_{n-1})^{2\ell} \bigl\| T_n \phi^\omega[0] \bigr\|_{H^{1-\delta_\ast}_x \times H^{-\delta_\ast}_x}.
 \end{equation*}
 Thanks to the smallness assumptions this then gives sufficient control for all ``mild'' error terms
 \begin{align*}
  \biggl\| \sum_{\ell=0}^\infty \calE^{n, [\ell]}_{mild} \biggr\|_{N_n} \lesssim \sum_{\ell=0}^\infty (\kappa_{n-1})^{2\ell} \bigl\| T_n \phi^\omega[0] \bigr\|_{H^{1-\delta_\ast}_x \times H^{-\delta_\ast}_x} \lesssim \bigl\| T_n \phi^\omega[0] \bigr\|_{H^{1-\delta_\ast}_x \times H^{-\delta_\ast}_x}.
 \end{align*}
 We describe in detail the estimates of the ``mild'' error $\calE^{n, [0]}_{mild}$ for the zeroth iterate and afterwards explain how to deal with the higher order iterates.
 Recall from~\eqref{equ:definition_Phi_n_rough} that 
 \begin{align*}
  \calE^{n, [0]}_{mild} = \text{Diff}^{[0]}_{1,(a)} + \text{Diff}^{[0]}_{1,(b)} + \sum_{j=2}^6 \text{Diff}^{[0]}_{j}.
 \end{align*}

 \noindent {\it The estimate for $\text{Diff}^{[0]}_{1,(a)}$}:
 Here we essentially argue exactly as in the treatment of the term $\text{Diff}_1$ in~\cite[Subsection 10.2]{KST}, only that we use the redeeming $L^2_t L^\infty_x$ norm for $A_{x,r}^{<n-1}$ and place $\calR_r^{n, \pm, [0]}$ into $S_n^{1-\delta_\ast}$. Importantly, the frequency separation $k \leq (1-\gamma)n$ ensures that we will gain a negative power in $n$ that compensates for the loss of $2^{\delta_*n}$ caused by placing $\calR_r^{n, \pm, [0]} $ into $S_n^{1-\delta_\ast}$. 
 
 \medskip 
 
 \noindent {\it The estimate for $\text{Diff}^{[0]}_{1,(b)}$}:
 Next, we consider the contribution of the smooth part. For the contribution of the spatial part $A_{x,s}^{<n-1}$, we generate the error term
 \[
  \sum_{k \leq (1-\gamma) n} \sum_\pm 2 \Bigl[ \Pi_{\leq 2^{\sigma \min\{k, -n\}}} \bigl( P_k A_{x,s}^{<n-1} \cdot \xi \bigr) e^{-i \psi_\pm^{n, mod}} \Bigr]_{<n-C} \calR_r^{n, \pm, [0]}. 
 \]
This term is of course a microlocal version of the interaction term $A_{x,s}^{<n-1, j} \partial_j$ and needs to be handled in analogy to it, but obtaining an extra exponential gain in $-n$, thanks to the additional tight angular localization. This additional gain allows to compensate a loss of $2^{\delta_\ast n}$ caused by placing $\calR_r^{n, \pm, [0]}$ into $S_n^{1-\delta_\ast}$. However, this will only be possible after re-iterating the equation for $A_{x,s}^{<n-1}$ and exploiting the subtle cancellation against the corresponding temporal contribution coming from the equation for $A_{0}^{<n-1}$. Since this is a recapitulation of the estimates in~\cite{KST} with one extra observation, we shall be relatively brief: 

\noindent {\it{(i)}} 
By translation invariance of all spaces involved, we can first dispose of the frequency localization $[\ldots]_{<n}$. As in~\cite{KST}, the goal shall be to reduce the contribution of the spatial part $A_{x, s}^{<n-1}$ of the connection form to an expression of the schematic form
\[
 \sum_{k \leq (1-\gamma) n}  \mathcal{H}^\ast \big[ \Pi_{\leq 2^{\sigma \min\{k, -n\}}} \bigl( \mathcal{H}_k A_{x,s}^{<n-1} \cdot \xi \bigr) e^{-i\psi^{n,mod}_{\pm}}\big] \calR_r^{n, \pm, [0]},
\]
where $\mathcal{H}^\ast [\ldots]$ is defined analogously to~\eqref{equ:calHast_def} and is understood as an operator acting on the frequency~$\sim 2^n$ function $\calR_r^{n, \pm, [0]}$. The above expression will then be combined with the corresponding one from $A_0^{<n-1}$ to result in the desired null form type cancellation. 

\noindent {\it{(ii) Reduction to $\mathcal{H}^\ast [\ldots] \calR_r^{n, \pm, [0]}$}}:
Consider 
\begin{align*}
&\sum_{k \leq (1-\gamma) n}  (1-\mathcal{H}^\ast)\big[ \Pi_{\leq 2^{\sigma \min\{k, -n\}}} \bigl( A_{x,s}^{<n-1} \cdot \xi \bigr) e^{-i\psi^{n,mod}_{\pm}}\big] \calR_r^{n, \pm, [0]}
\\
& = \sum_{k\leq(1-\gamma)n}\sum_{j}'Q_{\geq j-C}\big[\Pi_{\leq 2^{\sigma \min\{k, -n\}}} \bigl( Q_j P_k A_{x,s}^{<n-1} \cdot \xi \bigr) e^{-i\psi^{n,mod}_{\pm}}\big] \calR_r^{n, \pm, [0]}
\\
&\quad +  \sum_{k\leq(1-\gamma)n}\sum_{j}'\big[\Pi_{\leq 2^{\sigma \min\{k, -n\}}} \bigl( Q_jP_kA_{x,s}^{<n-1} \cdot \xi \bigr) Q_{\geq j-C} e^{-i\psi^{n,mod}_{\pm}}\big] \calR_r^{n, \pm, [0]},
\end{align*}
where the inner sum is over $\frac{j-k}{2}<-\sigma n$. For the first term on the right, we place the whole expression into $X_1^{0,-\frac12}$. For this we place $\Pi_{\leq 2^{\sigma \min\{k, -n\}}} Q_jP_kA_{x,s}^{<n-1}$ into $DL_t^2 L_x^\infty$, while $e^{-i\psi^{n,mod}_{\pm}}\calR_r^{n, \pm, [0]}$ gets placed into $L_t^\infty L_x^2$. Due the the null structure we gain $2^{\frac{j-k}{4}}\leq 2^{-\frac{\sigma n}{2}}$, which is enough.  
The second term on the right is handled similarly by placing the full expression into $L_t^1 L_x^2$ and the function $Q_{\geq j-C} e^{-i\psi^{n,mod}_{\pm}} \calR_r^{n, \pm, [0]}$ into~$L^2_t L^2_x$. 

\noindent {\it{(iii) Reduction to $\mathcal{H}^\ast \big[\Pi_{\leq 2^{\sigma \min\{k, -n\}}} \mathcal{H}A_{x,s}^{<n-1}\ldots\big] \calR_r^{n, \pm, [0]}$}}:
Consider the schematic expression 
\[
 \sum_{k\leq(1-\gamma)n} \mathcal{H}^* \big[\Pi_{\leq 2^{\sigma \min\{k, -n\}}}(1-\mathcal{H}_k) \bigl( A_{x,s}^{<n-1}\cdot\xi \bigr) e^{-i\psi^{n,mod}_{\pm}}\big] \calR_r^{n, \pm, [0]}.
\]
Iterating the equation, write this term as
\[
 \sum_{k\leq(1-\gamma)n} \sum_{j}' e^{i\psi^{n,mod}_{\pm}}\mathcal{H}^*\big[ \Pi_{\leq 2^{\sigma \min\{k, -n\}}} P_kQ_j\Box^{-1}\mathcal{N}(Q_{\geq j-C}\phi_1^{<n-1}, \phi_2^{<n-1})\cdot\xi e^{-i\psi^{n,mod}_{\pm}}\big] \calR_r^{n, \pm, [0]},
\]
where $\mathcal{N}$ denotes the null form which appears in the equation for $A_{x,s}^{<n-1}$. Then place the null form into $L_t^1 L_x^{\frac32}$ and pass from here to $L_t^1 L_x^\infty$ via Bernstein's inequality. Keeping track of the outer null form coming from the inner product with $\xi$, we gain a total of 
\[
2^{-j-k}\cdot 2^{\frac{j-k}{2}}\cdot 2^{3\cdot\frac23\cdot\frac{j-k}{2}} = 2^{\frac{j-k}{2}}\cdot 2^{-2k}, 
\]
and we can use the first factor to gain $2^{-\frac{\sigma}{2}n}$, as desired.

\noindent {\it{(iv) Dealing with the reduced term $\mathcal{H}^\ast \big[ \Pi_{\leq 2^{\sigma \min\{k, -n\}}} \mathcal{H}A_{x,s}^{<n-1}\ldots\big] \calR_r^{n, \pm, [0]}$}}:
This requires combination with the corresponding contribution from $A_0^{<n-1}$, which can be similarly reduced, and we omit the details for that. In order to deal with the combined term the idea is to argue precisely as in~\cite{KST}, using the generalized versions~\eqref{equ:generalization_136_KST}--\eqref{equ:generalization_138_KST} of the core trilinear null form estimates (136)--(138) from \cite[Theorem 12.1]{KST}, except that here the structure is slightly different due to the microlocal formulation. To get rid of this obstacle, we first observe that the multiplication with $\xi$ can be replaced by $\partial_x$, and similarly for the contribution from $A_0$ where multiplication with $|\xi|$ gets replaced by $\partial_t$. Moreover, localizing $\mathcal{H}A_{x,s}^{<n-1}$ to $P_k Q_j\mathcal{H}A_{x,s}^{<n-1}$ as in the proof of Theorem 12.1 in~\cite{KST}, the angular cut-off $\Pi_{\leq 2^{\sigma \min\{k, -n\}}}$ is smooth at the angular scale $2^{\frac{j-k}{2}}<2^{-\sigma n}$. Thus, we can expand the cut-off into a discrete Fourier series and decouple $P_k Q_j\mathcal{H}A_{x,s}^{<n-1}$ from the rest of the expression, i.e. we are formally allowed to replace the original expression 
\[
 \mathcal{H}^*\big[\Pi_{\leq 2^{\sigma \min\{k, -n\}}} P_kQ_j\mathcal{H}\bigl( A_{x,s}^{<n-1} \cdot \xi \bigr) e^{-i\psi^{n,mod}_{\pm}}\big]\calR_r^{n, \pm, [0]}
\]
by 
\[
 \mathcal{H}^*\big(P_kQ_j\mathcal{H}A_{x,s}^{<n-1} \cdot \partial_x \big(e^{-i\psi^{n,mod}_{\pm}} \calR_r^{n, \pm, [0]}\big)\big).
\]
But this term, combined with its analogue coming from $A_0^{<n-1}$, is now amenable to the generalized trilinear null form estimates~\eqref{equ:generalization_136_KST}--\eqref{equ:generalization_138_KST}. Moreover, we have $2^{\frac{j-k}{2}}<2^{-\sigma n}$, whence the fact that all estimates in the proofs of~\eqref{equ:generalization_136_KST}--\eqref{equ:generalization_138_KST} gain a small power of $j-k$ translates to an exponential gain in $-n$. This concludes the estimate for the term $\text{Diff}_{1, (b)}^{[0]}$.

\medskip 

\noindent {\it The estimate for $\text{Diff}^{[0]}_{j}$, $2 \leq j \leq 6$}. In all these terms there is an extra derivative falling on the low frequency term $\psi_{\pm}^{n,mod}$, and so the analogous estimates as in~\cite[Theorem 12.1]{KST} furnish an exponential gain in $-n$ due to the separation of the frequency support of $\psi_{\pm}^{n,mod}$ from $n$. 
 
\medskip 

Since by the mapping properties of the renormalization operators we have 
\begin{equation*}
 \|\calR_r^{n, \pm, [0]}\|_{S_n^{1-\delta_\ast}} \lesssim \bigl\| T_n \phi^\omega[0] \bigr\|_{H^{1-\delta_\ast}_x \times H^{-\delta_\ast}_x},
\end{equation*}
we obtain from the above that
\begin{equation*}
 \bigl\| \calE^{n, [0]}_{mild} \bigr\|_{N_n} \lesssim \bigl\| T_n \phi^\omega[0] \bigr\|_{H^{1-\delta_\ast}_x \times H^{-\delta_\ast}_x}.
\end{equation*}

The treatment of the ``mild'' error terms $\calE^{n, [\ell]}_{mild}$ generated by all higher iterates $\ell \geq 1$ proceeds analogously, placing $\calR_r^{n, \pm, [\ell]}$ into $S^{1-\delta_\ast}_n$. Then we gain additional smallness from the bound 
\begin{equation*}
 \bigl\| \calR_r^{n, \pm, [\ell]} \bigr\|_{S_n^{1-\delta_\ast}} \lesssim (\kappa_{n-1})^{2\ell} \bigl\| T_n \phi^\omega[0] \bigr\|_{H^{1-\delta_\ast}_x \times H^{-\delta_\ast}_x}, \quad \ell \geq 1,
\end{equation*}
which follows from the mapping properties of the renormalization operators and the definition of $\calE^{[\ell-1]}_{rough}$ (note that each magnetic potential term $Z_{k}^{\eta, \mp} \bigl( P_{k} A_{x, s}^{<n-1} \cdot \xi \mp P_{k} A_0^{<n-1} |\xi| \bigr)$ produces a smallness factor $(\kappa_{n-1})^2$, because to estimate these terms we insert the equations for $A_{x,s}^{<n-1}$ and $A_0^{<n-1}$, which are at least quadratic in the unknowns).
We remark that the additional error term $\text{Diff}^{[0]}_{7}$ that arises in the case of the inhomogeneous parametrix gains regularity from the difference to the error from the previous stage (see also the proof of the data error estimate in Proposition~\ref{prop:prob_data_error}).
 
\medskip 
 
\noindent {\bf ``Delicate'' error terms:}
Let $\ell \geq 0$. Recall that
\begin{equation}
 \calE^{n,[\ell]}_{del} = \text{Diff}^{[\ell]}_{1, rough} - \calE^{n, [\ell]}_{rough}
\end{equation} 
with
\begin{equation}
 \calE^{n, [\ell]}_{rough} := \sum_{\sum_j h_j \geq -\frac{\sigma n}{10}} \sum_{\sum_j \alpha_j \geq -\frac{\sigma n}{100} } \sum_{\substack{\text{small strings } (\underline{k}) \\ \text{of length } \ell}}    \Psi^{n, (\underline{k}), (\underline{\alpha}), (\underline{h})}_r.
\end{equation}
In order to estimate all ``delicate'' error terms we first dispose of the ``small angle'' cases $\sum_j \alpha_j \leq -\frac{\sigma n}{100}$ and $\sum_j h_j \leq -\frac{\sigma n}{10}$. Then we consider terminating situations, where (at least) one of the frequencies $k_{3\ell+3}$, $k_{3\ell+2}$, or $k_{3\ell+1}$ is greater than $\geq 3 \sigma n$. 
Once we have taken care of all these smoother parts of $\text{Diff}^{[\ell]}_{1, rough}$, we are left exactly with $\calE^{n,[\ell]}_{rough}$ and we are done.

\medskip 

We begin with the small angle case $\sum_j \alpha_j \leq -\frac{\sigma n}{100}$. We split $\text{Diff}_{1,rough}^{[\ell]}$ into two parts
\begin{align*}
 \sum_{\substack{k_i \leq (1-\gamma)n \\ 3\ell+1 \leq i \leq 3\ell+3}} \sum_{\substack{\text{small strings } (\underline{k}') \\ \text{of length } \ell-1}} \Psi^{n, (\underline{k})}_r &= \sum_{\sum_j \alpha_j \leq - \frac{\sigma n}{100} } \sum_{\substack{k_i \leq (1-\gamma)n \\ 3\ell+1 \leq i \leq 3\ell+3}} \sum_{\substack{\text{small strings } (\underline{k}') \\ \text{of length } \ell-1}} \Psi^{n, (\underline{k}), (\underline{\alpha})}_r \\
 &\qquad + \sum_{\sum_j \alpha_j > - \frac{\sigma n}{100} } \sum_{\substack{k_i \leq (1-\gamma)n \\ 3\ell+1 \leq i \leq 3\ell+3}} \sum_{\substack{\text{small strings } (\underline{k}') \\ \text{of length } \ell-1}}  \Psi^{n, (\underline{k}), (\underline{\alpha})}_r.
\end{align*}
For the small angle case we infer the following bound 
\begin{equation*}
 \biggl\| \sum_{\sum_j \alpha_j \leq - \frac{\sigma n}{100} } \sum_{\substack{k_i \leq (1-\gamma)n \\ 3\ell+1 \leq i \leq 3\ell+3}} \sum_{\substack{\text{small strings } (\underline{k}') \\ \text{of length } \ell-1}}  \Psi^{n, (\underline{k}), (\underline{\alpha})}_r \biggr\|_{N_n} \lesssim 2^{-\frac{\sigma n}{1000}} (\kappa_{n-1})^{\frac32 \ell} 2^{+\delta_\ast n} \bigl\| T_n \phi^\omega[0] \bigr\|_{H^{1-\delta_\ast}_x \times H^{-\delta_\ast}_x}, 
\end{equation*}
whence this term is in the smooth source space.
To see this, observe that we get an exponential gain from the small angles at each stage, gaining $2^{\frac14\sum_j\alpha_j}$ (by proceeding as in the treatment of the milder error term $\text{Diff}^{[0]}_{1, (b)}$ above). Moreover, at each stage we gain a power $(\kappa_{n-1})^{2}$, resulting in an overall smallness gain of $(\kappa_{n-1})^{2\ell}$. Now fixing the (integer) value of $\frac{\sigma n}{100} \leq \sum_j |\alpha_j| \leq \ell \sigma n$ and summing over all possible combinations of angles costs 
\[
\leq \left(\begin{array}{c}\sum_j |\alpha_j| + \ell \\ \ell \end{array}\right)\leq \left(\begin{array}{c}2 \sum_j |\alpha_j|\\ \ell \end{array}\right)\leq \frac{\big(2\sum_j|\alpha_j|\big)^\ell}{\ell!}.
\]
Combining this with a fractional power $(\kappa^{(n-1)})^{\frac{1}{4} \ell}$ of the overall smallness gain results in 
\[
\frac{\big(2\sum|\alpha_j|\big)^\ell}{\ell!}\cdot (\kappa_{n-1})^{\frac{1}{4}\ell}\leq e^{2 (\kappa_{n-1})^{\frac14} \sum_j |\alpha_j|} \lesssim 2^{C (\kappa_{n-1})^{\frac14} \sum_j |\alpha_j|}.
\]
Thus, the total effect of combining the gain $2^{\frac14\sum_j\alpha_j}$ and the smallness gain $(\kappa_{n-1})^{2\ell}$ with the loss due to counting all possible combinations and due to summing over $\frac{\sigma n}{100} \leq \sum_j |\alpha_j| \leq \ell \sigma n$ is bounded by
\[
 \ell \sigma n \, 2^{-\frac14 \sum_j |\alpha_j|}  2^{C (\kappa_{n-1})^{\frac14} \sum_j |\alpha_j|} (\kappa_{n-1})^{\frac74 \ell} 2^{+\delta_\ast n} \bigl\| T_n \phi^\omega[0] \bigr\|_{H^{1-\delta_\ast}_x \times H^{-\delta_\ast}_x} \lesssim 2^{-\frac{\sigma n}{1000}} (\kappa_{n-1})^{\frac32 \ell} 2^{+\delta_\ast n} \bigl\| T_n \phi^\omega[0] \bigr\|_{H^{1-\delta_\ast}_x \times H^{-\delta_\ast}_x}.
\]
It follows that we can reduce to considering 
\begin{equation*}
 \sum_{\sum_j \alpha_j > - \frac{\sigma n}{100} } \sum_{\substack{k_i \leq (1-\gamma)n \\ 3\ell+1 \leq i \leq 3\ell+3}} \sum_{\substack{\text{small strings } (\underline{k}') \\ \text{of length } \ell-1}}  \Psi^{n, (\underline{k}), (\underline{\alpha})}_r.
\end{equation*}

Next, we argue that we can further dispose of the error terms
\begin{equation*}
 \sum_{\sum_j h_j < -\frac{\sigma n}{10}} \sum_{\sum_j \alpha_j > - \frac{\sigma n}{100} } \sum_{\substack{k_i \leq (1-\gamma)n \\ 3\ell+1 \leq i \leq 3\ell+3}} \sum_{\substack{\text{small strings } (\underline{k}') \\ \text{of length } \ell-1}}  \Psi^{n, (\underline{k}), (\underline{\alpha}), (\underline{h})}_r,
\end{equation*}
where $\Psi^{n, (\underline{k}), (\underline{\alpha}), (\underline{h})}_r$ stands for the iterated integral expression when $m$ instances of the exponentials $e^{\pm i\psi^{n,mod}_{\pm}}$ are replaced by the integral expressions in the decomposition
\begin{align*}
 e^{\pm i\psi^{n,mod}_{\pm}} &= e^{\pm i\Pi_{>-10}\psi^{n,mod}_{\pm}} +  \bigl( e^{\pm i\psi^{n,mod}_{\pm}} - e^{\pm i\Pi_{>-10}\psi^{n,mod}_{\pm}} \bigr) \\
 &= e^{\pm i\Pi_{>-10}\psi^{n,mod}_{\pm}}  \mp \int_{-\sigma n}^{-10} i\frac{\partial}{\partial h}\big(\Pi_{>h}\psi^{n,mod}_{\pm}\big)e^{\pm i\Pi_{>h}\psi^{n,mod}_{\pm}}\,dh.
 \end{align*}
We recall that $\Pi_{>a}$ denotes a smooth cutoff localizing the angular separation of the Fourier support to direction $\eta := \frac{\xi}{|\xi|}$ to an angle $\gtrsim 2^a$.
 To see this, observe that for $a\leq -10$ the integral
 \[
  \int_{-\sigma n}^{-10} \chi_{h\sim a}\frac{\partial}{\partial h}\big(\Pi_{>h}\psi^{n,mod}_{\pm}\big)e^{\pm i\Pi_{>h}\psi^{n,mod}_{\pm}}\,dh
 \]
 defines a map $L_x^2 \longrightarrow L_x^2$ with norm $\lesssim 2^{\frac{a}{2}} \kappa_{n-1}$. This is a consequence of the bound 
 \[
  \Bigl\| \chi_{h\sim a}\frac{\partial}{\partial h}\big(\Pi_{>h}\psi^{n,mod}_{\pm}\big) \Bigr\|_{DL_t^\infty L_x^\infty} \lesssim 2^{\frac{a}{2}} \kappa_{n-1}.
 \]
 We can then iteratively bound the $L^1_t L^2_x$ norm of the contribution to
 \begin{equation*}
  \sum_{\sum_j \alpha_j > - \frac{\sigma n}{100} } \sum_{\substack{k_i \leq (1-\gamma)n \\ 3\ell+1 \leq i \leq 3\ell+3}} \sum_{\substack{\text{small strings } (\underline{k}') \\ \text{of length } \ell-1}}  \Psi^{n, (\underline{k}), (\underline{\alpha}), (\underline{h})}_r,
 \end{equation*}
 where $m$ exponentials $e^{\pm i\psi^{n,mod}_{\pm}}$ are replaced by the integral expression above,
 by invoking the schematic bound 
 \begin{align*}
  &\biggl\|\sum_{k_{3\ell+3} < (1-\gamma)n} Z^{\eta, \mp}_{k_{3\ell+3}, \alpha_\ell} \bigl( P_{k_{3\ell+3}} A_{x, s}^{<n-1} \cdot \xi \mp P_{k_{3\ell+3}} A_0^{<n-1} |\xi| \bigr) \Phi^{n,\pm,(\underline{k}), (\underline{\alpha})}_r \biggr\|_{L_t^1 L_x^2} \\
  &\lesssim \sum_{k_{3\ell+3} < (1-\gamma)n} \sup_\kappa \, \Bigl\| \Pi_\kappa Z^{\eta, \mp}_{k_{3\ell+3}, \alpha_\ell} \Bigl( P_{k_{3\ell+3}} A_{x, s}^{<n-1} \cdot \frac{\xi}{|\xi|} \mp P_{k_{3\ell+3}} A_0^{<n-1} \Bigr) \Bigr\|_{D L^1_t L^\infty_x} \bigl\| \nabla_x \Phi^{n,\pm,(\underline{k}), (\underline{\alpha})}_r \bigr\|_{L_t^\infty L_x^2} \\
  &\lesssim 2^{-\alpha_\ell}(\kappa_{n-1})^2 \bigl\| \nabla_x \Phi^{n,\pm,(\underline{k}), (\underline{\alpha})}_r \bigr\|_{L_t^\infty L_x^2}
 \end{align*}
 in conjunction with schematic bounds of the form
 \begin{align*}
  \biggl\| \nabla_x \biggl( \int_{-\sigma n}^{-10} \chi_{h\sim a}\frac{\partial}{\partial h} \big(\Pi_{>h}\psi^{n,mod}_{\pm}\big) e^{- i\Pi_{>h}\psi^{n,mod}_{\pm}} \, \ud h \biggr) \frac{K^{\pm}}{i|D|} e^{+ i\psi^{n,mod}_{\pm}} H \biggr\|_{L_t^\infty L_x^2}
  \lesssim  2^{\frac{a}{2}} \kappa_{n-1}\big\|H\big\|_{L_t^1 L_x^2}. 
 \end{align*}
 It follows that the contribution to 
 \begin{equation*}
  \sum_{\sum_j \alpha_j > - \frac{\sigma n}{100} } \sum_{\substack{k_i \leq (1-\gamma)n \\ 3\ell+1 \leq i \leq 3\ell+3}} \sum_{\substack{\text{small strings } (\underline{k}') \\ \text{of length } \ell-1}}  \Psi^{n, (\underline{k}), (\underline{\alpha}), (\underline{h})}_r
 \end{equation*}
 coming from those terms in the iterated expansion, where $m$ exponentials $e^{\pm i\psi^{n,mod}_{\pm}}$ are replaced by the integral expression above and where we impose the additional constraint 
 \[
 \sum_j h_j \leq - \frac{\sigma n}{10}
 \]
 for the integration variables, can be bounded with respect to $\|\cdot\|_{L_t^1 L_x^2}$ by 
 \begin{align*}
 2^{-\sum_j \alpha_j} 2^{\frac12 \sum_j h_j}  (\kappa_{n-1})^{2\ell} 2^{\delta_* n} \bigl\| T_n \phi^\omega[0] \bigr\|_{H^{1-\delta_\ast}_x \times H^{-\delta_\ast}_x} &\lesssim 2^{-(\frac{1}{40}-\frac{1}{100})\sigma n} (\kappa_{n-1})^{2\ell} 2^{\delta_* n} \bigl\| T_n \phi^\omega[0] \bigr\|_{H^{1-\delta_\ast}_x \times H^{-\delta_\ast}_x}.
 \end{align*}
 This easily allows us to place this contribution into the smooth source space, even after summation over all possible $\alpha_j$ as well as $m \leq 2 \ell$. 
 
It now follows that we may assume for the angular scales $h_j$ occurring in the phases in the exponentials $e^{\pm i\psi^{n,mod}_{\pm}}$ the additional constraint 
\[
\sum_j h_j > -\frac{\sigma n}{10}
\]
and we can henceforth omit the effect of the singular operator $\Delta_{\eta^{\perp}}^{-1}$ due to the angular degeneracy in the phases $\psi^{n,mod}_{\pm}$ up to paying a factor $2^{\frac{\sigma n}{10}}$ at the end. This is analogous to the restriction $\sum_j \alpha_j > -\frac{\sigma n}{100}$ that we impose on the angles occurring in the definition of the magnetic potential terms 
\[
 Z^{\eta, \mp}_{k, \alpha_j} \bigl( P_k A_{x, s}^{<n-1} \cdot \xi \mp P_k A_0^{<n-1} |\xi| \bigr). 
\]
We shall henceforth suppress these angular losses, and replace all operators $\Delta_{\eta^{\perp}}^{-1}$ by $\Delta^{-1}$, it being understood that at the end of the day we always have to have enough margin to absorb a loss of $2^{\frac{\sigma}{10}n}\cdot 2^{\frac{\sigma}{100}n}$.

\medskip 

At this point we are left to consider ``terminating situations'' where at least one of the frequencies $k_{3\ell+3}$, $k_{3\ell+2}$, or $k_{3\ell+1}$ is greater than $\geq 3\sigma n$. 
We describe in detail how to treat a ``delicate'' error term where the frequencies $k_{3\ell+3}$ are greater than $3 \sigma n$, noting that all other ``delicate'' error terms can be treated analogously. This error term is of the schematic form
\begin{align*}
 \sum_{\substack{\text{small strings } (\underline{k}) \\ \text{of length } \ell}} Z_{\geq 3 \sigma n}^{\eta, \mp} \bigl( P_{\geq 3 \sigma n} A_{x, s}^{<n-1} \cdot \xi \mp P_{\geq 3 \sigma n} A_0^{<n-1} |\xi| \bigr) \Phi^{n, \pm, (\underline{k})}_r,
\end{align*}
where we suppress the explicit notations $(\underline{\alpha})$ and $(\underline{h})$ for the angular restrictions, and where we recall that then
\begin{align*}
 \Phi^{n, \pm, (\underline{k})}_r := P_{k_{3\ell+2}} \bigl( e^{-i \psi_\pm^{n, mod}} \bigr)(t,x,D) \frac{K^{\pm}}{i|D|} P_{k_{3\ell+1}} \bigl( e^{+i \psi_{\pm}^{n, mod}} \bigr)(D,y,s) \Psi^{n, (\underline{k}')}_r.
\end{align*}
In order to show that this error term gains regularity and can be treated as a smooth source term, we bring in the crucial redeeming ``error control'' quantity $\calE \calC^{n, [\ell]}$. Suppressing the angular localizations, it reads
\begin{equation*}
 \calE \calC^{n, [\ell]} = 2^{-\sigma n} 2^{-\delta_\ast n} 2^{-\nu n} \sum_{\substack{\text{small strings } (\underline{k}) \\ \text{of length } \ell}} 2^{-\frac{r_a}{3}} \bigl\| \nabla_{t,x} \Phi^{n, \pm, (\underline{k})}_r \bigr\|_{L^M_t L^6_x},
\end{equation*}
where $M \gg 1$ is sufficiently large, $0 < \nu \equiv \nu(M) \ll 1$ with $\lim_{M\to\infty} \nu(M) = 0$, and for each small string $(\underline{k})$ of length $\ell$ we denote by $r_a$ its largest dominating frequency. 
Then we claim the following ``terminating bound''
\begin{equation} \label{equ:error_terminating_situation_estimate}
 \biggl\| \sum_{\substack{\text{small strings } (\underline{k}) \\ \text{of length } \ell}} Z_{\geq 3 \sigma n}^{\eta, \mp} \bigl( P_{\geq 3 \sigma n} A_{x, s}^{<n-1} \cdot \xi \mp P_{\geq 3 \sigma n} A_0^{<n-1} |\xi| \bigr) \Phi^{n, \pm, (\underline{k})}_r \biggr\|_{L^1_t L^2_x} \lesssim 2^{-(3-) \sigma n} 2^{\sigma n} 2^{\delta_\ast n} 2^{\nu n} \calE \calC^{n, [\ell]},
\end{equation}
with analogous bounds for all other ``terminating situations''.
Since $\sigma \gg \nu + \delta_\ast$ we have ample room to ensure the margin required to handle the losses arising from the angular degeneracies that we have suppressed. 
It follows that we can bound all ``delicate'' error terms accrued in the course of all inductive stages by
\begin{align*}
 \biggl\| \, \sum_{\ell=0}^\infty \calE^{n, [\ell]}_{del} \biggr\|_{L^1_t L^2_x} \lesssim \sum_{\ell =0}^\infty \calE \calC^{n, [\ell]} = \calE \calC^n.
\end{align*}

It remains to prove~\eqref{equ:error_terminating_situation_estimate}.
To see this we bound schematically 
\begin{align*}
 &\bigl\| Z_{\geq 3 \sigma n}^{\eta, \mp} \bigl( P_{\geq 3 \sigma n} A_{x, s}^{<n-1} \cdot \xi \mp P_{\geq 3 \sigma n} A_0^{<n-1} |\xi| \bigr) \Phi^{n, \pm, (\underline{k})}_r \bigr\|_{L^1_t L^2_x} \\
 &\lesssim \sup_\kappa \, \Bigl\| \Pi_\kappa Z_{\geq 3 \sigma n}^{\eta, \mp} \bigl( P_{\geq 3 \sigma n} A_{x, s}^{<n-1} \cdot \frac{\xi}{|\xi|} \mp P_{\geq 3 \sigma n} A_0^{<n-1} \bigr) \Bigr\|_{D L^{1+}_t L^3_x} \biggl( \sum_\kappa \, \bigl\| \Pi_\kappa \nabla_x \Phi^{n, \pm, (\underline{k})}_r \bigr\|_{L^M_t L^6_x}^2 \biggr)^{\frac12}.
\end{align*}
We show below that we can estimate
\begin{equation} \label{equ:error_estimate_Z_term_high_D}
 \sup_\kappa \, \Bigl\| \Pi_\kappa Z_{\geq 3 \sigma n}^{\eta, \mp} \bigl( P_{\geq 3 \sigma n} A_{x, s}^{<n-1} \cdot \frac{\xi}{|\xi|} \mp P_{\geq 3 \sigma n} A_0^{<n-1} \bigr) \Bigr\|_{D L^{1+}_t L^3_x} \lesssim 2^{-(4-) \sigma n} (\kappa_{n-1})^2,
\end{equation}
where we suppressed any losses arising from angular degeneracies. Then taking advantage of the probabilistic error control quantity $\calE \calC^{n, [\ell]}$ and suppressing any losses due to summations over angular caps, we obtain the desired bound 
\begin{align*}
 &\biggl\| \sum_{\substack{\text{small strings } (\underline{k}) \\ \text{of length } \ell}} Z_{\geq 3 \sigma n}^{\eta, \mp} \bigl( P_{\geq 3 \sigma n} A_{x, s}^{<n-1} \cdot \xi \mp P_{\geq 3 \sigma n} A_0^{<n-1} |\xi| \bigr) \Phi^{n, \pm, (\underline{k})}_r \biggr\|_{L^1_t L^2_x} \\
 &\lesssim 2^{-(4-) \sigma n} (\kappa_{n-1})^2 2^{\sigma n} \sum_{\substack{\text{small strings } (\underline{k}) \\ \text{of length } \ell}} 2^{-\frac{r_a}{3}} \bigl\| \nabla_{t,x} \Phi^{n, \pm, (\underline{k})}_r \bigr\|_{L^M_t L^6_x} \\
 &\lesssim 2^{-(3-)\sigma n} 2^{\sigma n} 2^{\delta_\ast n} 2^{\nu n} \calE \calC^{n, [\ell]}.
\end{align*}
In order to derive the estimate~\eqref{equ:error_estimate_Z_term_high_D}, we insert the equations for $A_{x,s}^{<n-1}$ and $A_0^{<n-1}$. 
Here we have very schematically that (ignoring angular localizations and replacing $\Delta_{\eta^\perp}^{-1}$ by $\Delta^{-1}$)
\begin{align*}
 &\Pi_\kappa Z_{\geq 3 \sigma n}^{\eta, \mp} P_{\geq 3 \sigma n} A_{j, s}^{<n-1} \cdot \frac{\xi}{|\xi|} \\
 &\qquad \simeq - \Delta^{-1} P_{\geq 3\sigma n} \Im \, \partial^k \Delta^{-1} \calN_{kj} \bigl( \phi^{<n-1}, \overline{\phi^{<n-1}} \bigr) \cdot \frac{\xi}{|\xi|} - \Delta^{-1} P_{\geq 3 \sigma n} \bigl( (\phi^{<n-1})^2 A^{<n-1} \bigr) \cdot \frac{\xi}{|\xi|}. 
\end{align*}
In the key quadratic term we can effectively ignore the $high \times high \to low$ case, then we obtain by just placing the inputs into $L^2_t L^6_x$ and $L^{2+}_t L^6_x$ that schematically \footnote{Strictly speaking, this bound is only valid provided the input frequencies in the null-form are at most comparable to the output frequencies. However, in case of $high \times high \to low$ situations with frequency differences $\geq \frac{\sigma}{100}n$, one can again easily place the corresponding contributions into the smooth space, and one reduces to strings with only small frequency differences by reasoning as for the angles $\alpha_j$.} 
\begin{equation*}
 \bigl\| \Delta^{-1} P_{\geq 3\sigma n} \Im \, \partial^k \Delta^{-1} \calN_{kj} \bigl( \phi^{<n-1}, \overline{\phi^{<n-1}} \bigr) \cdot \frac{\xi}{|\xi|} \bigr\|_{D L^{1+}_t L^3_x} \lesssim 2^{-(4-) \sigma n} (\kappa_{n-1})^2 
\end{equation*}
with analogous bounds for the cubic contribution to the equation for $A_x^{<n-1}$ and for the equation for $A_0^{<n-1}$.
\end{proof}

Finally, we turn to the proof of the derivation of moment bounds for the redeeming energy control quantity, for which we need the following technical refinement of certain decomposable bounds from~\cite{KST}.

\begin{lemma} \label{lem:decomprefinement} 
Let $0 < k \leq n-C$ be a positive integer or else $k$ equals the symbol $\leq 0$. Then the operators 
\[
 P_k \big(e^{\pm i\psi^{n,mod}_{\pm}}\big): S_n^{\sharp}\longrightarrow S^{Str},
\]
with mapping norms bounded by constants $a_k^{(n)}$ satisfy
\[
 \sum_k a_k^{(n)}\lesssim 1+\kappa_{n-1}.
\]
Moreover, we have the bound 
\[
 \big\|P_k\big(e^{\pm i\Pi_{\gtrsim h}\psi^{n,mod}_{\pm}}\big)\big\|_{DL_t^\infty L_x^2}\lesssim 2^{-2k} 2^{-h} a^{(n)}_k,
\]
where $\Pi_{\gtrsim h}$ localizes the scale of the angle between $\frac{\xi}{|\xi|}$ and the Fourier support of the phase to size $\gtrsim 2^{h}$, with say $h <-10$.
\end{lemma}
\begin{proof} 
This follows by writing schematically 
\begin{align*}
P_k\big(e^{\pm i\psi^{n,mod}_{\pm}}\big) &= 2^{-k}P_k\big(i\nabla_x\psi^{n,mod}_{\pm}\cdot e^{\pm i\psi^{n,mod}_{\pm}}\big)\\
&= 2^{-k}P_k\big(i\nabla_xP_{<k-C}\psi^{n,mod}_{\pm}\cdot e^{\pm i\psi^{n,mod}_{\pm}}\big)\\
&\quad +  2^{-k}P_k\big(i\nabla_xP_{[k-C,k+C]}\psi^{n,mod}_{\pm}\cdot e^{\pm i\psi^{n,mod}_{\pm}}\big)\\
&\quad + 2^{-k}P_k\big(i\nabla_xP_{\geq k+C}\psi^{n,mod}_{\pm}\cdot e^{\pm i\psi^{n,mod}_{\pm}}\big)\\
&\equiv I + II + III.
\end{align*}
Then we use that 
\begin{align*}
\sum_k 2^{-k}\big\|\nabla_xP_{<k-C}\psi^{n,mod}_{\pm}\big\|_{DL_{t,x}^\infty} + \sum_k 2^{-k}\big\|\nabla_xP_{[k-C,k+C]}\psi^{n,mod}_{\pm}\big\|_{DL_{t,x}^\infty}\lesssim \kappa_{n-1}
\end{align*}
to handle $I$ and $II$. To estimate $III$, we expand further
 \begin{align*}
 &2^{-k}P_k\big(i\nabla_xP_{\geq k+C}\psi^{n,mod}_{\pm}\cdot e^{\pm i\psi^{n,mod}_{\pm}}\big)\\
 & = \sum_{k_1\geq k+C}2^{-k}P_k\big(i\nabla_xP_{k_1}\psi^{n,mod}_{\pm}\cdot P_{k_1+C}[e^{\pm i\psi^{n,mod}_{\pm}}]\big)\\
 & = \sum_{k_1\geq k+C}2^{-k-k_1}P_k\big(i\nabla_xP_{k_1}\psi^{n,mod}_{\pm}\cdot P_{k_1+C}[i\nabla_x\psi^{n,mod}_{\pm} e^{\pm i\psi^{n,mod}_{\pm}}]\big)
 \end{align*}
 and then reiterate the splitting $I$--$III$ for the inner parentheses. Then we close the cases $I$--$II$ by using $DL_t^\infty L_x^{6+}$ for both factors $\nabla_x\psi^{n,mod}_{\pm}$ and Bernstein's inequality. The remaining case $III$ is treated by again expanding. We note that this infinite re-iteration procedure is required if we make no assumptions on the angular localisations of the phases $\psi^{n,mod}_{\pm}$. However, in the present setting, we in fact assume that the angles are bounded from below, in which case one can conclude after two-fold expansion, taking into account the loss from the degenerate operator $\Delta_{\eta^{\perp}}^{-1}$ in the definition of $\psi^{n,mod}_{\pm}$. 
The final estimate is proved similarly. 
\end{proof}

We are now in the position to establish moment bounds for the redeeming error control quantity.

\begin{proposition} \label{prop:moment_bound_error_control_quantity}
 Let $n \geq 1$. Assume that the functions $\{ \calA_{x,r}^m \}_{m=1}^{n-1}$, $\{ \calA_{x,s}^m \}_{m=0}^{n-1}$, $\{ \calA_0^m \}_{m=0}^{n-1}$, $\{ \Phi_r^m \}_{m=1}^{n-1}$, and $\{ \Phi_s^m \}_{m=0}^{n-1}$ are measurable with respect to the $\sigma$-algebra $\calF_{n-1}$ and that we have almost surely 
 \begin{equation*}
  \sum_{m=1}^{n-1} \bigl( \|\calA_{x,r}^m\|_{R_m} + \|\Phi_r^m\|_{R_m} \bigr)  + \sum_{m=0}^{n-1} \bigl( \|\calA_{x,s}^m\|_{S^1[m]} + \|\calA_0^m\|_{Y^1[m]} + \|\Phi^m_s\|_{S^1[m]} \bigr) < \infty.
 \end{equation*}
 Let $\mathds{1}_{[0,2C_0\varepsilon]}$ be the characteristic function of the interval $[0, 2C_0\varepsilon]$ and set
 \begin{equation*}
  \mathds{1}_\varepsilon^{<n-1} := \mathds{1}_{[0,2C_0\varepsilon]} \biggl( \sum_{m=1}^{n-1} \bigl( \|\calA_{x,r}^m\|_{R_m} + \|\Phi_r^m\|_{R_m} \bigr)  + \sum_{m=0}^{n-1} \bigl( \|\calA_{x,s}^m\|_{S^1[m]} + \|\calA_0^m\|_{Y^1[m]} + \|\Phi^m_s\|_{S^1[m]} \bigr) \biggr).
 \end{equation*}
 Let $\calE \calC^n$ be defined as in~\eqref{equ:definition_EC_n}. 
 Then we have for all $1 \leq p < \infty$ that
 \begin{equation} \label{equ:moment_bound_error_control_quantity}
  \bigl\| \mathds{1}_\varepsilon^{<n-1} \calE \calC^n \bigr\|_{L^p_\omega(\Omega)} \lesssim \sqrt{p} \, \bigl\| (P_n \phi_0, P_n \phi_1) \bigr\|_{H^{1-\delta_\ast}_x \times H^{-\delta_\ast}_x}.
 \end{equation}
\end{proposition}
\begin{proof}
 Recall from~\eqref{equ:definition_EC_n} that the redeeming error control quantity is defined as $\calE \calC^n := \sum_{\ell \geq 0} \calE \calC^{n, [\ell]}$ with
 \begin{equation*}
  \calE \calC^{n, [\ell]} := 2^{-\sigma n} 2^{-\delta_\ast n} 2^{-\nu n} \sum_{\sum_j h_j > - \frac{\sigma n}{10} } \sum_{\sum_j \alpha_j > - \frac{\sigma n}{100} } \sum_{\substack{\text{small strings } (\underline{k}) \\ \text{of length } \ell}} 2^{-\frac{r_a}{3}} \bigl\| \nabla_{t,x} \Phi^{n, \pm, (\underline{k}), (\underline{h}), (\underline{\alpha})}_r \bigr\|_{L^M_t L^6_x},
 \end{equation*}
 where $M \gg 1$ is sufficiently large, $0 < \nu \equiv \nu(M) \ll 1$ with $\lim_{M\to\infty} \nu(M) = 0$, and for each small string $(\underline{k})$ of length $\ell$ we denote by $r_a$ its largest dominating frequency. 
 
 Thanks to the angular restrictions $\sum_j \alpha_j > - \frac{\sigma n}{100}$ and $\sum_j h_j > - \frac{\sigma n}{10}$, in what follows we can omit the effect of the singular operator $\Delta_{\eta^{\perp}}^{-1}$ due to the angular degeneracy in the phases $\psi^{n,mod}_\pm$ as well as in the operators $Z^{\eta, \mp}_{k, \alpha_j}$ up to paying a factor $2^{\frac{\sigma}{10}n}\cdot 2^{\frac{\sigma}{100}n}$ at the end. We shall henceforth suppress these angular losses, and replace all operators $\Delta_{\eta^{\perp}}^{-1}$ by $\Delta^{-1}$, it being understood that at the end of the day we always have to have enough margin to absorb a loss of $2^{\frac{\sigma}{10}n}\cdot 2^{\frac{\sigma}{100}n}$. This is the purpose of the factor $2^{-\sigma n}$ in the definition of $\calE \calC^{n}$. Correspondingly, we omit the explicit notations $(\underline{\alpha})$ and $(\underline{h})$ from now on.
 
 We let $\Phi^{n, \pm, (\ulk), (c)}_r$ be defined inductively like $\Phi^{n, \pm, (\ulk)}_r$, only that the the data $(T_n \phi_0^\omega, T_n \phi_1^\omega)$ are replaced by $(P_c \phi_0, P_c \phi_1)$, where $P_c$ is a frequency projection to a unit-sized cube at distance $\sim 2^n$ from the origin of frequency space. Analogously, we define $ \Psi^{n, (\ulk), (c)}_r$.
 Conditioning on $\calF_{n-1}$ and using a conditional expectation version of Khintchine's inequality (as in the proof of Proposition~\ref{prop:prob_strichartz_phi}), the asserted moment bound~\eqref{equ:moment_bound_error_control_quantity} follows immediately from the following (deterministic) estimate
 \begin{equation} \label{equ:moment_bound_error_control_quantity_key_det_est}
  \begin{aligned}
   &\sum_{\substack{\text{small strings } (\underline{k}) \\ \text{of length } \ell}} 2^{-\frac{r_a}{3}} \biggl( \sum_c \, \bigl\| \nabla_{t,x} \Phi^{n, \pm, (\ulk), (c)}_r \bigr\|_{L^M_t L^6_x}^2 \biggr)^{\frac12} \\
   &\qquad \qquad \qquad  \lesssim \ell^{\frac43} C^\ell (\kappa_{n-1})^{2\ell} 2^{\delta_\ast n} 2^{\nu n} \bigl\| (P_n \phi_0, P_n \phi_1) \bigr\|_{H^{1-\delta_\ast}_x \times H^{-\delta_\ast}_x},
  \end{aligned}
 \end{equation}
 where $C \geq 1$ is some absolute constant.
 Indeed, thanks to the cutoff $\mathds{1}_\varepsilon^{<n-1}$ we may assume $\kappa_{n-1} \ll 1$ so that we can sum up $\sum_{\ell=0}^\infty \ell^{\frac43} C^\ell (\kappa_{n-1})^{2\ell} \lesssim 1$. Then we obtain that
 \begin{align*}
  \bigl\| \mathds{1}_\varepsilon^{<n-1} \calE \calC^n \bigr\|_{L^p_\omega(\Omega)} &\lesssim \sum_{\ell=0}^\infty \, \bigl\| \mathds{1}_\varepsilon^{<n-1} \calE \calC^{n, [\ell]} \bigr\|_{L^p_\omega(\Omega)} \lesssim \sqrt{p} \, \bigl\| (P_n \phi_0, P_n \phi_1) \bigr\|_{H^{1-\delta_\ast}_x \times H^{-\delta_\ast}_x}.
 \end{align*} 
  
 Let us therefore turn to the derivation of~\eqref{equ:moment_bound_error_control_quantity_key_det_est}. Fix a small string of frequencies $(\ulk)$ of length $\ell$ with dominating frequencies $r_1 < r_2 < \cdots < r_a$ and associated segments $b_1, \ldots, b_a$. Note that for a fixed cube $c$, the Fourier support of $\Phi^{n, \pm, (\ulk), (c)}_r$ is contained in a ball of radius $\sim b_1 2^{r_1} + b_2 2^{r_2} + \ldots + b_a 2^{r_a}$. Let $2+$ be such that $(M, 2+)$ is a sharp Strichartz admissible pair at regularity $\nu(M)$. Using Bernstein's inequality to go down from $L^6_x$ to $L^{2+}_x$ and using the fact that $\sum_{j=1}^a b_j = 3\ell+3$, we infer that
 \begin{equation} \label{equ:moment_bound_error_control_quantity_key_L^M_L^6}
  \begin{aligned}
   2^{-\frac{r_a}{3}} \biggl( \sum_c \, \bigl\| \nabla_{t,x} \Phi^{n, \pm, (\ulk), (c)}_r \bigr\|_{L^M_t L^6_x}^2 \biggr)^{\frac12} &\lesssim 2^{-\frac{r_a}{3}} \bigl( b_1 2^{r_1} + \ldots b_a 2^{r_a} \bigr)^{\frac{4}{3}-} \biggl( \sum_c \, \bigl\| \nabla_{t,x} \Phi^{n, \pm, (\ulk), (c)}_r \bigr\|_{L^M_t L^{2+}_x}^2 \biggr)^{\frac12} \\
   &\lesssim \ell^{\frac43} 2^{r_a} \biggl( \sum_c \, \bigl\| \nabla_{t,x} \Phi^{n, \pm, (\ulk), (c)}_r \bigr\|_{L^M_t L^{2+}_x}^2 \biggr)^{\frac12}.
  \end{aligned}
 \end{equation}
 For $1 \leq j \leq a$ let $d_j$ be such that $r_j \in \{ k_{3d_j}, k_{3d_{j+1}}, k_{3d_{j+2}} \}$ and denote $(\underline{r_1 r_j}) = (k_1 \ldots k_{3d_j+2})$. 
 Then by Strichartz estimates, the multilinear estimates from Section~\ref{sec:multilinear_estimates}, and Lemma~\ref{lem:decomprefinement} we have that 
 \begin{equation} \label{equ:moment_bound_error_control_quantity_key_LM_L2plus}
  \begin{aligned}
   \bigl\| \nabla_{t,x} \Phi^{n, \pm, (\ulk), (c)}_r \bigr\|_{L^M_t L^{2+}_x} &= \biggl\| \nabla_{t,x} P_{k_{3\ell+2}} \bigl( e^{-i\psi_\pm^{n,mod}} \bigr) \frac{K^\pm}{i|D|} P_{k_{3\ell+1}} \bigl( e^{+i\psi_\pm^{n,mod}} \bigr) \Psi^{n, (\ulk'), (c)}_r \biggr\|_{L^M_t L^{2+}_x} \\
   &\lesssim 2^{\nu n} a_{k_{3\ell+2}}^{(n)} a_{k_{3\ell+1}}^{(n)} \bigl\| \Psi^{n, (\ulk'), (c)}_r \bigr\|_{L^1_t L^2_x} \\
   &\lesssim 2^{\nu n} a_{k_{3\ell+2}}^{(n)} a_{k_{3\ell+1}}^{(n)} \beta_{k_{3\ell}}^{(n)} \bigl\| \nabla_{t,x} \Phi^{n, \pm, (\ulk'), (c)}_r \bigr\|_{L^\infty_t L^2_x} \\
   &\lesssim \ldots \\
   &\lesssim 2^{\nu n} \Biggl( \prod_{\tilde{l} = d_a + 1}^\ell a_{k_{3\tilde{l}+2}}^{(n)} a_{k_{3\tilde{l}+1}}^{(n)} \Biggr) \Biggl( \prod_{\tilde{l} = d_a + 1}^{\ell} \beta_{k_{3\tilde{l}}}^{(n)} \Biggr) \bigl\| \nabla_{t,x} \Phi^{n, \pm, (\underline{r_1 r_a}), (c)}_r \bigr\|_{L^\infty_t L^2_x},
  \end{aligned}
 \end{equation}
 where $\beta_{k_{3\tilde{l}}}^{(n)}$ denotes a bound for the estimate of the frequency-localized magnetic potential term 
 \begin{equation*}
  Z^{\eta, \mp}_{k_3} \bigl( P_{k_3} A_{x, s}^{<n-1} \cdot \xi \mp P_{k_3} A_0^{<n-1} |\xi| \bigr)
 \end{equation*}
 and is such that
 \begin{align*}
  \sum_{k_{3\tilde{l}} \leq 3 \sigma n} \beta_{k_{3\tilde{l}}}^{(n)} \lesssim (\kappa_{n-1})^2.
 \end{align*}
 We also recall that the factors $a_{k_{3\tilde{l}+2}}^{(n)}$ denote bounds on the mapping norms of $P_{k_{3\tilde{l}+2}} \big(e^{\pm i\psi^{n,mod}_{\pm}}\big): S_n^{\sharp}\longrightarrow S^{Str}$ satisfying $\sum_{k_{3\tilde{l}+2}} a_{k_{3\tilde{l}+2}}^{(n)} \lesssim 1+\kappa_{n-1}$, and similarly for the factors $a_{k_{3\tilde{l}+1}}^{(n)}$.
 Then we claim that for all $1 \leq j \leq a$ it holds that
 \begin{equation} \label{equ:moment_bound_error_control_quantity_key_Linfty_L2}
  \begin{aligned}
   &\biggl( \sum_c \, \bigl\| \nabla_{t,x} \Phi^{n, \pm, (\underline{r_1 r_j}), (c)}_r \bigr\|_{L^\infty_t L^2_x}^2 \biggr)^{\frac12} \\
   &\qquad \qquad \lesssim C^{\sum_{i=1}^{j-1} b_i} \biggl( \prod_{\tilde{l}=1}^{d_j} a_{k_{3\tilde{l}+2}}^{(n)} a_{k_{3\tilde{l}+1}}^{(n)} \biggr) \biggl( \prod_{\tilde{l}=1}^{d_j} \beta_{k_{3\tilde{l}}}^{(n)} \biggr) 2^{-r_j} 2^{\delta_\ast n} \bigl\| (P_n \phi_0, P_n \phi_1) \bigr\|_{H^{1-\delta_\ast}_x \times H^{-\delta_\ast}_x}.
  \end{aligned}
 \end{equation}
 Using~\eqref{equ:moment_bound_error_control_quantity_key_Linfty_L2} with $j=a$, combining with~\eqref{equ:moment_bound_error_control_quantity_key_LM_L2plus}, and square-summing over all cubes $c$, we conclude that
 \begin{equation*}
  \begin{aligned}
   &\biggl( \sum_c \bigl\| \nabla_{t,x} \Phi^{n, \pm, (\ulk), (c)}_r \bigr\|_{L^M_t L^{2+}_x}^2 \biggr)^{\frac12} \\
   &\qquad \qquad \lesssim C^\ell \biggl( \prod_{\tilde{l}=1}^{\ell} a_{k_{3\tilde{l}+2}}^{(n)} a_{k_{3\tilde{l}+1}}^{(n)} \biggr) \biggl( \prod_{\tilde{l}=1}^{\ell} \beta_{k_{3\tilde{l}}}^{(n)} \biggr) 2^{-r_a} 2^{\nu n} 2^{\delta_\ast n} \bigl\| (P_n \phi_0, P_n \phi_1) \bigr\|_{H^{1-\delta_\ast}_x \times H^{-\delta_\ast}_x}.
  \end{aligned}
 \end{equation*}
 Combining the previous estimate with~\eqref{equ:moment_bound_error_control_quantity_key_L^M_L^6} and summing over all small strings $(\ulk)$ yields~\eqref{equ:moment_bound_error_control_quantity_key_det_est}.
 
 It now remains to prove~\eqref{equ:moment_bound_error_control_quantity_key_Linfty_L2}. Recall that for $1 \leq j \leq a$ we let $d_j$ be such that $r_j \in \{k_{3d_j}, k_{3d_j+1}, k_{3d_j+2}\}$ and that we use the notation $(\underline{r_1 r_j}) = (k_1 \ldots k_{3d_j+2})$. We distinguish the cases $r_j = k_{3 d_j}$, $r_j = k_{3 d_j + 1}$, and $r_j = k_{3d_j + 2}$.
 
 We begin with the case $r_j = k_{3 d_j}$. For a fixed cube $c$ we use Lemma~\ref{lem:decomprefinement} to schematically estimate
 \begin{equation} \label{equ:moment_bound_error_control_quantity_Linfty_L2_single_cube}
  \begin{aligned}
   &\bigl\| \nabla_{t,x} \Phi^{n, \pm, (\underline{r_1 r_j}), (c)}_r \bigr\|_{L^\infty_t L^2_x} \\   
   &= \biggl\| \nabla_{t,x} P_{k_{3 d_j +2}} \bigl( e^{-i\psi_\pm^{n,mod}} \bigr) \frac{K^\pm}{i|D|} P_{k_{3 d_j +1}} \bigl( e^{+i\psi_\pm^{n,mod}} \bigr) \Psi^{n, (k_1 \ldots k_{3d_j}), (c)}_r \biggr\|_{L^\infty_t L^2_x} \\
   &\lesssim a_{k_{3d_j+2}}^{(n)} a_{k_{3d_j +1}}^{(n)} \Bigl\| Z_{r_j}^{\eta, \mp} \bigl( P_{r_j} A_{x, s}^{<n-1} \cdot \xi \mp P_{r_j} A_0^{<n-1} |\xi| \bigr) \Phi^{n, \pm, (k_1 \ldots k_{3d_j-1}), (c)}_r \Bigr\|_{L^1_t L^2_x} \\
   &\lesssim a_{k_{3d_j+2}}^{(n)} a_{k_{3d_j +1}}^{(n)} \sup_\kappa \, \Bigl\| \Pi_\kappa Z_{r_j}^{\eta, \mp} \Bigl( P_{r_j} A_{x, s}^{<n-1} \cdot \frac{\xi}{|\xi|} \mp P_{r_j} A_0^{<n-1} \Bigr) \Bigr\|_{D L^1_t L^3_x} \times \\
   &\qquad \qquad \qquad \qquad \qquad \qquad \qquad \times \Bigl\| \Bigl( \sum_\kappa \, \bigl\| \nabla_{t,x} \Pi_\kappa \Phi^{n, \pm, (k_1 \ldots k_{3d_j-1}), (c)}_r \bigr\|_{L^6_x}^2 \Bigr)^{\frac12} \Bigr\|_{L^\infty_t} \\
   &\lesssim a_{k_{3d_j+2}}^{(n)} a_{k_{3d_j +1}}^{(n)} 2^{-\frac{4}{3} r_j} \beta_{r_j}^{(n)} \bigl( b_{j-1} 2^{r_{j-1}} + \ldots + b_1 2^{r_1} \bigr)^{\frac43} \bigl\| \nabla_{t,x} \Phi^{n, \pm, (k_1 \ldots k_{3d_j-1}), (c)}_r \bigr\|_{L^\infty_t L^2_x}.
  \end{aligned}
 \end{equation}
 In order to achieve the last step, we bounded\footnote{Recall from the preceding footnote that we can again effectively ignore $high \times high \to low$ interactions here.}
 \begin{equation*}
  \sup_\kappa \, \Bigl\| \Pi_\kappa Z_{r_j}^{\eta, \mp} \Bigl( P_{r_j} A_{x, s}^{<n-1} \cdot \frac{\xi}{|\xi|} \mp P_{r_j} A_0^{<n-1} \Bigr) \Bigr\|_{D L^1_t L^3_x}
 \end{equation*}
 using the equations for $A_{x,s}^{<n-1}$ and $A_0^{<n-1}$, where for the key quadratic contribution to $\Box A_{x,s}^{<n-1}$ we just place both inputs into the $L^2_t L^6_x$ Strichartz space, and suppressing the errors accruing because of the angular localization. Moreover, we used that by construction $\Phi^{n, \pm, (k_1 \ldots k_{3d_j-1}), (c)}_r$ has Fourier support contained in a ball of radius $b_{j-1} 2^{r_{j-1}} + \ldots + b_1 2^{r_1}$ so that we can estimate 
 \begin{equation*}
  \Bigl\| \Bigl( \sum_\kappa \, \bigl\| \nabla_{t,x} \Pi_\kappa \Phi^{n, \pm, (k_1 \ldots k_{3d_j-1}), (c)}_r \bigr\|_{L^6_x}^2 \Bigr)^{\frac12} \Bigr\|_{L^\infty_t} 
 \end{equation*}
 by Bernstein to go from $L^6_x$ down to $L^2_x$, where we can then square-sum over the caps.
 
 Square-summing the estimate~\eqref{equ:moment_bound_error_control_quantity_Linfty_L2_single_cube} over the cubes $c$, and then repeatedly using the multilinear estimates from Section~\ref{sec:multilinear_estimates} along with Lemma~\ref{lem:decomprefinement} until we reach the next dominating frequency $r_{j-1}$, we obtain that
 \begin{equation*}
  \begin{aligned}
   &\biggl( \sum_c \, \bigl\| \nabla_{t,x} \Phi^{n, \pm, (\underline{r_1 r_j}), (c)}_r \bigr\|_{L^\infty_t L^2_x}^2 \biggr)^{\frac12} \\
   &\lesssim 2^{-\frac43 r_j} \beta_{r_j}^{(n)} \bigl( b_{j-1} 2^{r_{j-1}} + \ldots + b_1 2^{r_1} \bigr)^{\frac43} \biggl( \sum_c \, \bigl\| \nabla_{t,x} \Phi^{n, \pm, (k_1 \ldots k_{3d_j-1}), (c)}_r \bigr\|_{L^\infty_t L^2_x}^2 \biggr)^{\frac12} \\
   &\lesssim \Biggl( \prod_{\tilde{l}=d_{j-1}+1}^{d_j} a_{k_{3\tilde{l}+2}}^{(n)} a_{k_{3\tilde{l}+1}}^{(n)} \Biggr) 2^{-\frac43 r_j} \beta_{r_j}^{(n)} \bigl( b_{j-1} 2^{r_{j-1}} + \ldots + b_1 2^{r_1} \bigr)^{\frac43} \biggl( \sum_c \, \bigl\| \nabla_{t,x} \Phi^{n, \pm, (k_1 \ldots k_{3d_{j-1}+2}), (c)}_r \bigr\|_{L^\infty_t L^2_x}^2 \biggr)^{\frac12}.
  \end{aligned}
 \end{equation*}
 At this point we restart the process.
 
 If instead say $r_j = k_{3d_j+2}$, we obtain for a fixed cube $c$ the schematic bound
 \begin{equation*}
  \begin{aligned}
   &\bigl\| \nabla_{t,x} \Phi^{n, \pm, (\underline{r_1 r_j}), (c)}_r \bigr\|_{L^\infty_t L^2_x} \\
   &= \biggl\| \nabla_{t,x} P_{r_j} \bigl( e^{-i\psi_\pm^{n,mod}} \bigr) \frac{K^\pm}{i|D|} P_{k_{3 d_j +1}} \bigl( e^{+i\psi_\pm^{n,mod}} \bigr) \Psi^{n, (k_1 \ldots k_{3d_j}), (c)}_r \biggr\|_{L^\infty_t L^2_x} \\
   &\lesssim \sup_\kappa \, \bigl\| \Pi_\kappa P_{r_j} \bigl( e^{-i\psi_\pm^{n,mod}} \bigr) \bigr\|_{D L^\infty_t L^2_x} \biggl\| \nabla_{t,x} \sum_\kappa \, \Pi_\kappa \frac{K^\pm}{i|D|} P_{k_{3 d_j +1}} \bigl( e^{+i\psi_\pm^{n,mod}} \bigr) \Psi^{n, (k_1 \ldots k_{3d_j}), (c)}_r \biggr\|_{L^\infty_t L^\infty_x} \\
   &\lesssim \sup_\kappa \, \bigl\| \Pi_\kappa P_{r_j} \bigl( e^{-i\psi_\pm^{n,mod}} \bigr) \bigr\|_{D L^\infty_t L^2_x} \bigl( b_{j-1} 2^{r_{j-1}} + \ldots + b_1 2^{r_1} \bigr)^2 \times \\
   &\qquad \qquad \qquad \times \biggl\| \nabla_{t,x} \sum_\kappa \, \Pi_\kappa \frac{K^\pm}{i|D|} P_{k_{3 d_j +1}} \bigl( e^{+i\psi_\pm^{n,mod}} \bigr) \Psi^{n, (k_1 \ldots k_{3d_j}), (c)}_r \biggr\|_{L^\infty_t L^2_x},
  \end{aligned}
 \end{equation*}
 where we used Bernstein's inequality to go down from $L^\infty_x$ to $L^2_x$. Then we may invoke from Lemma~\ref{lem:decomprefinement} the following bound
 \begin{equation*}
  \sup_\kappa \, \bigl\| \Pi_\kappa P_{r_j} \bigl( e^{-i\psi_\pm^{n,mod}} \bigr) \bigr\|_{D L^\infty_t L^2_x} \lesssim a_{r_j}^{(n)} 2^{-2r_j},
 \end{equation*}
 where we exploit the assumption about the angular localizations of the Fourier support of the phases and we adhere to the convention of suppressing the accrued errors. Similarly to above we may estimate 
 \begin{equation*}
  \biggl\| \nabla_{t,x} \sum_\kappa \, \Pi_\kappa \frac{K^\pm}{i|D|} P_{k_{3 d_j +1}} \bigl( e^{+i\psi_\pm^{n,mod}} \bigr) \Psi^{n, (k_1 \ldots k_{3d_j}), (c)}_r \biggr\|_{L^\infty_t L^2_x} \lesssim a_{k_{3d_j+1}}^{(n)} \beta_{k_{3d_j}}^{(n)} \bigl\| \nabla_{t,x} \Phi^{n, \pm, (k_1 \ldots k_{3d_j-1}), (c)}_r \bigr\|_{L^\infty_t L^2_x}.
 \end{equation*}
 Combining the preceding bounds and square-summing over the cubes $c$, and then repeatedly using the multilinear estimates from Section~\ref{sec:multilinear_estimates} along with Lemma~\ref{lem:decomprefinement} until we reach the next dominating frequency~$r_{j-1}$, we obtain analogously to above that
 \begin{equation*}
  \begin{aligned}
   &\biggl( \sum_c \, \bigl\| \nabla_{t,x} \Phi^{n, \pm, (\underline{r_1 r_j}), (c)}_r \bigr\|_{L^\infty_t L^2_x}^2 \biggr)^{\frac12} \\
   &\lesssim \Biggl( \prod_{\tilde{l}=d_{j-1}+1}^{d_j} a_{k_{3\tilde{l}+2}}^{(n)} a_{k_{3\tilde{l}+1}}^{(n)} \Biggr) 2^{-2 r_j} \beta_{r_j}^{(n)} \bigl( b_{j-1} 2^{r_{j-1}} + \ldots + b_1 2^{r_1} \bigr)^2 \biggl( \sum_c \, \bigl\| \nabla_{t,x} \Phi^{n, \pm, (k_1 \ldots k_{3d_{j-1}+2}), (c)}_r \bigr\|_{L^\infty_t L^2_x}^2 \biggr)^{\frac12}.
  \end{aligned}
 \end{equation*}
 At this point we restart the process. The case $r_j = k_{3d_j+1}$ is similar. 
 
 Re-iterating the above procedure, we arrive at the bound 
 \begin{equation*}
  \begin{aligned}
   &\biggl( \sum_c \, \bigl\| \nabla_{t,x} \Phi^{n, \pm, (\underline{r_1 r_j}), (c)}_r \bigr\|_{L^\infty_t L^2_x}^2 \biggr)^{\frac12} \\      
   &\lesssim \Biggl( \prod_{\tilde{l}=1}^{d_j} a_{k_{3\tilde{l}+2}}^{(n)} a_{k_{3\tilde{l}+1}}^{(n)} \Biggr) \Biggl( \prod_{\tilde{l}=1}^{d_j} \beta_{k_{3\tilde{l}}}^{(n)} \Biggr) \Biggl( \prod_{q=1}^{j-1} 2^{-\gamma_q r_{q+1}} \bigl( b_q 2^{r_q} + \ldots + b_1 2^{r_1} \bigr)^{\gamma_q} \Biggr) 2^{-r_1} 2^{\delta_\ast n} \bigl\| (P_n \phi_0, P_n \phi_1) \bigr\|_{H^{1-\delta_\ast}_x \times H^{-\delta_\ast}_x},
  \end{aligned}
 \end{equation*}
 where $\gamma_q \in \{ \frac43, 2 \}$ for $1 \leq q \leq j-1$. In order to further estimate the third product we observe that 
 \begin{equation*}
  \Biggl( \prod_{q=1}^{j-1} 2^{-\gamma_q r_{q+1}} \bigl( b_q 2^{r_q} + \ldots + b_1 2^{r_1} \bigr)^{\gamma_q} \Biggr) 2^{-r_1} \lesssim 2^{-r_j} \prod_{q=1}^{j-1} \biggl( \frac{ b_q 2^{r_q} + \ldots + b_1 2^{r_1} }{2^{r_q}} \biggr)^{\gamma_q} 
 \end{equation*}
 and that 
 \begin{equation*}
  \sum_{q=1}^{j-1} \frac{ b_q 2^{r_q} + \ldots + b_1 2^{r_1} }{2^{r_q}} \lesssim \sum_{q=1}^{j-1} b_q.
 \end{equation*}
 Hence, by invoking the inequality of arithmetic and geometric means, we infer (for $j \geq 2$) that 
 \begin{equation*}
  \begin{aligned}
   \prod_{q=1}^{j-1} \biggl( \frac{ b_q 2^{r_q} + \ldots + b_1 2^{r_1} }{2^{r_q}} \biggr)^{\gamma_q} \leq \prod_{q=1}^{j-1} \biggl( \frac{ b_q 2^{r_q} + \ldots + b_1 2^{r_1} }{2^{r_q}} \biggr)^2 &\leq \biggl( \frac{\tilde{C} \sum_{q=1}^{j-1} b_q}{j-1} \biggr)^{2(j-1)} \lesssim C^{\sum_{q=1}^{j-1} b_q}.
  \end{aligned}
 \end{equation*}
 This gives~\eqref{equ:moment_bound_error_control_quantity_key_Linfty_L2} and thus finishes the proof of Proposition~\ref{prop:moment_bound_error_control_quantity}.
\end{proof}

\section{Proof of Theorem~\ref{thm:main}} \label{sec:proof_of_thm}

After the preparations in the previous sections, the main work to prove Theorem~\ref{thm:main} at this point goes into establishing the existence of an event $\Sigma \subset \Omega$ (with high probability) so that for all $\omega \in \Sigma$, we can obtain the corresponding solution to (MKG-CG) with random initial data $A_x^\omega[0]$, $\phi^\omega[0]$ as the limit of the sequence $(A_x^{<n}, A_0^{<n}, \phi^{<n})$ of solutions to (MKG-CG) with frequency truncated random initial data $T_{<n} A_x^\omega[0]$, $T_{<n} \phi^\omega[0]$, as described in Subsections~\ref{subsec:decomp_lin_nonlin}--\ref{subsec:global_forcedMKG}. Since it is not possible for the rough linear evolutions and the smooth nonlinear solution increments to almost surely satisfy the necessary smallness assumptions to apply the induction step Proposition~\ref{prop:induction_step} at every stage of the construction, we have to incorporate probabilistic cutoffs into the precise construction procedure.

More specifically, in the following we iteratively construct a sequence $\bigl\{ (\calA_{x,r}^{n,\chi}, \Phi_r^{n,\chi}) \bigr\}_{n \geq 1}$ of (possibly ``eventually cut off'') rough linear evolutions and a sequence $\bigl\{ (\calA_{x,s}^{n,\chi}, \calA_0^{n,\chi}, \Phi^{n,\chi}_s) \bigr\}_{n \geq 0}$ of (possibly ``eventually cut off'') smooth solutions to the sequence of systems of forced (fMKG-CG\textsubscript{n}) equations. The superscript~$\chi$ shall indicate this cutoff feature of the construction procedure. At the end we ensure that there exists an event $\Sigma \subset \Omega$ (with high probability) so that for every $\omega \in \Sigma$, the triples $(\calA_{x,s}^{n,\chi}, \calA_0^{n,\chi}, \Phi^{n,\chi}_s)$ are (non-trivial) solutions to the system of forced MKG-CG equations (fMKG-CG\textsubscript{n}) at dyadic level $n$ for every $n \geq 1$. Moreover, for every $\omega \in \Sigma$, the corresponding triple $(A_x, A_0, \phi)$ defined by
\begin{align*}
 A_x &:= \sum_{n=1}^\infty \calA_{x,r}^{n,\chi} + \sum_{n=0}^\infty \calA_{x,s}^{n,\chi} &&\in C^0_t H^s_x + S^1 \\
 A_0 &:= \sum_{n=0}^\infty \calA_0^{n, \chi} &&\in Y^1 \\
 \phi &:= \sum_{n=1}^\infty \Phi_r^{n,\chi} + \sum_{n=0}^\infty \Phi_s^{n,\chi} &&\in C^0_t H^s_x + S^1
\end{align*}
is then a solution to (MKG-CG) with random initial data $A_x[0] = (a^\omega, b^\omega)$, $\phi[0] = (\phi_0^\omega, \phi_1^\omega)$.

\medskip 

We begin by introducing various cutoff functions that will play a crucial role in the definition of the sequences $\bigl\{ (\calA_{x,r}^{n,\chi}, \Phi_r^{n,\chi}) \bigr\}_{n \geq 1}$ and $\bigl\{ (\calA_{x,s}^{n,\chi}, \calA_0^{n,\chi}, \Phi^{n,\chi}_s) \bigr\}_{n \geq 0}$ in what follows. To this end we denote by $\mathds{1}_{[0, \mu]}$ for any $\mu > 0$ the characteristic function of the interval $[0,\mu]$. Then we set
\begin{equation} \label{equ:proof_thm_def_chi_0}
 \chi_\varepsilon^0 := \mathds{1}_{[0, \varepsilon]} \Bigl( \| T_0 A_x^\omega[0] \|_{\dot{H}^1_x \times L^2_x} + \| T_0 \phi^\omega[0] \bigr\|_{\dot{H}^1_x \times L^2_x} \Bigr),
\end{equation}
and for every integer $n \geq 1$ we define 
\begin{equation} \label{equ:proof_thm_def_chi_less_n}
 \begin{aligned}
  \chi_\varepsilon^{<n-1} &:= \mathds{1}_{[0, \varepsilon]} \biggl( \| T_0 A_x^\omega[0] \|_{\dot{H}^1_x \times L^2_x} + \| T_0 \phi^\omega[0] \|_{\dot{H}^1_x \times L^2_x} + \sum_{m=1}^{n-1} \| \calA_{x,r}^{m,\chi} \|_{R_m} + \sum_{m=1}^{n-1} \| \Phi_r^{m,\chi} \|_{R_m} \\
  &\qquad \qquad \qquad \qquad + \sum_{m=1}^{n-1} \| \Phi^{m,\chi}_s[0] \|_{\dot{H}^1_x \times L^2_x} + \sum_{m=1}^{n-1} \| T_m \phi^\omega[0] \|_{H^{1-\delta_\ast}_x \times H^{-\delta_\ast}_x} + \sum_{m=1}^{n-1} \calE \calC^{m, \chi} \biggr) \times \\
  &\quad \, \, \,  \times \mathds{1}_{[0, C_0 \varepsilon]} \biggl( \sum_{m=0}^{n-1} \|\calA_{x,s}^{m,\chi}\|_{S^1[m]} + \sum_{m=0}^{n-1} \|\calA_0^{m,\chi}\|_{Y[m]} + \sum_{m=0}^{n-1} \|\Phi_s^{m,\chi}\|_{S^1[m]} \biggr)
 \end{aligned}
\end{equation}
as well as
\begin{equation} \label{equ:proof_thm_def_chi_n}
 \begin{aligned}
  \chi_\varepsilon^n &:= \mathds{1}_{[0, \varepsilon]} \biggl( \| \calA_{x,r}^{n,\chi} \|_{R_n} + \| \Phi_r^{n,\chi} \|_{R_n} + \| \Phi^{n,\chi}_s[0] \|_{\dot{H}^1_x \times L^2_x} + \| T_n \phi^\omega[0] \|_{H^{1-\delta_\ast}_x \times H^{-\delta_\ast}_x} + \calE \calC^{n, \chi} \biggr).  
 \end{aligned}
\end{equation}

\medskip 

\noindent \underline{Stage $n=0$:}
We define $(\calA_{x,s}^{0,\chi}, \calA_0^{0,\chi}, \Phi_s^{0,\chi})$ as the smooth solution to (MKG-CG) with smooth initial data
\begin{equation*}
 (\calA_{x,s}^{0,\chi}[0], \Phi_s^{0,\chi}[0]) = \chi_{\varepsilon}^0 \, (T_0 A_x^\omega[0], T_0 \phi^\omega[0]) \in \dot{H}^1_x \times L^2_x
\end{equation*}
provided by the induction base case Proposition~\ref{prop:induction_base_case}. Observe that the cutoff $\chi_\varepsilon^0$ ensures the necessary smallness of the data to apply Proposition~\ref{prop:induction_base_case}. In particular, it then holds almost surely that
\begin{equation} \label{equ:proof_thm_key_bound_ind_base_case}
 \begin{aligned}
  \|\calA_{x,s}^{0,\chi}\|_{S^1[0]} + \|\calA_0^{0,\chi}\|_{Y^1[0]} + \|\Phi_s^{0,\chi}\|_{S^1[0]} &\leq C_0 \Bigl(  \bigl\| \calA_{x,s}^{0,\chi}[0] \bigr\|_{\dot{H}^1_x \times L^2_x} + \bigl\| \Phi_s^{0,\chi}[0] \bigr\|_{\dot{H}^1_x \times L^2_x} \Bigr) \\
  &\leq C_0 \Bigl( \bigl\| T_0 A_x^\omega[0] \bigr\|_{\dot{H}^1_x \times L^2_x} + \bigl\| T_0 \phi^\omega[0] \bigr\|_{\dot{H}^1_x \times L^2_x} \Bigr).
 \end{aligned}
\end{equation}
Clearly, on an event with non-zero probability the initial data $\chi_{\varepsilon}^0 \, (T_0 A_x^\omega[0], T_0 \phi^\omega[0])$ vanishes and in those cases, $(\calA_{x,s}^{0,\chi}, \calA_0^{0,\chi}, \Phi_s^{0,\chi})$ is just the zero solution.

\medskip 

\noindent \underline{Stage $n \geq 1$:} 
Here we are given the smooth inhomogeneous parts $\{ (\calA_{x,s}^{m,\chi}, \calA_0^{m,\chi}, \Phi_s^{m,\chi}) \}_{m=0}^{n-1}$ and the rough linear evolutions $\{ (\calA_{x,r}^{m,\chi}, \Phi_r^{m,\chi}) \}_{m=1}^{n-1}$ from the previous stages $0, 1, \ldots, n-1$ of the construction. Importantly, these are measurable with respect to the $\sigma$-algebra $\calF_{n-1}$ (see for instance~\cite[Appendix A]{Bringmann18_2}).
Then we define the rough free wave evolution $\calA_{x,r}^{n,\chi}$ by
\begin{equation*}
  \calA_{x,r}^{n,\chi} := \chi_\varepsilon^{<n-1} \, S(t)\bigl[ T_n a^\omega, T_n b^\omega \bigr] = \chi_\varepsilon^{<n-1} \Bigl( \cos(t|D|) T_n a^\omega + \frac{\sin(t|D|)}{|D|} T_n b^\omega \Bigr)
 \end{equation*}
 and the rough adapted linear evolution $\Phi_r^{n,\chi}$ as in~\eqref{equ:definition_Phi_n_rough}, where the modified phase function $\psi_{\pm}^{n,mod}$ is defined in terms of $A_0^{<n-1, \chi} = \sum_{m=0}^{n-1} \calA_0^{m,\chi}$ and $A_x^{<n-1, \chi} = \sum_{m=1}^{n-1} \calA_{x,r}^{m,\chi} + \sum_{m=0}^{n-1} \calA_{x,s}^{m,\chi}$.
 Similarly, the redeeming error control quantity $\calE \calC^{n, \chi}$ is defined in terms of $A_{x,s}^{<n-1, \chi}$ and $A_0^{<n-1, \chi}$.
 Moreover, we define 
 \begin{equation*}
  \Phi_s^{n,\chi}[0] := \chi_{\varepsilon}^{<n-1} \bigl( T_n \phi^\omega[0] - \Phi_r^{n,\chi}[0] \bigr).
 \end{equation*}
 Observe that thanks to the cutoff $\chi_\varepsilon^{<n-1}$, we are in the position to invoke the moment bounds from Proposition~\ref{prop:prob_strichartz_phi}, Proposition~\ref{prop:prob_data_error}, Proposition~\ref{prop:prob_strichartz_A}, and Proposition~\ref{prop:moment_bound_error_control_quantity}. Hence, for any $1 \leq p < \infty$ it holds that
 \begin{equation} \label{equ:proof_thm_moment_bounds}
  \begin{aligned}
   \bigl\| \| \calA_{x,r}^{n,\chi} \|_{R_n} \bigr\|_{L^p_\omega} &\lesssim \sqrt{p} \, \bigl\| (P_n a, P_n b) \bigr\|_{H^{1-\delta_\ast}_x \times H^{-\delta_\ast}_x}, \\
   \bigl\| \| \Phi_r^{n,\chi} \|_{R_n} \bigr\|_{L^p_\omega} &\lesssim \sqrt{p} \, \bigl\| (P_n \phi_0, P_n \phi_1) \bigr\|_{H^{1-\delta_\ast}_x \times H^{-\delta_\ast}_x}, \\
   \bigl\| \| \Phi^{n,\chi}_s[0] \|_{\dot{H}^1_x \times L^2_x} \bigr\|_{L^p_\omega} &\lesssim \sqrt{p} \, \bigl\| (P_n \phi_0, P_n \phi_1) \bigr\|_{H^{1-\delta_\ast}_x \times H^{-\delta_\ast}_x}, \\
   \bigl\| \calE \calC^{n, \chi} \bigr\|_{L^p_\omega} &\lesssim \sqrt{p} \, \bigl\| (P_n \phi_0, P_n \phi_1) \bigr\|_{H^{1-\delta_\ast}_x \times H^{-\delta_\ast}_x}.
  \end{aligned}
 \end{equation}
 Now we use the induction step Proposition~\ref{prop:induction_step} to define $(\calA_{x,s}^{n,\chi}, \calA_0^{n,\chi}, \Phi^{n,\chi}_s)$ as the (smooth) solution to the system of forced Maxwell-Klein-Gordon equations (fMKG-CG\textsubscript{n}) at dyadic stage $n$ with forcing terms given by
 \begin{align*}
  &A_0^{<n-1} = \chi_\varepsilon^n \chi_\varepsilon^{<n-1} \biggl( \sum_{m=0}^{n-1} \calA_0^{m,\chi} \biggr), \quad A_x^{<n-1} = \chi_\varepsilon^n \chi_\varepsilon^{<n-1} \biggl( \sum_{m=1}^{n-1} \calA_{x,r}^{m,\chi} + \sum_{m=0}^{n-1} \calA_{x,s}^{m,\chi} \biggr), \\
  &\phi^{<n-1} = \chi_\varepsilon^n \chi_\varepsilon^{<n-1} \biggl( \sum_{m=1}^{n-1} \Phi_r^{m,\chi} + \sum_{m=0}^{n-1} \Phi_{s}^{m,\chi} \biggr), \\
  &\chi_\varepsilon^n \chi_\varepsilon^{<n-1} \calA_{x,r}^{n,\chi}, \quad \chi_\varepsilon^n \chi_\varepsilon^{<n-1} \Phi_r^{n,\chi}, \\
  &\chi_\varepsilon^n \chi_\varepsilon^{<n-1} \Box_{A^{<n-1}}^{p, mod} \Phi^{n, \chi}_r,
 \end{align*}
 and initial data for the scalar field given by 
 \begin{equation*}
  \chi_\varepsilon^n \chi_\varepsilon^{<n-1} \Phi^{n,\chi}_s[0].
 \end{equation*}
 Note that the cutoffs $\chi_\varepsilon^n \chi_\varepsilon^{<n-1}$ guarantee that the necessary smallness conditions in the statement of the induction step Proposition~\ref{prop:induction_step} are satisfied.
 Importantly, Proposition~\ref{prop:induction_step} also yields a bound on the $S^1[n]$ and $Y^1[n]$ norms of the solution $(\calA_{x,s}^{n,\chi}, \calA_0^{n,\chi}, \Phi_s^{n,\chi})$. Specifically, we have almost surely that
 \begin{equation} \label{equ:proof_thm_key_bound_from_ind_step_prop}
  \begin{aligned}
   &\|\calA_{x,s}^{n,\chi}\|_{S^1[n]} + \|\calA_0^{n,\chi}\|_{Y^1[n]} + \|\Phi_s^{n,\chi}\|_{S^1[n]} \\
   &\qquad \leq C_0 \Bigl( \| \chi_\varepsilon^n \chi_\varepsilon^{<n-1} \calA_{x,r}^{n,\chi} \|_{R_n} + \| \chi_\varepsilon^n \chi_\varepsilon^{<n-1} \Phi_r^{n,\chi} \|_{R_n} + \| \chi_\varepsilon^n \chi_\varepsilon^{<n-1} \Phi^{n,\chi}_s[0] \|_{\dot{H}^1_x \times L^2_x}  \\
   &\qquad \qquad \qquad \qquad \qquad \qquad \qquad \qquad + \| \chi_\varepsilon^n \chi_\varepsilon^{<n-1} T_n \phi^\omega[0] \|_{H^{1-\delta_\ast}_x \times H^{-\delta_\ast}_x} + \chi_\varepsilon^n \chi_\varepsilon^{<n-1} \calE \calC^{n, \chi} \Bigr) \\
   &\qquad \leq C_0 \Bigl( \| \calA_{x,r}^{n,\chi} \|_{R_n} + \| \Phi_r^{n,\chi} \|_{R_n} + \| \Phi^{n,\chi}_s[0] \|_{\dot{H}^1_x \times L^2_x} + \| T_n \phi^\omega[0] \|_{H^{1-\delta_\ast}_x \times H^{-\delta_\ast}_x} + \calE \calC^{n, \chi}  \Bigr).
  \end{aligned}
 \end{equation}
 Again, the cutoff clearly $\chi_\varepsilon^n \chi_\varepsilon^{<n-1}$ vanishes on an event with non-zero probability, and correspondingly $(\calA_{x,s}^{n,\chi}, \calA_0^{n,\chi}, \Phi_s^{n,\chi})$ is just the zero solution in those cases.
 
 \medskip 
 
 We carry out this construction for every integer $n \geq 1$. Then it remains to prove:
 \begin{enumerate}
  \item[(i)] The series of rough linear evolutions of the random data
 \begin{equation*}
  \sum_{n=1}^\infty \calA_{x,r}^{n,\chi} \quad \text{and} \quad \sum_{n=1}^\infty \Phi_r^{n,\chi} \quad \text{ converge in } L^2_\omega C^0_t H^s_x,
 \end{equation*}
 and the series of smooth nonlinear solution increments
 \begin{equation*}
  \sum_{n=0}^\infty \bigl( \calA_{x,s}^{n,\chi}, \calA_0^{n,\chi}, \Phi_s^{n,\chi} \bigr) \quad \text{ converges in } L^2_\omega \bigl( S^1 \times Y^1 \times S^1 \bigr).
 \end{equation*}
 Hence, for almost every $\omega \in \Omega$ these series converge in $C^0_t H^s_x$, respectively in $S^1 \times Y^1 \times S^1$. 

 \item[(ii)] There exists an event $\Sigma \subset \Omega$ with 
  \begin{equation*}
   \bbP(\Sigma) \geq 1 - C \exp \Bigl( - c \frac{\varepsilon^2}{D^2} \Bigr), \quad D := \|(a,b)\|_{H^s_x \times H^{s-1}_x} + \|(\phi_0, \phi_1)\|_{H^s_x \times H^{s-1}_x},
  \end{equation*}
  so that for every $\omega \in \Sigma$ none of the elements of the sequence $\bigl\{ \bigl( \calA_{x,r}^{n,\chi}, \Phi_r^{n,\chi}\bigr) \bigr\}_{n \geq 1}$ and of the sequence $\bigl\{ \bigl( \calA_{x,s}^{n,\chi}, \calA_0^{n,\chi}, \Phi_s^{n,\chi} \bigr) \bigr\}_{n \geq 0}$ are trivially set to zero in the above construction procedure. In particular, then for every $\omega \in \Sigma$ and for every $n \geq 1$ the triple $\bigl( \calA_{x,s}^{n,\chi}, \calA_0^{n,\chi}, \Phi_s^{n,\chi} \bigr)$ is a (non-trivial) solution to the system of forced Maxwell-Klein-Gordon equations (fMKG-CG\textsubscript{n}) at dyadic stage $n$ (with non-trivial forcing terms). Moreover, for every $\omega \in \Sigma$ the triple $(A_x, A_0, \phi)$ given by
  \begin{align*}
   A_x &:= \sum_{n=1}^\infty \calA_{x,r}^{n,\chi} + \sum_{n=0}^\infty \calA_{x,s}^{n,\chi} &&\in C^0_t H^s_x + S^1 \\
   A_0 &:= \sum_{n=0}^\infty \calA_0^{n, \chi} &&\in Y^1 \\
   \phi &:= \sum_{n=1}^\infty \Phi_r^{n,\chi} + \sum_{n=0}^\infty \Phi_s^{n,\chi} &&\in C^0_t H^s_x + S^1
  \end{align*}
  is a solution to (MKG-CG) with initial data $A_x[0] = (a^\omega, b^\omega)$, $\phi[0] = (\phi_0^\omega, \phi_1^\omega)$. 
 \end{enumerate}
 
\medskip 

\noindent \underline{Proof of (i):}
We begin with the rough linear evolutions.
For any $n \geq 1$ the energy estimate for the free wave evolution implies that almost surely
\begin{align*}
 \bigl\| \calA_{x,r}^{n,\chi} \bigr\|_{C^0_t H^s_x} \lesssim \bigl\| (T_n a^\omega, T_n b^\omega) \bigr\|_{H^s_x \times H^{s-1}_x}.
\end{align*}
Moreover, for any $n \geq 1$ we obtain from the mapping properties of the renormalization operators in Proposition~\ref{prop:prob_renormalization_mapping_properties} that almost surely
\begin{align*}
 \bigl\| \Phi_r^{n,\chi} \bigr\|_{C^0_t H^s_x} \lesssim \bigl\| (T_n \phi_0^\omega, T_n \phi_1^\omega) \bigr\|_{H^s_x \times H^{s-1}_x}. 
\end{align*}
Note that these bounds are trivial on the event where the cutoff $\chi_\varepsilon^{<n-1}$ vanishes. 
Thus, we have for any $N_2 \geq N_1 \geq 1$ by the almost orthogonality of the frequency supports that
\begin{align*}
 &\biggl\| \sum_{n=N_1}^{N_2} \calA_{x,r}^{n,\chi} \biggr\|_{L^2_\omega C^0_t H^s_x} + \biggl\| \sum_{n=N_1}^{N_2} \Phi_r^{n,\chi} \biggr\|_{L^2_\omega C^0_t H^s_x} \\
 &\lesssim \biggl\| \biggl( \sum_{n=N_1}^{N_2} \bigl\| \calA_{x,r}^{n,\chi} \bigr\|_{C^0_t H^s_x}^2 \biggr)^{\frac12} \biggr\|_{L^2_\omega} + \biggl\| \biggl( \sum_{n=N_1}^{N_2} \bigl\| \Phi_r^{n,\chi} \bigr\|_{C^0_t H^s_x}^2 \biggr)^{\frac12} \biggr\|_{L^2_\omega} \\
 &\lesssim \biggl\| \biggl( \sum_{n=N_1}^{N_2} \bigl\| (T_n a^\omega, T_n b^\omega) \bigr\|_{H^s_x \times H^{s-1}_x}^2 \biggr)^{\frac12} \biggr\|_{L^2_\omega} + \biggl\| \biggl( \sum_{n=N_1}^{N_2} \bigl\| (T_n \phi_0^\omega, T_n \phi_1^\omega) \bigr\|_{H^s_x \times H^{s-1}_x}^2 \biggr)^{\frac12} \biggr\|_{L^2_\omega} \\
 &\lesssim \biggl( \sum_{n=N_1}^\infty \bigl\| (P_n a, P_n b) \bigr\|_{H^s_x \times H^{s-1}_x}^2 \biggr)^{\frac12} + \biggl( \sum_{n=N_1}^\infty \bigl\| (P_n \phi_0, P_n \phi_1) \bigr\|_{H^s_x \times H^{s-1}_x}^2 \biggr)^{\frac12},
\end{align*}
which converges to zero as $N_1 \to \infty$. Thus, the series $\sum_{n=1}^\infty \calA_{x,r}^{n,\chi}$ and $\sum_{n=1}^\infty \Phi_r^{n,\chi}$ are Cauchy in $L^2_\omega C^0_t H^s_x$. 

Next, we turn to the smooth nonlinear components. Using the key bounds~\eqref{equ:proof_thm_key_bound_from_ind_step_prop} on the solutions $(\calA_{x,s}^{n,\chi}, \calA_0^{n,\chi}, \Phi_s^{n,\chi})$, $n \geq 1$, along with the moment bounds~\eqref{equ:proof_thm_moment_bounds}, we have for any $N_2 \geq N_1 \geq 1$ that
\begin{align*}
 &\biggl\| \sum_{n=N_1}^{N_2} (\calA_{x,s}^{n,\chi}, \calA_0^{n,\chi}, \Phi_s^{n,\chi}) \biggr\|_{L^2_\omega(S^1 \times Y^1 \times S^1)} \\
 &\lesssim \sum_{n=N_1}^{N_2} \Bigl\| \bigl\| (\calA_{x,s}^{n,\chi}, \calA_0^{n,\chi}, \Phi_s^{n,\chi}) \bigr\|_{S^1 \times Y^1 \times S^1} \Bigr\|_{L^2_\omega} \\
 &\lesssim \sum_{n=N_1}^{N_2} \Bigl\| \|\Phi^{n, \chi}_s\|_{S^1[n]} + \| \calA_{x,s}^{n,\chi} \|_{S^1[n]} + \| \calA_0^{n, \chi} \|_{Y^1[n]} \Bigr\|_{L^2_\omega} \\
 &\lesssim \sum_{n=N_1}^{N_2} \biggl( \bigl\| \| \Phi_s^{n, \chi}[0] \|_{\dot{H}^1_x \times L^2_x} \bigr\|_{L^2_\omega} + \bigl\| \|\Phi_r^{n, \chi}\|_{R_n} \bigr\|_{L^2_\omega} + \bigl\| \|\calA_{x,r}^{n, \chi}\|_{R_n} \bigr\|_{L^2_\omega} + \bigl\| \| T_n \phi^\omega[0] \|_{H^{1-\delta_\ast}_x \times H^{-\delta_\ast}_x} \bigr\|_{L^2_\omega} + \bigl\| \calE \calC^{n, \chi} \bigr\|_{L^2_\omega} \biggr) \\
 &\lesssim \sum_{n=N_1}^{N_2} \biggl( \bigl\| (P_n \phi_0, P_n \phi_1) \bigr\|_{H^{1-\delta_\ast}_x \times H^{-\delta_\ast}_x} + \bigl\| (P_n a, P_n b) \bigr\|_{H^{1-\delta_\ast}_x \times H^{-\delta_\ast}_x} \biggr).
\end{align*}
Since $1-\delta_\ast < s < 1$ by assumption, we may sum up the last line and bound it by
\begin{equation*}
 \biggl( \sum_{n=N_1}^{\infty} \bigl\| (P_n \phi_0, P_n \phi_1) \bigr\|_{H^{s}_x \times H^{s-1}_x}^2 + \bigl\| (P_n a, P_n b) \bigr\|_{H^{s}_x \times H^{s-1}_x}^2 \biggr)^{\frac12}, 
\end{equation*}
which converges to zero as $N_1 \to \infty$. Thus, the series $\sum_{n=0}^\infty (\calA_{x,s}^{n,\chi}, \calA_0^{n,\chi}, \Phi_s^{n,\chi})$ converges in $L^2_\omega (S^1 \times Y^1 \times S^1)$.

\medskip 

\noindent \underline{Proof of (ii):}
We need to show that there exists an event $\Sigma \subset \Omega$ (with high probability) on which none of the elements of the sequence of rough linear evolutions $\bigl\{ (\calA_{x,r}^{n,\chi}, \Phi_r^{n,\chi}) \bigr\}_{n=0}^\infty$ are trivially set to zero and on which none of the elements of the sequence of smooth nonlinear components $\bigl\{ (\calA_{x,s}^{n,\chi}, \calA_0^{n,\chi}, \Phi_s^{n,\chi}) \bigr\}_{n = 0}^\infty$ are trivial. In view of the definitions~\eqref{equ:proof_thm_def_chi_0}--\eqref{equ:proof_thm_def_chi_n} of the cutoffs $\chi_\varepsilon^0$, $\chi_\varepsilon^{<n-1}$, and $\chi_\varepsilon^n$ as well as in view of the crucial bound~\eqref{equ:proof_thm_key_bound_from_ind_step_prop} on the smooth nonlinear components, this is the case on the event $\Sigma \subset \Omega$ \emph{defined by the property that for all $\omega \in \Sigma$ it holds that}
\begin{equation} \label{equ:bound_ensuring_not_truncating}
 \begin{aligned}
  &\| T_0 A_x^\omega[0] \|_{\dot{H}^1_x \times L^2_x} + \| T_0 \phi^\omega[0] \|_{\dot{H}^1_x \times L^2_x} + \sum_{n=1}^\infty \| \calA_{x,r}^{n,\chi} \|_{R_n} + \sum_{n=1}^\infty \| \Phi_r^{n,\chi} \|_{R_n} \\
   &\qquad \qquad \qquad \qquad \qquad + \sum_{n=1}^\infty \| \Phi_s^{n,\chi}[0] \|_{\dot{H}^1_x \times L^2_x} + \sum_{n=1}^\infty \| T_n \phi^\omega[0] \|_{H^{1-\delta_\ast}_x \times H^{-\delta_\ast}_x} + \sum_{n=1}^\infty \calE \calC^{n, \chi} \leq \varepsilon.
 \end{aligned}
\end{equation}
Here the main point is that \eqref{equ:bound_ensuring_not_truncating} together with the key bounds~\eqref{equ:proof_thm_key_bound_ind_base_case} and~\eqref{equ:proof_thm_key_bound_from_ind_step_prop} on the solutions automatically ensure that the cutoff $\mathds{1}_{[0,C_0 \varepsilon]}(\cdot)$ in the definition~\eqref{equ:proof_thm_def_chi_less_n} of $\chi_\varepsilon^{<n-1}$ does not vanish on $\Sigma$ at every stage $n$. To determine a lower bound on the probability of the event $\Sigma$ we now establish $L^p_\omega$ bounds for the expression on the left-hand side of~\eqref{equ:bound_ensuring_not_truncating} and then invoke the tail estimate from Lemma~\ref{lem:probability_estimate}. By the moment bounds~\eqref{equ:proof_thm_moment_bounds} from above we have for all $1 \leq p < \infty$ that
\begin{align*}
 &\biggl\| \| T_0 A_x^\omega[0] \|_{\dot{H}^1_x \times L^2_x} + \| T_0 \phi^\omega[0] \|_{\dot{H}^1_x \times L^2_x} + \sum_{n=1}^\infty \| \calA_{x,r}^{n,\chi} \|_{R_n} + \sum_{n=1}^\infty \| \Phi_r^{n,\chi} \|_{R_n} \\
   &\qquad \qquad \qquad \qquad \qquad \qquad + \sum_{n=1}^\infty \| \Phi_s^{n,\chi}[0] \|_{\dot{H}^1_x \times L^2_x} + \sum_{n=1}^\infty \| T_n \phi^\omega[0] \|_{H^{1-\delta_\ast}_x \times H^{-\delta_\ast}_x} + \sum_{n=1}^\infty \calE \calC^{n, \chi}  \biggr\|_{L^p_\omega} \\
 &\lesssim \| T_0 A_x^\omega[0] \|_{L^p_\omega (\dot{H}^1_x \times L^2_x)} + \| T_0 \phi^\omega[0] \|_{L^p_\omega(\dot{H}^1_x \times L^2_x)} + \sum_{n=1}^\infty \| \calA_{x,r}^{n,\chi} \|_{L^p_\omega R_n} + \sum_{n=1}^\infty \| \Phi_r^{n,\chi} \|_{L^p_\omega R_n} \\
   &\qquad \qquad \qquad \qquad + \sum_{n=1}^\infty \| \Phi_s^{n,\chi}[0] \|_{L^p_\omega (\dot{H}^1_x \times L^2_x)} + \sum_{n=1}^\infty \| T_n \phi^\omega[0] \|_{L^p_\omega(H^{1-\delta_\ast}_x \times H^{-\delta_\ast}_x)} + \sum_{n=1}^\infty \| \calE \calC^{n, \chi} \|_{L^p_\omega} \\
 &\lesssim \sqrt{p} \, \biggl( \| (P_{\leq 0} a, P_{\leq 0} b) \|_{\dot{H}^1_x \times L^2_x} + \| (P_{\leq 0} \phi_0, P_{\leq 0} \phi_1) \|_{\dot{H}^1_x \times L^2_x} \\
 &\qquad \qquad \qquad + \sum_{n=1}^\infty \| (P_n a, P_n b) \|_{H^{1-\delta_\ast}_x \times H^{-\delta_\ast}_x} + \sum_{n=1}^\infty \| (P_n \phi_0, P_n \phi_1) \|_{H^{1-\delta_\ast}_x \times H^{-\delta_\ast}_x} \biggr) \\
 &\lesssim \sqrt{p} \bigl( \|(a,b)\|_{H^s_x \times H^{s-1}_x} + \|(\phi_0, \phi_1)\|_{H^s_x \times H^{s-1}_x} \bigr), 
\end{align*}
where in the last line we again used the assumption that $1-\delta_\ast < s < 1$ in order to sum up in $n$. By the tail estimate from Lemma~\ref{lem:probability_estimate} it follows that 
\begin{equation*}
 \bbP \bigl( \Sigma^c \bigr) \lesssim \exp \Bigl( - c \frac{\varepsilon^2}{D^2} \Bigr), \quad D := \|(a,b)\|_{H^s_x \times H^{s-1}_x} + \|(\phi_0, \phi_1)\|_{H^s_x \times H^{s-1}_x}.
\end{equation*}
Hence, we obtain that the probability of the event $\Sigma \subset \Omega$ is bounded from below by 
\begin{equation*}
 \bbP(\Sigma) \geq 1 - C \exp \Bigl( - c \frac{\varepsilon^2}{D^2} \Bigr),
\end{equation*}
which is close to $1$ for small (scaling super-critical) initial data with $0 < D \ll \varepsilon$.

\section{Multilinear estimates} \label{sec:multilinear_estimates}

In this section we establish generalized versions of the multilinear estimates from~\cite{KST} that allow for one or more rough inputs with redeeming space-time integrability properties.

\subsection{Core generic product estimates}

We begin with several generic product estimates that are immediate consequences of H\"older's inequality and Bernstein estimates.

\begin{lemma} \label{lem:core_generic_product_est}
 We have that 
 \begin{align}
  \bigl\| P_k \bigl( A_{k_1} \phi_{k_2} \bigr) \bigr\|_{L^1_t L^2_x} &\lesssim 2^{\delta (k - \max\{k_i\})} 2^{-\delta |k_1-k_2|} \|A_{k_1}\|_{L^2_t \dot{H}^{\hf}_x} \|\phi_{k_2}\|_{L^2_t \dot{W}^{6, \frac{1}{6}}_x + R_{k_2}} \label{equ:core_generic_product_est1}  \\
  \bigl\| P_k ( \phi_{k_1}^{(1)} \phi_{k_2}^{(2)} ) \bigr\|_{L^2_t \dot{H}^{\hf}_x} &\lesssim 2^{\delta (k - \max\{k_i\})} 2^{-\delta |k_1-k_2|} \| \phi_{k_1}^{(1)} \|_{S^1_{k_1} + R_{k_1}} \| \phi_{k_2}^{(2)} \|_{S^1_{k_2} + R_{k_2}} \label{equ:core_generic_product_est2} \\
  \bigl\| P_k ( \phi_{k_1}^{(1)} \phi_{k_2}^{(2)} ) \bigr\|_{L^2_t \dot{H}^{-\hf}_x} &\lesssim 2^{\delta (k - \max\{k_i\})} 2^{-\delta |k_1-k_2|} \|\phi_{k_1}^{(1)}\|_{L^\infty_t L^2_x} \| \phi_{k_2}^{(2)} \|_{L^2_t \dot{W}^{6, \frac{1}{6}}_x + R_{k_2}} \label{equ:core_generic_product_est3} \\
  \bigl\| P_k ( \nabla_{t,x} \phi_{k_1}^{(1)} \phi_{k_2}^{(2)} ) \bigr\|_{L^2_t \dot{H}^{-\hf}_x} &\lesssim 2^{\delta (k - \max\{k_i\})} 2^{-\delta |k_1-k_2|} \|\phi_{k_1}^{(1)}\|_{R_{k_1}} \| \phi_{k_2}^{(2)} \|_{L^2_t \dot{W}^{6, \frac{1}{6}}_x + R_{k_2}} \label{equ:core_generic_product_est4} 
 \end{align}
\end{lemma}

\begin{proof}[Proof of~\eqref{equ:core_generic_product_est1}]
 We may assume that $\phi_{k_2}$ is rough and that $k_2 \geq 1$. Otherwise, the estimate follows from (64) in~\cite{KST}. We begin with the low-high case $k_1 \leq k_2-C$. Then we obtain by H\"older's inequality and Bernstein estimates that 
 \begin{align*}
  \bigl\| P_k \bigl( A_{k_1} \phi_{k_2} \bigr) \bigr\|_{L^1_t L^2_x} &\lesssim \|A_{k_1}\|_{L^2_t L^3_x} \|\phi_{k_2}\|_{L^2_t L^6_x} \\
  &\lesssim 2^{+\frac{1}{6} k_1} \|A_{k_1}\|_{L^2_t \dot{H}^{\hf}_x} 2^{-(\frac12 - 20 \sigma)k_2} \|\phi_{k_2}\|_{R_{k_2}Str} \\
  &\lesssim 2^{-\frac{1}{6}(k_2-k_1)} \|A_{k_1}\|_{L^2_t \dot{H}^{\hf}_x} \|\phi_{k_2}\|_{R_{k_2}Str}.
 \end{align*}
 In the high-low case $k_2 \leq k_1-C$ we obtain in an analogous manner that
 \begin{align*}
  \bigl\| P_k \bigl( A_{k_1} \phi_{k_2} \bigr) \bigr\|_{L^1_t L^2_x} \lesssim \| A_{k_1} \|_{L^2_t L^2_x} \|\phi_{k_2}\|_{L^2_t L^\infty_x} &\lesssim 2^{-\hf(k_1-k_2)} \| A_{k_1} \|_{L^2_t \dot{H}^\hf_x} \|\phi_{k_2}\|_{L^2_t L^8_x} \\
  &\lesssim 2^{-\hf(k_1-k_2)} \| A_{k_1} \|_{L^2_t \dot{H}^\hf_x} \|\phi_{k_2}\|_{R_{k_2}Str}.
 \end{align*}
 Finally, in the high-high case $k_1 = k_2 + \calO(1) \gg k$ we bound by
 \begin{align*}
  \bigl\| P_k \bigl( A_{k_1} \phi_{k_2} \bigr) \bigr\|_{L^1_t L^2_x} &\lesssim 2^{\hf k} \bigl\| P_k \bigl( A_{k_1} \phi_{k_2} \bigr) \bigr\|_{L^1_t L^{\frac58}_x} \\
  &\lesssim 2^{\hf k} \|A_{k_1}\|_{L^2_t L^2_x} \|\phi_{k_2}\|_{L^2_t L^8_x} \\ 
  &\lesssim 2^{\hf(k-k_1)} \|A_{k_1}\|_{L^2_t \dot{H}^\hf_x} \|\phi_{k_2} \|_{R_{k_2}Str}.
 \end{align*}
\end{proof}

\begin{proof}[Proof of~\eqref{equ:core_generic_product_est2}]
 By symmetry considerations and the estimate~(67) from~\cite{KST}, it suffices to consider the two cases 
 \begin{itemize}
  \item[(i)] $\phi_{k_1}^{(1)}$ is rough ($k_1 \geq 1$) and $\phi_{k_2}^{(2)}$ is rough ($k_2 \geq 1$), 
  \item[(ii)] $\phi_{k_1}^{(1)}$ is rough ($k_1 \geq 1$) and $\phi_{k_2}^{(2)}$ is smooth.
 \end{itemize}
 We begin with the first case (i). For high-low interactions $k_2 \leq k_1-C$, we estimate 
 \begin{align*}
  \bigl\| P_k ( \phi_{k_1}^{(1)} \phi_{k_2}^{(2)} ) \bigr\|_{L^2_t \dot{H}^\hf_x} &\lesssim 2^{\hf k_1} \bigl\| \phi_{k_1}^{(1)} \phi_{k_2}^{(2)} \bigr\|_{L^2_t L^2_x} \\
  &\lesssim 2^{-(\hf-\delta_\ast)k_1} \| |\nabla|^{1-\delta_\ast} \phi_{k_1}^{(1)} \|_{L^\infty_t L^2_x} \|\phi_{k_2}^{(2)} \|_{L^2_t L^\infty_x} \\
  &\lesssim 2^{-(\hf-\delta_\ast)(k_1-k_2)} \| |\nabla|^{1-\delta_\ast} \phi_{k_1}^{(1)} \|_{L^\infty_t L^2_x} \|\phi_{k_2}^{(2)} \|_{L^2_t L^{8+}_x} \\
  &\lesssim 2^{-(\hf-\delta_\ast)(k_1-k_2)} \| \phi_{k_1}^{(1)} \|_{S^{1-\delta_\ast}_{k_1}} \|\phi_{k_2}^{(2)} \|_{R_{k_2}Str}.
 \end{align*}
 For low-high interactions we can proceed in the same manner by symmetry. For high-high interactions $k_1 = k_2 + \calO(1) \gg k$ we bound by
 \begin{align*}
  \bigl\| P_k ( \phi_{k_1}^{(1)} \phi_{k_2}^{(2)} ) \bigr\|_{L^2_t \dot{H}^\hf_x} &\lesssim 2^{\hf k} \bigl\| \phi_{k_1}^{(1)} \phi_{k_2}^{(2)} \bigr\|_{L^2_t L^2_x} \\
  &\lesssim 2^{\hf k} 2^{-(1-\delta_\ast)k_1} \| |\nabla|^{1-\delta_\ast} \phi_{k_1}^{(1)} \|_{L^\infty_t L^2_x} \|\phi_{k_2}^{(2)}\|_{L^2_t L^\infty_x} \\
  &\lesssim 2^{-\hf(k_1-k)} \|\phi_{k_1}^{(1)}\|_{S^{1-\delta_\ast}_{k_1}} \|\phi_{k_2}^{(2)}\|_{R_{k_2}Str}.
 \end{align*}
 Now we turn to the second case (ii). For the high-low interactions $k_2 \leq k_1-C$, we bound 
 \begin{align*}
  \bigl\| P_k ( \phi_{k_1}^{(1)} \phi_{k_2}^{(2)} ) \bigr\|_{L^2_t \dot{H}^\hf_x} &\lesssim 2^{\frac{1}{2} k_1} \| \phi_{k_1}^{(1)} \|_{L^{16}_t L^{\frac{24}{11}}_x} \| \phi_{k_2}^{(2)} \|_{L^{\frac{16}{7}}_t L^{24}_x} \\
  &\lesssim 2^{\hf k_1} \| \phi_{k_1}^{(1)} \|_{L^{16}_t L^{\frac{24}{11}}_x} 2^{\frac{19}{48} k_2} \| \phi_{k_2}^{(2)} \|_{L^{\frac{16}{7}}_t L^{\frac{64}{9}}_x} \\
  &\lesssim 2^{-\frac{19}{48}(k_1-k_2)} \bigl( 2^{\frac{43}{48}k_1} \|\phi_{k_1}\|_{L^{16}_t L^{\frac{24}{11}}_x} \bigr) \|\phi_{k_2}\|_{S_{k_2}^1} \\
  &\lesssim 2^{-\frac{19}{48}(k_1-k_2)} \bigl( 2^{\frac{43}{48}k_1} 2^{-(\frac{15}{16}-20\sigma)k_1} \|\phi_{k_1}\|_{R_{k_1}Str} \bigr) \|\phi_{k_2}\|_{S_{k_2}^1} \\
  &\lesssim 2^{-\frac{19}{48}(k_1-k_2)} \| \phi_{k_1}^{(1)} \|_{R_{k_1}Str} \|\phi_{k_2}^{(2)}\|_{S_{k_2}^1}.
 \end{align*}
 The bounds for the low-high interactions and the high-high interactions are more of the same. 
\end{proof}

\begin{proof}[Proof of~\eqref{equ:core_generic_product_est3} and~\eqref{equ:core_generic_product_est4}]
 These are generalizations of the estimate (65) in~\cite{KST}. The proofs are similar to the proofs of~\eqref{equ:core_generic_product_est1}--\eqref{equ:core_generic_product_est2}. 
\end{proof}

\subsection{Core bilinear null form estimates}

Here we present several generalized bilinear null form estimates. We begin with the generalization of the multilinear estimate (131) from~\cite{KST}.

\begin{lemma}
It holds that
\begin{equation} \label{equ:generalization_131_KST}
 \bigl\| P_{k_1} \calN ( \phi_{k_2}^{(2)}, \phi_{k_3}^{(3)} ) \bigr\|_{N_{k_1}} \lesssim 2^{k_1} 2^{\delta (k_1 - \max\{k_2, k_3\})} 2^{-\delta |k_2-k_3|} \|\phi_{k_2}^{(2)}\|_{S_{k_2}^1 + R_{k_2}} \|\phi_{k_3}^{(3)}\|_{S^1_{k_3} + R_{k_3}} 
\end{equation}
\end{lemma}
\begin{proof} 
This follows from Lemma~\ref{lem:KST12.4} and Lemma~\ref{lem:KST12.4addendum} below after localizing the modulations. 
\end{proof}

Similarly, we have the following generalized version of the multilinear estimate (132) from~\cite{KST}.
\begin{lemma}
The following holds
\begin{equation} \label{equ:generalization_132_KST}
 \bigl\| (I - \calH_{k_1}^\ast) \calN( \phi_{k_1}^{(1)}, \phi_{k_2}^{(2)}) \bigr\|_{N_{k_2}} \lesssim 2^{k_1} \|\phi_{k_1}^{(1)}\|_{S^1_{k_1} + R_{k_1}} \|\phi_{k_2}^{(2)}\|_{S_{k_2}^1}, \quad k_1 \leq k_2-C.
\end{equation}
\end{lemma}
\begin{proof} We write 
\begin{align*}
\| (I - \calH_{k_1}^\ast) \calN( \phi_{k_1}^{(1)}, \phi_{k_2}^{(2)}) & = \sum_{j<k_1+C} Q_{\geq j-C}\calN( Q_j\phi_{k_1}^{(1)}, \phi_{k_2}^{(2)}) +  \sum_{j<k_1+C} Q_{< j-C}\calN( Q_j\phi_{k_1}^{(1)}, Q_{\geq j-C}\phi_{k_2}^{(2)}) \\
&\equiv I + II.
\end{align*}
Due to the multilinear estimate (131) in~\cite{KST}, we may assume that $\phi_{k_1}^{(1)} \in R_{k_1}$, and in particular that $k_1>0$. 

\medskip 

\noindent {\it{Estimate for $I$}}:
Freezing the output modulation to $j_1\geq j-C$ and summing over $j\leq j_1-C$, we may localize the factors $Q_{<j_1+C}\phi_{k_1}^{(1)}, \phi_{k_2}^{(2)}$ to caps $\kappa_{1,2}$ of diameter $\sim 2^{\frac{j_1-k_1}{2}}$ and aligned or anti-aligned. Then for $j_1<k_1+O(1)$ estimate 
\begin{align*}
&\Big\| \sum_{\kappa_{1}\sim \pm\kappa_2}Q_{j_1}\calN( P_{\kappa_1}Q_{<j_1+C}\phi_{k_1}^{(1)}, P_{\kappa_2}\phi_{k_2}^{(2)}) \Big\|_{\dot{X}^{0,-\frac12}_1} \\
&\lesssim 2^{-\frac{j_1}{2}} 2^{\frac{j_1-k_1}{2}} \sum_{\kappa_{1}\sim \pm\kappa_2}\big\| P_{\kappa_1}Q_{<j_1+C}\phi_{k_1}^{(1)}\big\|_{L_t^2 L_x^\infty} \big\|P_{\kappa_2}\phi_{k_2}^{(2)}\big\|_{L_t^\infty L_x^2}\\
&\lesssim 2^{-\frac{k_1}{2}} 2^{\frac{j_1-k_1}{4}}  2^{\frac23 k_1} \biggl( \sum_{\kappa_{1}}\big\|\nabla_x P_{\kappa_1}Q_{<j_1+C}\phi_{k_1}^{(1)}\big\|_{L_t^2 L_x^6}^2 \biggr)^{\frac12} \biggl(\sum_{\kappa_{2}}\big\|\nabla_x P_{\kappa_2}\phi_{k_2}^{(2)}\big\|_{L_t^\infty L_x^2}^2 \biggr)^{\frac12}\\
&\lesssim  2^{\frac{j_1-k_1}{4}} 2^{\frac{k_1}{6}} 2^{-(\frac12 - 20\sigma)k_1} 2^{k_1} \big\|\phi_{k_1}^{(1)}\big\|_{R_{k_1}} \big\|\phi_{k_2}^{(2)}\big\|_{S_{k_2}^1},
\end{align*}
which can be summed over $j_1<k_1+O(1)$ to give (more than) the desired bound.  Note that the factor $2^{\frac{j_1-k_1}{2}}$ in the second line comes from the null-structure, and the factor $ 2^{\frac{j_1-k_1}{4}}$ in the third line comes from Bernstein's inequality passing from $L_t^2 L_x^6$ to $L_t^2 L_x^\infty$ and exploiting the angular localization. 
When $j_1>k_1+C$ one argues similarly but without angular localizations. 

\medskip 

\noindent {\it{Estimate for $II$}}: This is handled by placing the second factor $Q_{\geq j-C}\phi_{k_2}^{(2)}$ into $L_{t,x}^2$ and the first factor $Q_j\phi_{k_1}^{(1)}$ into $L_t^2 L_x^\infty$, thus placing the output into $L_t^1 L_x^2$. The details are similar to the preceding case. 
\end{proof}

The following is a variant of Lemma 12.4 in~\cite{KST}, which follows easily from the formulation there in case all factors are in the space $S_{k_j}^1$, and which suffices for the purposes of the core multilinear estimates in~\cite{KST}. 

\begin{lemma}[Core modulation estimates] \label{lem:KST12.4}
 The following estimate holds uniformly in the indices $j_i$, $k_i$, where $j_2, j_3 = j_1 + \calO(1)$:
 \begin{align*}
  &\big|\langle Q_{j_1}\phi_{k_1}^{(1)}, \mathcal{N}\big(Q_{<j_2}\phi_{k_2}^{(2)}, Q_{<j_3}\phi_{k_3}^{(3)}\rangle\big| \\
  &\quad \lesssim 2^{k_1} 2^{-\delta|j_1-k_2|}2^{-\delta|k_1 - k_3|}2^{-\delta|k_2-k_3|}\big\|\phi_{k_1}^{(1)}\big\|_{X_{\infty}^{0,\frac12}}\big\|\phi_{k_2}^{(2)}\big\|_{S_{k_2}^1+R_{k_2}}\big\|\phi_{k_3}^{(3)}\big\|_{S_{k_3}^1+R_{k_3}}.
 \end{align*}
 In addition, when $j>k_{\text{min}}+C$, we have the improved bound 
 \begin{align*}
  &\big|\langle Q_{j_1}\phi_{k_1}^{(1)}, \mathcal{N}\big(\phi_{k_2}^{(2)}, \phi_{k_3}^{(3)}\rangle\big| \\ 
  &\quad \lesssim 2^{-\delta(j-k_{\text{min}})}2^{2k_{\text{min}}}2^{k_{\text{max}}}\big\|\phi_{k_1}^{(1)}\big\|_{X_{\infty}^{0,\frac12}}\big\|\phi_{k_2}^{(2)}\big\|_{S_{k_2}+R_{k_2}}\big\|\phi_{k_3}^{(3)}\big\|_{N_{k_3}^*}.
 \end{align*}
\end{lemma} 
\begin{proof} By symmetry we may assume $k_2\geq k_3$. 
Then we may assume that $k_2>0$, since else the estimate coincides with the one from~\cite{KST}. To begin with, assume that $\phi_{k_2}^{(2)}\in R_{k_2}$. By duality, it suffices to place the null-form $P_{k_1}Q_{j_1}\mathcal{N}\big(Q_{<j_2}\phi_{k_2}^{(2)}, Q_{<j_3}\phi_{k_3}^{(3)})$ into $X_1^{0,-\frac12}$. We verify this for the different frequency interactions. 

\medskip 

\noindent {\it{High-High interactions $k_2 = k_3+O(1)\geq k_1+O(1)$}}: Assume first that $j_1\leq k_1+O(1)$. Localizing the inputs further to the upper or lower half-space, we can further write this as 
\begin{align*}
 P_{k_1}Q_{j_1}\mathcal{N}\big(Q_{<j_2}^{\pm}\phi_{k_2}^{(2)}, Q_{<j_3}^{\pm}\phi_{k_3}^{(3)})= \sum_{\kappa_2\sim \pm\kappa_3}P_{k_1}Q_{j_1}\mathcal{N}\big(Q_{<j_2}^{\pm}\phi_{k_2,\kappa_2}^{(2)}, Q_{<j_3}^{\pm}\phi_{k_3,\kappa_3}^{(3)}),
\end{align*}
where the caps $\kappa_{2,3}$ range over the collections of spherical caps of diameter $2^{\frac{j_1+k_1}{2}-k_2}$. Then we bound the expression by 
\begin{align*}
&2^{-\frac{j_1}{2}}\big\|P_{k_1}Q_{j_1}\mathcal{N}\big(Q_{<j_2}^{\pm}\phi_{k_2}^{(2)}, Q_{<j_3}^{\pm}\phi_{k_3}^{(3)})\big\|_{L_{t,x}^2}\\
&\lesssim 2^{-\frac{j_1}{2}} 2^{\frac{j_1+k_1}{2}}\sum_{\kappa_2\sim \pm\kappa_3}\big\|Q_{<j_2}^{\pm}\phi_{k_2,\kappa_2}^{(2)}\big\|_{L_t^2L_x^\infty}\big\|Q_{<j_3}^{\pm}\nabla_x\phi_{k_3,\kappa_3}^{(3)}\|_{L_t^\infty L_x^2}\\
&\lesssim 2^{\frac{k_1}{2}-k_2} 2^{k_2} \biggl(\sum_{\kappa_2}\big\|Q_{<j_2}^{\pm}\phi_{k_2,\kappa_2}^{(2)}\big\|_{L_t^2L_x^\infty}^2\biggr)^{\frac12} \biggl(\sum_{\kappa_3}\big\|Q_{<j_3}^{\pm}\nabla_x\phi_{k_3,\kappa_3}^{(3)}\big\|_{L_t^\infty L_x^2}^2 \biggr)^{\frac12}
\end{align*}
on account of the Cauchy-Schwarz inequality. Note that 
\begin{align*}
 \biggl( \sum_{\kappa_2}\big\|Q_{<j_2}^{\pm}\phi_{k_2,\kappa_2}^{(2)}\big\|_{L_t^2L_x^\infty}^2 \biggr)^{\frac12} &\lesssim 2^{-\delta|j_1-k_2|} 2^{-(\frac12-40\sigma)k_2} \big\|\phi_{k_2,\kappa_2}^{(2)}\big\|_{R_{k_2}} \\
 \biggl(\sum_{\kappa_3}\big\|Q_{<j_3}^{\pm}\nabla_x\phi_{k_3,\kappa_3}^{(3)}\big\|_{L_t^\infty L_x^2}^2 \biggr)^{\frac12} &\lesssim 2^{\delta_* k_3}\big\|\phi_{k_3,\kappa_3}^{(3)}\big\|_{S_{k_3}^1+R_{k_3}}. 
\end{align*}
The desired bound follows easily from this, if we choose $\sigma$ and $\delta_\ast$ sufficiently small. 
The estimate when $j_1>k_1+C$ is similar, except that it suffices to localize to caps of diameter $\sim 2^{k_1-k_2}$. 

\medskip

\noindent {\it{High-Low interactions $k_2 = k_1+O(1)\geq k_3+O(1)$}}: Assume first that $j_1\leq k_3+O(1)$. Then we can localize the factors to discs of radius $\sim 2^{\frac{j_1-k_3}{2}}$, and either aligned or anti-aligned. 
If $k_3>0$, we use the same bounds as in the preceding case, which gives
\begin{align*}
&2^{-\frac{j_1}{2}}\big\|P_{k_1}Q_{j_1}\mathcal{N}\big(Q_{<j_2}^{\pm}\phi_{k_2}^{(2)}, Q_{<j_3}^{\pm}\phi_{k_3}^{(3)})\big\|_{L_{t,x}^2}\\
&\lesssim 2^{-\frac{j_1}{2}} 2^{\frac{j_1-k_3}{2}} \sum_{\kappa_2\sim \pm\kappa_3} \big\|Q_{<j_2}^{\pm}\nabla_x\phi_{k_2,\kappa_2}^{(2)}\big\|_{L_t^2L_x^\infty}\big\|Q_{<j_3}^{\pm}\nabla_x\phi_{k_3,\kappa_3}^{(3)}\|_{L_t^\infty L_x^2}\\
&\lesssim 2^{k_2} 2^{-\frac{k_3}{2}} \biggl( \sum_{\kappa_2}\big\|Q_{<j_2}^{\pm}\phi_{k_2,\kappa_2}^{(2)}\big\|_{L_t^2L_x^\infty}^2\biggr)^{\frac12} \biggl(\sum_{\kappa_3}\big\|Q_{<j_3}^{\pm}\nabla_x\phi_{k_3,\kappa_3}^{(3)}\big\|_{L_t^\infty L_x^2}^2\biggr)^{\frac12}.
\end{align*}
Combining with the bounds from the preceding case, we infer the bound 
\begin{align*}
&2^{-\frac{j_1}{2}}\big\|P_{k_1}Q_{j_1}\mathcal{N}\big(Q_{<j_2}^{\pm}\phi_{k_2}^{(2)}, Q_{<j_3}^{\pm}\phi_{k_3}^{(3)})\big\|_{L_{t,x}^2}\\
&\lesssim 2^{k_2} 2^{-\frac{k_3}{2}} 2^{-\delta|j_1-k_2|} 2^{-(\frac12-40\sigma)k_2}\big\|\phi_{k_2,\kappa_2}^{(2)}\big\|_{R_{k_2}} 2^{\delta_* k_3}\big\|\phi_{k_3,\kappa_3}^{(3)}\big\|_{S_{k_3}^1+R_{k_3}},
\end{align*}
which is (more than) the required bound. In case $k_3<0$, we use the same bound provided $|k_3|<\sigma k_2$, while we place the high frequency term $Q_{<j_2}^{\pm}\nabla_x\phi_{k_2,\kappa_2}^{(2)}$ into $L_t^\infty L_x^2$ and the low frequency term $Q_{<j_3}^{\pm}\nabla_x\phi_{k_3,\kappa_3}^{(3)}$ into $L_t^2 L_x^\infty$, provided $|k_3|\geq \sigma k_2$. The low frequency gain neutralises the loss of $2^{\delta_*k_2}$ coming from the high frequency term. 

If $j_1>k_3+C$, one argues similarly but without the angular localizations. 

\medskip 

\noindent {\it{Low-High interactions $k_3 = k_1+O(1)\geq k_2+O(1)$}}:  This can be handled by using identical estimates to the preceding case, changing the roles of $\phi_{k_2}^{(2)}, \phi_{k_3}^{(3)}$ if necessary. 
\end{proof}

\begin{lemma} \label{lem:KST12.4addendum} 
The following estimate holds uniformly in all indices:
\begin{align*}
&\big\|Q_{<j_1-C}P_{k_1}\mathcal{N}\big(Q_{j_1}\phi_{k_2}^{(2)},\,Q_{<j_1+C}\phi_{k_3}^{(3)}\big)\big\|_{N_{k_1}} \\
&\quad \lesssim 2^{k_1} 2^{-\delta|j_1 - k_2|}  2^{\delta(k_1 - \max\{k_2,k_3\})} 2^{-\delta|k_2 - k_3|}\big\|\phi_{k_2}^{(2)}\big\|_{S_{k_2}^1+R_{k_2}} \big\|\phi_{k_3}^{(3)}\big\|_{S_{k_3}^1+R_{k_3}}. 
\end{align*}
\end{lemma}
\begin{proof} We consider the case when at least one of $\phi_{k_2}^{(2)}, \phi_{k_3}^{(3)}$ is in the space $R_{k_j}, j = 2,3$. To begin with, assume $\phi_{k_2}^{(2)}\in R_{k_2}$. 

\medskip

\noindent {\it{Low-High interactions}} $k_2\leq k_3,\,k_1 = k_3+O(1)$: Of course we may assume $k_3>0$ since else the estimate is covered by those in~\cite{KST}. If $j_1<k_2+O(1)$, we may localize the Fourier supports of $\phi_{k_2}, \phi_{k_3}$ to caps $\kappa_{2,3}$ of diameter $\sim 2^{\frac{j_1-k_2}{2}}$ and aligned or anti-aligned. Then we get 
\begin{align*}
&\biggl( \sum_{\kappa_2}\big\|\nabla_x P_{\kappa_2}Q_{j_1}\phi_{k_2}^{(2)}\big\|_{L_{t,x}^2}^2\biggr)^{\frac12}\lesssim 2^{-\frac{j_1}{2}} 2^{\delta_*k_2} \big\|\phi_{k_2}^{(2)}\big\|_{R_{k_2}},\\
&\biggl(\sum_{\kappa_2}\big\|\nabla_x P_{\kappa_2}Q_{j_1}\phi_{k_2}^{(2)}\big\|_{L_t^2 L_x^6}^2\biggr)^{\frac12}\lesssim 2^{k_2} 2^{-(\frac12 - 20\sigma)k_2}\big\|\phi_{k_2}^{(2)}\big\|_{R_{k_2}}. 
\end{align*}
Interpolating gives 
\begin{align*}
\biggl( \sum_{\kappa_2} \big\|\nabla_x P_{\kappa_2}Q_{j_1}\phi_{k_2}^{(2)}\big\|_{L_t^2 L_x^3}^2 \biggr)^{\frac12}\lesssim 2^{\frac12(\frac12+20\sigma)k_2}  2^{-\frac{j_1}{4}} 2^{\frac{\delta_*}{2}k_2} \big\|\phi_{k_2}^{(2)}\big\|_{R_{k_2}}.
\end{align*}
Furthermore, we have 
\begin{align*}
 \biggl(\sum_{\kappa_3}\big\|\nabla_x P_{\kappa_3}Q_{j_1}\phi_{k_3}^{(3)}\big\|_{L_t^2 L_x^6}^2\biggr)^{\frac12}\lesssim 2^{\frac{5}{6}k_3}\big\|\phi_{k_3}^{(3)}\big\|_{S_{k_3}^1+R_{k_3}}. 
\end{align*}
Since we gain $2^{\frac{j_1-k_2}{2}}$ from the null-structure, we infer the bound 
\begin{align*}
&\big\|Q_{<j_1-C}P_{k_1}\mathcal{N}\big(Q_{j_1}\phi_{k_2}^{(2)},\,Q_{<j_1+C}\phi_{k_3}^{(3)}\big)\big\|_{N_{k_1}}\\
&\leq \big\|Q_{<j_1-C}P_{k_1}\mathcal{N}\big(Q_{j_1}\phi_{k_2}^{(2)},\,Q_{<j_1+C}\phi_{k_3}^{(3)}\big)\big\|_{L_t^1 L_x^2}\\
&\lesssim 2^{\frac{j_1 - k_2}{2}}\biggl( \sum_{\kappa_2} \big\|\nabla_x P_{\kappa_2}Q_{j_1}\phi_{k_2}^{(2)}\big\|_{L_t^2 L_x^3}^2 \biggr)^{\frac12} \biggl( \sum_{\kappa_3}\big\|\nabla_x P_{\kappa_3}Q_{j_1}\phi_{k_3}^{(3)}\big\|_{L_t^2 L_x^6}^2\biggr)^{\frac12}.
\end{align*}
Since $k_1 = k_3+O(1)$, the above bounds allow us to bound the preceding by 
\begin{align*}
2^{k_1} 2^{\frac{j_1 - k_2}{4}} 2^{(10\sigma + \frac{\delta_*}{2}) k_2 -\frac{k_3}{6}} \big\|\phi_{k_2}^{(2)}\big\|_{R_{k_2}} \big\|\phi_{k_3}^{(3)}\big\|_{S_{k_3}^1+R_{k_3}}, 
\end{align*}
which is as desired if $\sigma, \delta_*$ are sufficiently small. If $j_1\geq k_2$, we can proceed similarly without the extra factor $2^{\frac{j_1-k_2}{2}}$ from the angular gain before. 

\noindent {\it{High-Low interactions}} $k_2\geq k_3,\,k_1 = k_2+O(1)$: Here if $j_1\leq k_3+O(1)$ we can localize the two factors $\phi_{k_j}^{(j)}$ to caps of diameter $\sim 2^{\frac{j_1-k_3}{2}}$. Then similarly to the preceding, we bound 
\begin{align*}
&\big\|Q_{<j_1-C}P_{k_1}\mathcal{N}\big(Q_{j_1}\phi_{k_2}^{(2)},\,Q_{<j_1+C}\phi_{k_3}^{(3)}\big)\big\|_{N_{k_1}}\\
&\leq \big\|Q_{<j_1-C}P_{k_1}\mathcal{N}\big(Q_{j_1}\phi_{k_2}^{(2)},\,Q_{<j_1+C}\phi_{k_3}^{(3)}\big)\big\|_{L_t^1 L_x^2}\\
&\lesssim 2^{\frac{j_1-k_3}{2}} \biggl( \sum_{\kappa_2}\big\|\nabla_x P_{\kappa_2}Q_{j_1}\phi_{k_2}^{(2)}\big\|_{L_t^2 L_x^3}^2\biggr)^{\frac12} \biggl(\sum_{\kappa_3}\big\|\nabla_x P_{\kappa_3}Q_{j_1}\phi_{k_3}^{(3)}\big\|_{L_t^2 L_x^6}^2 \biggr)^{\frac12}\\
&\lesssim  2^{\frac{j_1-k_3}{2}} 2^{\frac12(\frac12+20\sigma)k_2}  2^{-\frac{j_1}{4}} 2^{\frac{\delta_*}{2}k_2} \big\|\phi_{k_2}^{(2)}\big\|_{R_{k_2}} 2^{\frac{5}{6}k_3} \big\|\phi_{k_3}^{(3)}\big\|_{S_{k_3}^1+R_{k_3}}.
\end{align*}
The preceding can be rearranged as 
\[
2^{\frac{j_1 - k_2}{4}} 2^{(10\sigma + \delta_* - \frac16)k_2} 2^{\frac{k_3 - k_2}{3}} \big\|\phi_{k_2}^{(2)}\big\|_{R_{k_2}} \big\|\phi_{k_3}^{(3)}\big\|_{S_{k_3}^1+R_{k_3}},
\]
again acceptable if $\sigma, \delta_*$ are sufficiently small. In case $j_1\geq k_3$ we argue similarly without the angular gain. 

\medskip 

\noindent {\it{High-High interactions}} $k_2 = k_3+O(1)\geq k_1+O(1)$: Here we can localize the factors $\phi_{k_j}^{(j)}$ to caps $\kappa_j$ of radius $2^{\frac{j+k_1}{2}-k_2}$, either aligned or anti-aligned. Then we estimate 
\begin{align*}
&\big\|Q_{<j_1-C}P_{k_1}\mathcal{N}\big(Q_{j_1}\phi_{k_2}^{(2)},\,Q_{<j_1+C}\phi_{k_3}^{(3)}\big)\big\|_{N_{k_1}}\\
&\leq \big\|Q_{<j_1-C}P_{k_1}\mathcal{N}\big(Q_{j_1}\phi_{k_2}^{(2)},\,Q_{<j_1+C}\phi_{k_3}^{(3)}\big)\big\|_{L_t^1 L_x^2}\\
&\lesssim 2^{\frac{k_1}{3}} 2^{\frac{j_1+k_1}{2} - k_2}\sum_{\kappa_2\sim\pm \kappa_3}\big\|\nabla_x P_{\kappa_2}Q_{j_1}\phi_{k_2}^{(2)}\big\|_{L_t^2 L_x^{\frac{12}{5}}} \big\|\nabla_x P_{\kappa_3}Q_{<j_1+C}\phi_{k_3}^{(3)}\big\|_{L_t^2 L_x^{6}}
\end{align*}
by Bernstein's inequality as well as the gain from the angular alignment and the null-structure. Using interpolation we get 
\begin{align*}
\big\|\nabla_xP_{\kappa_2}Q_{j_1}\phi_{k_2}^{(2)}\big\|_{L_t^2 L_x^{\frac{12}{5}}} \leq \big\|\nabla_x P_{\kappa_2}Q_{j_1}\phi_{k_2}^{(2)}\big\|_{L_{t,x}^2}^{\frac34} \big\|\nabla_x P_{\kappa_2}Q_{j_1}\phi_{k_2}^{(2)}\big\|_{L_t^2L_x^6}^{\frac14},
\end{align*}
and square summing over $\kappa_2$, we get 
\begin{align*}
\biggl( \sum_{\kappa_2}\big\|\nabla_x P_{\kappa_2}Q_{j_1}\phi_{k_2}^{(2)}\big\|_{L_t^2 L_x^{\frac{12}{5}}}^2 \biggr)^{\frac12}\lesssim \big(2^{-\frac{j_1}{2}} 2^{\delta_* k_2}\big)^{\frac34} \big(2^{(\frac12 + 20\sigma)k_2}\big)^{\frac14} \big\|\phi_{k_2}^{(2)}\big\|_{R_{k_2}},
\end{align*}
while we have directly from the definition of $R_{k_3}$ that 
\begin{align*}
\biggl(\sum_{\kappa_3} \big\|\nabla_x P_{\kappa_3}Q_{<j_1+C}\phi_{k_3}^{(3)}\big\|_{L_t^2 L_x^{6}}^2 \biggr)^{\frac12}\lesssim 2^{k_3} 2^{-\frac{k_3}{6}} \big\|\phi_{k_3}^{(3)}\big\|_{S_{k_3}^1+R_{k_3}}. 
\end{align*}
Combining these estimates and also using Cauchy-Schwarz to reduce to square-summation over the caps, we infer the desired bound by also observing that necessarily $j_1\leq k_1+O(1)$: 
\begin{align*}
&\big\|Q_{<j_1-C}P_{k_1}\mathcal{N}\big(Q_{j_1}\phi_{k_2}^{(2)},\,Q_{<j_1+C}\phi_{k_3}^{(3)}\big)\big\|_{N_{k_1}}\\
&\lesssim 2^{\frac{k_1}{3}} 2^{\frac{j_1+k_1}{2} - k_2} 2^{k_2} 2^{-\frac{3j_1}{8}} 2^{(\frac18 + \frac34 \delta_* + 5\sigma)k_2} 2^{-\frac{k_2}{6}} \big\|\phi_{k_2}^{(2)}\big\|_{R_{k_2}} \big\|\phi_{k_3}^{(3)}\big\|_{S_{k_3}^1+R_{k_3}}\\
& = 2^{\frac{j_1}{8}} 2^{\frac{5k_1}{6}} 2^{(-\frac{1}{24} + \frac34 \delta_* + 5\sigma)k_2} \big\|\phi_{k_2}^{(2)}\big\|_{R_{k_2}} \big\|\phi_{k_3}^{(3)}\big\|_{S_{k_3}^1+R_{k_3}},\\
\end{align*}
which is easily seen to be of the desired form if $\delta_*, \sigma$ are sufficiently small, recalling that $j_1\leq k_1+O(1)$. 

\medskip

To conclude the proof, we also need to deal with the case when $\phi_{k_2}^{(2)}$ belongs to $S^1_{k_2}$ but $\phi_{k_3}^{(3)}$ is in $R_{k_3}$. This case is much easier though, because then it suffices to place $Q_{j_1}\phi_{k_2}^{(2)}$ in $L_{t,x}^2$ and exploit the redeeming bounds for $\phi_{k_3}^{(3)}$. We omit the details. 
\end{proof}

Finally, we present the following generalizations of the multilinear estimates (134) and (135) in~\cite{KST}.
\begin{proposition}
 We have that
 \begin{align}
  &\bigl\| (1-\calH_{k_1}) P_{k_1} \calN( \phi_{k_2}^{(2)}, \phi_{k_3}^{(3)} ) \bigr\|_{\Box Z} \nonumber \\
  &\quad \lesssim 2^{k_1} 2^{\delta (k_1 - \max\{k_2, k_3\})} 2^{-\delta |k_2-k_3|} \|\phi_{k_2}^{(2)}\|_{S^1_{k_2} + R_{k_2}} \|\phi_{k_3}^{(3)}\|_{S^1_{k_3} + R_{k_3}}, \label{equ:generalization_134_KST}  \\
  &\bigl\| \calH_{k_1} \calN( \phi_{k_2}^{(2)}, \phi_{k_3}^{(3)} ) \|_{\Box Z} \nonumber \\
  &\quad \lesssim 2^{k_1} 2^{\delta (k_1 - \max\{ k_2, k_3 \})} 2^{-\delta |k_2-k_3|} \|\phi_{k_2}^{(2)}\|_{S^1_{k_2} + R_{k_2}} \|\phi_{k_3}^{(3)}\|_{S^1_{k_3} + R_{k_3}}, \quad k_1 > \max\{k_2, k_3\} - C. \label{equ:generalization_135_KST}
 \end{align}
\end{proposition}
\begin{proof} 
We start with the first estimate~\eqref{equ:generalization_134_KST}. To this end we observe the identity
\begin{align*}
(1-\mathcal{H}_{k_1})\mathcal{N}\big(\phi_{k_2}^{(2)}, \phi_{k_3}^{(3)}\big) &= \sum_{j\geq k+C}Q_j\mathcal{N}\big(Q_{<j-C}\phi_{k_2}^{(2)}, Q_{<j-C}\phi_{k_3}^{(3)}\big)\\
&\quad + \sum_jQ_{<j+O(1)}\mathcal{N}\big(Q_{j}\phi_{k_2}^{(2)}, Q_{<j+O(1)}\phi_{k_3}^{(3)}\big)\\
&\quad +  \sum_jQ_{<j+O(1)}\mathcal{N}\big(Q_{<j+O(1)}\phi_{k_2}^{(2)}, Q_{j}\phi_{k_3}^{(3)}\big).
\end{align*}
The first term on the right-hand side does not contribute to the norm $\|\cdot\|_{\Box Z}$ due to the definition. Consider then the most delicate case where $k_2 =k_3+O(1)\gg k_1$. We only need to consider the case where at least one factor is in the space $R_{k_j}$. By symmetry, it suffices to bound the term
\[
\sum_j P_{k_1}Q_{<j+O(1)}\mathcal{N}\big(Q_{j}\phi_{k_2}^{(2)}, Q_{<j+O(1)}\phi_{k_3}^{(3)}\big)
\]
We may assume that $\phi_{k_2}^{(2)}\in R_{k_2}$. Here if $r>k_1+2l$ and $l>k_1-k_2$, we use the estimate 
\begin{align*}
&2^{\frac{l}{2}} \biggl( \sum_{\kappa\in K_l}\big\|P_{k_1,\kappa}Q_{k_1+2l+O(1)}\mathcal{N}\big(Q_{r}\phi_{k_2}^{(2)}, Q_{<r+O(1)}\phi_{k_3}^{(3)} \bigr) \big\|_{\Box L_t^1L_x^\infty}^2 \biggr)^{\frac12}\\
&= 2^{\frac{l}{2}} \biggl( \sum_{\kappa\in K_l}\sum_{\substack{C_{2}\sim C_3\in C_{k_1}(l)\\\kappa(C_2)\in 2\kappa}}\big\|P_{k_1,\kappa}Q_{k_1+2l+O(1)}\mathcal{N}\big(Q_{r}\phi_{k_2,C_2}^{(2)}, Q_{<r+O(1)}\phi_{k_3, C_3}^{(3)}\big)\big\|_{\Box L_t^1L_x^\infty}^2 \biggr)^{\frac12} \\
&\lesssim 2^{k_1+l-k_2}2^{\frac23 k_1} \biggl( \sum_{C_2\in  C_{k_1}(l)}\big\|Q_{r}\nabla_x\phi_{k_2,C_2}^{(2)}\big\|_{L_t^2 L_x^{2}}^2 \biggr)^{\frac12} \biggl(\sum_{C_3\in  C_{k_1}(l)}\big\|Q_{<r+O(1)}\nabla_x\phi_{k_3,C_3}^{(3)}\big\|_{L_t^2 L_x^{6}}^2 \biggr)^{\frac12},
\end{align*}
where we have used Bernstein's inequality to pass from $L_x^{\frac32}$ to $L_x^2$ and we used the fact that 
\[
2^{\frac{l}{2}}P_{k_1,\kappa}Q_{k_1+2l+O(1)}L_t^1 L_x^2\subset \Box L_t^1 L_x^\infty.
\]
Then use the estimate 
\begin{align*}
\biggl( \sum_{C_2\in  C_{k_1}(l)}\big\|Q_{r}\nabla_x\phi_{k_2,C_2}^{(2)}\big\|_{L_t^2 L_x^{2}}^2 \biggr)^{\frac12}\lesssim 2^{\frac{k_1+2l-r}{2}} 2^{-\frac{k_1+2l}{2}} 2^{\delta_* k_2}\big\|\phi_{k_2}^{(2)}\big\|_{R_{k_2}}, 
\end{align*}
while we also have the improved Strichartz type estimate 
\begin{align*}
\biggl( \sum_{C_3\in  C_{k_1}(l)}\big\|Q_{<r+O(1)}\nabla_x\phi_{k_3,C_3}^{(3)}\big\|_{L_t^2 L_x^{6}}^2 \biggr)^{\frac12}\lesssim 2^{k_3} 2^{-\frac{k_3}{6}} 2^{\frac{k_1-k_3}{3}}\big\|\phi_{k_3}^{(3)}\big\|_{S_{k_3}^1}. 
\end{align*}
Combining the preceding estimates we infer the first bound we need 
\begin{align*}
&2^{\frac{l}{2}} \biggl(\sum_{\kappa\in K_l}\big\|P_{k_1,\kappa}Q_{k_1+2l+O(1)}\mathcal{N}\big(Q_{r}\phi_{k_2}^{(2)}, Q_{<r+O(1)}\phi_{k_3}^{(3)}\big)\big\|_{\Box L_t^1L_x^\infty}^2\biggr)^{\frac12}\\
&\lesssim  2^{k_1+l-k_2}2^{\frac23 k_1} 2^{\frac{k_1+2l-r}{2}} 2^{-\frac{k_1+2l}{2}} 2^{\delta_* k_2}2^{k_3} 2^{-\frac{k_3}{6}} 2^{\frac{k_1-k_3}{3}}\big\|\phi_{k_2}^{(2)}\big\|_{R_{k_2}}\big\|\phi_{k_3}^{(3)}\big\|_{S_{k_3}^1}.
\end{align*}
Using the assumption $k_2 = k_3+O(1)$, the preceding simplifies to 
\[
\lesssim 2^{\frac{3}{2}k_1 +(\delta_* - \frac12)k_2}  2^{\frac{k_1+2l-r}{2}} \big\|\phi_{k_2}^{(2)}\big\|_{R_{k_2}}\big\|\phi_{k_3}^{(3)}\big\|_{S_{k_3}^1}.
\]
This is good in terms of the decay in $k_2$ but bad since overall we leak $\delta_*$ in terms of the frequencies; this is as expected since we have not used the redeeming features of $\phi_{k_2}^{(2)}$, which we do next.  Using interpolation between $L_t^2L_x^\infty$ and $L_{t,x}^2$, we obtain the bound 
\begin{align*}
\biggl(\sum_{C_2\in  C_{k_1}(l)}\big\|Q_{r}\nabla_x\phi_{k_2,C_2}^{(2)}\big\|_{L_t^2 L_x^{3}}^2\biggr)^{\frac12}\lesssim \big(2^{k_2} 2^{-(\frac12-\delta_*)k_2}\big)^{\frac13} \big(2^{-\frac{k_1 + 2l}{2}+\delta_*k_2}\big)^{\frac23}\big\|\phi_{k_2}^{(2)}\big\|_{R_{k_2}}.
\end{align*}
We then infer the second bound 
\begin{align*}
&2^{\frac{l}{2}} \biggl(\sum_{\kappa\in K_l}\big\|P_{k_1,\kappa}Q_{k_1+2l+O(1)}\mathcal{N}\big(Q_{r}\phi_{k_2}^{(2)}, Q_{<r+O(1)}\phi_{k_3}^{(3)}\big)\big\|_{\Box L_t^1L_x^\infty}^2 \biggr)^{\frac12} \\
&\lesssim 2^{k_1+l-k_2}\biggl(\sum_{C_2\in  C_{k_1}(l)}\big\|Q_{r}\nabla_x\phi_{k_2,C_2}^{(2)}\big\|_{L_t^2 L_x^{3}}^2\biggr)^{\frac12} \biggl(\sum_{C_3\in  C_{k_1}(l)}\big\|Q_{<r+O(1)}\nabla_x\phi_{k_3,C_3}^{(3)}\big\|_{L_t^2 L_x^{6}}^2\biggr)^{\frac12}\\
&\lesssim  2^{k_1+l-k_2} 2^{\frac{k_2}{6}} 2^{-\frac{k_1+2l}{3}} 2^{k_2} 2^{-\frac{k_2}{2}} 2^{\frac{k_1}{3}} 2^{\delta_*k_2} \big\|\phi_{k_2}^{(2)}\big\|_{R_{k_2}}
\big\|\phi_{k_3}^{(3)}\big\|_{S_{k_3}^1}.
\end{align*}
This simplifies to 
\begin{align*}
\lesssim  2^{k_1} 2^{\frac{l}{3}} 2^{(\delta_* - \frac13)k_2} \big\|\phi_{k_2}^{(2)}\big\|_{R_{k_2}}
\big\|\phi_{k_3}^{(3)}\big\|_{S_{k_3}^1}.
\end{align*}
Interpolating between this bound and the preceding one results beyond the factor $2^{k_1}$ in exponential gains in $l, k_1+2l-r$, as well as $-k_2$, which is more than what we need. 

Consider next the second estimate~\eqref{equ:generalization_135_KST}. For symmetry reasons, we may assume that $k_1 = k_2+O(1)$. We need to bound 
\begin{align*}
\biggl\| \sum_{j\leq k+1+O(1)}P_{k_1}Q_j\mathcal{N}\big(Q_{<j-C}\phi_{k_2}^{(2)}, Q_{<j-C}\phi_{k_3}^{(3)}\big) \biggr\|_{\Box Z}. 
\end{align*}
We can further restrict the summation to $j\leq k_3+O(1)$, and we can localize $\phi_{k_j}^{(j)}$, $j = 2,3$, to angular caps $\kappa_j$ of size $\sim 2^{\frac{j-k_3}{2}}$ and either aligned or anti-aligned. The whole expression then also has Fourier support on $C\kappa_2$, and square summation over caps is handled by using the square-summation over caps inherent in the definitions of the norms $\|\cdot\|_{S^1_{k_2}}$ and $\|\cdot\|_{R_{k_2}}$. Then taking advantage of Bernstein's inequality, we have the bound (for $\kappa$ a cap of radius $\sim 2^{\frac{j-k_1}{2}}$)
\begin{align*}
&\big\|P_{k_1,\kappa}Q_j\mathcal{N}\big(Q_{<j-C}P_{k_2,\kappa_2}\phi_{k_2}^{(2)}, Q_{<j-C}P_{k_3,\kappa_3}\phi_{k_3}^{(3)}\big)\big\|_{\Box Z}\\
&\lesssim 2^{\frac{j-k_1}{4}} 2^{-k_1-j} 2^{\frac{j-k_3}{2}} 2^{\frac23 k_1}\big\|Q_{<j-C}P_{k_2,\kappa_2}\phi_{k_2}^{(2)}\big\|_{L_t^2 L_x^6} \big\|Q_{<j-C}P_{k_3,\kappa_3}\phi_{k_3}^{(3)}\big\|_{L_t^2 L_x^\infty}.
\end{align*}
Square-summing over the caps results in the bound 
\begin{align*}
\lesssim 2^{k_1} 2^{\frac{j-k_1}{4}} 2^{-\frac{k_1}{2}+\frac{k_3}{2}} \big\|\phi_{k_2}^{(2)}\big\|_{S_{k_2}^1+R_{k_2}} \big\|\phi_{k_3}^{(3)}\big\|_{S_{k_3}^1+R_{k_3}},
\end{align*}
which can then be summed over $j\leq k_3+O(1)$ to result in the desired bound. 
\end{proof}

\subsection{Core quadrilinear null form bounds}

Here we present the generalized versions of the key quadrilinear null form bounds (136)--(138) in~\cite{KST}.
\begin{proposition}
 The following quadrilinear form bounds hold under the condition $k < k_i - C$:
 \begin{align}
  &\big|\langle \Box^{-1}\mathcal{H}_k\big(\phi_{k_1}^{(1)}\partial_{\alpha}\phi_{k_2}^{(2)}\big), \mathcal{H}_k\big(\partial^{\alpha}\phi_{k_3}^{(3)}\psi_{k_4}\big)\rangle\big| \nonumber \\
  &\quad \lesssim 2^{\delta(k-\min\{k_i\})} \big\|\phi_{k_1}^{(1)}\big\|_{S^1_{k_1} + R_{k_1}} \big\|\phi_{k_2}^{(2)}\big\|_{S^1_{k_2} + R_{k_2}} \big\|\phi_{k_3}^{(3)}\big\|_{S_{k_3}^1} \big\|\psi_{k_4}\big\|_{N^*}, \label{equ:generalization_136_KST} \\
  &\big|\langle (\Box \Delta)^{-1}\mathcal{H}_k\partial^{\alpha} \big(\phi_{k_1}^{(1)}\partial_{\alpha}\phi_{k_2}^{(2)}\big), \mathcal{H}_k\partial_t\big(\partial_t\phi_{k_3}^{(3)}\psi_{k_4}\big)\rangle\big| \nonumber \\
  &\quad \lesssim 2^{\delta(k-\min\{k_i\})} \big\|\phi_{k_1}^{(1)}\big\|_{S^1_{k_1} + R_{k_1}}\big\|\phi_{k_2}^{(2)} \big\|_{S^1_{k_2}+R_{k_2}}\big\|\phi_{k_3}^{(3)}\big\|_{S_{k_3}^1} \big\|\psi_{k_4}\big\|_{N^*}, \label{equ:generalization_137_KST} \\
  &\big|\langle (\Box \Delta)^{-1}\mathcal{H}_k\nabla_x\big(\phi_{k_1}^{(1)}\nabla_x\phi_{k_2}^{(2)}\big), \mathcal{H}_k\partial_\alpha\big(\partial^{\alpha}\phi_{k_3}^{(3)}\psi_{k_4}\big)\rangle\big| \nonumber \\ &\quad \lesssim 2^{\delta(k-\min\{k_i\})} \big\|\phi_{k_1}^{(1)}\big\|_{S^1_{k_1} + R_{k_1}} \big\|\phi_{k_2}^{(2)}\big\|_{S^1_{k_2}+R_{k_2}} \big\|\phi_{k_3}^{(3)}\big\|_{S^1_{k_3}} \big\|\psi_{k_4}\big\|_{N^*}. \label{equ:generalization_138_KST}
\end{align}
\end{proposition}

\begin{proof}
 We present the details for the derivation of the first estimate~\eqref{equ:generalization_136_KST}. The remaining estimates~\eqref{equ:generalization_137_KST}--\eqref{equ:generalization_138_KST} can be handled analogously.

 We microlocalize as in (148) in~\cite{KST}. In particular, the modulation of $\mathcal{H}_k\big(\ldots\big)$ is restricted to $\sim 2^{j}$ and we set $2^l: = 2^{\frac{j-k}{2}}$. We may assume that at least one of the inputs $\phi_{k_{1}}^{(1)}$ or $\phi_{k_{2}}^{(2)}$ are in the space $R_{k_j}$, $j=1$ or $j=2$, and in particular that $k_1>0$. Following the argument in the proof of (136) in~\cite{KST}, we consider various situations depending on the angle between $\phi_{k_2}^{(2)}$ and $\phi_{k_3}^{(3)}$. 

\medskip

\noindent {\it{Case 1: $\angle (\phi^2,\phi^3) \,\text{mod}\,\pi \lesssim 2^l 2^{k-k_2}$.}} As in~\cite{KST} we obtain the bound 
\begin{align*}
&\lesssim 2^{k_3 - k_2}\sum_{\mathcal{C}_k(l)}\big\|P_{\mathcal{C}_k(l)}Q_{<j-C}\phi_{k_1}^{(1)}\big\|_{L_t^2 L_x^\infty} \big\|P_{-\mathcal{C}_k(l)}Q_{<j-C}\phi_{k_2}^{(2)}\big\|_{L_t^2 L_x^\infty} \times\\
&\qquad \qquad \qquad \times \sup_t\sum_{\mathcal{C}_{k'}(l)}\big\|P_{\mathcal{C}_{k'}(l)}Q_{<j-C}\phi_{k_3}^{(3)}\big\|_{L_x^2}\big\|P_{-\mathcal{C}_{k'}(l)}Q_{<j-C}\psi_{k_4}\big\|_{L_x^2}.
\end{align*}
By symmetry, we may assume that the first factor $\phi_{k_1}^{(1)}\in R_{k_1}$, while we use control over $\|\cdot\|_{S^1_{k_2}}$ for the second factor, potentially with a $2^{\delta_*k_2}$-loss. Then we can bound 
\begin{align*}
&2^{-k_2}\sum_{\mathcal{C}_k(l)}\big\|P_{\mathcal{C}_k(l)}Q_{<j-C}\phi_{k_1}^{(1)}\big\|_{L_t^2 L_x^\infty} \big\|P_{-\mathcal{C}_k(l)}Q_{<j-C}\phi_{k_2}^{(2)}\big\|_{L_t^2 L_x^\infty}\\
&\lesssim 2^{-k_2} 2^{-\frac{k_1}{2+}} 2^{\frac{l}{2}} 2^{\frac{k_2}{2}} 2^{\frac{k-k_2}{2}} 2^{\delta_*k_2}\big\|\phi_{k_1}^{(1)}\big\|_{R_{k_1}} \big\|\phi_{k_2}^{(2)}\big\|_{S^1_{k_2}+R_{k_2}}.
\end{align*}
Completing the estimate as in~\cite{KST}, we arrive at a bound that is indeed much better than what is required, due to additional exponential gains in $-k_1$. 

\medskip

\noindent {\it{Case 2: $\angle (\phi^2, \phi^3) \,\text{mod}\,\pi \lesssim 2^l 2^{k-k_3}$.}} Here we may assume $k_3<k_2$, in light of the previous case. 
This time we use the fixed-time bound 
\begin{align*}
&\lesssim 2^{k_2 - k_3} \sum_{\mathcal{C}_k(l)}\big\|P_{\mathcal{C}_k(l)}Q_{<j-C}\phi_{k_1}^{(1)}\big\|_{L_x^\infty} \big\|P_{-\mathcal{C}_k(l)}Q_{<j-C}\phi_{k_2}^{(2)}\big\|_{L_x^2} \times \\
&\qquad \qquad \qquad \qquad \times \sum_{\mathcal{C}_{k'}(l)}\big\|P_{\mathcal{C}_{k'}(l)}Q_{<j-C}\phi_{k_3}^{(3)}\big\|_{L_x^\infty}\big\|P_{-\mathcal{C}_{k'}(l)}Q_{<j-C}\psi_{k_4}\big\|_{ L_x^2}.
\end{align*}
Applying Cauchy-Schwarz in the second sum over rectangular boxes in order to reduce to $\big\|\psi_{k_4}\big\|_{L_x^2}$ and then integrating in time and using H\"older's inequality, we can estimate things as before by using the $L_t^2 L_x^\infty$ based norm for the factors $\phi_{k_1}^{(1)}, \phi_{k_3}^{(3)}$, and $L_t^\infty L_x^2$ for $\phi_{k_2}^{(2)}$(more precisely, we use square sums over pieces microlocalized to rectangular boxes). 
Note that if the other high-frequency factor $\phi_{k_2}^{(2)}$ is in $R_{k_2}$ and not the first one, we simply interchange the roles of these factors. Then the preceding expression can further be bounded by 
\begin{align*}
\lesssim 2^{-k_3} 2^{\frac{l}{2}} 2^{\frac{k-k_3}{3}} 2^{-\frac{k_3}{6}} 2^{\frac{k}{2}} 2^{\frac{k}{6}} 2^{\frac23 k} 2^{-\frac{k_1}{2+}} 2^{\delta_*k_2} \big\|\phi_{k_1}^{(1)}\big\|_{R_{k_1}} \big\|\phi_{k_2}^{(2)}\big\|_{S^1_{k_2}+R_{k_2}}\big\|\psi_{k_4}\big\|_{N_{k_4}^*}.
\end{align*}
This can again be summed over all relevant parameters to give (more than) the required bound. 

\medskip

\noindent {\it{Case 3: $2^l\gtrsim \angle (\phi^2, \phi^3) \, \text{mod}\,\pi \gg 2^l 2^{k-\min\{k_{2,3}\}}$.}}  Again we may assume that $\phi_{k_1}^{(1)}$ is in $R_{k_1}$, since both $\phi_{k_{1}}^{(1)}$ and $\phi_{k_2}^{(2)}$ form an angle $\gg 2^l 2^{k-\min\{k_{2,3}\}}\,\text{mod}\,\pi$ with $\phi_{k_3}^{(3)}$. Set $\angle (\phi^2,\phi^3) \sim 2^{l'}$. In analogy with~\cite{KST}, and specifically {Case 3} in the proof of estimate (148) there, we infer the bound 
\begin{align*}
\lesssim 2^{-2k} 2^{-2l} 2^{l'} 2^{\min\{k,0\}} 2^{-\frac{k_1}{2+}} 2^{k_2+k_3} I_{23}(l') \big\|\phi_{k_1}^{(1)}\big\|_{R_{k_1}} \big\|\psi_{k_4}\big\|_{N_{k_4}^*},
\end{align*}
 where we have 
 \[
 I_{23}(l')\lesssim 2^{\frac32 (k+l)} 2^{\delta_\ast k_2} \big\|\phi_{k_2}\big\|_{R_{k_2}} \big\|\phi_{k_3}\big\|_{S^1_{k_3}}.
 \]
 It is then straightforward to sum over $l'<l+C<O(1)$ to infer the desired bound. 
\end{proof}

We conclude with a generalized version of the multilinear estimate (141) in~\cite{KST}.
\begin{proposition}[Additional core product estimate]
 We have that
 \begin{equation} \label{equ:generalization_141_KST}
  \bigl\| (I - \calH_{k_1}) P_{k_1} ( \phi_{k_2}^{(2)} \partial_t \phi_{k_3}^{(3)} ) \bigr\|_{\Box^{\frac12} \Delta^{\frac12} Z} \lesssim 2^{\delta (k_1 - k_2)} \|\phi_{k_2}^{(2)}\|_{S^1_{k_2} + R_{k_2}} \|\phi_{k_3}^{(3)}\|_{S^1_{k_3} + R_{k_3}}, \quad k_1 \leq k_2-C.
 \end{equation}
\end{proposition}
\begin{proof} From the definition, we have 
\begin{align*}
 (I - \calH_{k_1}) P_{k_1} ( \phi_{k_2}^{(2)} \partial_t \phi_{k_3}^{(3)} ) &= \sum_{j\geq k_1+C} P_{k_1}Q_j (Q_{<j-C} \phi_{k_2}^{(2)} \partial_t Q_{<j-C}\phi_{k_3}^{(3)} )\\
 &\quad + \sum_{j<k_1+C} P_{k_1}Q_j (Q_{\geq j-C} \phi_{k_2}^{(2)} \partial_t Q_{<j-C}\phi_{k_3}^{(3)} )\\
 &\quad +  \sum_{j< k_1+C} P_{k_1}Q_j ( \phi_{k_2}^{(2)} \partial_t Q_{\geq j-C}\phi_{k_3}^{(3)} )\\
 &\equiv I + II + III. 
\end{align*}
Recall that we have 
\[
\big\|\phi\big\|_{Z_k}^2 = \sup_{l<C}2^l\big\|P^\kappa_lQ_{k+2l}\phi\big\|_{L_t^1 L_x^\infty}^2.
\]
In particular, the term $I$ does not contribute. We treat the term $II$, the remaining term $III$ being similar. 
Consider then the term $II$. We need to estimate (with $l = \frac{j-k}{2}$)
\begin{align*}
&2^{-\frac{j}{2}} 2^{-\frac{3k_1}{2}} 2^{\frac{j-k_1}{4}} \biggl( \sum_{\kappa}\big\|P^{\kappa}_l P_{k_1}Q_j (Q_{\geq j-C} \phi_{k_2}^{(2)} \partial_t Q_{<j-C}\phi_{k_3}^{(3)} )\big\|_{L_t^1 L_x^\infty}^2 \biggr)^{\frac12} \\
&\lesssim  2^{-\frac{j}{2}} 2^{-\frac{3k_1}{2}} 2^{\frac{j-k_1}{4}} \sum_{j_1\geq j-C}\sum_{\kappa_2\sim \pm \kappa_3} \biggl( \sum_{\kappa} \big\|P^\kappa_l P_{k_1}Q_j (Q_{j_1} P^{\kappa_2}\phi_{k_2}^{(2)} \partial_t Q_{<j-C}P^{\kappa_3}\phi_{k_3}^{(3)} )\big\|_{L_t^1 L_x^\infty}^2 \biggr)^{\frac12},
\end{align*}
where the caps $\kappa_{2,3}$ are of diameter $\sim 2^{\frac{j_1+k_1}{2}-k_2}$. Then from the proof of Lemma~\ref{lem:KST12.4addendum} recall the estimate 
\begin{align*}
2^{k_2} \biggl( \sum_{\kappa_2} \big\|Q_{j} P_{\kappa_2}\phi_{k_2}^{(2)}\big\|_{L_t^2 L_x^{\frac{12}{5}}}^2 \biggr)^{\frac12}\lesssim  \big(2^{-\frac{j}{2}} 2^{\delta_* k_2}\big)^{\frac34} \big(2^{(\frac12 + 20\sigma)k_2}\big)^{\frac14} \big\|\phi_{k_2}^{(2)}\big\|_{R_{k_2}},
\end{align*}
and furthermore that we have 
\begin{align*}
\biggl(\sum_{\kappa_3}\big\|\partial_t Q_{<j-C}P_{\kappa_3}\phi_{k_3}^{(3)}\big\|_{L_t^2 L_x^6}^2\biggr)^{\frac12}\lesssim 2^{\frac{5}{6}k_3} \big\|\phi_{k_3}^{(3)}\big\|_{S^1_{k_3}+R_{k_3}}. 
\end{align*}
Then use that 
\begin{align*}
\biggl(\sum_{\kappa}\big\|P^{\kappa}_l P_{k_1}Q_j f\big\|_{L_t^1 L_x^\infty}^2\biggr)^{\frac12}\lesssim 2^{\frac{ 3l}{2}} 2^{2k_1} \biggl(\sum_{\kappa}\big\|P^{\kappa}_l P_{k_1}Q_j f\big\|_{L_t^1 L_x^2}^2 \biggr)^{\frac12}\lesssim 2^{\frac{ 3l}{2}} 2^{2k_1} \big\|P_{k_1}Q_j f\big\|_{L_t^1 L_x^2},
\end{align*}
and apply the Cauchy-Schwarz inequality to $\sum_{\kappa_2\sim \pm \kappa_3}$, as well as Bernstein's inequality to pass from $L_t^1 L_x^{\frac{12}{7}}$ to $L_t^1 L_x^2$. It follows that 
\begin{align*}
&2^{-\frac{j}{2}} 2^{-\frac{3k_1}{2}} 2^{\frac{j-k_1}{4}} \biggl( \sum_{\kappa}\big\|P^{\kappa}_l P_{k_1}Q_j (Q_{\geq j-C} \phi_{k_2}^{(2)} \partial_t Q_{<j-C}\phi_{k_3}^{(3)} )\big\|_{L_t^1 L_x^\infty}^2 \biggr)^{\frac12}\\
&\lesssim 2^{-\frac{j}{2}} 2^{-\frac{3k_1}{2}} 2^{\frac{j-k_1}{4}} 2^{3\frac{j-k_1}{4}} 2^{\frac{7}{3}k_1} \big(2^{-\frac{j}{2}} 2^{\delta_* k_2}\big)^{\frac34} \big(2^{(\frac12 + 20\sigma)k_2}\big)^{\frac14} \big\|\phi_{k_2}^{(2)}\big\|_{R_{k_2}}  2^{-\frac{1}{6}k_3} \big\|\phi_{k_3}^{(3)}\big\|_{S^1_{k_3}+R_{k_3}}\\
&\simeq 2^{\frac{j-k_1}{8}} 2^{-\frac{k_1}{24}} 2^{(\delta_* + 20\sigma - \frac{1}{24})k_2}  \big\|\phi_{k_2}^{(2)}\big\|_{R_{k_2}}  \big\|\phi_{k_3}^{(3)}\big\|_{S^1_{k_3}+R_{k_3}},\\
\end{align*}
which is good provided that $\phi_{k_2}^{(2)}\in R_{k_2}$ and $k_1\geq 0$. 
If $k_1<0$, one places $Q_{\geq j-C} \phi_{k_2}^{(2)}$ into $L_{t,x}^2$ and $\partial_t Q_{<j-C}\phi_{k_3}^{(3)}$ into $L_t^2 L_x^6$, since the gain of $2^{-\frac{k_3}{6}}$ is then enough to neutralize the loss of $2^{\delta_* k_2}$. 
The case when $\phi_{k_2}^{(2)}\in S_{k_2}^1$ but $\phi_{k_3}^{(3)}\in R_{k_3}$ is simpler since one only needs to place $Q_{\geq j-C} \phi_{k_2}^{(2)}$ in $L_{t,x}^2$ while using the redeeming version of $L_t^2 L_x^6$ for $\partial_t Q_{<j-C}\phi_{k_3}^{(3)}$.  
\end{proof}

\bibliographystyle{amsplain}
\bibliography{references}

\end{document}